\documentclass[10pt,twoside]{amsart}
\usepackage{latexsym,amssymb,graphicx}
\usepackage{mathrsfs,stmaryrd,dsfont}
\usepackage[all]{xy}

\newdir{ >}{{}*!/-5pt/\dir{>}}

\def\smono[#1]{\ar@{->}[#1]|@{|}}
\newcommand{\xymono}{\ar@{ >->}}
\newcommand{\xyepi}{\ar@{->>}}
\newcommand{\xyhookrightarrow}{\ar@{^(->}}

\newtheorem{theorem}{Theorem}[section]
\newtheorem{lemma}[theorem]{Lemma}
\newtheorem{corollary}[theorem]{Corollary}
\newtheorem{proposition}[theorem]{Proposition}

\theoremstyle{definition}
\newtheorem{definition}[theorem]{Definition}
\newtheorem{example}[theorem]{Example}
\newtheorem{remark}[theorem]{Remark}
      
 \numberwithin{equation}{section}

\newcommand{\Z}{\mathbb{Z}}

\newcommand{\A}{\mathcal{A}}
\newcommand{\scA}{\mathscr{A}}
\newcommand{\B}{\mathcal{B}}
\newcommand{\scB}{\mathscr{B}}
\newcommand{\C}{\mathcal{C}}
\newcommand{\scC}{\mathscr{C}}
\newcommand{\D}{\mathcal{D}}
\newcommand{\E}{\mathscr{E}}
\newcommand{\F}{\mathscr{F}}

\renewcommand{\H}{\mathcal{H}}
\newcommand{\M}{\mathcal{M}}
\newcommand{\N}{\mathbb{N}}
\newcommand{\scP}{\mathscr{P}}
\renewcommand{\P}{\mathcal{P}}
\newcommand{\scS}{\mathscr{S}}
\newcommand{\X}{\mathcal{X}}

\newcommand{\ffi}{\varphi}
\newcommand{\eps}{\varepsilon}

\newcommand{\colim}{\operatorname{colim}}
\newcommand{\Spec}{\operatorname{Spec}}

\newcommand{\Sch}{\operatorname{Sch}}

\newcommand{\im}{\operatorname{Im}}
\newcommand{\Cat}{\operatorname{Cat}}

\newcommand{\Hom}{\operatorname{Hom}}
\newcommand{\T}{\mathcal{T}}

\newcommand{\Fun}{\operatorname{Fun}}
\newcommand{\Vect}{\operatorname{Vect}}

\newcommand{\Ch}{\operatorname{Ch}}
\newcommand{\can}{\operatorname{can}}
\newcommand{\Cone}{\operatorname{Cone}}
\newcommand{\ad}{\operatorname{ad}}
\newcommand{\Ar}{\operatorname{Ar}}
\newcommand{\RR}{\mathcal{R}}

\newcommand{\pt}{\operatorname{pt}}
\newcommand{\dgCatWD}{\operatorname{dgCatWD}}
\newcommand{\dgCat}{\operatorname{dgCat}}
\newcommand{\dgCatD}{\operatorname{dgCatD}}
\newcommand{\Sp}{\operatorname{Sp}}
\newcommand{\Top}{\operatorname{Top}}
\newcommand{\CatD}{\operatorname{CatD}}
\newcommand{\GW}{\mathbb{G}W}

\newcommand{\str}{\operatorname{str}}
\newcommand{\bH}{\mathbb{H}}

\newcommand{\Mod}{\operatorname{Mod}}
\newcommand{\dgMod}{\operatorname{dgMod}}
\newcommand{\dgFun}{\operatorname{dgFun}}
\newcommand{\quis}{\operatorname{quis}}
\newcommand{\rk}{\operatorname{rk}}

\newcommand{\K}{\mathcal{K}}
\newcommand{\TriD}{\operatorname{TriD}}
\newcommand{\pretr}{\operatorname{ptr}}
\newcommand{\Map}{\operatorname{Map}}
\newcommand{\SH}{\operatorname{SH}}
\newcommand{\Sing}{\operatorname{Sing}}
\newcommand{\cpt}{\operatorname{cpt}}
\newcommand{\Ab}{\operatorname{Ab}}
\newcommand{\sat}{\operatorname{sat}}

\newcommand{\rr}{\mathbb{R}}
\newcommand{\cc}{\mathbb{C}}
\newcommand{\hh}{\mathbb{H}}
\newcommand{\ff}{\mathbb{F}}

\newcommand{\U}{\mathcal{U}}
\newcommand{\V}{\mathcal{V}}
\newcommand{\W}{\mathcal{W}}

\renewcommand{\1}{\mathds{1}}
\newcommand{\Qed}{\hfil\penalty-5\hspace*{\fill}$\square$}

\newcommand{\poSet}{\operatorname{poSet}}
\newcommand{\poSetD}{\operatorname{poSetD}}

\newcommand{\bK}{\mathbb{K}}
\newcommand{\meps}{\varepsilon}
\newcommand{\dgCatW}{\operatorname{dgCatW}}

\newcommand{\Ev}{\operatorname{Ev}}
\newcommand{\BiSp}{\operatorname{BiSp}}
\newcommand{\Sym}{\operatorname{Sym}}
\newcommand{\ie}{{\it i.e.}}

\newcommand{\pp}{\mathbb{P}}
\newcommand{\Proj}{\operatorname{Proj}}
\newcommand{\sPerf}{\operatorname{sPerf}}
\newcommand{\coker}{\operatorname{coker}}

\newcommand{\bA}{\mathbb{A}}

\newcommand{\scU}{\mathscr{U}}

\newcommand{\Qcoh}{\operatorname{Qcoh}}
\newcommand{\Coh}{\operatorname{Coh}}
\newcommand{\bL}{\mathbb{L}}
\newcommand{\Q}{\mathbb{Q}}

\newcommand{\bk}{k}

\newcommand{\supp}{\operatorname{supp}}

\usepackage{color}
\newcommand{\red}[1]{#1}%{{\color{red}{#1}}}
\newcommand{\Ac}{\operatorname{Ac}}
\newcommand{\scT}{\mathscr{T}}
\newtheorem*{exampleNoNb}{Example}
\newtheorem*{remarkNoNb}{Remark}

\title[Hermitian $K$-theory and derived equivalences]{Hermitian $K$-theory, derived equivalences and Karoubi's fundamental theorem}
\author{Marco Schlichting}
\address{Marco Schlichting, Mathematics Institute,
Zeeman Building,
University of Warwick,
Coventry CV4 7AL, UK} 

\thanks{}

\email{m.schlichting@warwick.ac.uk}



\begin{document}
\bibliographystyle{alpha}

\begin{abstract}
Within the framework of dg categories with weak equivalences and duality that have uniquely $2$-divisible mapping complexes,
we show that higher Grothendieck-Witt groups (aka. hermitian $K$-groups) are invariant under derived equivalences and that Morita exact sequences induce long exact sequences of Grothendieck-Witt groups.
This implies an algebraic Bott sequence and a new proof and generalisation of Karoubi's Fundamental Theorem.
For the higher Grothendieck-Witt groups of vector bundles of (possibly singular) schemes $X$ with an ample family of line-bundles such that $\frac{1}{2}\in \Gamma(X,O_X)$, we obtain Mayer-Vietoris long exact sequences for Nisnevich coverings and blow-ups along regularly embedded centers, projective bundle formulas, and a Bass fundamental theorem.
For coherent Grothendieck-Witt groups, we obtain a localization theorem analogous to Quillen's $K'$-localization theorem.
\end{abstract}

\maketitle

\tableofcontents

\section*{Introduction}

By a result of Thomason \cite[Theorem 1.9.8]{TT}, algebraic K-theory is ``invariant under derived equivalences''.
This is a tremendously powerful fact.
Together with Waldhausen's Fibration Theorem \cite[Theorem 1.6.4]{wald:spaces}
it implies a Localization Theorem which - omitting hypothesis - says that a ``short exact sequence of triangulated categories'' induces a long exact sequence of algebraic $K$-groups.
This property of algebraic $K$-theory implies long exact Mayer-Vietoris sequences for Nisnevich coverings \cite{TT} and for blow-ups along regularly embedded centers \cite{Thomason:eclate}.
Both results are fundamental for the recent advances in our understanding of the algebraic $K$-theory of singular varieties \cite{CHSW}, \cite{CHW:Vorst}.

In this article, we investigate the problem of ``invariance under derived equivalences'' for the higher Grothendieck-Witt groups introduced by the author in \cite{myMV}.
These groups naturally occur in $\bA^1$-homotopy theory \cite{Morel:book}, \cite{Fasel:Asok} and are to oriented Chow groups what algebraic K-theory is to ordinary Chow groups.
They are the algebraic analog of the topological $KO$-groups, or, in the context of schemes with involution, they are the algebraic analog of Atiyah's $KR$-theory.

Within the framework of dg categories with weak equivalences and duality that have uniquely $2$-divisible mapping complexes,
we show in Theorems \ref{thm:Invariance} and \ref{thm:InvarianceKGW} that the higher Grothendieck-Witt groups are invariant under derived equivalences.
This is in some sense the best one can hope for since without an assumption such as ``uniquely $2$-divisible homomorphism groups'', there are examples of derived equivalences that do not induce isomorphisms of Grothendieck-Witt groups; see Proposition \ref{enum:prop:Counterex}.

Together with our Fibration Theorem \cite[Theorem 6]{myMV} this implies 
Theorems \ref{thm:LcnConn1} and \ref{thm:locnForKGW} which are the analogs for higher Grothendieck-Witt groups of the Thomason-Waldhausen Localization Theorem mentioned above.
We obtain several new results as application of our Localization Theorem: 
algebraic versions of the Bott sequence in topology
(Theorems \ref{thm:PeriodicityExTriangle} and \ref{thm:PeriodicityExTriangleForKGW}), a new and more general proof of Karoubi's Fundamental Theorem \cite{Karoubi:Annals}  (Theorem \ref{thm:KaroubiFundThm}), Nisnevich descent for (possibly singular) noetherian schemes with an ample family of line bundles (Theorem \ref{thm:NisnDescent}), 
a Mayer-Vietoris property for blow-ups along regularly embedded centers (Theorem \ref{thm:blowup}),
projective bundle formulas (Theorem \ref{thm:projLineBld} and Remark \ref{rmk:otherProjBdlFormulas}),
a Bass fundamental theorem for Grothendieck-Witt groups (Theorem \ref{thm:Bassfund}), and
an analog for Grothendieck-Witt groups of Quillen's $K'$-theory localization theorem based on coherent sheaves (Theorem \ref{thm:LocnCohGW}).

Special cases of our results have been obtained by other authors using different methods.
Notably, Karoubi proved his Fundamental Theorem for $\Z[1/2]$-algebras with involution in \cite{Karoubi:Annals}, and Hornbostel proved Nisnevich descent and a version of the Projective Line Bundle Theorem both in the case of regular noetherian separated $\Z[1/2]$-schemes in \cite{hornbostel:A1reps}. 
At the very least all geometric results regarding singular schemes are new.
Our interpretation of Karoubi's $U$ and $V$-theories as the odd shifted higher Grothendieck-Witt groups and our careful study of the multiplicative structure of higher Grothendieck-Witt groups in Sections \ref{sec:GWspectrum} and \ref{sec:KGWspectrum} should be of independent interest.

Results of this article have already been used in \cite{fasel:srinivas}, \cite{zibrowius}, \cite{PaninWalter}, in the solution of a conjecture of Suslin on the structure of stably free projective modules \cite{Fasel:Rao:Swan} and in the solution of a conjecture of Williams and the analog of the Quillen-Lichtenbaum conjecture for hermitian K-theory \cite{meBerrickKaroubiOestvar}.
\vspace{1ex}

We give an outline of the article.
Section \ref{sect:GWgrpsDGcats} doesn't contain anything essentially new.
We recall definitions and basic facts about dg categories and introduce the obvious notion of a dg category with weak equivalences and duality.
The main point here is our treatment of the pretriangulated hull $\A^{\pretr}$ of a dg category $\A$ (Definition \ref{dfn:pretr}) which makes it clear that $\A^{\pretr}$ is a dg category with weak equivalences and duality whenever $\A$ is, and the assignment $\A \mapsto \A^{\pretr}$ is symmetric monoidal in some sense; see Section \ref{subsec:GWsymmSeq}.
The Grothendieck-Witt group $GW_0(\scA)$ of a dg category with (weak equivalences and) duality is defined to be the Grothendieck-Witt group of the exact category with weak equivalences and duality in the sense of \cite[Definition 1]{myMV} of the category $Z^0\scA^{\pretr}$ of closed degree zero maps in the pretriangulated hull of $\scA$.
When $\scA$ is a dg algebra $A$ then $GW_0(A)$ is the Grothendieck-Witt group of finitely generated semi-free dg $A$-modules.
%In particular, if $A$ is a ring then $GW_0(A)$ is the Grothendieck-Witt group of finitely generated free $A$-modules.
When $\scA = \sPerf(X)$ is the dg category of bounded  complexes of vector bundles on a scheme $X$ with weak equivalences the quasi-isomorphisms, then 
$GW_0(\sPerf(X))$ is the Grothendieck-Witt group \red{$GW_0(X)$} of $X$ as introduced by Knebusch in \cite{Knebusch:L(X)}.

In Section \ref{sec:ConeCounterex}, we introduce the ``duality preserving'' cone functor, a version for dg categories with duality of the functor that sends a map of  complexes to its cone.
We use the cone functor to construct in Proposition \ref{enum:prop:Counterex}  examples of functors between dg categories with weak equivalences and duality which induce equivalences of derived categories but the induced map on Grothendieck-Witt groups is not an isomorphism.
Our examples use dg categories for which the mapping complexes are not uniquely $2$-divisible.

In Section \ref{sec:GWofTriD}
we slightly generalize the notion of a triangulated category with duality given in \cite{Balmer:TWGI} in that we don't require the shift of the triangulated category to be an isomorphism; it is only required to be an equivalence.
As in \cite{Balmer:TWGI} and \cite{walter:GWTrPreprint} we have a Witt and a Grothendieck-Witt group of such categories, but, contrary to {\it loc.cit.}, our presentation doesn't use shifted dualities.
For a dg category with weak equivalences and duality $\scA$, we verify in Lemma \ref{lem:TAisTriD} that the triangulated category $\T\scA$ of $\scA$, that is, the localization of the homotopy category $H^0\scA^{\pretr}$ of the pretriangulated hull of $\scA$ with respect to the weak equivalences, is indeed a triangulated category with duality.
We show in Proposition \ref{prop:GW0AisGW0TA} that the Grothendieck-Witt group of a dg category with weak equivalences and duality $\scA$ agrees with the Grothendieck-Witt group of its triangulated category $\T\scA$ provided its mapping complexes are uniquely $2$-divisible, a condition we denote by $\frac{1}{2}\in \scA$.
Without an assumption such as $\frac{1}{2}\in \scA$, this proposition cannot hold as shown by the examples in Section \ref{sec:ConeCounterex}.

In Section \ref{sec:multiSimplRdot}, we recall from \cite{myMV} the hermitian analog $\RR_{\bullet}$ (denoted by $S^e_{\bullet}$ in \cite{myMV}) of Waldhausen's $S_{\bullet}$-construction and introduce its iterated versions $\RR_{\bullet}^{(n)}$, $n\in \N$.
%denoted $S^e_{\bullet}$ in \cite{myMV}.
The main result here is Proposition \ref{prop:SpaceKaroubiProto1} which is the basis for the construction of the Grothendieck-Witt spectrum functor to be introduced in Section \ref{sec:GWspectrum}.
Together with Corollary \ref{cor:RFunIsSdot} it implies Karoubi's Fundamental Theorem \cite{Karoubi:Annals}.

In Section \ref{sec:GWspectrum}, we construct the Grothendieck-Witt theory functor $GW$ which is a symmetric monoidal functor from dg categories with weak equivalences and duality (that have uniquely $2$-divisible mapping complexes) to symmetric spectra of topological spaces.
The higher Grothendieck-Witt groups of a dg category with weak equivalences and duality $\scA$ are the homotopy groups $GW_i(\scA)$ of its Grothendieck-Witt spectrum $GW(\scA)$.
To explain the idea of the construction, recall that Waldhausen's $S_{\bullet}$-construction can be iterated to yield a symmetric spectrum
$\{ w\scA^{\pretr}, wS_{\bullet}\scA^{\pretr}, wS_{\bullet}^{(2)}\scA^{\pretr},...\}$ which is a positive $\Omega$-spectrum.
Similarly, the $\RR_{\bullet}$-construction can be iterated.
However, contrary to the $K$-theory case, $(w\RR_{\bullet}^{(2)}\scA^{\pretr})_h$
is not a delooping of $(w\RR_{\bullet}\scA^{\pretr})_h$.
To construct a spectrum as in the $K$-theory case, we note that the pretriangulated hull $\scA^{\pretr}$ can be equipped with many dualities.
For each integer $n$, we have a dg category with weak equivalences and duality $\scA^{[n]}$ all having the same pretriangulated hull as $\scA$,
but they are equipped with dualities that depend on $n$.
When $\scA = \sPerf(X)$, the categories $\scA^{[n]}$ correspond to the dg categories $\sPerf(X)$ equipped with the shifted dualities $E\mapsto Hom^{\bullet}_{O_X}(E,O_X[n])$.
With appropriate models $\scA^{(n)}$ for $\scA^{[n]}$, we show in Theorem \ref{thm:GWAOmegaInfinity}
that the sequence 
$$GW(\scA) = \{ (w\scA)_h, (w\RR_{\bullet}\scA^{(1)})_h, (w\RR_{\bullet}^{(2)}\scA^{(2)})_h,...\}$$
together with the bonding maps of Definition \ref{dfn:GWspectrum} is a positive symmetric $\Omega$-spectrum.
By Proposition \ref{prop:GWspecIsGWspace}, its infinite loop space is the Grothendieck-Witt space of $Z^0\scA^{\pretr}$ as defined in \cite{myMV}.
In particular, its zero-th homotopy group is the Grothendieck-Witt group of Section \ref{sect:GWgrpsDGcats}.
The spectrum $GW(\scA)$ is the Grothendieck-Witt analog of connective $K$-theory.
However (!), this spectrum is rarely connective: its negative homotopy groups are Balmer's triangular Witt groups of $\T\scA$; see Proposition \ref{prop:NegGWistriangularW}.

Section \ref{sec:PeriodInvLocn} contains our main results regarding the higher Grothendieck-Witt groups of dg categories with weak equivalences and duality.
Write $GW^{[n]}(\scA)$ for the Grothendieck-Witt spectrum of $\scA^{[n]}$.
Then $GW^{[0]}(\scA) \simeq GW(\scA)$ and $GW^{[n]}(\scA) \simeq GW^{[n+4]}(\scA)$.
We show in Theorem \ref{thm:PeriodicityExTriangle} an algebraic version of the Bott sequence in topology.
Our theorem asserts that if $\frac{1}{2}\in \scA$ then the sequence
$$GW^{[n]}(\scA) \stackrel{F}{\longrightarrow} K(\scA) \stackrel{H}{\longrightarrow} GW^{[n+1]}(\scA) \stackrel{\eta\cup}{\longrightarrow}S^1\wedge GW^{[n]}(\scA)$$
is an exact triangle in the homotopy category of spectra.
Here $F$ and $H$ are forgetful and hyperbolic functors, and $\eta \in GW^{[-1]}_{-1}(k)$ corresponds to the unit $\langle 1\rangle \in W(k)\cong GW^{[-1]}_{-1}(k)$ in the Witt ring of the base ring $k$.
This theorem implies Karoubi's Fundamental Theorem for dg categories (Theorem \ref{thm:KaroubiFundThm}) and the Localization Theorem \ref{thm:LcnConn1} asserting that a sequence $\scA_0 \to \scA_1 \to \scA_2$ of dg categories with weak equivalences and duality such that $\frac{1}{2}\in \scA_i$ and such that the sequence $\T\scA_0 \to \T\scA_1 \to \T\scA_2$ of associated triangulated categories is exact induces a homotopy fibration 
$$GW(\scA_0) \to GW(\scA_1) \to GW(\scA_2)$$
of Grothendieck-Witt spectra.
In particular, our higher Grothendieck-Witt groups are invariant under derived equivalences (Theorem \ref{thm:Invariance}).

In Section \ref{section:LandTate} we generalize results of Kobal and Williams \cite{kobal} to dg categories with weak equivalences and duality.
For such a category $\scA$, its connective $K$-theory spectrum $K(\scA)$ is equipped with a canonical $C_2$-action induced by the duality on $\scA$.
When $\frac{1}{2}\in \scA$, we show in Theorem \ref{thm:williamsKobal} a homotopy fibration 
$$K(\scA)_{hC_2} \to  GW(\scA) \to L(\scA)$$
and a homotopy cartesian square of spectra
$$
\xymatrix{
GW(\scA) \ar[r]\ar[d] & K(\scA)^{hC_2} \ar[d]\\
L(\scA) \ar[r] & \hat{\bH}(C_2,K(\scA))
}$$
where $K(\scA)_{hC_2}$, $K(\scA)^{hC_2}$ and $\hat{\bH}(C_2,K(\scA))$ denote the homotopy orbit, homotopy fixed point and the Tate spectrum of the $C_2$-spectrum $K(\scA)$.
The spectrum $L(\scA) = \eta^{-1}GW(\scA)$ is $4$-periodic and has homotopy groups the Balmer triangular Witt groups of $\scA$.
The map $GW(\scA) \to K(\scA)^{hC_2}$ is the subject of Williams' conjecture which predicts it to be a $2$-adic equivalence.
The results of this section are essential in the solution of this conjecture for schemes in \cite{meBerrickKaroubiOestvar}.

In Section \ref{sec:KGWspectrum}, we define the analog $\GW$, called {\em Karoubi-Grothendieck-Witt spectrum functor}, of the non-connective $K$-theory functor $\bK$.
As a spectrum it is equivalent to the mapping telescope of a sequence of spectra
$$GW(\scA) {\longrightarrow}  \Omega GW(S\scA) {\longrightarrow} \Omega^2 GW(S^2\scA) {\longrightarrow} \cdots $$
where $S$ denotes the algebraic suspension ring (Remark \ref{rmk:BiSpVsSpForGW}).
However, this mapping telescope does not define a symmetric monoidal functor from dg categories to spectra for the same reason that the infinite loop space functor $\Omega^{\infty}GW$ of $GW$ is not symmetric monoidal.
That's why our symmetric monoidal functor $\GW$ takes values in the category of bispectra.
The symmetric monoidal model category of bispectra is yet another model for the homotopy category of spectra.
It contains the category of symmetric spectra as a full monoidal subcategory, the inclusion of spectra into bispectra preserves stable equivalences and induces an equivalence of associated homotopy categories. 
There is a natural map $GW(\scA) \to \GW(\scA)$ which is an isomorphism in degrees $\geq 1$ and a monomorphism in degree $0$ (Theorem \ref{thm:KbKGWKGW}).
If $\T\scA$ is idempotent complete, then this map is also an isomorphism in degree $0$.
For instance, for a ring with involution $A$ with $\frac{1}{2}\in A$, the group $\GW_0(A)$ is the Grothendieck-Witt group of the category of finitely generated projective $A$-modules.
We show in Theorems \ref{thm:InvarianceKGW}, \ref{thm:locnForKGW} and \ref{thm:PeriodicityExTriangleForKGW} the $\GW$-analogs of the Invariance, Localization, and Bott Sequence theorems proved for $GW$ in Section \ref{sec:PeriodInvLocn}.

In Section \ref{sec:GWschemes} we prove the results regarding the higher Grothendieck-Witt groups of schemes mentioned above.

In Section \ref{section:Bott} we show that our results imply the $8$ homotopy equivalences in Bott's Periodicity Theorem for the infinite orthogonal group as envisioned by Karoubi in \cite{karoubi:battelle}.
Whereas Karoubi's proof of his Fundamental Theorem in \cite{Karoubi:Annals} is based on classical $8$-fold Bott periodicity, our new proof of his theorem actually implies $8$-fold Bott periodicity.
This is explained in Section \ref{subsec:ClassicalBottPeriod}.
Of course, Bott periodicity is not new, but the reader may find it a little reward at the end of a long paper.

There are two appendices to this article.
One, Section \ref{appdx:Homology}, shows that for a ring with involution $A$ such that $\frac{1}{2}\in A$, the connected component of the infinite loop space of $GW(A)$ is homotopy equivalent to Quillen's plus construction $BO(A)^+$ applied to the classifying space of the infinite orthogonal group of $A$.
This is necessary if Theorem \ref{thm:KaroubiFundThm} is truly to be a new proof of Karoubi's Fundamental Theorem.
In \cite{mygiffen}, we gave proofs of the results of Section \ref{appdx:Homology} based on Karoubi's version of the Fundamental Theorem.
Here we avoid Karoubi's theorem.
The other appendix, Section \ref{Appx:TateSp}, recollects definitions and results about symmetric spectra, homotopy orbit, homotopy fixed and Tate spectra, and about bispectra used throughout the paper.
\vspace{1ex}

{\bf DG-categories versus complicial exact categories}.
In some preprint versions of this paper I used ``complicial exact categories'' (in the sense of, for instance, \cite[Definition 3.2.2]{mySedano}) whereas I decided to work with ``dg-categories'' here. 
In practice, all complicial exact categories have a ``dg-enhancement'' even though the exact structures may differ (this, however, doesn't change the $GW$-theories, by \cite[Lemma 7]{myMV}).
The passage from exact categories to  complexes which are complicial exact categories underlying a dg-category is explained in \cite[\S 6]{myMV}; see also \cite[Remark 14]{myMV}.

The advantage of the category of dg-categories is that it has a symmetric monoidal tensor-product. 
This considerably simplifies the construction of the $GW$-spectrum and the treatment of products in $GW$-theory.
\vspace{1ex}

\red{
{\bf Advice for the hurried reader.}
For a first reading, I suggest to skip the technical sections \ref{sect:GWgrpsDGcats} -- \ref{sec:KGWspectrum} while nevertheless bearing in mind Examples \ref{exn:VectAsExCatDual} and \ref{ex:ChbEDGwithWeaksAndDuals} and the analogies of $GW$ and $\GW$ with connective and non-connective $K$-theory $K$ and $\bK$.
Take note, however, that the negative homotopy groups $GW_i$ of $GW$ are the Balmer Witt groups $W^{-i}$ for $i<0$ which are non-zero in general; see the example following Proposition \ref{prop:GWspecIsGWspace}.
Then take the Localization Theorems \ref{thm:LcnConn1} and \ref{thm:locnForKGW} as axioms and use Remark \ref{rmk:FundThmFromLocnThm} to deduce the Bott sequences \ref{thm:PeriodicityExTriangle} and \ref{thm:PeriodicityExTriangleForKGW}.
Now you are ready to jump right into Section \ref{sec:GWschemes}.
}
\vspace{1ex}

\red{
{\bf Acknowledgements.}
This paper is part of a collection of papers published in honor of Charles A. Weibel's 65th birthday. 
Ever since I wrote my PhD, Chuck has been incredibly supportive; I warmly thank him for that.

It should be clear to the reader that this paper owes a lot to the mathematics of Max Karoubi, Robert Thomason, and Paul Balmer. 
I would like to thank them.
Thanks also go to Baptiste Calmes and Heng Xie for comments on earlier drafts, and to the anonymous referee for his comments.

 It took me about 10 years to put the results of this paper into its final form.
 I gave the first talks about the main results of Section \ref{sec:KGWspectrum} and its consequences at a workshop on algebraic $K$-theory in Montreal, 2004, and in a series of lectures at ETH Zurich in 2005, though proofs and framework were very different then.
 I apologize for the delay in publication.
 
 Over the years, various institutions and grant agencies have given financial support, such as I.H.E.S. and Max Planck in Bonn, NSF and EPSRC. 
 I would like to thank them.
 }

\section{Grothendieck-Witt groups of dg categories}
\label{sect:GWgrpsDGcats}

Let $\bk$ be a commutative ring.
Our work takes place in the framework of differential graded categories over $\bk$.
We recall standard notation and terminology; see also \cite{keller:ICM}.

\subsection{Differential graded $\bk$-modules}
\label{subsec:difk}
A dg $\bk$-module is a $\bk$-module $M$ together with a direct sum decomposition $M=\bigoplus_{i\in \Z}M^i$ and a $\bk$-linear map $d:M\to M$, called differential, such that $d(M^i)\subset  M^{i+1}$ for all $i\in \Z$ and $d\circ d =0$.
Elements $x \in M^i$ are said to be homogeneous of degree $|x| = i$.
A morphism $f:M\to N$ of dg $\bk$-modules is a morphism of graded $\bk$-modules commuting with differentials, that is,
$f(M^i)\subset N^i$ and $f(dx)=d(fx)$ for all $x\in M$.
Composition is composition of $\bk$-modules.
This defines the category $\dgMod_{\bk}$ of dg $\bk$-modules.

There are additive functors $Z^0, B^0, H^0: \dgMod_{\bk} \to \Mod_{\bk}$ to the category $\Mod_{\bk}$ of $\bk$-modules defined by
$$\renewcommand\arraystretch{2}
\begin{array}{lcl}
Z^0M &=& \ker(d:M^0\to M^1)\\
B^0M &=& \im(d:M^{-1}\to M^0)\\ 
H^0M &=& Z^0M/B^0M.
\end{array}
$$
The tensor product of two dg $\bk$-modules $M$ and $N$ is given by 
$$(M\otimes N)^n=\bigoplus_{i+j=n}M^i\otimes N^j$$ with differential $d(x\otimes y) = dx \otimes y + (-1)^{|x|}x\otimes dy$.
The function dg $\bk$-module $[M,N]$ from $M$ to $N$ is the dg $\bk$-module given by  
$$[M,N]^n=\prod_{-i+j=n}\Hom_{\bk}(M^i,N^j)$$ with differential 
$df = d\circ f - (-1)^{|f|} f \circ d$.
Note that the dg $\bk$-module morphisms from $M$ to $N$ are precisely the elements of $Z^0[M,N]$. 
Tensor product $\otimes$, function object $[\phantom{A},\phantom{A}]$ and unit $\1 = \bk$
make the category of dg $\bk$-modules into a closed symmetric monoidal category where evaluation, coevaluation and symmetry (or switch) are given by 
$$
\renewcommand\arraystretch{2}
\begin{array}{ll}
e:[M,N]\otimes M \to N: f \otimes x \mapsto f(x),& \hspace{-8ex}\nabla:M \mapsto [N,M\otimes N]: x \mapsto (y \mapsto x\otimes y),\\
\tau: M\otimes N \to N\otimes M: x\otimes y \mapsto (-1)^{|x||y|}y\otimes x. &
\end{array}$$

There is a coherence theorem for closed symmetric monoidal categories \cite{KellyMacLane:Coherence}.
This will be convenient when checking the commutativity of diagrams. 
It is therefore important to consider the various categories to be defined below   as closed symmetric monoidal categories, and it is important to insist that functors and natural transformations be defined using only the basic building blocks of closed symmetric monoidal categories such as unit, tensor product, function object, evaluation and coevaluation.

\subsection{Pointed differential graded categories}
\label{subsec:pointedDGcats}
DG categories are categories enriched in the closed symmetric monoidal category of dg $\bk$-modules.
The category of small dg-categories is therefore endowed with the structure of a closed symmetric monoidal category \cite{kelly:book}.
In particular, it is equipped with internal tensor products, function objects, unit, evaluation and coevaluation.
For the purpose of $K$-theory and Grothendieck-Witt theory it will be convenient to assume the following.
\vspace{1ex}

\begin{center}
{\it All dg categories in this paper are pointed.}
\end{center}
\vspace{1ex}

\noindent
This means that they are equipped with a choice of zero object called base-point.
All dg functors are to preserve base points, and tensor products, function objects and all other constructions pertaining to dg categories are to be understood to be performed in the pointed category.
The functor $\A \mapsto \A_+$ from dg categories to pointed dg categories that adds a base point zero object is symmetric monoidal and doesn't change $K$-theory nor Grothendieck-Witt theory.
So, there is no loss in assuming our dg categories to be pointed.
To be less repetitive we may drop the adjective pointed and write ``dg category'' for what should be ``pointed dg category''.

For the reader's convenience and to fix notation, we recall what some of this means in detail.

\begin{definition}
A small {\em pointed dg category} $\A$ \red{(henceforth called {\em dg category})} is given by the following data:
\begin{itemize}
\item a non-empty set of objects, 
\item
for any two objects $A,B$ of $\A$ a dg $\bk$-module $\A(A,B) = \bigoplus_{i\in \Z}\A(A,B)^i$, called {\em mapping complex} from $A$ to $B$, 
\item
for any object $A$ a unit morphism $1_A \in \A(A,A)^0$,
\item
for any three objects $A,B,C$ a map of dg $\bk$-modules
$$\circ: \A(B,C)\otimes \A(A,B) \to \A(A,C):f\otimes g \mapsto f\circ g$$
satisfying the usual associativity and unit constraints of a category,
\item
a choice of zero object $0_{\A}$ called base point.
\end{itemize}
\end{definition}

Note that $d(1_A)=0$ for every object $A\in \A$ (since $d(1)=d(1\circ 1) = d(1)\circ 1 + 1\circ d(1)$) and that $d(f\circ g) = 0$ whenever $df=0$ and $dg=0$.
In particular, we have a subcategory $Z^0\A \subset \A$ which has the same objects as $\A$ and morphism $\bk$-modules $(Z^0\A)(A,B)=Z^0(\A(A,B))$.
We call $Z^0\A$ the {\em underlying category} of the dg category $\A$
and think of a dg category $\A$ as the $\bk$-linear category $Z^0\A$ 
together
%where each morphism group $Z^0\A(X,Y)$ is equipped 
with the additional homological algebra data of the mapping complexes $\A(X,Y)$ \red{for all pairs of objects  $X$, $Y$ of $\A$}.
%For us therefore, a {\em morphism} $f:A \to B$ between objects $A$, $B$ of the dg category $\A$ is a morphism in $Z^0\A$, and two objects $A$ and $B$ of a dg category $\A$ are called {\em isomorphic} if they are isomorphic in $Z^0\A$.
%\edit{shall we really do that? Morphisms in $\A$ equal morphisms in $Z^0\A$?
%Try to avoid this}
%there are morphisms $f:A \to B$ and $g:B \to A$ in $\A$ (that is, $f\in Z^0(A,B)$ and $g\in Z^0(B,A)$) such that $fg=1$ and $gf=1$.
%Composition in $\A$ makes the set $Z^0\A$ of morphisms in $\A$ into a category with the same set of objects as $\A$ and morphisms from $A$ to $B$ the set $Z^0\A(A,B)$.
%It knows a lot about the dg category $\A$ itself.
\vspace{1ex}

A {\em morphism of pointed dg categories}, also called dg functor, $F:\A \to \B$ from $\A$ to $\B$ is given by the following data: to every object $A$ of $\A$ is assigned an object $FA$ of $\B$, for every two objects $A,B$ of $\A$ there is given a morphism of dg $\bk$-modules $\A(A,B) \to \B(FA,FB):f \mapsto F(f)$ such that  $F(1_A) = 1_{F(A)}$ and for any three objects $A,B,C$ of $\A$ and any $f\in \A(B,C)$ and $g\in \A(A,B)$ we have $F(f\circ g)=F(f)\circ F(g) \in \B(FA,FC)$.
Moreover, $F$ is required to send the base point of $\A$ to the base point of $\B$.
We denote by $\dgCat_{\bk}$ the category of small pointed dg $\bk$-categories.
The category of small pointed $\bk$-linear categories will be denoted by 
$\Cat_{\bk}$.
We may drop the index $\bk$ if the base ring $\bk$ is understood.

A {\em natural transformation} $\alpha:F \to G$ between dg functors $F,G:\A \to \B$ 
is a collection $\{\alpha_A\}_{A\in \A}$ of morphisms $\alpha_A\in Z^0\B(FA,GA)$ such that for all $A,B\in \A$ and $f\in \A(A,B)$ we have $\alpha_B \circ F(f) = G(f)\circ \alpha_A$. 
A dg functor $F: \A \to \B$ is an {\em equivalence of dg categories}
if there exists a dg functor $G: \B \to \A$ such that $FG$ and $GF$ are naturally isomorphic to the identity functors.
If $F:\A \to \B$ is an equivalence of dg categories, then $Z^0F:Z^0\A \to Z^0\B$ is an equivalence of linear categories.
If $F$ is fully faithful then the converse also holds.

For $\A,\B \in \dgCat_{\bk}$, the {\em tensor product dg category} $\A \otimes \B$ has objects the pairs $(A,B)$ of objects $A$ of $\A$ and $B$ of $\B$.
The objects of the form $(A,0_{\B})$ and $(0_{\A},B)$ are identified with the base point $0_{\A\otimes\B}$ of $\A\otimes \B$, that is,
$Ob(\A\otimes \B) = Ob\A \wedge Ob\B$.
We may sometimes write $A\otimes B$ or $AB$ for the object $(A,B)$ of $\A \otimes \B$.
Morphism complexes are
$$(\A\otimes \B)((A_1,B_1),(A_2,B_2)) = \A(A_1,A_2)\otimes \B(B_1,B_2)$$
with composition defined by 
$(f_1\otimes g_1)\circ (f_2\otimes g_2) = (-1)^{|g_1||f_2|}(f_1\circ f_2)\otimes (g_1\circ g_2)$ and unit $1_{(A,B)}=1_A\otimes 1_B$.
The switch dg-functor $\A\otimes \B \to \B\otimes \A$ is the switch of pointed sets $(A,B)\mapsto (B,A)$ on objects and the switch of dg $k$-modules (\S \ref{subsec:difk}) on morphism complexes.

The {\em homomorphism dg category} $\dgFun(\A,\B) \in \dgCat_{\bk}$ for $\A, \B\in \dgCat_{\bk}$ has objects the dg functors $F:\A \to \B$ (which, by our convention, preserve base points) and the morphism complex
$[F,G]$ in degree $n$ is the set $[F,G]^n$ of collections 
$\alpha = (\alpha_A)_{A\in \A}$ with $\alpha_A \in \B(FA,GA)^n$ satisfying 
$G(f)\circ \alpha_A = (-1)^{|\alpha||f|}\ \alpha_B \circ F(f)$ for all homogeneous $f \in \A(A,B)$.
The differential $d\alpha$ is defined by $(d\alpha)_A = d^{\B}\alpha_A$ where $d^{\B}$ is the differential in $\B$.

If $\C$ is a small category and $\A \in \dgCat_{\bk}$, we write $\Fun(\C,\A)$ for the dg category
$$\Fun(\C,\A) = \dgFun(k[\C]_+,\A)$$
where $k[\C]$ is the dg category whose objects are the objects of $\C$ and the mapping complex $k[\C](A,B)$ in degree $0$ is the free $\bk$-module  $k[\C](A,B)^0=k[\C(A,B)]$ generated by the elements of $\C(A,B)$ and in degree $i\neq 0$ we have $k[\C](A,B)^i=0$.
Note that dg functors $k[\C]_+ \to \A$ are the same as functors $\C \to Z^0\A$.

The functors $Z^0, H^0: \dgMod_{\bk} \to \Mod_{\bk}$ extend to functors
$$Z^0,H^0: \dgCat_{\bk} \to \Cat_{\bk}.$$
For example, the category $Z^0\A$ is the underlying $\bk$-linear category of $\A$, and $H^0\A$ 
has the same objects as $\A$ and the morphism $\bk$-module $H^0\A(A,B)$ is the module $H^0(\A(A,B))$.

\begin{example}[The category $\scC_{\bk}$]
\label{ex:Ck}
Let $\scC_{\bk}$, or simply $\scC$ if $\bk$ is understood, be the full subcategory of $\dgMod_{\bk}$ of those dg $\bk$-modules $M$ such that 
$M^i = k^{n_i}$ for some $n_i \in \N$ and $M^i \neq 0$ for only finitely many $i\in \Z$.
In other words, $\scC_{\bk}$ is (a small model of) the category of bounded  complexes of finitely generated free $\bk$-modules.
Tensor product, function space and unit in the category of dg $\bk$-modules provide $\scC_{\bk}$ with a structure of a closed symmetric monoidal category where evaluation, coevaluation and symmetry are as in Section \ref{subsec:difk}.
This makes $\scC_{\bk}$ into a dg category.
We consider $\scC_{\bk}$ as a pointed dg category with base point the unique zero object in $\scC_{\bk}$.

In $\scC_{\bk}$, we have the following distinguished objects $\1$, $C$ and $T$ and an exact sequence $\Gamma$ relating them.
%They are defined as follows.
The object $\1$ is the unit $\bk$ of the tensor product in $\scC_{\bk}$.
Let $C$ be the dg $\bk$-algebra $\bk[\eps]/\eps^2$, $|\eps|=-1$, $d\eps = 1$.
As a graded $\bk$-module, we have $C=k\cdot 1 \oplus k\cdot \eps$, and the differential is determined by $d \eps = 1$.
Note that $C$ is a commutative dg $\bk$-algebra containing $\1$ as a dg submodule.
We let $T$ be the quotient dg $k$-module $C/\1$.
As a graded $k$-module, we have $T = k\eps$, $|\eps| = -1$, and the differential is zero.
If we denote by $i:\1 \to C: 1 \mapsto 1$ the inclusion and by $p: C\to T: 1\mapsto 0,\ \eps \mapsto \eps$ the projection map, we have
an exact sequence of dg $k$-modules in $\scC_{\bk}$
\begin{equation}
\label{eqn:fundExSeq}
\xymatrix{\Gamma \hspace{5ex} & {\1} \xymono[r]^i  & C \xyepi[r]^p  & T.}
\end{equation}
Often we will write $\scC_{\bk}\A$ \red{(or $\scC \A$ if $k$ is understood)} for the tensor \red{product} dg category $\scC_{\bk}\otimes \A$.
There is a canonical fully faithful embedding of dg categories 
$\A \to \scC_{\bk}\A:A\mapsto \1\otimes A$.
\end{example}

\subsection{Exact and pretriangulated dg categories}
\label{subsec:ExDGCats}
Recall from \cite{quillen:higherI}, \cite{keller:obscureAxiom}, \cite[\S 2.1]{myEx} that a {\it small exact category} is a small additive category $\E$ equipped with a set of
sequences of maps in $\E$, called {\it conflations} (or admissible exact
sequences),
\begin{equation}
\label{eqn:Conflation}
%$$
X \stackrel{i}{\to} Y \stackrel{p}{\to} Z
%\eqno(jhk)
%$$
\end{equation}
satisfying the properties (Ex0) -- (Ex4) below.
In a conflation (\ref{eqn:Conflation}), the map $i$ is called {\it inflation}
(or admissible monomorphism) and may be depicted as $\rightarrowtail$, and the map $p$ is called {\it deflation} (or admissible
epimorphism) and may be depicted as $\twoheadrightarrow$.

\begin{itemize}
\item[(Ex0)]
\label{enum:ExCat2}
Conflations are closed under isomorphisms.
\item[(Ex1)]
\label{enum:ExCat1}
In a conflation (\ref{eqn:Conflation}), the map $i$ is a kernel of $p$, and
$p$ is a cokernel of $i$.
\item[(Ex2)]
\label{enum:ExCat6}
For any two objects $X,Y\in \E$, the following sequence is a conflation
$$
X \stackrel{\left(\begin{smallmatrix}1\\0\end{smallmatrix}\right)}{\to} X
\oplus Y \stackrel{\left(\begin{smallmatrix}0&1\end{smallmatrix}\right)}{\to}
Y.
$$
\item[(Ex3)]
\label{enum:ExCat3}
Inflations are closed under compositions, and deflations are closed under
compositions.
\item[(Ex4)]
\label{enum:ExCat4}
The pushout of an inflation along any map in $\E$ exists and is an inflation.
Dually, the pull-back of a deflation along any map in $\E$ exists and is a deflation.
\end{itemize}

\red{
\begin{example}
A typical example of an exact category is the category $\Vect(X)$ of vector bundles (that is, locally free sheaves of finite rank) on a scheme $X$.
The conflations are the usual short exact sequence of locally free sheaves.
\end{example}
}

\red{
\begin{example}
\label{exn:ChbE}
If $\E$ is a $k$-linear exact category, we make the category of bounded complexes in $\E$ into a dg category $\Ch^b(\E)$ as follows.
Its objects are the bounded  complexes in $\E$ (unless stated otherwise, all complexes in this article are cohomologically indexed and thus have their differentials increase degrees).
Let $M$ and $N$ be two objects of $\Ch^b(\E)$.
Their mapping complex is the dg $k$-module $[M,N]$ with
$$[M,N]^n=\prod_{-i+j=n}\E(M^i,N^j)$$ and differential 
$df = d_N\circ f - (-1)^{|f|} f \circ d_M$.
The underlying category $Z^0\Ch^b(\E)$ is the usual category of bounded  complexes in $\E$ and maps the chain maps (that is, degree $0$ maps that commute with differentials).
The category $H^0\Ch^b(\E)$ is the usual homotopy category of bounded  complexes in $\E$ where the maps are the homotopy classes of chain maps.
Note that the dg category $\Ch^b(\E)$ only depends on $\E$ as an additive category.
At this point, the exact structure is irrelevant.
\end{example}
}

\begin{definition}
Let $\A$ be a dg category.
A sequence $A \stackrel{f}{\to} B \stackrel{g}{\to} C$ of morphisms in the underlying category $Z^0\A$
with $gf=0$ is called {\em exact} if there are $s\in \A(C,B)^0$ and $r \in \A(B,A)^0$ such that $1_A = rf$, $1_B= fr + sg$, and $1_C = gs$.
Note that $s$ determines $r$ and vice versa.
A pointed dg category $\A$ is called {\em exact} if these exact sequences make the underlying linear category $Z^0\A$ into an exact category.
\end{definition}

\red{
\begin{example}
\label{exn:dgChbEisExact}
Let $\E$ be a $k$-linear exact category.
Recall from Example \ref{exn:ChbE} the dg category $\Ch^b\E$ of bounded  complexes in $\E$.
This is an exact dg category.
The conflations in the underlying category $Z^0\Ch^b\E$ are the short exact sequences of complexes which degree-wise (that is, forgetting differentials) are split exact.
In particular, the exact structure on $Z^0\Ch^b\E$ does not depend on the exact structure of $\E$.
\end{example}
}

\begin{definition}
\label{dfn:pretriangulated}
A pointed dg category $\A$ is called {\em pretriangulated} if it is exact
and if 
the functor $\A \to \scC_{\bk}\A: A \mapsto \1 \otimes A$ is an equivalences of dg categories.
Since $\A \to \scC_{\bk}\A$ is fully faithful, the last condition is equivalent to requiring that $Z^0\A \to Z^0(\scC_{\bk} \A)$ be an equivalences of $k$-linear categories.
\end{definition}

\red{
\begin{example}
\label{ex:ChbEPreTr}
We already noticed in Example \ref{exn:dgChbEisExact} that for a $k$-linear exact category $\E$, the dg category $\Ch^b\E$ of Example \ref{exn:ChbE} is exact.
In fact, it is pretriangulated.
The inverse of the dg functor $\Ch^b\E \to \scC_k\otimes \Ch^b\E$ is the total complex dg functor
$\scC_k\otimes \Ch^b\E \to \Ch^b\E$ which sends the object $(A,E)$ of $\scC_k\otimes \Ch^b\E$ to the total complex object
$A\otimes E$ with 
$$(A\otimes E)^n = \bigoplus_{i+j=n} A^i\otimes E^j$$
and differential restricted to $A^i\otimes E^j$ the map $d_A^i\otimes 1_{E_j} + (-1)^i 1_{A^i}\otimes d_E^j$.
Here,  for $A=k^n$, the expression $A^i\otimes E^j=k^n\otimes E^j$ stands for $(E^j)^{\oplus n}$.
The total complex dg functor is the "identity" on mapping complexes.
\end{example}
}

\begin{remark}
For a dg category to be exact (pretriangulated) is a property, not an extra structure.
Note also that any dg functor $\A \to \B$ preserves exact sequences.
In particular, if $\A \to \B$ is a dg functor between exact dg categories, then the functor $Z^0\A \to Z^0\B$ between underlying exact categories is exact.
\end{remark}

\begin{lemma}
\label{lem:ABdgFunProps}
Let $\A$ and $\B$ be pointed dg categories.
\begin{enumerate}
\item
\label{item1:lem:ABdgFunProps}
If $\B$ is exact then so is  $\dgFun(\A,\B)$.
\item
\label{item2:lem:ABdgFunProps}
If $\B$ is pretriangulated then so is $\dgFun(\A,\B)$.
\item
\label{item3:lem:ABdgFunProps}
If $\B$ is pretriangulated then so is $\scC_{\bk}\B$.
\end{enumerate}
\end{lemma}

\begin{proof}
We prove part (\ref{item1:lem:ABdgFunProps}).
Since $\B$ is exact, $Z^0\B$, $\B^0$ and $\B$ are all additive categories.
This implies that $Z^0\dgFun(\A,B)$, $\dgFun(\A,B)^0$ and $\dgFun(\A,B)$ are all additive categories with $(F\oplus G)(A)=F(A)\oplus G(A)$ for $F,G\in \dgFun(\A,\B)$ and $A\in \A$.
Given a map $i:F_0 \to F_1$ in $Z^0\dgFun(\A,\B)$ such that there is a homogeneous degree $0$ map $r: F_1 \to F_0$ with $ri=1$, then $i$ is an admissible monomorphism if and only if for every $A\in \A$ the map $i(A):F_0(A) \to F_1(A)$ is an admissible monomorphism in $Z^0\B$.
This implies that the composition of two admissible monomorphisms in $Z^0\dgFun(\A,\B)$ is an admissible monomorphism.
It also implies that the push-out of a map in $Z^0\dgFun(\A,\B)$ along an admissible monomorphism exists and is an admissible monomorphism, the push-out being defined object-wise and made into a dg functor by functoriality of the push-out and the fact that the inclusions $Z^0\B \subset \B^0 \subset \B$ are exact where the additive categories $\B^0$ and $\B$ are equipped with only those exact sequences which split.
This shows (Ex3) and (Ex4) for inflations.
By a dual argument, (Ex3) and (Ex4) hold for deflations.
The remaining axioms (Ex0) - (Ex2) trivially hold, and thus $Z^0\dgFun(\A,\B)$ is indeed an exact category.

Next we prove part (\ref{item2:lem:ABdgFunProps}).
We already know that $\dgFun(\A,\B)$ is exact.
We are left with proving that 
$\dgFun(\A,\B) \to \scC_k\otimes \dgFun(\A,\B): F \mapsto \1\otimes F$ 
is an equivalence.
Since that functor is fully faithful, we need to show essential surjectivity.
Now, for any dg categories $\A$ and $\B$, the dg functor
$$\scC_k \otimes \dgFun(\A,\B) \to \dgFun(\A,\scC_k\otimes \B): X \otimes F \mapsto (A \mapsto X \otimes F(A))$$
is fully faithful.
This follows from the fact that for any $X,Y\in \scC_k$ the dg $k$-module $[X,Y]$ is finitely generated free as $k$-module and thus tensoring with $[X,Y]$ commutes with arbitrary equalizers of dg $k$-modules.
Therefore, all functors in the diagram
$$\xymatrix{
Z^0\dgFun(\A,\B) \ar[rd]\ar[d] & \\
Z^0( \scC_k\dgFun(\A,\B)) \ar[r] & Z^0\dgFun(\A,\scC_k\B) }$$
are fully faithful.
If $\B$ is pretriangulated, then the diagonal map is an equivalence.
By fully faithfulness of the horizontal functor, the vertical functor is also an equivalence.

Finally, we prove (\ref{item3:lem:ABdgFunProps}).
If $\A$ and $\B$ are equivalent dg categories, then $\A$ is pretriangulated if and only if $\B$ is.
Therefore, if $\B$ is pretriangulated, then $\B$ and $\scC_k\B$ are equivalent and hence $\scC_k\B$ is pretriangulated.
\end{proof}

\subsection{The pretriangulated hull. Introduction}
\red{
Pretriangulated dg categories are dg categories whose homotopy categories are naturally tringulated.
Most dg categories of interest are pretriangulated, for instance $\Ch^b\E$ as noted in Example \ref{ex:ChbEPreTr}.
Our treatment of products in $GW$-theory, in fact, the very definition of the $GW$-spectrum, 
do require us to consider dg categories of the form
$\A \otimes \B$ which, in general, are not pretriangulated even if $\A$ and $\B$ are.
Thus, we need to complete $\A \otimes \B$ into a pretriangulated dg category.
Now, for any dg category $\A$ there is a pretriangulated completion $\A^{\pretr}$, called {\em pretriangulated hull of $\A$}, which 
}
was introduced in \cite{bondalKapranov:pretr}; see also \cite{Drinfeld:DGquotient}.

\red{
In short, the pretriangulated hull of a dg category $\A$  is the category of finitely generated semi-free dg $\A$-modules, that is, those dg $\A$-modules $M: \A^{op} \to \dgMod_k$ which have a
finite filtration 
$$0=F^0M\subset ... \subset F^rM \subset F^{r+1}M \subset ... \subset F^nM=M$$  by dg submodules such that $F^{r+1}M/F^rM$ is finitely generated free (that is, a finite direct sum of (shifts of) objects of $\A$).
In \cite{bondalKapranov:pretr}, \cite{Drinfeld:DGquotient} the authors give an explicit model for  
$\A^{\pretr}$.
Below, we give a variant of their construction.
For us, the pretriangulated hull of $\A$}
is the dg category of extensions of $\scC_k\A$; see Definition \ref{dfn:pretr}.
Our treatment here is somewhat heavier than one would expect.
This is because we will need a description of $\A^{\pretr}$ in the context of dg categories with dualities which is in some sense symmetric monoidal with respect to the tensor product of such categories.
So, along the way we will take care of the monoidal structures.
We start with the definition of the dg category of extensions.
%\vspace{1ex}

\subsection{The dg category of extensions}
\label{subsec:dgCatExtn}
Let $\poSet$ be the category of small posets (that is, partially ordered sets) with injective order preserving maps as morphisms.
This is a symmetric monoidal category under the cartesian product as monoidal product.
Recall that the category of small pointed dg categories $\dgCat$ is symmetric monoidal under the tensor product of dg categories (Section \ref{subsec:pointedDGcats}).
In particular, the product category $\poSet \times \dgCat$ is symmetric monoidal with component-wise monoidal product.
We will define a symmetric monoidal functor
\begin{equation}
\label{eqn:PAmonoidalFuntor}
\poSet \times \dgCat \to \dgCat: (\P,\A) \mapsto \P\A
\end{equation}
as follows.
For a poset $\P$ and a pointed dg category $\A$
the pointed dg category $\P\A$ has objects the pairs $A = (A,q)$ where $A:Ob \P\to Ob\A$ is a function that sends all but a finite number of elements of $\P$ to the base point zero object of $\A$, and $q=(q_{ij})_{i,j\in \P}$ is a matrix of elements $q_{ij}\in \A(A_j,A_i)^1$ of degree $1$ such that $q_{ij}=0$ for $i\red{\nless} j\in P$ and $dq +q^2 =0$.
The support $\supp(A)$ of $(A,q)$ is the set of \red{indices} $i\in \P$ such that $A_i\neq 0$.
This is a finite set.
In particular, only finitely many entries in $q$ are non-zero, and thus, the matrix product $q^2$ makes sense where $(q^2)_{ij}=\sum_kq_{ik}q_{kj}$.
The base point of $\P\A$ is the pair $(A,0)$ where the function $A$ sends all objects of $\P$ to the base point of $\A$. 
For two objects $A=(A,q^A)$ and $B=(B,q^B)$ of $\P\A$, the morphism complex is
$$\P\A(A,B)^n = \prod_{i,j\in \P}\A(A_i,B_j)^n$$
whose elements are the matrices $f = (f_{ji})_{i,j\in \P}$ with entries 
$f_{ji} \in \A(A_i,B_j)^n$.
The differential $d^{\P\A} f$ for such a matrix $f$ is defined by 
$$d^{\P\A} f= d^{\A}f + q^B\circ f - (-1)^{|f|}f\circ q^A$$
where $(d^{\A}f)_{i,j}=d^{\A}(f_{i,j})$ and $d^{\A}$ is the differential in $\A$.
Composition in $\P\A$ is matrix multiplication in $\A$.
One checks that $\P\A$ is indeed a pointed dg category.
Clearly, the assignment $(\P,\A) \mapsto \P\A$ is functorial in $\A$.
It is also functorial in $\P$ as an embedding $\P_0 \subset \P_1$ defines a full embedding of dg categories $\P_0\A \subset \P_1\A$ sending $(A,q) \in \P_0\A$ to its extension by zero $(\bar{A},\bar{q})$ in $\P_1\A$ where $\bar{A}_i=A_i$ if $i\in \P_0$, $\bar{A}_i=0$ for $i\notin \P_0$, $\bar{q}_{ij}=q_{ij}$ if $i,j\in \P_0$ and $\bar{q}_{ij}=0$ otherwise.

The functor $(\P,\A) \mapsto \P\A$ is symmetric monoidal with monoidal compatibility map
\begin{equation}
\label{eqn:MonoidalComMap}
\P_0\A{\otimes}\P_1\X \longrightarrow (\P_0\times\P_1)(\A{\otimes}\X)
\end{equation}
sending the object $(A,q^A)\otimes (X,q^X)$ of $\P_0\A{\otimes}\P_1\X$ to the object 
$(AX,q^{A X})$ where 
$(A X)_{ax}=(A_a,X_x)$ for $a\in \P_0$, $x\in \P_1$, 
and $q^{AX}$ is the sum of matrix tensor products $q^{AX} = q^A\otimes 1_X + 1_A \otimes q^X$.
On morphism complexes the dg functor sends the tensor product of matrices
$(f)\otimes (g)$ to the matrix tensor product $(f\otimes g)$, that is,
$(f\otimes g)_{ax,by}=f_{ab}\otimes g_{xy}$ for $a,b\in \P_0$ and $x,y \in \P_1$.
%One checkes that this commutes with differentials.

\begin{remark}
\label{rmk:P0AP1AessentiallySurj}
An embedding of posets $\P_0 \subset \P_1$ induces a fully faithful
embedding $\P_0\A \subset \P_1\A$ of dg categories.
Any two embeddings $i_0,i_1: \P_0 \subset \P_1$ induce naturally isomorphic embeddings $i_0,i_1:\P_0\A \subset \P_1\A$ of dg categories.
This is because for $(A,q)\in \P_0\A$ the ``identity matrix'' induces an isomorphism in $Z^0\P_1\A$ between the two objects $i_0(A,q)$ and $i_1(A,q)$.
In particular, the embedding $\P_0\A \subset \P_1\A$ is an equivalence of dg categories if for every embedding of posets $P \subset \P_1$ with $P$ finite, there is an embedding $P \subset P_0$.
For example, any poset embedding $\N \subset \P$ induces an equivalence of dg categories $\N\A \subset \P\A$ since every finite poset can be embedded into $\N$.
\end{remark}

\begin{remark}
\label{rmk:PPA}
Let $\A$ be a dg category and $\P_0$ and $\P_1$ be posets.
Then there is a fully faithful embedding of dg categories
\begin{equation}
\label{P0(P1)=P0P1}
\P_0(\P_1\A) \to (\P_0\P_1)\A: (A,Q) \mapsto (B,q)
\end{equation}
where $\P_0\P_1=\P_0\times \P_1$ is equipped with the lexicographic order: $(i,r)< (j,s)$ if and only if $i<j$, or $i=j$ and $r< s$.

To define the functor, recall that an object $(A,Q)$ of 
$\P_0(\P_1\A)$ is given by a map of sets $\P_0 \to Ob\P_1\A:i\mapsto (A_i,q^{A_i})$ with finite support together with a matrix $Q=(Q_{ij})_{i,j\in \P_0}$ where $Q_{ij}\in \P_1\A(A_j,A_i)^1$ such that $Q_{ij}=0$ whenever $i\red{\nless}\ j$ in $\P_0$
and $Q^2+d^{\P_1\A}Q=0$.
Here $(A_i,q^{A_i})$ is an object of $\P_1\A$, that is, $A_i:\P_1\to Ob\A:r\mapsto (A_i)_r$ is a finitely supported map of sets, and $q^{A_i}= (q^{A_i}_{rs})_{r,s\in \P_1}$ is a matrix with $q^{A_i}_{rs}\in \A((A_i)_s,(A_i)_r)^1$ such that 
$q^{A_i}_{rs}=0$ for $r\red{\nless} s$ in $\P_1$ and $(q^{A_i})^2+d^{\A}q^{A_i}=0$.
The map $Q_{ij}$ itself is a matrix with entries $(Q_{ij})_{rs}\in \A((A_j)_s, (A_i)_r)^1$.
Now, the image $(B,q)$ of $(A,Q)$ under the map (\ref{P0(P1)=P0P1}) is given by the map of sets $B:\P_0\times \P_1 \to \A$ defined by 
$$B_{i,r}=(A_i)_r,\hspace{3ex}i\in \P_0,\ r\in \P_1$$
and the matrix $q=(q_{ir,js})_{i,j\in \P_0,\ r,s\in\P_1}$ has entries
$$\renewcommand\arraystretch{1.5}
q_{ir,js} = \left\{\begin{matrix} (Q_{ij})_{rs} & i\neq j\\
(q^{A_i})_{rs} & i=j. \end{matrix}\right.$$
On morphisms, the functor (\ref{P0(P1)=P0P1}) is given by the ``identity''.
\end{remark}

\begin{definition}
\label{dfn:dgcatofextensions}
Let $\A$ be a pointed dg category.
The {\em dg category of extensions} of $\A$ is the pointed dg category
$$\Z(\A)$$
defined in Section \ref{subsec:dgCatExtn}
where $\Z$ is the poset of integers with its usual ordering.
The dg category of extensions is
equipped with the fully faithful inclusion of dg categories $\A \subset \Z\A$
induced by the embedding of posets $\{0\} \subset \Z$. 
\end{definition}

\red{
\begin{example}
\label{exn:EasDG}
If we consider an exact category $\E$ as a dg category with $\E^0=\E$ and $\E^i=0$ for $i\neq 0$,
then its dg category of extensions $\Z\E$ is (equivalent) to $\E$.
If we consider a ring $R$ as a dg category with one object, then $\Z R$ is the category of finitely generated free $R$-modules.
\end{example}
}

\begin{lemma}
\label{lem:ZAisExact}
Let $\A$ be a pointed dg category.
Then the dg category of extensions $\Z\A$ is exact.
\end{lemma}

\begin{proof}
The axioms (Ex0) and (Ex1) hold in any dg category.
For (Ex2), we have to show that $\Z\A$ is an additive category.
Every object $(A,q)$ in $\Z\A$ is isomorphic to its shift $(A[1],q[1])$ where $A[1]_i=A_{i+1}$ and $q[1]_{i,j} = q_{i+1,j+1}$.
Hence, for any two objects $(A,q^A)$ and $(B,q^{B})$ we can assume that they have disjoint support.
Then $(A,q^A)\oplus (B,q^B) = (C,q^C)$ where 
$$\renewcommand\arraystretch{1.5}
 C_i = \left\{ \begin{matrix} A_i & i\in \supp(A) \\ B_i & i\in \supp(B) \\ 0 & \text{otherwise}, \end{matrix} \right.
\hspace{8ex}
q^C_{ij} = \left\{ \begin{matrix} q^A_{ij} & i,j\in \supp(A) \\ q^B_{ij} & i,j\in \supp(B) \\ 0 & \text{otherwise.} \end{matrix} \right.
$$

We are left with showimg axioms (Ex3) and (Ex4) for inflations, the case of deflations being dual.
Let $[2]$ be a the poset $\{0<1<2\}$.
Let $\B$ be a dg category and let 
$$A_{i-1} \stackrel{f_i}{\rightarrowtail} A_i \stackrel{g_{i}}{\twoheadrightarrow} B_{i} $$
be admissible exact sequences in $Z^0\B$ where $i = 1,2$.
So, there are $r_i:A_i \to A_{i-1}$, $s_{i}:B_{i} \to A_i$ in $\B^0$ such that
$1_{A_{i-1}}= r_if_{i}$, $1_{A_i}= f_ir_i + s_{i}g_{i}$, and $1_{B_{i}} = g_{i}s_{i}$.
Note that 
$(dr_i)f_i = 0$, $g_{i}(ds_{i})=0$ and thus $(dr_i)(ds_{i}) = (dr_i)(f_ir_i + s_{i}g_{i})(ds_{i}) = 0$.
Note also that $(dr_i)s_i+r_i(ds_i) = 0$ because $r_is_i=0$.
%The sequence of conflations $A_0 \rightarrowtail A_1 \rightarrowtail A_2$ 
%is isomorphic in $Z^0[2]\B$ to the sequence
In $Z^0([2]\B)$ we have a commutative diagram 
$$\xymatrix{
A_0 \ar[r]^{f_1} \ar[d]^1 
& A_1 \ar[rr]^{f_2} 
\ar[d]^{\left(\begin{smallmatrix}r_1 \\ g_{1} \end{smallmatrix}\right)} 
&& A_2 
\ar[d]_{\left(\begin{smallmatrix}r_1r_2 \\ g_{1}r_2\\g_{2} \end{smallmatrix}\right)}\\
 A_0  \ar[r]_{\hspace{-10ex}\left(\begin{smallmatrix}1\\ 0\end{smallmatrix}\right)} 
& 
\left(A_0,B_{1},
{\left(\begin{smallmatrix} 0 & r_1(ds_{1}) \\ 0&0\end{smallmatrix}\right)}
\right) 
\ar[rr]_{\hspace{-12ex}\left(\begin{smallmatrix}1 & 0 \\ 0 & 1 \\ 0&0 \end{smallmatrix}\right)} 
& &
\hspace{1ex}
\left(A_0,\hspace{1ex}B_{1},\hspace{1ex}B_{2}, \hspace{2ex}{\left(\begin{smallmatrix}0 & r_1(ds_{1}) & r_1r_2(ds_{2})\\ 0&0 & g_{1}r_2(ds_{2}) \\ 0&0&0\end{smallmatrix}\right)}\right)
}$$
in which the vertical arrows are isomorphisms with inverses given by
$1_{A_0}$, $(f_1,s_{1})$ and $(f_2f_1, f_2s_{1},s_{2})$.
It follows that the composition $f_2f_1$ is an inflation in $Z^0([2]\B)$ with cokernel 
$$\coker(f_2f_1)\cong (B_{1},B_{2},\left(\begin{smallmatrix}0 & g_{1}r_2(ds_{2}) \\ 0&0\end{smallmatrix}\right)).$$
If $h:A_0 \to X$ is a map in $Z^0\B$, we have a push-out diagram in $Z^0([2]\B)$ as follows
$$\xymatrix{ 
 A_0 \ar[r]^{\hspace{-8ex}\left(\begin{smallmatrix}1\\ 0\end{smallmatrix}\right)} \ar[d]_h & 
(A_0,B_{1},{\left(\begin{smallmatrix}0 & r_1(ds_{1})\\ 0&0\end{smallmatrix}\right)}) \ar[d]_{\left(\begin{smallmatrix}h & 0\\ 0&1\end{smallmatrix}\right)} \\
 X \ar[r]^{\hspace{-8ex}\left(\begin{smallmatrix}1\\ 0\end{smallmatrix}\right)} & 
(X,B_{1},{\left(\begin{smallmatrix}0 & hr_1(ds_{1})\\ 0&0\end{smallmatrix}\right)}).}$$
Therefore, compositions of inflations in $Z^0\B$ are inflations in $Z^0([2]\B)$, and push-outs of inflations in $Z^0\B$ along any map in $Z^0\B$ exist in $Z^0([2]\B)$ and are inflations.

This applies in particular to the dg category $\B = \Z\A$.
We claim that the inclusion $\{0\} \subset [2]$ induces an equivalence of dg categories $\Z\A \subset [2](\Z\A)$, and hence, (Ex3) and (Ex4) hold in $\Z\A$ for inflations.
To prove the claim, note that by Remark \ref{rmk:PPA} we have a fully faithful embedding of dg categories
$[2](\Z\A) \subset ([2]\Z)\A$ where $[2]\Z=[2]\times \Z$ is equipped with the lexicographic order.
Since both posets $[2]\Z$ and $\Z$ contain $\N$, the full embeddings
$\Z\A \subset [2](\Z\A) \subset ([2]\Z)\A$ are equivalences of dg categories,
in view of Remark \ref{rmk:P0AP1AessentiallySurj}.
\end{proof}

\subsection{The pretriangulated hull of a dg category}
Finally, we come to the definition of the pretriangulated hull of a dg category;
compare \cite{bondalKapranov:pretr}, \cite{Drinfeld:DGquotient}.

\begin{definition}
\label{dfn:pretr}
Let $\A$ be a pointed dg category.
The {\em pretriangulated hull} of $\A$ is the pointed dg category 
$$\A^{\pretr} = \Z\left(\scC_{\bk}\A\right).$$
It is equipped with the full inclusion $\A \subset \A^{\pretr}$
induced by the inclusions in Example \ref{ex:Ck} and Definition \ref{dfn:dgcatofextensions}:
$\A \subset \scC_{\bk}\A  \subset \Z(\scC_{\bk}\A)$.
By Lemma \ref{lem:ptrprops} below, the dg category $\A^{\pretr}$ is indeed pretriangulated, and the inclusion $\A\subset \A^{\pretr}$ is an equivalence if and only if $\A$ is pretriangulated.
\end{definition}

\red{
\begin{remark}
The model of the pretriangulated hull of a dg category $\A$ given in \cite{bondalKapranov:pretr} and\cite{Drinfeld:DGquotient} is the dg category $\Z(\scT_k \otimes \A)$ where $\scT_k \subset \scC_k$ is the full dg subcategory whose objects are all shifts $k[i]$ of $k$,  $i\in \Z$.
The canonical inclusion $\Z(\scT_k \otimes \A) \subset \Z(\scC_k \otimes \A)$ is an equivalence of dg categories. This follows from Lemma \ref{lem:ptrprops} below since both dg categories are pretriangulated (hulls of $\A$).
\end{remark}
}

\begin{remark}
\label{rmk:ActiononAptr}
There is an action
\begin{equation}
\label{eqn:CkActiononAptr}
\otimes : \scC_{\bk} \otimes \A^{\pretr} \longrightarrow \A^{\pretr}
\end{equation}
of the symmetric monoidal category $\scC_{\bk}$ on $\A^{\pretr}$ 
given by the composition of dg functors
$$\scC_{\bk} \otimes \Z(\scC_{\bk}\A) \longrightarrow \Z(\scC_{\bk}\otimes \scC_{\bk}\A) \longrightarrow \Z( \scC_{\bk}\otimes \A)$$
in which the first map is the monoidal compatibility map (\ref{eqn:MonoidalComMap}) with $\P_0=[0]$ and $\P_1=\Z$, and the second map is induced by $\otimes: \scC_{\bk}\otimes \scC_{\bk} \to \scC_{\bk}$.

For objects $A$ of $\scC_{\bk}$ and $X$ of $\A^{\pretr}$, we may write 
$AX$ for the image in $\A^{\pretr}$ of the object $A\otimes X$ under the map (\ref{eqn:CkActiononAptr}).
In this sense, for any object $X$ of $\A^{\pretr}$ the objects $CX$ and $TX$ are defined where $C,T\in \scC_k$ are as in Example \ref{ex:Ck}.

The action (\ref{eqn:CkActiononAptr}) makes $Z^0\A^{\pretr}$ into a complicial exact category in the sense of 
\cite[Definition 3.2.2]{mySedano}.
By \cite[A.2.16]{mySedano}, the exact category $Z^0\A^{\pretr}$ is therefore a Frobenius exact category, that is, an exact category with enough injective and projective objects, and injective and projective objects coincide. 
By \cite[A.2.16(f)]{mySedano}, an object of $Z^0\A^{\pretr}$ is injective and projective if and only if it is a direct factor of an object of the form $CX$ with $X\in \A^{\pretr}$.

In fact, the exact structure on $Z^0\A^{\pretr}$ coincides with its Frobenius exact structure defined in \cite[A.2.15]{mySedano}.
This is because for an inflation $f:X \rightarrowtail Y$ in $Z^0\A^{\pretr}$ with retract $r:Y \to X$ in $\red{(\A^{\pretr})^0}$ and any map $g: X \to CU$ in $Z^0\A^{\pretr}$, the map $g$ is the composition 
$$X \stackrel{f}{\longrightarrow} Y \stackrel{1\otimes r-\eps\otimes dr}{\longrightarrow} CX \stackrel{1_C\otimes g}{\longrightarrow} CCU \stackrel{m\otimes 1}{\longrightarrow} CU$$
where $m:CC \to C$ is the multiplication in the dga $C$ \red{(use that $(dr)f=0$)}.
\end{remark}

\red{
\begin{remark}
The action (\ref{rmk:ActiononAptr}) of $\scC_k$ on $\A^{\pretr}$ plays an important role in the rest of the paper. 
This is the reason we chose $\A^{\pretr} = Z(\scC_k\A)$ as our model of the pretriangulated hull of $\scA$ rather than the one given in  \cite{bondalKapranov:pretr}, \cite{Drinfeld:DGquotient}.
\end{remark}
}

\begin{lemma}
\label{lem:ptrprops}
\phantom{1}
\vspace{0ex}

\begin{enumerate}
\item
\label{item1:lem:ptrprops}
Let $\A \subset \B$ be a full inclusion of pointed dg categories with $\A$ pretriangulated. Then if in an exact sequence in $Z^0\B$
$$0 \to X \to Y \to Z \to 0,$$
$X$ and $Z$ are in $\A$ then $Y$ is (isomorphic to an object) in $\A$. 
\item
\label{item2:lem:ptrprops}
If $\A$ is a pretriangulated dg category, then the inclusion 
$$\A \subset \Z \A$$
is an equivalence of dg categories.
\item
\label{item3:lem:ptrprops}
For any pointed dg category $\A$, the dg category $\A^{\pretr}$ is pretriangulated.
\item
\label{item4:lem:ptrprops}
A pointed dg category $\A$ is pretriangulated if and only if the inclusion
$$\A \subset \A^{\pretr}$$
is an equivalence.
\end{enumerate}
\end{lemma}

\begin{proof}
(\ref{item1:lem:ptrprops})
Embedding $\B$ into $\B^{\pretr}$, we can assume $Z^0\B$ to be a Frobenius exact category as in Remark \ref{rmk:ActiononAptr}.
Since $\A$ is pretriangulated, there is an admissible monomorphism $i\otimes 1_X: X \rightarrowtail CX$ in $Z^0\A$ (use the equivalence $\A \subset \scC \A$).
The map $X \to Y$ is an admissible monomorphism in $Z^0\B$, and $CX$ is an injective object of $Z^0\B$ (Remark \ref{rmk:ActiononAptr}).
It follows that the map $X \rightarrowtail CX$ extends to a map $Y \to CX$, and we obtain a map of exact sequences in $Z^0\B$
$$\xymatrix{X \xymono[r] \ar@{=}[d] & Y \ar[d]\xyepi[r] & Z \ar[d]\\
X \xymono[r] & CX \xyepi[r] & TX.}
$$
The right hand square is a pull-back square in $Z^0\B$.
Since $\A$ is exact and the maps $CX \to TX$ and $Z \to TX$ are in $Z^0\A$ with $CX \twoheadrightarrow TX$ an admissible epimorphism in $Z^0\A$, the pull-back $P$ of 
$CX \twoheadrightarrow TX \leftarrow Z$ exists in $Z^0\A$.
Since $\A \subset \B$ is exact, it preserves pull-backs along admissible epimorphisms.
Hence, $Y$ is isomorphic to $P$, an object of $\A$.

(\ref{item2:lem:ptrprops})
Every object $X$ of $\Z\A$ has a filtration $0=X_0 \subset X_1\subset ...\subset X_n=X$ with $X_i\subset X_{i+1}$ an admissible monomorphism in $\Z\A$ and quotients $X_{i+1}/X_i$ in $\A$.
By (\ref{item1:lem:ptrprops}), the object $X$ is (isomorphic to an object) in $\A$.

(\ref{item3:lem:ptrprops})
By Lemma \ref{lem:ZAisExact}, the dg category $\A^{\pretr}$ is exact.
Moreover, the inclusion $\A^{\pretr}\subset \scC_{\bk}\A^{\pretr}$ is an equivalence 
since the fully faithful dg functors in Remark \ref{rmk:ActiononAptr} 
$$\Z\scC_k\A \subset \scC_k \Z\scC_k\A \to \Z\scC_k\scC_k\A \stackrel{\otimes}{\to} \Z\scC_k\A$$
are essentially surjective.
Therefore, $\A^{\pretr}$ is indeed pretriangulated.

(\ref{item4:lem:ptrprops})
If $\A$ is pretriangulated, then $\scC\A$ is equivalent to $\A$ and pretriangulated.
Therefore, and in light of (\ref{item2:lem:ptrprops}), 
the inclusions $\A \subset \scC\A \subset \Z(\scC\A) = \A^{\pretr}$ are equivalences of dg categories.
Conversely, if $\A \subset \A^{\pretr}$ is an equivalence then $\A$ is pretriangulated in view of (\ref{item3:lem:ptrprops}).
\end{proof}

\red{
\begin{example}
Let $\E$ be an exact category considered as a dg category as in Example \ref{exn:EasDG}.
Then $\E^{\pretr}=\Ch^b\E$ since every object of the pretriangulated dg category $\Ch^b\E$ has a filtration with quotients shifts of objects of $\E$.
If $R$ is a ring considered as a dg category with one object, then $R^{\pretr}$ is the dg category of bounded  complexes of finitely generated free $R$-modules.
\end{example}
}

\begin{remark}
\label{rmk:H0isFrobeniusStableCat}
If $\A$ is a pretriangulated dg category, then we know (Remark \ref{rmk:ActiononAptr} and Lemma \ref{lem:ptrprops} (\ref{item4:lem:ptrprops})) that its underlying category $Z^0\A$ is a Frobenius exact category.
I claim that a map $f$ in $Z^0\A$ is zero in $H^0\A$ if and only if it factors through some injective projective object of $Z^0\A$, and therefore $H^0\A$ is the stable category $\underline{Z^0\A}$ of the Frobenius category $Z^0\A$:
$$H^0\A = \underline{Z^0\A}.$$
To prove the claim, note that objects of the form $CX$ are zero in $H^0\A$ because its identity morphism $1_{CX}=1_C\otimes 1_X = d(\eps \otimes 1_X)$ is the zero map in $H^0\A$. 
Since injective projective objects in $Z^0\A$ are the direct factors of objects of the form $CX$ (Remark \ref{rmk:ActiononAptr}), this shows that they all are zero in $H^0\A$ and that a map $f$ in $Z^0\A$ which is zero in the stable category is zero in $H^0\A$.
Conversely, assume that the map $f:X \to Y$ in $Z^0\A$ is zero in $H^0\A$ then $f = dg$ for some $g \in \A(X,Y)^1$.
But then $f$ factors through $CX$ and hence is zero in the stable category since $f$ is the composition in $Z^0\A$ 
$$X \stackrel{i\otimes 1}{\longrightarrow} CX \stackrel{\eps^{\vee}\otimes g}{\longrightarrow} Y$$
where $\eps^{\vee} \in [C,\1]^1$ is the unique degree-one map $C \to \1$ with 
$\eps^{\vee}(\eps)=1$.
\end{remark}

\subsection{The homotopy category of a dg category}
\label{subsec:htpycat}
Let $\A$ be a dg category.
Recall that the homotopy category $H^0\A^{\pretr}$ of its pretriangulated hull 
is the stable category of the Frobenius category $Z^0\A^{\pretr}$ (Remark \ref{rmk:H0isFrobeniusStableCat}).
Therefore, it is equipped with the structure of a triangulated category \cite{Happel:FrobTriang}, \cite{Keller:uses}.
The {\em shift functor} of the triangulated category is the functor
$T: H^0\A^{\pretr} \to H^0\A^{\pretr}: X \mapsto TX$ defined by the object $T\in \scC_k$ using the action (\ref{eqn:CkActiononAptr}).
The exact triangles in $H^0\A^{\pretr}$ are 
those which are isomorphic in $H^0\A^{\pretr}$
to a triangle of the following form, called {\em standard exact triangle}, 
\begin{equation}
\label{eqn:standardTriangle}
X \stackrel{f}{\to} Y \stackrel{g}{\to} C(f) \stackrel{h}{\to} TX
\end{equation}
defined by the following map of exact sequences in $\A^{\pretr}$
\begin{equation}
\label{eqn:exTrsDfn}
\xymatrix{
X\hspace{1ex} \ar@{}[dr]|{\square} \ar@{>->}[r]^{i\otimes 1} \ar[d]_f & CX \ar@{->>}[r]^{p\otimes 1} \ar[d] & TX \ar[d]^1\\
Y\hspace{1ex} \ar@{>->}[r]_{g} & C(f) \ar@{->>}[r]_h & TX}
\end{equation}
where $f:X \to Y$ is any morphism of $\A^{\pretr}$.
The left-hand square is cocartesian in $\A^{\pretr}$.
It exists because $\A^{\pretr}$ is exact and is equipped with an action of $\scC_k$; see Remark \ref{rmk:ActiononAptr}.
The object $TX$ is the {\em shift} of $X$, the object $CX$ is called the {\em cone} of $X$ and the object $C(f)$ is called the {\em cone} of $f$ where $T$ and $C$ are the objects of $\scC_k$ defined in Example \ref{ex:Ck}.

\begin{remark}
As a triangulated category, the category $H^0\A^{\pretr}$ is generated by $\A$, that is, every object of $H^0\A^{\pretr}$ is obtained from $\A$ by iteration of taking finite direct sums, shifts and cones.
This is because every object of $\Z\scC_{\bk}\A$ has a filtration with quotients in $\scC_{\bk}\A$, and every object in $\scC_{\bk}\A$ is obtained from $\A$ by iteration of taking finite direct sums, shifts and cones (since $\scC_{\bk}$ is obtained from $\bk$ in this way).
\end{remark}

\begin{definition}
A {\em localization pair} \cite{keller:cyclic}, also called {\em dg pair}, is a pair $(\A,\A_0)$ consisting of a pointed dg category $\A$ and a full pointed dg subcategory $\A_0 \subset \A$.
A morphism of localization pairs $(\A,\A_0) \to (\B,\B_0)$ is a dg functor $\A\to \B$ which sends $\A_0$ into $\B_0$.
The {\em associated triangulated category} of a localization pair 
$(\A,\A_0)$ is the
Verdier quotient triangulated category 
$$\T(\A,\A_0) = (H^0\A^{\pretr})/(H^0\A_0^{\pretr}).$$
The associated triangulated category is equipped with an additive functor 
$$Z^0\A \to \T(\A,\A_0)$$
 obtained as the composition
$Z^0\A \subset Z^0\A^{\pretr} \to H^0\A^{\pretr} \to \T(\A,\A_0)$.
Write $A^{\sat}$
for the full dg subcategory of $\A$ consisting of the objects $A \in\A$ which are zero in the associated triangulated category $\T(\A,\A_0)$.
The dg pair is called {\em saturated} if $\A_0 = \A^{\sat}$.
The saturation of $(\A,\A_0)$ is the dg pair $(\A,\A^{\sat})$.
Note that a dg pair and its saturation have the same associated triangulated category.
\end{definition}

\red{
\begin{example}
\label{ex:DGpairChbE}
Let $\E$ be an exact category, and let $\Ac^b_0\E \subset \Ch^b\E$ be the full dg subcategory of those complexes $E^*$ such that each $d^i:E^i \to E^{i+1}$ has a factorization $E^i \twoheadrightarrow \im(d^i) \rightarrowtail E^{i+1}$ in $\E$ into a deflation followed by an inflation such that
$\im(d^{i-1}) \rightarrowtail E^i \twoheadrightarrow \im(d^i)$ is a conflation for all $i\in \Z$.
Then $(\Ch^b\E,\Ac^b_0\E)$ is a dg pair with saturation $(\Ch^b\E,\Ac^b\E)$ where $\Ac^b\E \subset \Ch^b\E$ is the full dg subcategory of those complexes which are homotopy equivalent to an object of $\Ac^b_0\E$.
If $\E$ is idempotent complete (or more generally, if every map that admits a retraction is an inflation), then $\Ac_0^b\E = \Ac^b\E$, and $(\Ch^b\E,\Ac^b\E)$ is saturated.
The associated triangulated category $\T(\Ch^b\E,\Ac^b\E)$ is the usual bounded derived category $\D^b(\E)$ of the exact category $\E$.
For details; see \cite{neeman:excat}.
\end{example}
}

\subsection{DG categories with weak equivalences}
\label{subsec:DGweakEq}
Let $\A$ be a pointed dg category, and let $w\subset Z^0\A^{\pretr}$ be a set of morphisms containing all isomorphisms.
Let $\A^w$ be the full dg subcategory of $\A$ consisting of those objects $A$ for which the map $0_{\A} \to A$ is in $w$.
Such objects are called $w$-acyclic objects (in $\A$).
Note that the base point zero object $0_{\A}$ is in $\A^w$ since the identity on $0_{\A}$ is in $w$.
We call the set $w$ {\em saturated in $\A$} if a map $f$ in $Z^0\A^{\pretr}$ is in the set $w$ if and only if 
the functor $\A \to \T(\A,\A^w)$ sends $f$ to an isomorphism.

\begin{definition}
\label{dfn:dgCatW}
A {\em dg category with weak equivalences} is a pair $\scA = (\A,w)$ where $\A$ is a pointed dg category and $w\subset Z^0\A^{\pretr}$ is a set of morphisms
which is saturated in $\A$.
A map $f$ in $w$ is called {\em weak equivalence}.
The {\em triangulated category of a dg category with weak equivalences} $\scA = (\A,w)$ is the triangulated category
$$\T\scA = \T(\A,w)= \T(\A,\A^w)$$
of the associated saturated dg pair $(\A,\A^w)$.
\end{definition}

\red{
\begin{example}
\label{ex:ChbEisDGwithWeaks}
Continuing Example \ref{ex:DGpairChbE}, denote by $\quis$ the set of chain maps in $Z^0\Ch^b\E$ which are quasi-isomorphisms, that is, whose cones are in $\Ac^b\E$.
Then $(\Ch^b\E,\quis)$ is a dg category with weak equivalences.
The associated triangulated  category $\T(\Ch^b\E,\quis)$ is the usual bounded derived category $\D^b\E$ of the exact category $\E$.
\end{example}
}

A dg functor $F: \A \to \B$ between dg categories with weak equivalences $(\A,w)$ and $(\B,w)$ is called {\em exact} if it sends weak equivalences to weak \red{equivalences}, or equivalently, if it sends $\A^w$ to $\B^w$.
Such a functor induces a triangle functor $\T(\A,w) \to \T(\B,w)$ which strictly commutes with the shifts in both categories.
We write 
$$\dgCatW$$
 for the category of small dg categories with weak equivalences. 
Note that this category is the same as the category of saturated dg pairs 
where a dg category with weak equivalences $(\A,w)$ corresponds to the saturated dg pair $(\A,\A^w)$.
This allows us to switch freely between the two concepts.

A pointed dg category $\A$ is considered a dg category with weak equivalences where the weak equivalences are the homotopy equivalences, that is the maps which are isomorphisms in $H^0\A \subset H^0\A^{\pretr}$.
Equivalently, it is the saturation of the localization pair $(\A,0)$.

The category $\dgCatW$ is closed symmetric monoidal.
The tensor product 
$$(\A,w)\otimes (\B,w)$$ is the saturation of the dg pair 
$(\A\otimes\B,\A^w\otimes \B \cup \A \otimes \B^w)$.
The function object 
$$\dgFun(\A,w;\B,w)$$ is the saturation of the dg pair $(\dgFun^w(\A,\B), \dgFun(\A;\B^w))$
where $\dgFun^w(\A,\B)$ is the full dg subcategory of $\dgFun(\A,\B)$ of those dg functors that preserve weak equivalences.
The unit is the pointed dg category $\bk_+$ associated with the base ring $\bk$.

\subsection{Categories with duality}
\label{subsec:CatsWDual}
Recall from \cite{myEx}, \cite{myMV} that
a {\it category with duality} is a triple $(\C,*,\eta)$ with $\C$ a category, $*: \C^{op} \to
\C$ a functor, $\eta: 1 \to *\circ *^{op}$ a natural transformation, called {\em double
  dual identification}, such that 
$1_{A^{*}}=\eta_A^{*}\circ \eta_{A^{*}}$ for all objects $A$ in $\C$.
If $\eta$ is a natural isomorphism, we say that the duality is {\it strong}.
In case $\eta$ is the identity (in which case $**=id$), we call the duality
{\it strict}.

A {\it symmetric form}  in a category with duality
$(\C,*,\eta)$ is a pair $(X,\ffi)$ where $\ffi:X \to
X^{*}$ is a morphism in $\C$ satisfying $\ffi^{*}\eta_X=\ffi$.
A map of symmetric forms $(X,\ffi) \to (Y,\psi)$ is a map $f:X \to Y$ in $\C$
such that $\ffi=f^*\circ\psi\circ f$.
Composition of such maps is composition in $\C$.
For a category with duality $(\C,*,\eta)$, we denote by
$\C_h$ the {\em category of symmetric forms in $\C$}.
It has objects the symmetric forms in $\C$ and its morphisms are the maps of symmetric forms as defined above. 

A {\em form functor}  from a category with duality
$(\A,*,\alpha)$ to another such category $(\B,*,\beta)$  is a pair $(F,\ffi)$
with 
$F:\A \to \B$ a functor and $\ffi: F* \to *F$ a natural transformation, called
{\em duality compatibility morphism},
such that $\ffi_A^*\beta_{FA}=\ffi_{A^*}F(\alpha_A)$ for every object $A$ of
$\A$.
There is an evident definition of composition of form functors; see
\cite[3.2]{myEx}. 
We write $\CatD$ for the category of small categories with duality and form functors as morphisms.
A natural transformation $f:(F,\ffi) \to (G,\psi)$ of form functors $\A \to \B$ is a natural transformation of functors $f:F\to G$ such that for all objects $A$ of $\A$ we have $f_A^*\circ \psi_A\circ f_{A^*} = \ffi_A$.

If $\A$ and $\B$ are categories with duality, then
the category $\Fun(\A,\B)$ of functors $\A \to \B$ is a category with duality,
where the dual
$F^{\sharp}$ of a functor $F$ is $*F*$, and the double dual identification 
$\eta_F: F \to F^{\sharp\sharp}$ at an object $A$ of $\A$ is the map 
$\beta_{F(A^{**})}\circ F(\alpha_A) = F(\alpha_A)^{**}\circ \beta_{FA}$.
To give a form functor $(F,\ffi)$ is the same as to give a symmetric form
$(F,\hat{\ffi})$ in the category with duality $\Fun(\A,\B)$
in view of the formulas ${\ffi}_A=F(\alpha_A)^*\circ \hat{\ffi}_{A^*}$ and
$\hat{\ffi}_A={\ffi}_{A^*}\circ F(\alpha_A)$.
A natural transformation $(F,\ffi) \to (G,\psi)$ of form functors is the same as a map
$(F,\hat{\ffi}) \to (G,\hat{\psi})$ of symmetric forms in $\Fun(\A,\B)$.

A {\em duality preserving functor} between categories with duality
$(\A,*,\alpha)$ and $(\B,*,\beta)$ is a functor $F: \A \to \B$ which commutes
with dualities and double dual identifications, that is, we have
$F*=*F$ and $F(\alpha)=\beta_{F}$.
In this case, $(F,id)$ is a form functor.
We write $\CatD_{str}$ for the category of small categories with strict duality and duality preserving functors as morphisms.

\red{
A {\em $k$-linear exact category with duality} is a $k$-linear category with duality $(\E,*,\eta)$ where $\E$ is an exact category, $*:\E^{op}\to \E$ is a $k$-linear exact functor, and $\eta$ is a natural isomorphism.
}

\red{
\begin{example}
\label{exn:VectAsExCatDual}
Let $X$ be a $k$-scheme, $L$ a line bundle on $X$, and $\Vect(X)$ the exact category of vector bundles on $X$.
Then $$(\Vect(X),\sharp_L,\can^L)$$
is a $k$-linear exact category with duality where $E^{\sharp_L}$ is the sheaf of $O_X$-module homomorphisms $Hom_{O_X}(E,L)$ for $E\in \Vect(X)$, and $\can^L_E:E \to E^{\sharp_L\sharp_L}$ is the canonical double dual identification which (locally on X) is the evaluation at $x$ for $x\in E(U)$ and $U\subset X$ an affine open subset.
Note that $(\Vect(X),\sharp_L,-\can^L)$ is also a $k$-linear exact category with duality.
\end{example}
}

\subsection{DG categories with duality}
If $\A$ is a pointed dg category, its opposite pointed dg category 
$\A^{op}$ has the same objects and base point as $\A$ and the mapping complexes are
$\A^{op}(X,Y) = \A(Y,X)$.
Composition is defined by $f^{op}\circ g^{op} = (-1)^{|f||g|}(g\circ f)^{op}$ where $f^{op}$ and $g^{op}$ are the maps in $\A^{op}$ corresponding to the homogeneous maps $f$ and $g$ in $\A$.
For a dg functor $F:\A \to \B$, the assignment $F^{op}:\A^{op} \to \B^{op}: X \mapsto F(X), f^{op} \mapsto F(f)^{op}$ defines a dg functor between opposite categories.
The identity $\A^{op}\otimes \B^{op} \to (\A\otimes\B)^{op}:  (A,B)\mapsto (A,B), f^{op}\otimes g^{op}\mapsto (f\otimes g)^{op}$ is an isomorphism of dg categories and lets us identify these two dg categories.
Note that $(\A^{op})^{op}=\A$.

\begin{definition}
\label{dfn:dgCatD}
A {\it pointed dg category with duality} is a triple $(\A,\vee,\can)$ where $\A$ is a pointed dg category, $\vee:\A^{op} \to \A$ is a dg functor, called {\em duality functor}, and $\can:1 \to \vee\circ \vee^{op}$ is a natural transformation of dg functors (that is, an element of $Z^0[1,\vee\circ \vee^{op}]$), called {\em double dual identification}, such that $\can_A^{\vee}\circ \can_{A^{\vee}}=1_{A^{\vee}}$ for all objects $A$ in $\A$.
Note that the duality functor satisfies $(f\circ g)^{\vee}=(-1)^{|f|\cdot |g|}g^{\vee}\circ f^{\vee}$ for composable homogeneous morphisms $f$ and $g$ of degrees $|f|$ and $|g|$. 
Any dg category with duality $(\A,\vee,\can)$ defines a category with duality $(Z^0\A,\vee,\can)$ by restriction of structure along the inclusion $Z^0\A \subset \A$.
\end{definition}

\red{
\begin{example}
\label{ex:ChbEduality}
Let $(\E,*,\eta)$ be a $k$-linear exact category with duality.
We will endow the dg category $\Ch^b\E$ of bounded  complexes in $\E$ (Example \ref{exn:ChbE}) with the structure 
$$(\Ch^b\E,*,\eta)$$ 
of a dg category with strong duality such that the canonical embedding $\E \to \Ch^b\E$ is duality preserving.
On objects, the dg functor $\ast:(\Ch^b\E)^{op} \to \Ch^b\E$ is given by 
$$(E^*)^i = (E^{-i})^*,\hspace{3ex} (d^*)^i=(-1)^{i+1}(d^{-i-1})^*$$
for $(E,d) \in \Ch^b\E$.
On function complexes, the dg functor $\ast$ is the map of complexes $\ast:[M,N] \to [N^*,M^*]$ given by
$$\E(M^i,N^j) \to \E(N^{-i},M^{-j}): f\mapsto (-1)^{i(j-i)}f^*$$
for $M,N\in \Ch^b\E$.
The canonical double dual identification on $E\in \Ch^b\E$ is 
$(\eta_E)^i=(-1)^i\eta_{E^i}$.
The signs occurring in these formulas are chosen to be compatible with Subsection \ref{subsec:DualtiesCk}.
In case $\E=\Vect(X)$, they agree with the signs coming from the closed symmetric monoidal structure on $\Ch^b\Vect(X)$ as in Section \ref{sec:GWschemes}.
The signs here differ from the choices in \cite{Balmer:TWGII}.
For a comparison isomorphism; see \cite[Section 6.1]{myMV}.
\end{example}
}

A morphism $(\A,\vee,\can) \to (\B,\vee,\can)$ of dg categories with duality, also called {\em dg form functor}, is a pair $(F,\ffi)$ where $F:\A \to \B$ is a dg functor and $\ffi: F\circ \vee\to \vee\circ (F^{op})$ is a natural transformation of dg functors, called {\em duality compatibility morphism}, such that 
$\ffi_{A^{\vee}}\circ F(\can_A) = \ffi^{\vee}_A\circ \can_{FA}$ for all objects $A$ of $\A$.
Composition of $(F,\ffi):(\A,\vee,\can) \to (\B,\vee,\can)$ and $(G,\psi): (\B,\vee,\can) \to (\C,\vee,\can)$ is $(G\circ F, \psi_F\circ G(\ffi)): (\A,\vee,\can) \to (\C,\vee,\can)$.
Composition is associative and unital with unit on $\A$ the identity dg form functor $(id_{\A},id)$.
This defines the category 
$$\dgCatD_{\bk}$$
 whose objects are the small pointed dg categories with duality (over $\bk$) and whose morphisms are the dg form functors.

The category $\dgCatD_{\bk}$ is closed symmetric monoidal.
The tensor product 
$$(\A,\vee,\can^{\A})\otimes (\B,*,\can^{\B}) = (\A\otimes \B,\vee\otimes *,\can^{\A}\otimes \can^{\B})$$ has duality functor
$\vee \otimes *: (\A\otimes \B)^{op}=\A^{op}\otimes \B^{op} \to \A\otimes \B$ 
and double dual identification $\can_A^{\A}\otimes \can_B^{\B}:(A,B) \to (A^{\vee\vee},B^{**})$.
The unit of the tensor product is the dg category $\bk_+$ equipped with the trivial duality.
The switch $\tau:\A \otimes \B \to \B \otimes \A$ in $\dgCatD_{\bk}$ is the switch in $\dgCat_{\bk}$ with identity as duality compatibility map.
The internal function object of $(\A,\vee,\can^{\A})$ and $(\B,*,\can^{\B})$ is the dg category $\dgFun(\A,\B)$ of dg functors equipped with the duality
$$\sharp: \dgFun(\A,\B)^{op} \to \dgFun(\A,\B):F \mapsto F^{\sharp}= *\circ F^{op} \circ \vee^{op}$$
and double dual identification $\can^{\B}_F\circ F(\can^{\A}):F \to F^{\sharp\sharp}$.
Note that a dg form functor $(F,\ffi):\A \to \B$ between dg categories with duality is the same as a symmetric form in the dg category with duality $\dgFun(\A,\B)$ of dg functors from $\A$ to $\B$, or, in the notation of Section \ref{subsec:CatsWDual}, an object of the category $(Z^0\dgFun(\A,\B))_h$ of symmetric forms in $Z^0\dgFun(\A,\B)$.

If $\C$ is a category with duality and $\A$ a pointed dg category with duality, we write $\Fun(\C,\A)$ for the pointed dg category with duality 
$$\Fun(\C,\A) = \dgFun(k[\C]_+,\A).$$

\subsection{Dualities in $\scC_k$}
\label{subsec:DualtiesCk}
Recall that tensor products and function complexes make the category $\scC_{\bk}$ of
bounded  complexes of finitely generated free $\bk$-modules
into a closed symmetric monoidal category.
Therefore, an object $A$ in $\scC_{\bk}$ defines a pointed dg category with duality 
\red{
$$\scC^{[A]}_{\bk} = (\scC_{\bk},\vee_A,\can^A).$$
On objects the duality is defined by
$$\vee_A:\scC_{\bk}^{op} \to \scC_{\bk}: X \mapsto [X,A].$$
On morphism complexes it is the unique admissble natural transformation
$$[X,Y] \stackrel{\nabla}{\to} [[Y,A],[X,Y]\otimes [Y,A]] \stackrel{[1,\tau]}{\longrightarrow} [[Y,A],[Y,A]\otimes [X,Y]] \stackrel{[1,\circ]}{\longrightarrow} [[Y,A],[X,A]],$$
that is, the map
$$[X,Y] \to [[Y,A],[X,A]]: f \mapsto \left\{ g\mapsto (-1)^{|g||f|}g f \right\} $$
for homogenerous composable $f\in [X,Y]$ and $g\in [Y,A]$.
}
The canonical double dual identification 
$\can^A_X: X \to X^{\vee_A\vee_A}: x \mapsto \can_X^A(x)$
is given by 
$$\can_X^A(x)(f) = (-1)^{|x||f|}f(x)$$
for homogeneous $f\in [X,A]$ and $x\in X$.

Tensor product of complexes defines an equivalence of dg categories with duality
\begin{equation}
\label{equn:CACBtoCAB}
(\otimes, \can): \scC^{[A]}\otimes \scC^{[B]} \stackrel{\simeq}{\longrightarrow} \scC^{[A\otimes B]}
\end{equation}
where the duality compatibility isomorphism is 
\begin{equation}
\label{eqn:dualityCompIsoForTensorProds}
\can: [X,A]\otimes [Y,B] \to [X\otimes Y,A\otimes B]: f\otimes g \mapsto f\otimes g
\end{equation}
with $(f\otimes g)(x\otimes y) = (-1)^{|x||g|}f(x) \otimes g(y)$.
For later reference, we note that the following diagram of dg form functors commutes up \red{to} a natural isomorphism of form functors defined below
\begin{equation}
\label{eqn:CAandSwitches}
\xymatrix{
\scC^{[A]}\otimes \scC^{[B]} \ar[r]^{\hspace{2ex}(\otimes,\can)} \ar[d]_{\tau} & \scC^{[A\otimes B]} \ar[d]^{(id,\tau)} \\
\scC^{[B]}\otimes \scC^{[A]} \ar[r]^{\hspace{2ex}(\otimes,\can)} & \scC^{[B\otimes A]}
}
\end{equation}
In the diagram, the left vertical dg form functor is the switch in the symmetric monoidal 
$\dgCatD_{\bk}$ and the right vertical dg form functor is the identity functor equipped with the duality compatibility map 
$[1,\tau]: [X,A\otimes B] \to [X,B\otimes A]$ induced by the switch $\tau: A\otimes B \to B\otimes A$ in $\scC_{\bk}$.
The isomorphism of form functors between the two compositions in the diagram is given by the switch isomorphism $\tau:X\otimes Y \to Y\otimes X$ in $\scC_{\bk}$.

\begin{remark}
\label{rem:SymmToForm}
As in any closed symmetric monoidal category, a map $\mu: X \otimes Y \to A$ in $\scC_{\bk}$ 
defines, by adjunction, a map  
$$\ffi_{\mu}: X \stackrel{\nabla}{\to} [Y,X\otimes Y] \stackrel{[1,\mu]}{\to}
[Y,A]=Y^{\vee_A}$$
which satisfies $\ffi_{\mu}^{\vee_A} \circ \can_Y^A = \ffi_{\mu\circ \tau}$ where
$\mu\circ \tau: Y\otimes X \stackrel{\tau}{\to} X \otimes Y \stackrel{\mu}{\to} A$.
In particular, a map $\mu: X \otimes X \to A$ with $\mu\circ \tau = \mu$ defines
a symmetric form $\ffi_{\mu}$ on $X$, that is, a map $\ffi_{\mu}: X \to X^{\vee_A}$ satisfying $\ffi_{\mu}^{\vee_A}\can_X^A=\ffi_{\mu}$.
\end{remark}

\subsection{$\scC^{[n]}$ and shifted dualities}
\label{subsec:DGshiftedDuals}
For $n\in \Z$, we write $\scC_{\bk}^{[n]}$, or simply $\scC^{[n]}$, for the dg $\bk$-category with duality $\scC^{k[n]}$ where $k[n]$ is the dg $\bk$-module with $k[n]^{-n}=k$ and $k[n]^i=0$ for $i \neq -n$.
Tensor product defines an equivalence of dg categories with duality
\begin{equation}
\label{eqn:CiCj=Ci+j}
\mu_{i,j}:\scC^{k[i]}\otimes \scC^{k[j]} \stackrel{(\otimes,\can)}{\longrightarrow}\scC^{k[i]\otimes k[j]} \stackrel{(1,\mu)}{\longrightarrow} \scC^{k[i+j]}
\end{equation}
where the second form functor is the identity together with the duality compatibility map $[X,k[i]\otimes k[j]] \to [X,k[i+j]]$ induced by the isomorphism of dg $k$-modules
$k[i]\otimes k[j] \to k[i+j]:x\otimes y \mapsto xy$.
Note that under this isomorphism, the switch map $k[j] \otimes k[i] \to k[i]\otimes k[j]$ is identified with multiplication by $(-1)^{ij}:k[i+j]\to k[i+j]$.
Therefore, in the following diagram, the left hand square commutes up to the natural isomorphism of dg form functors as in diagram (\ref{eqn:CAandSwitches}), and the right hand square commutes
\begin{equation}
\label{equn:CiCjcommuting}
\xymatrix{
\scC^{k[i]}\otimes \scC^{k[j]} \ar[d]^{\tau} \ar[r]^{\hspace{2ex}(\otimes,\can)} & \scC^{k[i]\otimes k[j]} \ar[d]^{(1,\tau)} \ar[r]^{(1,\mu)} & \scC^{k[i+j]} \ar[d]^{(1,(-1)^{ij})}\\
\scC^{k[j]}\otimes \scC^{k[i]} \ar[r]^{\hspace{2ex}(\otimes,\can)} & \scC^{k[j]\otimes k[i]} \ar[r]^{(1,\mu)} & \scC^{k[i+j]}.
}\end{equation}

\begin{remark}
Let $(X,\ffi)$ be a symmetric form in $\scC^{[n]}$ and assume that $\ffi:X \to [X,k[n]]$ is an isomorphism.
The symmetric form defines a dg form functor
$$(X,\ffi): k_+ \to \scC^{[n]}: k\mapsto X,\ [k,k] \to [X,X]: a\mapsto a\cdot 1_X$$
with duality compatibility map the isomorphism $\ffi: X \to [X,k[n]]$.
Therefore, we obtain dg form functors
$$(X,\ffi)\otimes: \scC^{[m]} \cong k_+\otimes \scC^{[m]} \stackrel{(X,\ffi)\otimes  1}{\longrightarrow} \scC^{[n]}\otimes \scC^{[m]} \stackrel{\mu_{n,m}}{\longrightarrow} \scC^{[n+m]}.$$
\end{remark}

\begin{lemma}
\label{lem:CiCjcommuting}
The following diagram in $\dgCatD_{\bk}$ commutes up to natural isomorphism of dg form functors
$$\xymatrix{
\scC^{k[i]}\otimes \scC^{k[j]} \ar[d]^{\tau} \ar[r]^{\hspace{2ex}\mu_{i,j}} & \scC^{k[i+j]} \ar[d]^{ \langle -1\rangle^{ij}\otimes }\\
\scC^{k[j]}\otimes \scC^{k[i]} \ar[r]^{\hspace{2ex}\mu_{j,i}} &  \scC^{k[i+j]}
}$$
where $\langle -1 \rangle$ is the inner product space $(k,-1)$ in $\scC^{[0]}$ given by $x,y\mapsto -xy$.
\end{lemma}

\begin{proof}
This follows from diagram (\ref{equn:CiCjcommuting}) and the identification of its right vertical map with the right vertical map in the Lemma.
\end{proof}

\begin{remark}
If we denote by $m:k[2] \otimes k[2] \to k[4]$ the multiplication map, then tensor product with the symmetric form $(k[2],m)$ in $\scC^{[4]}$ induces isomorphisms of dg categories with duality
\begin{equation}
\label{eqn:Cj=Cj+4}
(k[2],m)\otimes :\scC^{[n]} \stackrel{\cong}{\longrightarrow} \scC^{[n+4]};
\end{equation}
compare \cite[Proposition 7]{myMV}.
\end{remark}

\begin{definition}
\label{dfn:nshifteddgcat}
Let $\A$ be a dg category with duality.
The {\em $n$-th shifted} dg category with duality is 
$$\A^{[n]} = \scC_{\bk}^{[n]}\otimes \A.$$
The equivalences (\ref{eqn:CiCj=Ci+j}) and isomorphisms (\ref{eqn:Cj=Cj+4}) induce equivalences and isomorphisms of dg-categories with duality
$$(\A^{[n]})^{[m]}=\scC_{\bk}^{[m]}\otimes \A^{[n]} \stackrel{\simeq}{\longrightarrow} \A^{[m+n]}\hspace{4ex}\text{and}$$
$$\A^{[n]} \cong \A^{[n+4]}.$$
The dg categories with duality $\A^{[n]}$ all have the same underlying dg category but are equipped with a duality depending on $n\in \Z$.
If the dg-category with duality $\A$ is pretriangulated, then so is $\A^{[n]}$, $n\in \Z$, and $\A \to \A^{[0]}$ is an equivalence of dg categories with duality.
In general, the dg form functor $\A \to \A^{[0]}$ always induces an equivalence on pretriangulated hulls for any dg category with duality $\A$.
\end{definition}

\subsection{The pretriangulated hull of a dg category with duality}
Let $\A$ be a dg category with duality.
We will make its pretriangulated hull $\A^{\pretr}$ into a dg category with duality such that the inclusion $\A \subset \A^{\pretr}$ is duality preserving.

For that, let $(\A,\vee,\can)$ be a dg category with duality.
Consider the ordered set $\Z$ as a category with strict duality $\Z^{op}\to \Z: n\mapsto -n$.
Then the dg category of extensions $\Z\A$ is equipped with the duality 
$$\vee: (\Z\A)^{op} \to \Z\A: (A,q) \mapsto (A,q)^{\vee} = (A^{\vee}, -q^{\vee})$$ where $(A^{\vee})_{i}=(A_{-i})^{\vee}$ and $q^{\vee}$ has entries $(q^{\vee})_{ij}=(q_{-j,-i})^{\vee}$.
On morphism complexes, the duality sends a matrix $f$ to the matrix $f^{\vee}$ with entries $(f^{\vee})_{ij}=(f_{-j,-i})^{\vee}$.
The double dual identification $\can:(A,q) \to (A,q)^{\vee\vee} = (A^{\vee\vee}, q^{\vee\vee})$ is the matrix with entries $(\can_{(A,q)})_{ij} = \can_{A_i}$ for $i=j\in \Z$ and $(\can_{(A,q)})_{ij} = 0$ for $i\neq j \in \Z$; see \S \ref{subsec:MoreOnExtCats} below.
Clearly, the inclusion $\A \subset \Z\A$ preserves dualities.

For a dg category with duality $\A$, the category $\scC_{\bk}^{[0]}\A$ is a dg category with duality, and the inclusion $\A \subset \scC_{\bk}^{[0]}\A:X\mapsto \1\otimes X$ is duality preserving.
Therefore, the pretriangulated hull $\A^{\pretr} = \Z\scC^{[0]}_{\bk}\A$ is a pretriangulated dg category with duality containing $\A$ as a full dg subcategory with duality.
If $\A$ is a pretriangulated dg category with duality then the inclusion 
$\A \subset \A^{\pretr}$ is an equivalence of dg categories with duality.

\red{
\begin{example}
Let $(\E,*,\eta)$ be a $k$-linear exact category with duality.
Recall from Example \ref{ex:ChbEduality} and the above discussion that $\Ch^b\E$ and $\E^{\pretr}$ are naturally equipped with the structure of a dg category with strong duality.
Since both dg categories are pretriangulated hulls for $\E$, the natural dg form functors 
$$\E^{\pretr} \to (\Ch^b\E)^{\pretr} \leftarrow \Ch^b\E$$
are equivalences of dg categories with dualities.
\end{example}
}

\begin{definition}
A {\em dg category with weak equivalences and duality} is a quadruple
$\scA = (\A,w,\vee,\can)$ where 
$(\A,w)$ is a dg category with weak equivalences and
$(\A,\vee,\can)$ is a dg category with duality such that the dg subcategory $\A^w\subset \A$ of $w$-acyclic objects is closed under the duality functor $\vee$ and
$\can_A:A \to A^{\vee\vee}$ is a weak equivalence for all objects $A$ of $\A$.
Note that then $f\in w$ if and only if $f^{\vee}\in w$.
We may sometimes omit $w$, $\vee$ or $\can$ from the notation when they are understood.
\end{definition}

\red{
\begin{example}
\label{ex:ChbEDGwithWeaksAndDuals}
Let $(\E,*,\eta)$ be an exact category with duality.
From Examples \ref{ex:ChbEisDGwithWeaks} and \ref{ex:ChbEduality} we obtain 
the dg category with weak equivalences and (strong) duality
$$(\Ch^b\E,\quis,*,\eta).$$
\end{example}
}

A morphism of dg categories with weak equivalences and duality (also called exact dg form functor) is a dg form functor $(F,\ffi)$ where $F$ an exact dg functor.
Compositions of exact dg form functors are exact dg form functors.
This defines the category $\dgCatWD_{\bk}$ of small dg $\bk$-categories with weak equivalences and duality.
Note that if $\scA = (\A,w)$ is a dg category with weak equivalences and duality, then so is its pretriangulated hull $\scA^{\pretr} = (\A^{\pretr},w)$.

Tensor product and function object given in Definition \ref{dfn:dgCatW} and \ref{dfn:dgCatD} 
make $\dgCatWD_{\bk}$ into a closed symmetric monoidal category.
If $\scA = (\A,w)$ is a dg category with weak equivalences and duality, then so is $\scA^{[n]} = (\A^{[n]}, w)$ which is the saturation of the dg pair $(\A^{[n]}, (\A^w)^{[n]})$.

\subsection{Grothendieck-Witt groups of dg categories}
\label{dfn:GW0}
A dg category with weak equivalences and duality 
$\scA = (\A,w,\vee,\can)$ defines an exact category with weak equivalences and duality 
$(Z^0\A^{\pretr},w,\vee,\can)$ in the sense of \cite[\S 2.3]{myMV}
with exact sequences as defined in Section \ref{subsec:ExDGCats}.
As such its Grothendieck Witt group $GW_0(\scA)$ was defined in \cite[Definition 1]{myMV}.
We will remind the reader of the definition below.
But first, recall that a symmetric space in $\scA^{\pretr}$ is a pair $(A,\ffi)$ where 
$\ffi: A\to A^{\vee}$ is a weak equivalence in $\scA^{\pretr}$ (in particular $\ffi \in Z^0\A$) such that $\ffi^{\vee}\can_A=\ffi$.

\begin{definition}[\cite{myMV}]
\label{dfn:GW0A}
Let $\scA = (\A,w,\vee,\can)$ be a dg category with weak equivalences and duality.
The Grothendieck-Witt group $GW_0(\scA)$ of $\scA$ is the abelian group 
generated by symmetric spaces $[X,\ffi]$ in $(Z^0\A^{\pretr},w,\vee,\can)$, subject to the following
relations 
\begin{enumerate}
\item
\label{cor:itm0:Kh_0}
$[X,\ffi]+[Y,\psi] = [X \oplus Y,\ffi \oplus \psi]$
\item
\label{cor:itm1:Kh_0}
if $g:X \to Y$ is a weak equivalence, then $[Y,\psi]=[X,g^{\vee}\psi g]$,
and
\item
\label{cor:itm2:Kh_0}
if $(E_{\bullet},\ffi_{\bullet})$ is a symmetric space in the category
of exact sequences in $Z^0\A^{\pretr}$, that is, a map 
$$\xymatrix{
\hspace{2ex}E_{\bullet} \ar[d]^{\ffi_{\bullet}}_{\wr}:  & E_{-1}
\xymono[r]^i\ar[d]^{\ffi_{-1}}_{\wr} & E_0 \xyepi[r]^p\ar[d]^{\ffi_{0}}_{\wr}
&  E_1 \ar[d]^{\ffi_{1}}_{\wr} \\
\hspace{2ex}E_{\bullet}^{\vee}: & E_1^{\vee} \xymono[r]_{p^{\vee}} & E_0^{\vee} \xyepi[r]_{i^{\vee}} &
E_{-1}^{\vee}}$$ 
of exact sequences with
$(\ffi_{-1},\ffi_0,\ffi_1)=(\ffi_1^{\vee}\can,\ffi_0^{\vee}\can,\ffi_{-1}^{\vee}\can) $ a 
weak equivalence, then $$[E_0,\ffi_0] = \left[ E_{-1}\oplus E_1,
\left(\begin{smallmatrix} 0 & \ffi_1\\ \ffi_{-1} &
    0\end{smallmatrix}\right)\right].
$$
\end{enumerate}
\end{definition}

Similarly, the Witt group $W_0(\scA)$ of $\scA$ is
the abelian group generated by symmetric spaces $[X,\ffi]$ in 
$(Z^0\A^{\pretr},w,\vee,\can)$, subject to the
relations (\ref{cor:itm0:Kh_0}), (\ref{cor:itm1:Kh_0}) and
\begin{itemize}
\item[(3')]
if $(E_{\bullet},\ffi_{\bullet})$ is a symmetric space in the category
of exact sequences in $(Z^0\A^{\pretr},w,\vee,\can)$, then $[E_0,\ffi_0]=0$.
\end{itemize}

Furthermore, we define the {\it shifted Witt and Grothendieck-Witt groups} of $\scA$ as
$$W^{[n]}(\scA)= W_0(\scA^{[n]}) \hspace{3ex}\text{and}\hspace{3ex} GW^{[n]}_0(\scA)
 = GW_0(\scA^{[n]}).$$
\red{Note that for $n=0$ we have $GW^{[0]}_0(\scA) = GW_0(\scA)$ and $W^{[0]}(\scA) = W_0(\scA)$.}

\red{
\begin{example}
Let $(\E,*,\eta)$ be a $k$-linear exact category with duality.
Then the Grothendieck-Witt group $GW_0(\Ch^b\E,\quis,*,\eta)$ of the dg category with weak equivalences and duality from Example \ref{ex:ChbEDGwithWeaksAndDuals} is naturally isomorphic to Knebusch's Grothendieck-Witt group $GW_0(\E,*,\eta)$; see \cite[Remark 1]{myMV}.
This follows for instance from \cite[Proposition 6 and Remark 14]{myMV}.
The same applies to the associated Witt groups.
\end{example}
}

\subsection{More on extension categories}
\label{subsec:MoreOnExtCats}
Let $\poSetD_{\str}$ be the category of posets with strict duality.
Morphisms in that category are the order preserving embeddings which commute with dualities.
Cartesian product makes $\poSetD_{\str}$ into a symmetric monoidal category.
We will extend the symmetric monoidal functor (\ref{eqn:PAmonoidalFuntor}) to a symmetric monoidal functor
\begin{equation}
\label{eqn:poSetDdgCatDFun}
\poSetD_{\str} \times \dgCatD \to \dgCatD : \P, (\A,\vee,\can) \mapsto (\P\A,\vee,\can).
\end{equation}
Let $\P$ be a poset with strict duality $\P^{op} \to \P:x \mapsto x'$, and let $(\A,\vee,\can)$ be a pointed dg category with duality.
Then $\P\A$ is equipped with the duality 
$\vee: (\P\A)^{op} \to \P\A: (A,q) \mapsto (A,q)^{\vee} = (A^{\vee}, -q^{\vee})$ where $(A^{\vee})_{i}=(A_{i'})^{\vee}$ and $q^{\vee}$ has entries $(q^{\vee})_{ij}=(q_{j'i'})^{\vee}$.
Note that $(-q^{\vee})^2-d(q^{\vee})=0$ because $(q^2)^{\vee}=-(q^{\vee})^2$ and $d(q^{\vee})=(dq)^{\vee}$.
On morphism complexes, the duality sends a matrix $f$ to the matrix $f^{\vee}$ with entries $(f^{\vee})_{ij}=(f_{j'i'})^{\vee}$.
The double dual identification $\can:(A,q) \to (A,q)^{\vee\vee} = (A^{\vee\vee}, q^{\vee\vee})$ is the matrix with entries $(\can_{(A,q)})_{ij} = \can_{A_i}$ for $i=j\in \P$ and $(\can_{(A,q)})_{ij} = 0$ for $i\neq j \in \P$.
In the obvious way this construction is functorial in posets with strict duality $\P$ and in pointed dg categories with duality $(A,\vee,\can)$.
The monoidal compatibility map (\ref{eqn:MonoidalComMap}) is duality preserving
and thus equipped with the identity as duality compatibility map.

Let $\scA = (\A,w)$ be a dg category with weak equivalences.
Then we define $\P\scA$ as the saturation of the dg pair $(\P\A,\P\A^w)$.
If $\scA = (\A,w)$ is a dg category with weak equivalences and duality, then 
the duality on $\P\A$  makes $\P\scA$ into a dg category with weak equivalences and duality.
Finally, the symmetric monoidal functor (\ref{eqn:PAmonoidalFuntor}) extends to a symmetric monoidal functor
$$\poSetD_{\str} \times \dgCatWD \to \dgCatWD.$$

\section{The cone functor and a counter-example to invariance}
\label{sec:ConeCounterex}

Let $\A$ be a pretriangulated dg category with duality.
Any morphism $f$ in $\A$ has a cone $C(f)$ in $\A$ defined by diagram 
(\ref{eqn:exTrsDfn}).
This defines a (dg) functor 
$$\Cone:  \Fun([1],\A) \to \A: f \mapsto C(f)$$
where $[n]$ denotes, as usual,  the category (with unique duality) $0 \to 1 \to \cdots \to n$ associated with the poset $0 <1 < ... <n$.
Thus, both categories $\Fun([1],\A)$ and $\A$ are exact dg categories with duality.
In what follows, we make the cone functor into a dg form functor
$$\Cone: \Fun([1],\A) \to \A^{[1]}$$
with double dual identification a natural isomorphism.
Note that the target will be equipped with a shifted duality.
The construction of the cone functor and its properties will be fundamental for the rest of this paper.
In this section, we will use it in Proposition \ref{enum:prop:Counterex} to give a counter-example to invariance under derived equivalences.

\subsection{The mapping cone as a dg form functor}
\label{subsec:mappingConedgForm}
Recall from Example \ref{ex:Ck} the commutative dg $\bk$-algebra $C$ and the exact sequence (\ref{eqn:fundExSeq})
in $\scC_{\bk}$.
For a dg category $\A$, let $S_2\A$ denote the full dg subcategory of $\Fun([2],\A)$ of those functors $A: [2] \to \A$ for which the sequence 
$A_0 \to A_1 \to A_2$ is exact in $\A$.
Thus, $\Gamma$ defines an object in $S_2\scC_{\bk}$.

Recall from Remark \ref{rem:SymmToForm} that the multiplication map $\mu:C \otimes C \to C$ of the commutative dg $\bk$-algebra $C$ together with
the composition $p\mu: C \otimes C \to T$ defines
a symmetric form $(C,\mu\circ p)$ in $\scC_{\bk}^{[1]}$.
Note that its adjunction $\ffi_{\mu\circ p}: C \to [C,T]$ is an isomorphism.
In this way we obtain a symmetric isomorphism of short exact sequences
\begin{equation}
\label{eqn:FormOnFundExSeq}
\xymatrix{
\Gamma: \ar[d]^{\gamma}&k\hspace{1ex} \ar@{>->}[r]^i \ar[d]^{\nabla}_{\gamma_{-1}=} & C \ar@{->>}[r]^p \ar[d]^{\ffi_{\mu\circ p}}_{\gamma_0=} & T \ar[d]^{id}_{\gamma_1=}\\
 [\Gamma,T]:& [T,T] \hspace{1ex} \ar@{>->}[r]^{[p,1]} &  [C,T]
\ar@{->>}[r]^{[i,1]} & [k,T].}
\end{equation}
In other words, the pair $(\Gamma,\gamma)$ defines a symmetric space in $S_2\scC_{\bk}^{[1]}$
Tensoring with $(\Gamma,\gamma)$ defines a dg form functor 
$$(\Gamma,\gamma)\otimes id: \A \to (S_2\scC^{[1]}) \A \subset S_2(\scC^{[1]}\A) =S_2\A^{[1]}.$$ 
with duality compatibility map a natural isomorphism.
Note that $S_2(\A^{[n]}) = (S_2\A)^{[n]}$.

Next, for any exact dg category with duality $\A$ we define a dg form functor 
$$(\Delta, \delta): \Fun([1],S_2\A) \to \A$$
as follows.
The exact dg-category with duality 
$\Fun([2],S_2\A)$ 
has objects the sequences $A_{\bullet}^{0} \stackrel{f^{0}}{\to} A_{\bullet}^{1}
\stackrel{f^{1}}{\to} A_{\bullet}^{2}$ of morphisms of short exact sequences 
$A_{\bullet}^i: A_{0}^i \rightarrowtail A_{1}^i \twoheadrightarrow A_{2}^i$ of objects in $\A$.
The evaluation at $(1,1)$
$$e: \Fun([2],S_2\A) \to \A:\hspace{1ex} (A_{\bullet}^{0} \stackrel{f^{0}}{\to} A_{\bullet}^{1} \stackrel{f^{1}}{\to} A_{\bullet}^{2}) \mapsto A_1^1$$
preserves dualities and thus defines a dg form functor between exact dg categories with duality.

Let $\M(\A) \subset \Fun([2],S_2\A)$ be the full dg-subcategory of those sequences $A_{\bullet}^{0} \stackrel{f^{0}}{\to} A_{\bullet}^{1} \stackrel{f^{1}}{\to} A_{\bullet}^{2}$ for which the maps $f_{0}^1: A_{0}^1 \to A_{0}^2$ and  $f_{2}^{0}: A_2^{0} \to A_2^1$ are the identity maps.
The duality on $\Fun([2],S_2\A)$ preserves the subcategory $\M(\A)$ and thus makes $\M(\A)$ into a dg category with duality.
The dg-functor $\M(\A) \to \Fun([1],S_2\A): (f^{0},f^1) \mapsto f^1\circ f^{0}$ preserves dualities.
If $\A$ is exact then this is an equivalence of dg-categories.
By Lemma \ref{lem:Aheq} below, we can choose an inverse dg form functor $\Fun([1],S_2\A) \to \M(\A)$ which is unique up to natural isomorphism of dg form functors.
The dg form functor $(\Delta, \delta)$ is the composition 
$$(\Delta, \delta): \Fun([1],S_2\A) \stackrel{\simeq}{\longrightarrow} \M(\A) \subset \Fun([2],S_2\A) \stackrel{e}{\longrightarrow} \A$$

Finally, the mapping cone dg form functor is the composition
\begin{equation}
\label{eqn:cone}
\Cone: \Fun([1],\A) \stackrel{(\Gamma,\gamma) \otimes id}{\longrightarrow} \Fun([1],S_2\A^{[1]})
\stackrel{(\Delta,\delta)}{\longrightarrow} \A^{[1]}.
\end{equation}
More precisely, it is the zigzag
\begin{equation}
\label{eqn:ConeDfnZigzag}
\Fun([1],\A) \stackrel{(\Gamma,\gamma) \otimes id}{\longrightarrow} \Fun([1],S_2A^{[1]}) \stackrel{\simeq}{\longleftarrow}
\M(\A)
\stackrel{e}{\longrightarrow} \A^{[1]}
\end{equation}
of dg form functors in which the arrow in the wrong direction is an equivalence of dg categories with dualities for which we may choose an inverse if we wish.
Replacing $\A$ with $\A^{[n]}$ we obtain the dg form functor
$$\Cone: \Fun([1],\A^{[n]}) \to \A^{[n+1]}$$
for any exact dg category with duality and $n\in \Z$.
Note that the duality compatibility morphism for this form functor is an isomorphism.

\subsection{A counter-example to invariance under derived equivalences}
The following proposition shows that an exact dg form functor $\scA \to \scB$ between pretriangulated dg categories with weak equivalences and duality which induces an equivalence of associated triangulated categories $\T\scA \stackrel{\cong}{\to} \T\scB$ need not induce an isomorphism of Grothendieck-Witt groups 
$GW_0(\scA) \to GW_0(\scB)$ contrary to the situation in $K$-theory.

To state the proposition, let $R$ be commutative ring.
Equip the exact dg category $\scC_R$ of bounded  complexes of finitely generated free $R$-modules with the set of weak equivalences which are the quasi-isomorphisms, that is,
the morphisms which become isomorphisms in $H^0\scC_R$, or in other words, which are homotopy equivalences of bounded  complexes.
Recall from \S \ref{subsec:mappingConedgForm} the cone dg form functor
\begin{equation}
\label{enq:App:CounterexCone}\Cone:\Fun([1],\scC_R^{[-1]}) \to \scC_R^{[0]}.
\end{equation}
Let $w$ be the set of morphisms $f$ in $\Fun([1],\scC_R^{[-1]})$ for which $\Cone(f)$ is a quasi-isomorphism.

\begin{proposition}
\label{enum:prop:Counterex}
Let $R$ be a commutative ring.
\begin{enumerate}
\item
\label{enum:prop:Counterex:DerEq}
The cone functor (\ref{enq:App:CounterexCone}) induces an equivalence of
triangulated categories 
$$w^{-1}\Fun([1],\scC_R^{[-1]}) \stackrel{\simeq}{\to} \quis^{-1}\scC_R^{[0]}$$
\item
\label{enum:prop:Counterex:notSurj}
If $2$ is not a unit in $R$, then the map 
$$GW_0(\Fun([1],\scC_R^{[-1]}),w) \to GW_0(\scC_R^{[0]},\quis)$$ 
induced by the cone functor (\ref{enq:App:CounterexCone}) is not surjective.
\end{enumerate}
\end{proposition}

\begin{proof}
An inverse to the functor in (\ref{enum:prop:Counterex:DerEq}) is given by 
the functor $\scC_R \to \Fun([1],\scC_R)$
sending and object $A$ of $\scC_R$ to the map $0 \to A$.
Of course, this does not preserve dualities.

For part (\ref{enum:prop:Counterex:notSurj}),
recall that $GW_0(\scC^{[0]}_R,\quis)$ is isomorphic to the usual
Grothendieck-Witt group $GW^{free}_0(R)$ of non-degenerate symmetric bilinear forms on finitely generated free $R$-modules \cite[Proposition 6]{myMV}.
The isomorphism is induced by the map that sends a finitely generated
free $R$-module equipped with a non-degenerated symmetric bilinear form
to the complex concentrated in degree zero where it is that $R$-module
together with the induced form on the complex.
Let $m\subset R$ be a maximal ideal containing $2$ which exists since $2$ is
not a unit in $R$.
Then $k=R/m$ is a field of characteristic $2$.
The composition 
$$\rk_m: GW_0^{free}(R) \to GW^{free}_0(k)=GW_0(k) \stackrel{\rk}{\to} \Z:[M,\ffi]
\mapsto \dim(M\otimes_R k)$$
is surjective since $[R,1]$ is sent to $1$. 
We will show that for every symmetric space $(M,\ffi)$ in
$(\scC_R^{[-1]},w)$, the rank $\rk_m(\Cone M)$ of $\Cone M$ at $m$
is even.
For that, we can assume $R=k$ and it suffices to show that the composition
\begin{equation}
\label{enq:App:CounterexRkmap}
W_0(\Fun([1],\scC_{\bk}^{[-1]}),w) \stackrel{\Cone}{\longrightarrow} W_0(k)
\stackrel{\rk}{\longrightarrow} \Z/2
\end{equation}
is zero.
Since $k$ has characteristic $2$, we will ignore all signs.
As mentioned above, inclusion as complexes concentrated in degree zero induces
an isomorphism 
$W_0(k) \to W_0(\scC_{\bk}^{[0]},\quis)$.
The inverse $W_0(\scC_{\bk}^{[0]},\quis) \to W_0(k)$ is given by the
zero-homology functor $[M,\ffi] \mapsto [H_0M,H_0\ffi]$.
Let $(M,\ffi)= (P\stackrel{f}{\to}Q, \ffi^P,\ffi^Q)$ be a symmetric space in
$(\Fun([1],\scC_{\bk}^{[-1]}),w)$.
We have to show that the symmetric space 
$$(N,\psi)=H_0\Cone(M,\ffi)$$
in $\Vect_{\bk}$ has even rank.
Symmetry of the map $(\ffi^P,\ffi^Q):(P,Q) \to (Q^*[-1],P^*[-1])$ means that
$\ffi^Q_i=(\ffi^P_{1-i})^*\eta_{Q_i}$
where $V^*=Hom_{\bk}(E,k)$ is the usual dual of a $k$-vector space $V$ and $\eta$ is the usual canonical double dual identification $\eta_V:V \to V^{**}$.  
The cone symmetric space of $(M,\ffi)$ in
degree zero is the symmetric map of $k$-vector spaces
\begin{equation}
\label{eqn:propCounterex}
\left(\begin{smallmatrix} \alpha & 0 \\ 0 & \alpha^*\eta
  \end{smallmatrix}\right):  Q_0 \oplus 
  P_1 \to P_1^* \oplus Q_0^*,\hspace{4ex} {\rm where}\hspace{2ex} \alpha=\ffi^Q_0.
\end{equation}
The particular shape of the symmetric map shows that $Q_0 \oplus P_1$ 
has a basis
$v_1,...,v_n$, $n=\dim Q_0
+ \dim P_1$, of isotropic vectors, that is, of vectors $v_i$ for which the
(possibly singular) associated symmetric bilinear form (\ref{eqn:propCounterex}) satisfies
$\langle v_i,v_i\rangle = 0$, $i=1,...,n$.
Since the field $k$ has characteristic $2$ this implies that every vector $v \in Q_0
\oplus P_1$ has to be isotropic.
As a subquotient, the non-singular symmetric space
$(N,\psi)$ also consists 
of isotropic vectors.
Hence, the symmetric space $(N,\psi)$ over the field $k$ of characteristic $2$
is symplectic.
Every symplectic inner product space has even rank \cite{MilnorHusemoeller}, hence, the map (\ref{enq:App:CounterexRkmap}) is the zero map.
\end{proof}

We conclude this section with two lemmas used in the construction of the cone dg form functor.
Recall terminology and notation from \S \ref{subsec:CatsWDual}.

\begin{lemma}
\label{lem:PreAheq}
Let $(F,\ffi):(\A,*,\alpha) \to (\B,*,\beta)$ be a form functor between categories with duality. 
If $F$ is an equivalence of categories and $\ffi$ a natural isomorphism, then $(F,\ffi)$ induces an equivalence of categories of symmetric forms
$$(F,\ffi)_h: \A_h \stackrel{\simeq}{\longrightarrow} \B_h.$$
\end{lemma}

\begin{proof}
Easy verification left to the reader.
\end{proof}

\begin{lemma}
\label{lem:Aheq}
Let $(F,\ffi):(\A,*,\alpha) \to (\B,*,\beta)$ be a form functor between categories with duality. 
If $F$ is an equivalence of categories and $\ffi$ a natural isomorphism, then $(F,\ffi)$ is an equivalence of categories with duality, that is, there is a form functor $(G,\gamma): (\B,*,\beta) \to (\A,*,\alpha)$ such that
$(F,\ffi)\circ (G,\gamma)$ and $(G,\gamma) \circ (F,\ffi)$ are naturally equivalent to the identity form functors.
Moreover, any two inverses of $(F,\ffi)$ are naturally isomorphic.
\end{lemma}

\begin{proof}
Composition with $(F,\ffi)$ induces a form functor between categories with duality
$$(F,\ffi): \Fun(\B,\A) \to \Fun(\B,\B): G \mapsto F\circ G$$
which is an equivalence of underlying categories (since $F$ is) and whose duality compatibility map is an isomorphism (since $\ffi$ is).
By Lemma \red{\ref{lem:PreAheq}}, the induced functor between categories of symmetric forms is an equivalence of categories
$$\Fun(\B,\A)_h \stackrel{\simeq}{\longrightarrow} \Fun(\B,\B)_h: (G,\gamma)\mapsto (F,\ffi)\circ (G,\gamma).$$
In particular, there is an object $(G,\gamma)$ of $\Fun(\B,\A)_h$ such that
$(F,\ffi) \circ (G,\gamma)$ is naturally isomorphic to the identity form functor on $\B$.
This object is unique up to natural isomorphism of form functors.
Thus, we have shown that any form functor $(F,\ffi)$ with $F$ an equivalence and $\ffi$ an isomorphism has a right inverse (up to natural isomorphism of form functors) which is unique up to natural isomorphisms of form functors.
Let $(G,\gamma)$ be such an inverse. 
Then since $F$ is an equivalence, so is $G$, and since, furthermore, $\ffi$ is an isomorphism, so is $\gamma$.
Therefore, $(G,\gamma)$ has a right inverse, too, say $(H,\eta)$.
But then $(F,\ffi)$ is naturally isomorphic to $(F,\ffi)\circ [(G,\gamma)\circ (H,\eta)] = [(F,\ffi)\circ (G,\gamma)]\circ (H,\eta)$ which is naturally isomorphic to $(H,\eta)$.
Hence, $(G,\gamma)$ is not only a left but also a right inverse of $(F,\ffi)$.
\end{proof}

\section{Grothendieck-Witt groups of triangulated categories}
\label{sec:GWofTriD}

\red{The theory of Witt groups has experienced a renaissance some 15 years ago due to the introduction of triangulated category methods by Balmer \cite{Balmer:TWGI}, \cite{Balmer:TWGII}.
The purpose of this section is to associate to every dg category with weak equivalences and duality a triangulated category with duality (Lemma \ref{lem:TAisTriD}), and to show that both have isomorphic Grothendieck-Witt groups (Proposition \ref{prop:GW0AisGW0TA}).
This is the base case of the Invariance Theorem \ref{thm:Invariance} which plays an important role in applications. 
Unfortunately, Balmer's framework is too restrictive for our purpose.
So, part of this section recasts some of the definitions in \cite{Balmer:TWGI}.
}

Our reference for triangulated categories is \cite{Keller:uses}.
In particular, we will only assume that the shift functor $T: \K \to \K$ 
in a triangulated category  $\K$ is an equivalence of categories.

\begin{definition}
\label{dfn:TriDwithoutTed}
A {\em triangulated category with duality} is a triangulated category $\K$ together with an additive functor $\sharp:\K^{op} \to \K$ and natural isomorphisms $\lambda: \sharp \stackrel{\cong}{\to} T\sharp T$ and $\varpi: 1 \stackrel{\cong}{\to} \sharp \sharp$ satisfying (\ref{TriD1}) - (\ref{TriD3}) below.
\begin{enumerate}
\item
\label{TriD1}
The following diagram commutes
$$\xymatrix{
T \ar[r]^{\varpi_T} \ar[d]_{T\varpi} & \sharp\sharp T \ar[d]^{\lambda_{\sharp T}} \\
T\sharp\sharp & T\sharp T\sharp T. \ar[l]^{T\sharp\lambda}}$$
\item
\label{TriD2}
For all objects $X$ of $\K$ we have $\varpi^{\sharp}_X\circ\varpi_{X^{\sharp}}=1_{X^{\sharp}}$.
\item
\label{TriD3}
If
\begin{equation}
\label{eqn:genericExTriangle}
X \stackrel{f}{\longrightarrow} Y \stackrel{g}{\longrightarrow} Z \stackrel{h}{\longrightarrow} TX
\end{equation}
is an exact (also called distinguished) triangle in $\K$.
Then the following triangle, called dual triangle, is also exact in $\K$
\begin{equation}
\label{eqn:genericDualExTriangle}
Z^{\sharp} \stackrel{g^{\sharp}}{\longrightarrow} Y^{\sharp} \stackrel{f^{\sharp}}{\longrightarrow} X^{\sharp} \stackrel{T(h^{\sharp})\circ\lambda_X}{\longrightarrow} T(Z^{\sharp}).
\end{equation}
\end{enumerate}
\end{definition}

\begin{remark}
\label{rem:TriangularSharpExact}
Given a triangulated category $\K$, an additive functor $\sharp:\K^{op} \to \K$, a natural transformation $\varpi:1 \to \sharp\sharp$ and a natural isomorphism $\lambda:\sharp \to T\sharp T$ satisfying (\ref{TriD1}) and (\ref{TriD2}) of Definition \ref{dfn:TriDwithoutTed}.
Then (\ref{TriD3}) is equivalent to requiring that the following triangle be exact
\begin{equation}
\label{eqn:AlternativeDualExactTriangle}
(TY)^{\sharp} \stackrel{(Tf)^{\sharp}}{\longrightarrow} (TX)^{\sharp}
\stackrel{h^{\sharp}}{\longrightarrow} Z^{\sharp} \stackrel{\lambda_Y\circ g^{\sharp}}{\longrightarrow} T(TY)^{\sharp}.
\end{equation}
Indeed, under the natural isomorphism $\lambda$, the triangle (\ref{eqn:genericDualExTriangle}) becomes the triangle (\ref{eqn:AlternativeDualExactTriangle}) shifted twice.
\end{remark}

\begin{remark}
There is an evident category $\Delta\K$ of exact triangles in $\K$.
The duality functor $\sharp$ on $\K$ makes $\Delta\K$ into a category with duality $(\Delta\K,\sharp,\varpi)$ where the duality functor 
$$\sharp:(\Delta\K)^{op} \to \Delta\K$$
sends the exact triangle (\ref{eqn:genericExTriangle}) to the exact triangle (\ref{eqn:genericDualExTriangle}).
The double dual identification for the exact triangle (\ref{eqn:genericExTriangle}) is the map of triangles $(\varpi_X,\varpi_Y,\varpi_Z)$ which is indeed a map of triangles in view of the commutative diagram in Definition \ref{dfn:TriDwithoutTed} (\ref{TriD1}).
\end{remark}

\begin{definition}
\label{dfn:MorphTriD}
A morphism of triangulated categories with duality 
$$(F,\rho,\ffi): (\K_1,\sharp,\varpi,\lambda) \to (\K_2,\sharp,\varpi,\lambda)$$
is a triple $(F,\rho,\ffi)$ where $F:\K_1 \to \K_2$ is an additive functor, $\rho:FT \stackrel{\cong}{\to} TF$ and $\ffi:F\sharp \stackrel{\cong}{\to} \sharp F$ are natural isomorphisms such that $(F,\rho): \K_1 \to \K_2$ is a triangle functor and such that the following diagrams commute
$$
\xymatrix{
F \ar[r]^{F\varpi} \ar[d]_{\varpi_F} & F\sharp\sharp \ar[d]^{\ffi_{\sharp}} & F\sharp 
\ar[r]^{F(\lambda)} \ar[d]_{\ffi} & FT\sharp T \ar[r]^{\rho_{\sharp T}} & TF\sharp T 
\ar[d]^{T(\ffi_T)} \\
\sharp\sharp F \ar[r]_{\ffi^{\sharp}} & \sharp F \sharp & \sharp F \ar[r]_{\lambda_F} & T\sharp T F \ar[r]_{T\sharp (\rho)} & T\sharp F T.
}$$
Composition is defined as 
$$(\bar{F},\bar{\rho},\bar{\ffi}) \circ (F,\rho,\ffi) = (\bar{F}\circ F,\bar{\rho}_F\circ \bar{F}\rho,\bar{\ffi}_F\circ \bar{F}\ffi).$$
In this way we obtain the category $\TriD$ of small triangulated categories with duality and their morphisms. 
Note that a morphism of triangulated categories with duality $(F,\rho,\ffi)$ as above induces morphisms of additive categories with duality $(F,\ffi):(\K_1,\sharp,\varpi) \to (\K_2,\sharp,\varpi)$ and $(F,\rho,\ffi): (\Delta\K_1,\sharp,\varpi) \to (\Delta\K_2,\sharp,\varpi)$.
\end{definition}

For the purpose of the next definition, an inner product space in an additive category with duality $(\A,\sharp,\varpi)$ will mean a pair $(A,\ffi)$ where $\ffi:A \to A^{\sharp}$ is an isomorphism satisfying $\ffi^{\sharp}\varpi_A=\ffi$.
An isometry from $(A,\ffi)$ to $(B,\psi)$ is an isomorphism $f:A \to B$ in $\A$ such that $\ffi=f^{\sharp}\psi f$.
This applies in particular to the additive categories with duality $\K$ and $\Delta\K$ associated with a triangulated category with duality $\K$.

\begin{definition}
\label{dfn:TriGW0andW}
The {\em Grothendieck-Witt group} 
$$GW^0(\K)$$
of a triangulated category with duality $\K=(\K,\sharp,\varpi,\lambda)$
is the  abelian group generated by isometry classes $[X,\ffi]$ of
inner product spaces $(X,\ffi)$ in $\K$, subject to the following
relations 
\begin{enumerate}
\item
\label{dfn:itm1:GW0Tri}
$[X,\ffi]+[Y,\psi] = [X \oplus Y,\ffi \oplus \psi]$
\item
\label{dfn:itm2:GW0Tri}
Given an inner product space in the category
of exact triangles $\Delta\K$ 
$$\xymatrix{
%\hspace{2ex}
E_{-1}
\ar[rr]^f \ar[d]^{\ffi_{-1}}_{\cong} && E_0 \ar[rr]^g \ar[d]^{\ffi_{0}}_{\cong}
&&  
E_1 \ar[d]^{\ffi_{1}}_{\cong} 
\ar[rr]^h && 
T(E_{-1}) \ar[d]^{T\ffi_{-1}}_{\cong}
\\
E_{1}^{\sharp} \ar[rr]_{g^{\sharp}} && E_0^{\sharp}  \ar[rr]_{f^{\sharp}} &&
E_{-1}^{\sharp} 
\ar[rr]_{T(h^{\sharp})\circ \lambda_{X}}  && 
T(E_{1}^{\sharp}),
}$$ 
that is,
$(\ffi_{-1},\ffi_0,\ffi_1)=(\ffi_1^{\sharp}\varpi,\ffi_0^{\sharp}\varpi,\ffi_{-1}^{\sharp}\varpi) $ is an isomorphism, then 
$$[E_0,\ffi_0] = \left[ E_{-1}\oplus E_1,
\left(\begin{smallmatrix} 0 & \ffi_1\\ \ffi_{-1} &
    0\end{smallmatrix}\right)\right] \in GW^0(\K).
$$
\end{enumerate}
The {\em Witt group} 
$$W^0(\K)$$
of $\K$
is the  abelian group generated by isometry classes $[X,\ffi]$ of
inner product spaces $(X,\ffi)$ in $\K$, subject to the 
relation (\ref{dfn:itm1:GW0Tri}) above and $[E_0,\ffi_0] = 0$ for every inner product space $(\ffi_{-1},\ffi_0,\ffi_1)$ in the category of exact triangles as in 
(\ref{dfn:itm2:GW0Tri}).
\end{definition}

\subsection{The functor $\T: \dgCatWD \to \TriD$.}
Next we want to verify that the triangulated category of a dg category with weak equivalences and duality is canonically a triangulated category with duality.
For that, consider the natural transformation in $\scC_{\bk}$
$$\alpha_{B,X,Y}: B[BX,Y] \to [X,Y]:\ \alpha(b\otimes f)(x) = (-1)^{|b||f|}f(b\otimes x)$$
where $b\in B$, $f\in [BX,Y]$ and $x\in X$. 
It is the unique admissible natural transformation of this shape in the sense of \cite{KellyMacLane:Coherence}.
For $B=T = k[1] \in \scC_{\bk}$, this map is an isomorphism.

Let $\A = (\A,\sharp,\varpi)$ be a dg category with duality.
To simplify notation, we may write $(\sharp,\varpi)$ for the duality of $\scC_{\bk}^{[0]}\A$ though it is strictly speaking $(\vee \otimes \sharp,\can \otimes \varpi)$
where $(\vee,\can) = ([\phantom{A},\1],\can)$ is the duality on $\scC^{[0]}$.
We define a natural transformation 
$\lambda: \sharp \to T\sharp T$ of functors $\scC_{\bk}^{[0]}\A \to \scC_{\bk}^{[0]}\A$ 
by
$$\lambda_{AX} = \lambda \otimes 1:  (AX)^{\sharp} = [A,\1]\otimes X^{\sharp} \to T\sharp T (AX) = T[TA,\1]\otimes X^{\sharp}$$
where $\lambda: [A,\1] \to T[TA,\1]$ is the inverse of $\alpha_{T,A,\1}$.

\begin{lemma}
\label{lem:TAisTriD}
Let $\scA = (\A,w,\sharp,\varpi)$ be a dg category with weak equivalences and duality.
Then the datum
$$\T\scA = (\T\scA, \sharp,\varpi,\lambda)$$ 
defines a triangulated category with duality.
\end{lemma}

\begin{proof}
As a localization of a category with duality $Z^0\A^{\pretr}$ where $\varpi$ is a natural weak equivalence, the datum $(\T\scA, \sharp,\varpi)$ defines a category with duality where $\varpi$ is a natural isomorphism.
We are left with checking conditions (\ref{TriD1}) and (\ref{TriD3}) of Definition \ref{dfn:TriDwithoutTed}. 

For the first condition, consider the following diagram in $\scC_{\bk}$
$$\xymatrix{
BX \ar[rr]^{\can^Y_{BX}} \ar[d]_{1\otimes \can^Y_X} && [[BX,Y],Y] \\
B[[X,Y],Y] \ar[rr]_{1\otimes[\alpha_{B,X,Y},1]} && B[B[BX,Y],Y]. \ar[u]_{\alpha_{B,[BX,Y],Y}}
}$$
The diagram commutes for every $X$, $Y$, $B$ in $\scC_{\bk}$.
This can be checked directly by a diagram chase.
Alternatively, it also follows from Kelly-MacLane's Coherence Theorem \cite{KellyMacLane:Coherence} as there is a unique admissible natural transformation $BX \to [[BX,Y],Y]$.
Since in $\T\scA$, we have $\varpi_{\T\scA} = \vee_{\scC_{\bk}}\otimes \varpi_{\scA}$ and $\lambda = \lambda \otimes id$, 
commutativity of the diagram for $B=T$ and $Y=\1$ implies the commutativity of diagram (\ref{TriD1}) in Definition \ref{dfn:TriDwithoutTed}.

We are left with checking condition (\ref{TriD3}) of Definition \ref{dfn:TriDwithoutTed}.
The condition is invariant under isomorphisms of exact triangles.
Therefore, it suffices to check it for the standard triangles (\ref{eqn:standardTriangle}).
For that, let $f:X \to Y$ be a morphism in $Z^0\A^{\pretr}$, and
consider the commutative diagram
$$\xymatrix{
 X \hspace{1ex} \ar[d]_{f} \xymono[rr]^{i\otimes 1} \ar@{}[drr]|{\square} &&
CX \xyepi[rr]^{p\otimes 1} \ar[d] &&
TX \ar[d]^1 \\
 Y \hspace{1ex} \xymono[rr]^{g} \ar[d]_{\gamma_{-1}\otimes 1} &&
C(f) \xyepi[rr]^h \ar[d] \ar@{}[drr]|{\square} &&
TX \ar[d]^{\gamma_1\otimes f} \\
[T,T]Y \hspace{1ex} \xymono[rr]_{[p,1]\otimes 1} &&
[C,T]Y \xyepi[rr]_{[i,1]\otimes 1} &&
[1,T]Y
}$$
in which the outer diagram is the map of exact sequences
(\ref{eqn:FormOnFundExSeq}) tensored with $f$.
The rest of the diagram is given by the definition of the cone $C(f)$.
As usual, the squares with a $\square$-sign in it are bicartesian.
Applying the (exact) duality functor $\sharp$ to the lower two squares of the diagram yields the lower part of the following commutative diagram
$$
\xymatrix{
(TY)^{\sharp}\hspace{1ex} \xymono[rr]^{i\otimes 1} \ar[d]_{\tilde{\beta}}^{\cong} &&
C(TY)^{\sharp} \xyepi[rr]^{p\otimes 1} \ar[d]^{\tilde{\beta}}_{\cong} &&
T(TY)^{\sharp} \ar[d]^{\tilde{\beta}}_{\cong}\\
([1,T]Y)^{\sharp}\hspace{1ex}\xymono[rr]^{([i,1]\otimes 1)^{\sharp}} \ar[d]_{(\gamma_1\otimes f)^{\sharp}} \ar@{}[drr]|{\square} &&
([C,T]Y)^{\sharp} \xyepi[rr]^{([p,1]\otimes 1)^{\sharp}} \ar[d] &&
([T,T]Y)^{\sharp} \ar[d]^{(\gamma_{-1}\otimes 1)^{\sharp}} \\
(TX)^{\sharp} \hspace{1ex}\xymono[rr]_{h^{\sharp}} &&
(Cf)^{\sharp} \xyepi[rr]_{g^{\sharp}} &&
Y^{\sharp}.
}$$
The upper vertical arrows are given by the natural isomorphism
$\tilde{\beta}: ([B,D]Y)^{\sharp} \to B(DY)^{\sharp}$ with $B,D\in \scC_{k}$ and $Y\in \scC_{\bk}\A$.
For $Y=E\otimes Y_0$ with $E\in \scC_{\bk}$ and $Y_0\in \A$ the natural transformation $\tilde{\beta}$ has the form
$$\tilde{\beta} = \beta\otimes id: B(DE)^{\sharp}\otimes Y_0^{\sharp} \to ([B,D]E)^{\sharp}\otimes Y_0^{\sharp}$$
where $\beta: B\otimes (DE)^{\sharp} \to ([B,D]E)^{\sharp}$ is the unique admissible natural transformation of this shape in $\scC_{\bk}$.
It is given by the formula
$$\beta: B[DE,A] \to [[B,D]E,A]: \beta(b\otimes f) (g\otimes e) = (-1)^{|f||b|+|g||b|}f(g(b)\otimes e).$$
The map $\beta$ is an isomorphism for $A,B,D,E \in \scC_k$.

Composition of the left two vertical arrows is $(Tf)^{\sharp}$.
This is because $(\gamma_1\otimes f)^{\sharp} = (Tf)^{\sharp}\circ (\gamma_1 \otimes 1_Y)^{\sharp}$ reducing the claim to verifying $(\gamma_1 \otimes 1_Y)^{\sharp} \circ \beta = 1_{(TY)^{\sharp}}$ which can be checked within $\scC_k$ where, by adjunction, it boils down to the composition  
$$T\otimes \1 \stackrel{\gamma_1\otimes 1}{\longrightarrow} [\1,T]\otimes \1 \stackrel{e}{\longrightarrow} T$$
being the usual isomorphism $T\otimes \1 \cong T$.
Similarly, the composition of the right two vertical arrows is the natural transformation $\alpha: T\sharp T \to \sharp$, inverse of $\lambda$.
Again, this can be checked within $\scC$ where, by adjunction, it boils down to the composition 
$$\1 \otimes T \stackrel{\gamma_{-1}\otimes 1}{\longrightarrow} [T,T]T 
\stackrel{e}{\longrightarrow} T$$
being the usual identification $\1\otimes T \cong T$.
The outer part of the diagram shows that the following is an exact triangle in $\T\A$
$$(TY)^{\sharp} \stackrel{(Tf)^{\sharp}}{\longrightarrow} (TX)^{\sharp}
\stackrel{h^{\sharp}}{\longrightarrow} (Cf)^{\sharp} \stackrel{\lambda_Y\circ g^{\sharp}}{\longrightarrow} T(TY)^{\sharp}.$$
By Remark \ref{rem:TriangularSharpExact}, we are done.
\end{proof}

Next, we want to show that the Grothendieck-Witt group of a pretriangulated dg category with weak equivalences and duality $\scA$ coincides with the Grothendieck-Witt groups of its derived category $\T\scA$ when $1/2 \in \scA$.
For that, let $\scA = (\A,w, \sharp,\varpi)$ be a pretriangulated dg category with weak equivalences and duality, and consider the localization functor
\begin{equation}
\label{eqn:AToTAdualpreserve}
Z^0\A \to \T\scA:A \mapsto A
\end{equation}
sending weak equivalences to isomorphisms.
This functor preserves dualities and thus defines a form functor between categories with duality.
Similarly, consider the functor
\begin{equation}
\label{eqn:S2ToDeltaT}
S_2Z^0\A \to \Delta\T\scA
\end{equation}
from exact sequences in $\A$ to exact triangles in $\T\scA$.
It sends an exact sequence 
\begin{equation}
\label{eqn:exSeqfg}
X \stackrel{f}{\rightarrowtail} Y \stackrel{g}{\twoheadrightarrow} Z
\end{equation}
 to the exact triangle
\begin{equation}
\label{eqn:TriangleOfExSeq}
\xymatrix{
X \ar[r]^{f} & Y \ar[r]^{g} & Z \ar[r]^{q\circ r^{-1}} & TX}
\end{equation}
where the maps $q = q(f,g):C(f) \to TX$ and $r = r(f,g): C(f) \to Z$ are defined by the commutative diagram
\begin{equation}
\label{eqn:dfn:rq(fg)}
\xymatrix{X \xymono[r]^f \xymono[d]_{i\otimes 1} \ar@{}[dr]|{\square}
& Y \xyepi[r]^g \xymono[d] & Z \\
CX \xymono[r] \xyepi[d]_{p\otimes 1} & C(f) \xyepi[ur]_{r(f,g)} \xyepi[dl]^{q(f,g)} & \\
TX, &&
}\end{equation}
in which the square is cocartesian,
and the map $r(f,g):C(f) \to Z$ is an isomorphism in $\T\scA$ since it is an admissible epimorphism in $\scC_{\bk}\A$ with contractible kernel $CX$.
It follows from the following lemma that the functor (\ref{eqn:S2ToDeltaT}) equipped with the identity as duality compatibility map is a form functor between categories with dualities.

\begin{lemma}
Given an exact sequence (\ref{eqn:exSeqfg}) in a pretriangulated dg category $\A$,
then the following diagram is a map of exact triangles in $\T\A$
$$
\xymatrix{
Z^{\sharp}  \ar[r]^{g^{\sharp}} \ar[d]^1 & Y^{\sharp} \ar[r]^{f^{\sharp}} \ar[d]^1 & X^{\sharp} \ar[r]^{\bar{q}\circ \bar{r}^{-1}} \ar[d]^1 & T(Z^{\sharp}) \ar[d]^1 \\
Z^{\sharp}  \ar[r]_{g^{\sharp}} & Y^{\sharp} \ar[r]_{f^{\sharp}} & X^{\sharp} \ar[r]_{\hspace{-2ex}T(h^{\sharp})\circ \lambda_X} & T(Z^{\sharp})}
$$
where 
$h = q \circ r^{-1}$ with $q=q(f,g)$, $r = r(f,g)$ and $\bar{q}=q(g^{\sharp},f^{\sharp})$, $\bar{r} = r(g^{\sharp},f^{\sharp})$ defined by the commutative diagram (\ref{eqn:dfn:rq(fg)}).
%In particular, the functor (\ref{eqn:S2ToDeltaT}) preserves dualities.
\end{lemma}

\begin{proof}
The only thing that needs justification is the commutativity of the right hand square in the diagram, that is, the commutativity of
$$
\xymatrix{
X^{\sharp} \ar[d]^{\lambda_X} & C(g^{\sharp}) \ar[l]_{\bar{r}} \ar[r]^{\bar{q}} & T(Z^{\sharp}) \ar[d]^1 \\
T\sharp TX \ar[r]_{T(q^{\sharp})} & T\sharp C(f) & T(Z^{\sharp}). \ar[l]^{T(r^{\sharp})}}
$$
This diagram commutes due to the following two facts.

Firstly, in the following diagram
\begin{equation}
\label{eqn:Pf:Lem:3.7}
\xymatrix{
{\sharp}X \xymono[r]^{i\otimes 1} \ar[d]^{\lambda_X} & C{\sharp}X \xyepi[r]^{p\otimes 1}\ar[d] & T{\sharp}X \ar[d]^1\\
T{\sharp}TX \xymono[r]_{T{\sharp}(p\otimes 1)} & T{\sharp}CX \xyepi[r]_{T{\sharp}(i\otimes 1)} & T{\sharp}X,
}
\end{equation}
the left square anti-commutes and the right square commutes where the middle vertical map is the natural transformation
$$
\xymatrix{
C \otimes X^{\sharp} \ar[r]^{\gamma_0\otimes 1}_{\cong} & [C,T]\otimes X^{\sharp} &   T[C,\1]\otimes X^{\sharp} \ar[l]_{\beta \otimes 1}^{\cong} }
$$
with $\gamma_0: C \to [C,T]$ the map defined in \S \ref{subsec:mappingConedgForm}  and
$\beta: T[C,\1] \to [C,T]$ the (admissible) natural transformation
$$\beta: A[B,E] \to [B,AE]: \beta(a\otimes f)(b)= a \otimes f(b).$$
Since all maps in diagram (\ref{eqn:Pf:Lem:3.7}) are of the form $?\otimes 1_{X^{\sharp}}$, (anti-) commutativity can be checked in $\scC_{\bk}$ with $X=k$ in which case it is a direct verification.

Secondly, 
given a commutative diagram in an exact category 
$$\xymatrix{A_{00} \xymono[r] \xymono[d] & A_{01} \xyepi[r] \xymono[d] & A_{02} \xymono[d] \\
A_{10} \xymono[r] \xyepi[d] & A_{11} \xyepi[r] \xyepi[d] & A_{12} \xyepi[d] \\
A_{20} \xymono[r] & A_{21} \xyepi[r]  & A_{22} 
}$$
with exact rows and columns, let $E$ be the push-out of the upper left corner, and $P$ the pull-back of the lower right corner, and let $c_{ij}:E \to A_{ij}$, and $p_{ij}: A_{ij} \to P$, $i=0,1,2$, $i+j=2$,  be the natural maps given by the universal properties defining $E$ and $P$.
Then we have the equation 
\begin{equation}
\label{eqn:GridOfExSeqs}
p_{11}c_{11}= p_{20}c_{20} + p_{02}c_{02}.
\end{equation}
This can be checked by composing both sides of the equation with $P \to A_{12}$, $P \to  A_{21}$, $A_{01} \to E$ and $A_{10}\to E$.

We apply the second fact to the diagram obtained by tensoring the (column) exact sequence (\ref{eqn:fundExSeq}) with the (row) exact sequence $(g^{\sharp},f^{\sharp})$.
Then $E = C(g^{\sharp})$, $c_{20}=\bar{q}$, $c_{02}=\bar{r}$.
Applying the exact functor $T\sharp$ to diagram (\ref{eqn:dfn:rq(fg)}) and using the first fact, we obtain $P=T\sharp C(f)$, 
$p_{20} = T(r^{\sharp})$, and $p_{02}=-T(q^{\sharp})\lambda_X$.
In the equation (\ref{eqn:GridOfExSeqs}), the map $p_{11}c_{11}$ is zero in $\T\A$ because it factors through the object $C\sharp Y$ which is zero in $\T\A$.
Therefore, equation (\ref{eqn:GridOfExSeqs}) yields
$$ 0 = T(r^{\sharp}) \circ \bar{q}  -  T(q^{\sharp})\lambda_X \circ \bar{r}.$$
\end{proof}

Since we have checked that the functors (\ref{eqn:AToTAdualpreserve})
and (\ref{eqn:S2ToDeltaT}) preserve dualities, we obtain a well-defined map of abelian groups
\begin{equation}
\label{eqn:GW0AToGW0TA}
GW_0(\scA) \to GW^0(\T\scA): [A,\ffi] \mapsto [A,\ffi]
\end{equation}
for any dg category with weak equivalences and duality $\scA$.
Recall that $\frac{1}{2}\in \A$ means that $\A$ is a dg category over a $\Z[1/2]$-algebra.

\begin{proposition}
\label{prop:GW0AisGW0TA}
Let $\scA=(\A,w,\sharp,\varpi)$ be a dg category with weak equivalences and duality such that $\frac{1}{2}\in \scA$.
Then the map (\ref{eqn:GW0AToGW0TA}) is an isomorphism of abelian groups
$$GW_0(\scA) \stackrel{\cong}{\longrightarrow} GW^0(\T\scA).$$
\end{proposition}

\begin{proof}
We can assume $\A$ to be pretriangulated.
For a small category $\C$, write $\pi_0\C$ for its set of connected components, that is, the quotient of the set of objects of $\C$ modulo the relation (generated by) $A \sim B$ whenever there is a morphism $A \to B$ in $\C$.
If $\C$ is a symmetric monoidal category, then $\pi_0\C$ is an abelian monoid, and we write $K_0(\C)$ for the Grothendieck group of the abelian monoid $\pi_0\C$.
Recall from Section \ref{subsec:CatsWDual} the category $\C_h$ of symmetric forms in a category with duality $\C$.
So, $K_0\C_h$ denotes the Grothendieck group of symmetric forms in $\C$, that is, the Grothendieck group of $\pi_0(\C_h)$ in case $\C_h$ is symmetric monoidal.

From Definitions \ref{dfn:GW0A} and \ref{dfn:TriGW0andW}, we have a map of short exact sequences of abelian groups
$$
\xymatrix{
K_0(wS_2\A)_h \ar[r] \ar[d] & K_0(w\A)_h \ar[r]\ar[d] & GW_0(\scA) \ar[r] \ar[d] & 0\\
K_0(i\Delta\T\scA)_h \ar[r] & K_0(i\T\scA)_h \ar[r] & GW^0(\T\scA) \ar[r] & 0
}$$
where the vertical maps are induced by the functors (\ref{eqn:AToTAdualpreserve}) and (\ref{eqn:S2ToDeltaT}) and the left horizontal maps are the differences of the two sides of the equations in Definitions \ref{dfn:GW0A} (\ref{cor:itm2:Kh_0}) and \ref{dfn:TriGW0andW} (\ref{dfn:itm2:GW0Tri}).
Let $R_1$ and $R_2$ be the images of the top left and the bottom left horizontal maps.
So, $R_1$ is generated by $[E_0,\ffi_0] - H(E_{-1})$ where $(E_*,\ffi_*)$ is a  symmetric space in the category of exact sequences as in Definition \ref{dfn:GW0A} (3), and $R_2$ is generated by $[E_0,\ffi_0] - H(E_{-1})$ where $(E_*,\ffi_*)$ is a  symmetric space in the category of exact triangles as in Definition \ref{dfn:TriGW0andW} (2).
By Lemma \ref{lem:hD=Dh} below, the middle vertical arrow in the diagram is an isomorphism.
Hence, we are left with showing that the induced map $R_1 \to R_2$ is surjective.

Let $(E_*,\ffi_*)$ be a symmetric space in the category of exact triangles in $\T\scA$ as in Definition \ref{dfn:TriGW0andW} (2).
Recall that we can assume $\A$ to be pretriangulated, so $\T\scA=w^{-1}H^0\A$.
By the definition of exact triangles in $w^{-1}H^0\A$, we can assume that $(f,g,h)$ is an exact triangle in $H^0\A$.
The maps $\ffi_i$ are fractions $\ffi_i=\alpha_is_i^{-1}$ where $\alpha_i:A_i \to E_{-i}^{\sharp}$ and $s_i:A_i \to E_i$ are weak equivalences in $\A$.
By the calculus of fractions,
there is \red{a} map $a:A_{-1} \to A_0$ in $Z^0\A$ such that $s_0f=as_{-1}$ and $\alpha_0a=g^{\sharp}\alpha_{-1}$ in $H^0\A$.
Furthermore, replacing $a:A_{-1}\to A_0$ with the homotopy equivalent $(a,i): A_{-1} \to A_0\oplus CA_{-1}$, we can assume that $a:A_{-1}\to A_0$ is an inflation with quotient map $b:A_0 \to A_1=A_0/A_{-1}$ in $Z^0\A$.
The pair $(s_{-1},s_0)$ extends to a map of exact triangles $(s_{-1},s_0,s_1)$ in $H^0\A$ from $A_*$ to $E_*$.
Since $s_{-1}$ and $s_0$ are weak equivalences, so is $s_1$.
Therefore, the symmetric exact triangle $(E_*,\ffi_*)$ is isometric in $w^{-1}H^0\A$ to the symmetric exact triangle $(A_*,\psi_*)$  where $\psi_{i} = s_i^{\sharp}\ffi_{i}s_{i}$.
In particular, $[E_0,\ffi_0]-H(E_{-1}) = [A_0,\psi_0]-H(A_{-1})$ in $R_2$.
We will show that the latter expression is in the image of $R_1 \to R_2$.
Since
$(s_1^{\sharp}\ffi_{-1}s_{-1},s_0^{\sharp}\ffi_{0}s_0, s_{-1}1^{\sharp}\ffi_{1}s_{1}) = (s_1^{\sharp}\alpha_{-1},s_0^{\sharp}\alpha_{0},\alpha_{-1}^{\sharp}s_1)$ in $w^{-1}H^0\A$, we can assume that the components of $(\psi_{-1},\psi_{0},\psi_{1}) = (s_1^{\sharp}\alpha_{-1},s_0^{\sharp}\alpha_{0},\alpha_{-1}^{\sharp}s_1)$ are maps in $Z^0\A$.
In $H^0\A$ we have
$\psi_0a=s_0^{\sharp}\alpha_0a=s_0^{\sharp}g^{\sharp}\alpha_{-1}=b^{\sharp}s_1^{\sharp}\alpha_{-1}=b^{\sharp}\psi_{-1}$.
Therefore, the difference $\psi_0a - b^{\sharp}\psi_{-1}: A_{-1} \to A_0^{\sharp}$ in $Z^0\A$ factors through an injective object of $Z^0\A$.
Since $a:A_{-1}\to A_0$ is an inflation in the Frobenius exact category $Z^0\A$, there is a map $\eps:A_0 \to A_0^{\sharp}$ in $Z^0\A$ which factors through that injective object such that $(\psi_0+\eps)a = b^{\sharp}\psi_{-1}$ in $Z^0\A$.
Replacing $\psi_0$ with $\psi_0+\eps$ we can assume that
$\psi_0a = b^{\sharp}\psi_{-1}$ in $Z^0\A$ without changing the image of $\psi_0$ in $H^0\A$.
From the last equality, we see that the pair $(\psi_{-1},\psi_0)$ induces a unique map on quotients $\psi_{0/-1}:A_1 \to A_0^{\sharp}$ in $Z^0\A$.
In particular, we have two maps of exact triangles in $w^{-1}H^0\A$ extending $(\psi_{-1},\psi_0)$, namely $(\psi_{-1},\psi_0,\psi_1)$ and $(\psi_{-1},\psi_0,\psi_{0/-1})$.
Since $\psi_1$ and $\psi_{0/{-1}}$ are isomorphisms in $w^{-1}H^0\A$, 
\cite[Lemma 4.6]{Balmer:TWGI} provides us with 
an endomorphism $x: A_{-1}^{\sharp} \to A_{-1}^{\sharp}$ such that $x^3=0$ and $\psi_{0/-1} = (1+x)\psi_1$.
In $w^{-1}H^0\A$, the maps $\psi_{0/-1}+\psi_{-1}^{\sharp}\varpi = (2+x)\psi_1$ and  $\psi_0 +\psi_0^{\sharp}\varpi=2\psi_0$ are isomorphisms since $x$ is nilpotent and $1/2\in \A$.
It follows that 
$$(\gamma_{-1},\gamma_0,\gamma_1) = \frac{1}{2}(\psi_{0/-1}^{\sharp}\varpi_{A_{-1}}+\psi_{-1}, \psi_0 +\psi_0^{\sharp}\varpi_{A_0}   ,\psi_{0/-1}+\psi_{-1}^{\sharp}\varpi_{A_1})$$
is a symmetric weak equivalence in the category of exact sequences in $(Z^0\A,w,\sharp,\varpi)$.
This defines an element $[A_0,\gamma_0]-H(A_{-1}) \in R_1$ which maps to 
$[A_0,\psi_0]-H(A_{-1}) \in R_2$ since $\psi_0=\gamma_0$ in $w^{-1}H^0\A$.
\end{proof}

\begin{lemma}
\label{lem:hD=Dh}
Let $\scA = (\A,w,\sharp,\varpi)$ be a pretriangulated dg category with
weak equivalences and duality such that $\frac{1}{2}\in \scA$.
Then the functor $(w\A)_h \to (i\T\scA)_{h}:(X,\ffi) \mapsto (X,\ffi)$ induces
a bijection
$$\pi_0\ (w\A)_h \stackrel{\cong}{\to} \pi_0\ (i\T\scA)_{h}.$$
\end{lemma}

\begin{proof}
The map is surjective  by the following argument.
Let $(A,\alpha)$ be a symmetric space in $\T\scA = w^{-1}Z^0\A$.
We can write $\alpha$ as a fraction $\alpha=a\circ s^{-1}$ with $a:B \to
A^{\sharp}$ and $s:B \to A$ weak equivalences in $\A$.
Let $\ffi= (s^{\sharp}a+a^{\sharp}\varpi s)/2:B \to B^{\sharp}$.
Then $(B,\ffi)$ is a symmetric weak equivalence in $\A$, and $s:B \to A$ defines an isometry in $\T\scA$ between $(B,\ffi)$ and $(A,\alpha)$.

The map in the lemma is injective by the following argument.
Let $(A,\alpha)$ and $(B,\beta)$ be objects of $(w\A)_{h}$, and let
$[fs^{-1}]$ be an isometry in $\T\scA$ from $(A,\alpha)$ to $(B,\beta)$
where $f:E\to B$ and  $s: E \to A$ are weak equivalences in $\A$.
Using the calculus of fractions in $\T\scA = w^{-1}H^0\A$, 
we can assume that 
$\alpha_0=s^{\sharp}\alpha s$ and $\beta_0=f^{\sharp}\beta f$ are homotopic.
Therefore, the difference $\alpha_0-\beta_0$ factors as $E \stackrel{\eps}{\to} I
\stackrel{\delta}{\to} E^{\sharp}$ where $I$ is an injective object in $Z^0\A$.
Then the pair $(E\oplus I,\ffi)$ with  
$$\ffi = \left(\begin{smallmatrix}\beta_0 &
    \delta/2\\ 
    \delta^{\sharp}\varpi/2 & 0\end{smallmatrix}\right): E\oplus I
\to E^{\sharp}\oplus I^{\sharp}$$
defines an object of $(w\A)_{h}$.
The string of maps in $(w\A)_{h}$
$$(A,\alpha) \stackrel{s}{\leftarrow} (E,\alpha_0) \stackrel{\left(\begin{smallmatrix}1\\ \eps\end{smallmatrix}\right)}{\to} (E\oplus I, \ffi) \stackrel{\left(\begin{smallmatrix}1\\ 0\end{smallmatrix}\right)}{\leftarrow} (E,\beta_0) \stackrel{f}{\to} (B,\beta)$$
shows that $[A,\alpha]=[B,\beta] \in \pi_0(w\A)_h$.
\end{proof}

\begin{definition}[Shifted dualities]
Let $\K=(\K,\sharp,\varpi,\lambda)$ be a triangulated category with duality.
The first shifted triangulated category with duality $\K^{[1]} = (\K, \sharp^{[1]}, \varpi^{[1]}, \lambda^{[1]})$ has the same underlying triangulated category as $\K$ and 
$$\sharp^{[1]} = T\sharp,\hspace{3ex} \varpi^{[1]}=-\lambda_{\sharp}\circ\varpi,\hspace{3ex} \lambda^{[1]} = -T(\lambda).$$
One checks that this defines indeed a triangulated category with duality in the sense of Definition \ref{dfn:TriDwithoutTed}.
Similarly, we have a triangulated category with duality $\K^{[-1]}=
(\K, \sharp^{[-1]}, \varpi^{[-1]}, \lambda^{[-1]})$ defined by
$$\sharp^{[-1]} = \sharp T,\hspace{3ex} \varpi^{[-1]}=(\lambda^{\sharp})^{-1}\circ\varpi,\hspace{3ex} \lambda^{[-1]} = -\lambda_T.$$
Iterating, we obtain triangulated categories with duality $\K^{[n]}$ for $n\in \Z$ where $\K^{[n]} = (\K^{[n-1]}){[1]}$ for $n>0$, $\K^{[0]} = \K$, and $\K^{[n]} = (\K^{[n+1]})^{[-1]}$ for $n<0$.
For $m,n\in \Z$, there are natural isomorphisms $(K^{[m]})^{[n]} \cong \K^{[m+n]}$.
For instance 
$$(id,1,\lambda): \K^{[0]} = (\K,\sharp,\varpi,\lambda) \stackrel{\cong}{\longrightarrow} (\K^{[-1]})^{[1]} = (\K, T\sharp T, \lambda_{T\sharp T}\circ (\lambda^{\sharp})^{-1} \circ \varpi, T\lambda_T).$$
A morphism $(F,\rho,\ffi): \K_1 \to \K_2$ of triangulated categories with duality induces another such morphism $(F,\rho,\ffi^{[1]}): \K_1^{[1]} \to \K_2^{[1]}$ where 
$\ffi^{[1]} = T(\ffi) \circ \rho_{\sharp}$.
\end{definition}

\begin{definition}
For a triangulated category with duality $\K$, one defines the {\em shifted Witt and Grothendieck-Witt groups} for $n\in \Z$ by
$$W^n(\K) = W^0(\K^{[n]})\hspace{3ex}\text{and}\hspace{3ex}GW^n(\K) = GW^0(\K^{[n]}).$$
\end{definition}

Finally, we want to construct for $n\in \Z$ equivalences of triangulated categories with duality $\T(\scA^{[n]})\simeq (\T\scA)^{[n]}$.
It suffices to construct such equivalences for $n=\pm 1$.
Consider the following (admissible) natural transformation in $\scC_{\bk}$ 
$$\beta: A[B,E] \to [B,AE]: \beta(a\otimes f)(b)= a \otimes f(b).$$
For $A=T$, $E=\1$ and $\scA \in \dgCatWD$ it defines a natural transformation $\beta: T \sharp \to \sharp_T$ which for $A \otimes X \in \scC_{\bk}\otimes\A$ is
$$\beta \otimes 1: T[A,\1] \otimes X^{\sharp} \to [A,T]\otimes X^{\sharp}.$$
Similarly, the natural isomorphism in $\scC_{\bk}$ 
$$\phi:[EA,BD] \underset{\cong}{\stackrel{\can}{\longleftarrow}} [E,B][A,D] \underset{\cong}{\stackrel{\ffi}{\longrightarrow}} [A,[E,B]D]$$
with $E=T$ and $B=D=\1$ defines a natural isomorphism of functors 
$\phi: \sharp T \to \sharp_{T^{-1}}$ which for $A \otimes X \in \scC_{\bk}\otimes\A$ is
$$\phi \otimes 1: [TA,\1]\otimes X^{\sharp} \to [A,[T,1]]\otimes X^{\sharp}$$
where $T^{-1}=[T,\1] \cong k[-1]$.

\begin{lemma}
\label{lem:TriShif=ShiftTri}
Let $\scA$ be a pretriangulated dg category with weak equivalences and duality.
Then $(id,1,\beta)$ and $(id,1,\phi)$ define equivalences of triangulated categories with duality
$$(\T\scA)^{[1]} \stackrel{\simeq}{\longrightarrow} \T(\scA^{[1]}) \hspace{3ex}\text{and}\hspace{3ex}(\T\scA)^{[-1]} \stackrel{\simeq}{\longrightarrow} \T(\scA^{[-1]}).
$$
\end{lemma}

\begin{proof}
The functors $(id,1)$ are equivalences of triangulated categories.
So, the only thing to prove is commutativity of the 
two diagrams in Definition \ref{dfn:MorphTriD}.
Since all natural transformations involved are of the form $? \otimes 1$ it suffices to check the commutativity of the corresponding diagrams in $\scC_{\bk}$.

For the first equivalence, the first diagram in Definition \ref{dfn:MorphTriD} is is the outer diagram in
$$
\xymatrix{
A \ar[r]^{\varpi} \ar[d]_{\varpi} & [[A,\1],\1] \ar@{-->}[drr]_{\eps_{[A,\1]}} && T[T[A,\1],\1]\ar[ll]_{-\alpha_{[A,\1]}}^{\cong} \ar[d]^{\beta} \\
[[A,T],T] \ar[rrr]_{\beta} &&& [T[A,\1],T]
}$$
where the dashed arrow is the (admissible) natural transformation 
$$\eps_B: [B,E] \to [TB,TE]: \eps_B(f)(t\otimes b)= (-1)^{|t||f|}t\otimes f(b)$$
with $B=[A,\1]$ and $E=\1$.
Commutativity of this diagram can be checked directly, or by coherence.
Similarly, commutativity of the second diagram in Definition \ref{dfn:MorphTriD} for the first equivalence requires us to check commutativity in $\scC_{\bk}$ of the diagram 
$$
\xymatrix{
B[A,E] \ar[d]_{\beta} & BT[TA,E] \ar[l]_{1 \otimes \alpha} & TB[TA,E] \ar[l]_{c \otimes 1} \ar[d]^{1\otimes \beta} \\
[A,BE] && T[TA,BE]. \ar[ll]^{\alpha}
}
$$
with $B=T$ and $E=\1$.
Since $c_{T,T}=-1$, commutativity of the diagram follows from 
coherence, or by diagram chase.
This proves the first equivalence of triangulated categories with duality in the lemma.
The second is \red{proved} similarly, and we omit the details.
\end{proof}

\begin{corollary}
\label{cor:BalmerWnismyWn}
Let $\scA$ be a pretriangulated dg category with weak equivalences and duality such that $\frac{1}{2}\in \A$.
Then for all $n\in \Z$ there are natural isomorphisms
$$GW_0^{[n]}(\scA) \cong GW^n(\T\scA)\hspace{3ex}\text{and}\hspace{3ex}W^{[n]}(\scA) \cong W^n(\T\scA).$$
\Qed
\end{corollary}

\begin{remark}
Let $(\K,\sharp,\varpi,\lambda)$ be a triangulated category with duality.
Our definition of the Witt and Grothendieck-Witt groups $W^n(\K)$ and $GW^n(\K)$ of $\K$ agree with Balmer's and Walter's definitions in \cite{Balmer:TWGI} and \cite{walter:GWTrPreprint}.
This is because for an inner product space in the category $\Delta\K$ of exact triangles as in Definition \ref{dfn:TriGW0andW} (\ref{dfn:itm2:GW0Tri}),
the map $(T\ffi_{-1})h: Z \to T(Z^{\sharp})$ is a symmetric map in the category with duality
$$(\K,T\sharp,\lambda_{\sharp}\varpi)= (\K,\sharp^{[1]},-\varpi^{[1]}) \stackrel{T}{\simeq} \K^{[-1]}$$
which shows that the map from our definition to Balmer's definition is well-defined.
Conversely, Balmer shows in \cite{Balmer:TWGI} that any 
$(\K,T\sharp,\lambda_{\sharp}\varpi)$-symmetric map $Z \to T(Z^{\sharp})$ is part of a symmetric space in $\Delta\K$ as in Definition \ref{dfn:TriGW0andW} (\ref{dfn:itm2:GW0Tri}) (with $\ffi_{-1}=id$).
This shows that the map from his definition to our definition is well-defined.
Since both maps are the identity maps on generators, they are isomorphism.
\end{remark}

\section{The multi-simplicial $\RR_{\bullet}$-construction}
\label{sec:multiSimplRdot}

\subsection{Waldhausen's iterated $S_{\bullet}$-construction}
\label{subsec:ItertdSdot}
To fix notation and to motivate our construction of the Grothendieck-Witt spectrum in Definition \ref{dfn:GWspectrum} below,
we remind the reader of Waldhausen's definition of the $K$-theory of an exact category with weak equivalences \cite{wald:spaces} and its explicit deloopings \cite[Proposition 1.5.3]{wald:spaces}.
Recall that an exact category with weak equivalences is a pair $(\E,w)$ where $\E$ is an exact category and $w$ is a set of morphisms in $\E$ called weak equivalences satisfying certain properties; see for instance \cite[\S 2.2]{myMV}.
It is called pointed if it is equipped with a choice of zero object in $\E$ which we will call base point.

For an integer $n\geq 0$, let $[n]$ be the totally ordered set  
$$[n] = \{0<1<...<n\}$$
of \red{$n+2$} elements.
As usual, an ordered set is considered as a category with unique non-identity morphism $a\to b$ if $a<b$.
Let $\Ar[n]=\Fun([1],[n])$ be the category of arrows in $[n]$.
For an integer $n \geq 0$ and a pointed exact category with weak equivalences $(\E,w)$, Waldhausen constructs 
another pointed exact category with weak equivalences $S_n\E$ as the full subcategory of $\Fun(\Ar[n],\E)$ of those functors
$$A:\Ar[n] \to \E: (p\leq q) \mapsto A_{p,q}$$
for which $A_{p,p} = 0$ (the base point zero object of $\E$) for all $p\in [n]$  and 
$0\to A_{p,q} \to A_{p,r}\to A_{q,r}\to 0$ is an admissible short exact sequence in $\E$ whenever $p\leq q \leq r \in [n]$.
A sequence in $S_n\E$ is exact if it is exact at each $(p,q)$-spot and a morphism in $S_n\E$ is a weak equivalence if it is a weak equivalence at each $(p,q)$-spot, $p\leq q \in [n]$.

Recall \cite[p. 330 Definition]{wald:spaces} that the $K$-theory space of $(\E,w)$ is the loop space $\Omega |wS_{\bullet}\E|$ of the topological realization of the simplicial category $n\mapsto wS_n\E$.
The {\em connective $K$-theory spectrum $K(\E,w)$} of $(\E,w)$ is the spectrum
$$\{|w\E|, |wS_{\bullet}\E|,...,|wS^{(n)}_{\bullet}\E|,...\}$$
with bonding maps defined in \cite[p. 341]{wald:spaces}
where $i\mapsto S^{(n)}_{i}\E$ denotes the $n$-th iterate of the $S_{\bullet}$-construction:
$$S^{(n)}_{i}\E=\underset{n}{\underbrace{S_i\dots S_i}}\E.$$
This is a positive $\Omega$-spectrum \cite[Proposition 1.5.3 and Remark thereafter]{wald:spaces}.
It was given the structure of a symmetric spectrum in \cite{geisserHesselholt:cyclicSchemes}.

\subsection{The $\RR_{\bullet}$-construction}
\label{subsec:TheRRconstr}
Recall \cite[\S 2.3]{myMV}
that an {\it exact category with weak equivalences and duality} is a quadruple
$(\E,w,*,\eta)$
with $(\E,w)$ an exact category with weak equivalences and
$(\E,*,\eta)$ a category with duality such that 
$*:(\E^{op},w) \to (\E,w)$ is an exact functor (in particular, $*(w) \subset
w$) and 
$\eta: id \to **$ is a natural weak equivalence. 
Note that if $\E$ is an exact category with weak equivalences and duality, the
category $w\E$ of weak equivalences in $\E$ is a category with duality.

For an integer $n \geq 0$, let $\underline{n}$ be the totally ordered set of $2n+1$ elements
$$\underline{n} = \{n'<(n-1)' < \cdots < 0'<0<\cdots < n\}.$$
This is a category with unique (strict) duality $n \leftrightarrow n'$.
The assignment $[n] \mapsto \underline{n}$ defines a simplicial category with strict duality where for $\theta:[n] \to [m]$ the induced map $\underline{\theta}:\underline{n} \to \underline{m}$ is $\underline{\theta}(i) = \theta(i)$ and $\underline{\theta}(i') = \theta(i)'$.
Let $\Ar(\underline{n})=\Fun([1],\underline{n})$ be the category of arrows in $\underline{n}$.
When $n$ varies, this is a simplicial category with strict duality.

For a pointed exact category with weak equivalences and duality $\E$, 
the assignment  
$$[n] \mapsto \Fun(\Ar(\underline{n}),\E)$$
is a pointed simplicial exact category with duality with base point the unique functor $\Ar(\underline{n}) \to 0$. 
As in Section \ref{subsec:ItertdSdot}, we write
$$\RR_n\E \subset \Fun(\Ar(\underline{n}),\E)$$ 
for the full subcategory of those objects $A: \Ar(\underline{n}) \to \E$ for which $A_{i,i}=0$ (the base point zero object of $\E$), and for which for all $i\leq j\leq k \in \underline{n}$ the sequence $0\to A_{i,j} \to A_{i, k} \to A_{j,k}\to 0$ is a admissible exact sequence in $\E$.
A map or sequence in $\RR_n\E$ is a weak equivalence or admissible exact sequence if it is at each $(p,q)$-spot for all objects $(p,q)$ of  $\Ar(\underline{n})$.
The duality on $\Fun(\Ar(\underline{n}),\E)$ makes
$\RR_n\E$
into a pointed exact category with weak equivalences and duality.
Varying $n$, we obtain a simplicial pointed exact category with weak equivalences and duality $\RR_{\bullet}\E$.
The simplicial category $\RR_{\bullet}\E$ is the edge-wise subdivision \cite[\S 2.4]{myMV}, \cite[\S 1.9]{wald:spaces} of Waldhausen's $S_{\bullet}$-construction and was denoted $S^e_{\bullet}\E$ in \cite{myMV}.

The inclusion $\iota: [n]\subset \underline{n}:i\mapsto i$ induces a functor of simplicial exact categories with weak equivalences
$\RR_{\bullet}\E \to S_{\bullet}\E: A \mapsto A\circ \iota$
and thus a map $(w\RR_{\bullet}\E)_h \to wS_{\bullet}\E$ as the composition 
\begin{equation}
\label{eqn:RtoSmap}
(w\RR_{\bullet}\E)_h \stackrel{(A,\ffi)\mapsto A}{\longrightarrow} w\RR_{\bullet}\E  \stackrel{A\mapsto A\iota}{\longrightarrow} wS_{\bullet}\E.
\end{equation}
where $\C_h$ is the category of symmetric forms in $\C$; see Section \ref{subsec:CatsWDual}.

\subsection{Homotopy pull-backs}
Let $f: X \to Z$ and \red{$g: Y \to Z$} be maps of topological spaces.
Then the homotopy pull-back of the diagram $X \to Z \leftarrow Y$ is the topological space
$$X \underset{Z}{\stackrel{h}{\times}}\hspace{1ex}Y
=
\{(x,\sigma,y) \in X\times Z^I \times Y|\ f(x)= \sigma(0), g(y)=\sigma(1)\ \}
$$
equipped with the topology making it a subspace of $X\times Z^I \times Y$ where $Z^I$ is the space of paths $I \to Z$ in $Z$ and $I$ is the unit interval $[0,1] \subset \mathbb{R}$.
It is equipped with the two projections maps 
$p_X: X \underset{Z}{\stackrel{h}{\times}}\hspace{1ex}Y \to X: (x,\sigma,y) \mapsto x$
and 
$p_Y: X \underset{Z}{\stackrel{h}{\times}}\hspace{1ex}Y \to Y: (x,\sigma,y) \mapsto y$
making the square
$$\xymatrix{
X \underset{Z}{\stackrel{h}{\times}}\hspace{1ex}Y \ar[r]^{\hspace{2ex}p_X} \ar[d]_{p_Y} & X \ar[d]^f\\
Y \ar[r]_g & Y 
}$$
commute up to (a preferred) homotopy.

The homotopy pull-back is associative in the sense that
for a diagram $X \to Z \leftarrow Y \to U \leftarrow V$, we have 
$$(X \underset{Z}{\stackrel{h}{\times}}\hspace{1ex}Y) \underset{U}{\stackrel{h}{\times}}\hspace{1ex}V = 
X \underset{Z}{\stackrel{h}{\times}}\hspace{1ex} (Y \underset{U}{\stackrel{h}{\times}}\hspace{1ex}V).
$$

If $g:Y=\pt \to Z$ is the inclusion of the base point of $Z$ then
$X \underset{Z}{\stackrel{h}{\times}}\hspace{1ex}\pt$ is the homotopy fibre of $f: X \to Z$.
If $f:X=\pt \to Z$ and $g:Y=\pt \to Z$ are the base point inclusions then 
$\pt \underset{Z}{\stackrel{h}{\times}}\hspace{1ex}\pt$ is the loop space $\Omega Z$ of $Z$.

\subsection{The Grothendieck-Witt space}
\label{subsec:GWspace}
The {\em Grothendieck-Witt space} of a pointed exact category with weak equivalences and duality $\E$ was defined in \cite[\S 2.7 Definition 3]{myMV} as the homotopy fibre of the topological realization of the map (\ref{eqn:RtoSmap}), that is, it is the homotopy fibre product
\begin{equation}
\label{eqn:GWspace}
|(w\RR_{\bullet}\E)_h| \hspace{1ex} \underset{|wS_{\bullet}\E|}{\stackrel{h}{\times}}\hspace{1ex}\pt.
\end{equation}
For $i\geq 0$, we write $GW_i(\E)$ for the $i$-th homotopy group of the pointed topological space (\ref{eqn:GWspace}).

Inclusion of degree zero simplices defines a map 
$(w\E)_h = (w\RR_{0}\E)_h \to (w\RR_{\bullet}\E)_h$ of simplicial categories such that the composition 
\begin{equation}
\label{eqn:PrimAddtyFib}
(w\E)_h \to (w\RR_{\bullet}\E)_h \to wS_{\bullet}\E
\end{equation}
is trivial (because $wS_{0}\E=\pt$ is the one-object one-morphism category).
This induces a canonical map from $|(w\E)_h|$ to the Grothendieck-Witt space (\ref{eqn:GWspace}) of $\E$ which should be thought of as some kind of a group completion map (it really is a group completion if $\E$ is a split exact category over $\Z[\frac{1}{2}]$ with weak equivalences the set of isomorphisms; see Appendix \ref{appdx:Homology}).

\subsection{The multi-simplicial $\RR_{\bullet}$-construction}
For a pointed exact category with weak equivalences and duality $\E$,
we write $\RR^{(n)}_{\bullet\cdots \bullet}\E$ for the $n$-th iterate of the $\RR_{\bullet}$-construction.
This is the $n$-simplicial pointed exact category with weak equivalences and duality
$$\RR^{(n)}_{k_1,...,k_n}\E = \RR_{k_1}\RR_{k_2}\dots \RR_{k_n}\E,\vspace{2ex}$$
which is the full pointed exact subcategory with weak equivalences and duality of
$$\Fun(\Ar(\underline{k_1})\times \Ar(\underline{k_2})\times \cdots \times \Ar(\underline{k_n}), \E)$$
consisting of those functors $A: \Ar(\underline{k_1})\times \Ar(\underline{k_2})\times \cdots \times \Ar(\underline{k_n}) \longrightarrow \E$
for which for all $r,s=1,...,n$ and all $i_r \leq j_r \in \underline{k_r}$, $r\neq s$, the functor
$$A_{(i_1,j_1), \dots (i_{s-1},j_{s-1}), \bullet, (i_{s+1},j_{s+1}),\dots (i_{n}, j_{n}}): \Ar(\underline{k_s})\to \E$$
is an object of $\RR_{k_s}\E$.
Weak equivalences and admissible exact sequences are those maps and sequences of diagrams which are weak equivalences and admissible exact sequences when evaluated at every object of $\Ar(\underline{k_1})\times \cdots \times \Ar(\underline{k_n})$.
We denote by $\RR^{(n)}_{\bullet}\E$ the diagonal of the multi-simplicial pointed exact category with weak equivalences and duality $\RR^{(n)}_{\bullet\cdots \bullet}\E$.

If we replace $\E$ with $\RR_{\bullet}^{(n)}\E$ in the sequence (\ref{eqn:PrimAddtyFib}) we obtain the sequence of pointed topological spaces
\begin{equation}
\label{eqn:AddtyFib1}
|(w\RR_{\bullet}^{(n)}\E)_h| \longrightarrow |(w\RR_{\bullet}\RR_{\bullet}^{(n)}\E)_h| \longrightarrow |wS_{\bullet}\RR_{\bullet}^{(n)}\E|
\end{equation}
with trivial composition.

\begin{proposition}
%\label{prop:AddtyFib1}
Let $\E$ be a pointed exact category with weak equivalences and duality.
Then for all $n\geq 1$, the sequence
(\ref{eqn:AddtyFib1})
of pointed topological spaces
is a homotopy fibration.
\end{proposition}

\begin{proof}
We will use the fact that a sequence of pointed simplicial spaces which is degree-wise a homotopy fibration with connected base induces a homotopy fibration after topological realization \cite[Lemma 5.2]{wald:genProdsI}.
By \cite[\S 3 Remark 7]{myMV}, for every integer $q\geq 0$, the sequence 
$$|(w\RR_{\bullet}^{(n)}\E)_h| \longrightarrow |(w\RR_{q}\RR_{\bullet}^{(n)}\E)_h| \longrightarrow |wS_{q}\RR_{\bullet}^{(n)}\E|$$
is a homotopy fibration when $n=1$.
In particular, for all integers $p\geq 0$ and $n\geq 1$, the sequence
$$|(w\RR_{\bullet}\RR_{p}^{(n-1)}\E)_h| \longrightarrow |(w\RR_{q}\RR_{\bullet}\RR_{p}^{(n-1)}\E)_h| \longrightarrow |wS_{q}\RR_{\bullet}\RR_{p}^{(n-1)}\E|$$
is a homotopy fibration.
Since the base of that fibration is connected as it is homotopy equivalent to the connected space $|wS_{q}S_{\bullet}\RR_{p}^{(n-1)}\E|$ ({\em cf.} \cite[\S 2.4 Lemma 1]{myMV}), the realization of the sequence in the $p$-direction (which is the first sequence above) is a homotopy fibration.
By the same argument, realizing the first sequence in the $q$-direction yields the homotopy fibration (\ref{eqn:AddtyFib1}).
\end{proof}

Composing the second map in (\ref{eqn:AddtyFib1}) with the homotopy equivalence 
$$|wS_{\bullet}\RR_{\bullet}^{(n)}\E| \stackrel{\sim}{\to} |wS_{\bullet}S_{\bullet}^{(n)}\E|: A \mapsto A\iota$$ of \cite[\S 2.4 Lemma 1]{myMV}, we obtain the following.

\begin{corollary}
\label{cor:AddtyExFib1}
Let $\E$ be a pointed exact category with weak equivalences and duality.
Then for every integer $n\geq 1$, the sequence
$$|(w\RR_{\bullet}^{(n)}\E)_h| \longrightarrow |(w\RR_{\bullet}^{(1+n)}\E)_h| \longrightarrow |wS_{\bullet}^{(1+n)}\E|$$
is a homotopy fibration of pointed topological spaces.
\Qed
\end{corollary}

\subsection{The $\RR_{\bullet}^{(n)}$-construction for dg categories}
Let $\scA = (\A,w,\sharp,\varpi)$ be a pointed pretriangulated (or just exact) dg category with weak equivalences and duality.
Then the quadruple $Z^0\scA = (Z^0\A,w,\sharp,\varpi)$ is a pointed exact category with weak equivalences and duality, and therefore, the constructions and results above apply to $Z^0\scA$.
But it will be useful to stay within the category of dg categories.
So, we will make $\RR_i^{(n)}\A$ into a dg category with weak equivalences and duality for $\A \in \dgCatWD_*$ even when $\A$ is not exact.

For a pointed dg category with duality $\A$, the assignment  
$$[n] \mapsto \Fun(\Ar(\underline{n}),\A)$$
is a pointed simplicial dg-category with duality with base point the unique functor $\Ar(\underline{n}) \to 0$.
We write
$$\RR_n\A \subset \Fun(\Ar(\underline{n}),\A)$$ 
for the full dg subcategory of those objects $A: \Ar(\underline{n}) \to \A$ for which $A_{i,i}=0$, and for which for all $i\leq j\leq k \in \underline{n}$ the sequence $0\to A_{i,j} \to A_{i, k} \to A_{j,k}\to 0$ is exact.
Recall that exact sequences in a dg category are defined even if the dg category itself is not exact.
The duality on $\Fun(\Ar(\underline{n}),\A)$ makes
$\RR_n\A$
into a dg-category with duality.
Varying $n$, we obtain a pointed simplicial dg category with duality $\RR_{\bullet}\A$.
For a pointed dg category with weak equivalences and duality $\scA=(\A,w)$, we obtain a pointed simplicial dg category with weak equivalences and duality
$$n\mapsto \RR_{n}\scA = (\RR_{n}\A,w).$$
corresponding the the (saturation of the) localization pair $(\RR_{n}\A,\RR_{n}\A^w)$.

To define the iterated $\RR_{\bullet}$-construction for dg categories,
write $\RR^{(n)}_{\bullet\cdots \bullet}\A$ for the $n$-simplicial dg-category with duality which in degree $({k_1,...,k_n})$
is the full dg-subcategory
$$\RR^{(n)}_{k_1,...,k_n}\A \subset \Fun(\Ar(\underline{k_1})\times \Ar(\underline{k_2})\times \cdots \times \Ar(\underline{k_n}), \A)$$
of those dg-functors $A: \Ar(\underline{k_1})\times \Ar(\underline{k_2})\times \cdots \times \Ar(\underline{k_n}) \longrightarrow \A$
for which for all $r,s=1,...,n$ and all $i_r \leq j_r \in \underline{k_r}$, $r\neq s$, the dg functor
$$A_{(i_1,j_1), \dots (i_{s-1},j_{s-1}), \bullet, (i_{s+1},j_{s+1}),\dots (i_{n}, j_{n}}): \Ar(\underline{k_s})\to \A$$
is an object of $\RR_{k_s}\A$.
We denote by $\RR^{(n)}_{\bullet}\A$ the diagonal of the multi-simplicial dg category with duality $\RR^{(n)}_{\bullet\cdots \bullet}\A$.
If $\scA = (\A,w)$ is a dg category with weak equivalences and duality then so is
$$\RR^{(n)}_{\bullet}\scA = (\RR^{(n)}_{\bullet}\A, w)$$
given by the (saturation of the) localization pair
$(\RR^{(n)}_{\bullet}\A,\RR^{(n)}_{\bullet}\A^w)$.

\begin{remark}
For an exact dg category with weak equivalences and duality $(\A,w,\sharp,\varpi)$, we have equalities of categories with duality
\begin{equation}
\label{eqn:dgRisnotiteratedR}
w\RR_{\bullet}^{(n)}\A = wZ^0\RR_{\bullet}^{(n)}\A = w\RR_{\bullet}^{(n)}Z^0\A.
\end{equation}
This lets us apply the results obtained for exact categories with weak equivalences and duality to exact dg categories with weak equivalences and duality.
Note, however, that in general, the exact category 
$Z^0\RR_{k}^{(n)}\A$ has fewer exact sequences as the exact category $\RR_{k}^{(n)}Z^0\A$, and so, $\RR_{k}\RR_{l}\A \neq \RR_{k,l}^{(2)}\A$ even when $\A$ is exact, and thus, for dg categories, $\RR_{\bullet}^{(n)}\A$ is not the iterate of $\RR_{\bullet}\A$.
\end{remark}

Using equality (\ref{eqn:dgRisnotiteratedR}), we obtain from Corollary 
\ref{cor:AddtyExFib1} the following.

\begin{proposition}
\label{prop:AddtyFib1}
Let $\scA$ be a pointed exact dg category with weak equivalences and duality.
Then for every integer $n\geq 1$, the sequence (\ref{eqn:PrimAddtyFib}) induces a homotopy fibration of pointed topological spaces
$$|(w\RR_{\bullet}^{(n)}\scA)_h| \longrightarrow |(w\RR_{\bullet}^{(1+n)}\scA)_h| \longrightarrow |wS_{\bullet}^{(1+n)}\scA|.$$
\Qed
\end{proposition}

\begin{remark}
Recall (Section \ref{subsec:GWspace}) that the Grothendieck-Witt space of an exact dg category with weak equivalences and duality is the homotopy fibre of the second map in Proposition \ref{prop:AddtyFib1} when $n=0$.
Proposition \ref{prop:AddtyFib1} therefore lets us deduce results for the functors $(w\RR_{\bullet}^{(n)})_h$ from the knowledge of similar results for the functors $wS_{\bullet}^{(n)}$, that is for $K$-theory, and for the Grothendieck-Witt space functor.
For example, an exact dg form functor $\scA \to \scB$ between pretriangulated dg categories with weak equivalences and duality which induces isomorphisms of $K$-groups and $GW_i$-groups for $i\geq 0$ induces homotopy equivalences
$|(w\RR_{\bullet}^{(n)}\scA)_h| \to |(w\RR_{\bullet}^{(n)}\scB)_h|$
for all $n\geq 0$.
In the sequel, we will frequently apply this argument. 
\end{remark}

\begin{corollary}
\label{cor:pi0RRn}
Let $\scA$ be a pointed pretriangulated dg category with weak equivalences and duality.
Then for every integer $n\geq 1$, the map
$$\pi_i |(w\RR_{\bullet}^{(n)}\scA)_h| \longrightarrow \pi_i |(w\RR_{\bullet}^{(n+1)}\scA)_h|$$
is an isomorphism for $0\leq i < n $ and a surjection for $i=n>0$.
In particular, for every $n\geq 1$, the (composite) map 
$\pi_0 |(w\scA)_h| \longrightarrow \pi_0 |(w\RR_{\bullet}^{(n)}\scA)_h|$
induces an isomorphism 
$$W_0(\scA) \stackrel{\cong}{\longrightarrow} \pi_0 |(w\RR_{\bullet}^{(n)}\scA)_h|.$$
\end{corollary}

\begin{proof}
The first statement follows from Proposition \ref{prop:AddtyFib1} together with the fact that $\pi_i|wS_{\bullet}^{(n)}\scA| = 0$ for $0 \leq i < n$.
For the second statement, recall that the dg functor $\A \to \A^{\pretr}$ is an equivalence for $\A$ pretriangulated.
Therefore,
$$W_0(\scA) =  W_0(Z^0\scA^{\pretr},w) \stackrel{\cong}{\leftarrow} 
W_0(Z^0\scA,w)\stackrel{\cong}{\rightarrow} \pi_0 |(w\RR_{\bullet}\scA)_h|$$
where the last isomorphism is
\cite[\S3 Proposition 3 and Remark 5]{myMV}.
\end{proof}

Recall from Section \ref{subsec:DGweakEq} the category $\A^w$ of $w$-acyclic objects in an exact category with weak equivalences $(\A,w)$.

\begin{proposition}
\label{prop:pretrFibPrp}
Let $(\A,w,\vee,\can)$ be a pointed pretriangulated dg category with weak equivalences and duality.
Let $v \subset \A \simeq \A^{\pretr}$ be a larger set of weak equivalences such that $(\A,v,\vee,\can)$ is also a pretriangulated dg category with weak equivalences and duality.
Then for $n\geq 1$ the functor $|(w\RR_{\bullet}^{(n)}\phantom{\scA})_h|$ applied to the commutative square of pointed pretriangulated dg categories with weak equivalences and duality
$$
\xymatrix{ 
(\A^v,w) \hspace{1ex} \ar[r] \ar[d] &
(\A,w) \ar[d]\\
(\A^v,v) \ar[r] & (\A,v)
}
$$
yields a homotopy cartesian square of pointed topological spaces with contractible lower left corner.
\end{proposition}

\begin{proof}
Since $\A$ is pretriangulated, the map $\A \to \scC_{\bk}^{[0]}\A$ is an equivalence of dg categories with duality.
Therefore, the statement of the proposition holds for $\A$ if and only if it holds for $\scC_{\bk}^{[0]}\A$.
We will prove the proposition for the latter category.

The pretriangulated dg category with duality $\scC_{\bk}^{[0]}\A$ has a symmetric cone in the sense of \cite[\S 4 Definition 4]{myMV} given by the functors
$E \mapsto PE = [C,\1]E$ and $E\mapsto CE$, the natural admissible epimorphism
$[i,1]: PE= [C,\1]E \to [\1,\1]E\cong E$, the natural admissible monomorphism $i:E \to CE$, and the natural map $P(E^{\vee}) \to (CE)^{\vee}$ which for $E=AX$ with $A\in \scC$ and $X\in \A$ is given by
$$\can \otimes 1_{X^{\vee}}: [C,\1][A,\1]X^{\vee} \to [CA,\1]X^{\vee}$$
where $\can: [C,\1][A,\1] \to [CA,\1]$ is the map (\ref{eqn:dualityCompIsoForTensorProds}).
Therefore, the square in the proposition induces homotopy cartesian squares of $K$-theory and Grothendieck-Witt spaces with contractible lower left corner, by \cite[Theorem 1.6.4]{wald:spaces} and \cite[\S 4 Theorem 6]{myMV}.
By Proposition \ref{prop:AddtyFib1}, we are done.
\end{proof}

In the proof of the space-level version (Proposition \ref{prop:SpaceKaroubiProto1}) of the algebraic Bott sequence (Theorem \ref{thm:PeriodicityExTriangle}), we will need the following lemma.

\begin{lemma}
\label{lem:KhInvI}
Let $(F,\ffi):\scA \to \scB$ be an exact dg form functor of pretriangulated dg categories with weak equivalences and duality.
Suppose that $\frac{1}{2}\in \scA, \scB$ and that there are an exact pointed dg functor $G:\scB \to \scA$ (no compatibility with dualities required) and zig-zags
of natural weak equivalences of pointed exact dg functors between $FG$ and 
$id_{\scB}$ and between $GF$ and $id_{\scA}$.
Then for $n\geq 1$, the form functor $(F,\ffi)$
induces homotopy equivalences of pointed spaces
$$|(w\RR_{\bullet}^{(n)}\scA)_h| \stackrel{\sim}{\longrightarrow} |(w\RR_{\bullet}^{(n)}\scB)_h|.$$
\end{lemma}

\begin{proof}
Since the lemma also holds for $K$-theory (same proof as below), it suffices to show that $\scA \to \scB$ induces a homotopy equivalence of Grothendieck-Witt spaces, by Proposition \ref{prop:AddtyFib1}.
The hypothesis of the lemma imply that the dg-functor $F:\dgFun(\scB,\scA) \to \dgFun(\scB,\scB):
H \mapsto F\circ H$
induces an equivalence of triangulated categories $w^{-1}H^0\dgFun(\scB,\scA) \to w^{-1}H^0\dgFun(\scB,\scB)$ with inverse $G$.
By Lemma \ref{lem:hD=Dh}, $(F,\ffi)$ induces an isomorphism of abelian monoids
$$\pi_0\  (w\dgFun(\scB,\scA))_{h} \stackrel{\cong}{\to} \pi_0\
(w\dgFun(\scB,\scB))_{h}.$$
So, there is an object $(H,\psi)\in (w\dgFun(\scB,\scA))_{h}$ which under this isomorphism 
goes to $id_{\scB} \in (w\dgFun(\scB,\scB))_{h}$.
This means that there is a zigzag of natural weak equivalences between
$(F,\ffi)\circ (H,\psi)$  and $id_{\scB}$ compatible with forms.
By \cite[\S 2 Lemma 2]{myMV},
the dg form functors $(F,\ffi)\circ (H,\psi)$  and $id_{\scB}$ induce homotopic maps on Grothendieck-Witt spaces.
In particular, $(F,\ffi)$ is surjective, and $(H,\psi)$ is injective on
higher Grothendieck-Witt groups $GW_i$, $i\geq 0$.

By construction of $H$ and the existence of $G$, there are zig-zags of weak
equivalences of dg form functors between $HF$ and $id_{\scA}$ and between $FH$ and $id_{\scB}$. 
By the argument above (with $(H,\psi)$ in place of $(F,\ffi)$), we see that 
$(H,\psi)$ is also surjective and hence bijective on higher Grothendieck-Witt
groups. 
From the previous paragraph it follows that
$(F,\ffi)$ induces isomorphisms $GW_i(\scA) \to GW_i(\scB)$, $i\geq 0$.
\end{proof}

The following proposition is a space-level version of the homotopy fibration in Theorem \ref{thm:PeriodicityExTriangle} below.
Together with Corollary \ref{cor:RFunIsSdot} it implies Karoubi's Fundamental Theorem \cite{Karoubi:Annals}.

Let $\scA$ be a pretriangulated dg category with weak equivalences and duality.
The unique map $[1] \to [0]$ induces a duality preserving exact dg functor 
$I:\scA \to \Fun([1],\scA): A \mapsto 1_A$ whose composition with the cone functor of Section \ref{sec:ConeCounterex} has image in the full dg subcategory $(\scA^{[1]})^w\subset \scA^{[1]}$ of $w$-acyclic objects.
In particular, the square in the following proposition commutes.

\begin{proposition}
\label{prop:SpaceKaroubiProto1}
Let $\scA$ be a pretriangulated dg category with weak equivalences and duality.
Assume that $\frac{1}{2}\in \scA$.
Then for $n\geq 1$ the functor $|(w\RR_{\bullet}^{(n)}\phantom{\scA})_h|$ applied to the commutative square of exact dg form functors\\
%\begin{equation}
%\label{eqn:prop:funSpaces}
$\displaystyle{
\begin{array}{cc}
\hspace{3ex}
(\square_{\scA})&
\hspace{17ex}
\xymatrix{ 
\scA \hspace{1ex} \ar@{^(->}[r]^{\hspace{-5ex}I} \ar[d]_{\Cone_{|\scA}} &
\Fun([1], \scA) \ar[d]^{\Cone}\\
(\scA^{[1]})^w\phantom{2} \ar@{^(->}[r] &
 \scA^{[1]}
}
\end{array}
}$
\vspace{1ex}

\noindent
%\end{equation}
yields a homotopy cartesian square of pointed topological spaces with contractible lower left corner.
In particular, for $n\geq 1$ we obtain a homotopy fibration of pointed spaces
$$|(w\RR^{(n)}_{\bullet}\scA)_h| \to |(w\RR^{(n)}_{\bullet}\Fun([1],\scA))_h|  \to |(w\RR_{\bullet}^{(n)}\scA^{[1]})_h|.$$
\end{proposition}

\begin{proof}
Let $w \subset Z^0\A$ be the set of weak equivalences in $\scA$, and let
 $v\subset Z^0\Fun([1],\A)$ be the set of morphisms $f$ in $\Fun([1],\A)$  for 
which $\Cone(f)$ is a weak equivalence in $\scA$.
By Proposition \ref{prop:pretrFibPrp}, the diagram 
$$\xymatrix{
(\Fun([1],\A)^v,w) \ar[d] \ar[r] & (\Fun([1],\A),w) \ar[d]\\
(\Fun([1],\A)^v,v) \ar[r] & (\Fun([1],\A),v)
} $$
induces a homotopy cartesian square of $|(w\RR_{\bullet}^{(n)}\phantom{\scA})_h|$ spaces with contractible lower left corner.
The functor $\A \to \Fun([1], \A)$ factors through 
the full subcategory $\Fun([1], \A)^v \subset \Fun([1], \A)$ of $v$-acyclic objects.
By \cite[\S 2 Lemma 2]{myMV},
the induced form functor $(G,id):(\A,w) \to (\Fun([1], \A)^v,w)$
yields an equivalence of Grothendieck-Witt spaces.
More precisely, the form functor
$$(F,\ffi): (\Fun([1], \A)^v \to \A: (f:A_0\to A_1) \mapsto A_0$$
with duality compatibility map 
$\ffi_f = f^*$ defines an inverse up to homotopy of $(G,id)$ in view of the 
identity $(F,\ffi)\circ (G,id) = id$ and the
natural weak equivalences of form functors $(G,id)\circ (F,\ffi) \to id$ 
 given by $(1_{A_0},f)$ for $(f:A_0\to A_1) \in \Fun([1],\A)$.
Similarly, $(G,id)$ induces an equivalence of $K$-theory spaces, and hence of $|(w\RR_{\bullet}^{(n)}\phantom{\scA})_h|$ spaces in view of Proposition \ref{prop:AddtyFib1}.

Finally, the dg form functor $(F,\ffi) = \Cone: (\Fun([1],\A,v) \to (\A^{[1]},w)$ 
induces an equivalence of $|(w\RR_{\bullet}^{(n)}\phantom{\scA})_h|$ spaces
by Lemma \ref{lem:KhInvI} where $G: (\A,w) \to (\Fun([1],\A),v)$ sends $X$
to the object $0 \to X$ of $\Fun([1],\A)$.
We have $FG=id_{\A}$ and a zigzag of natural weak equivalences $id \to H
\leftarrow GF$ where $H$ sends an object $f:X \to Y$ of $\Fun([1],\A)$ 
to the object $\Cone(id_X)\to \Cone(f)$, and the natural weak equivalences are
given by the natural maps from $f:X \to Y$ to $\Cone(id_X)\to \Cone(f)$
and from $0 \to \Cone(f)$ to $\Cone(id_X) \to \Cone(f)$.
\end{proof}

\subsection{Hyperbolic and forgetful functors}
\label{subsec:hyperforgetfuncs}
Next, we want to identify the Grothen-dieck-Witt space of the upper right corner in the square of Proposition \ref{prop:SpaceKaroubiProto1} with the $K$-theory space of $\scA$.

For a dg category with weak equivalences $\scA$ (no duality given), the hyperbolic dg category with weak equivalences and duality is the dg category with weak equivalences 
$$\H\scA = \scA \times \scA^{op}$$
equipped with the strict duality $(A,B)^*=(B,A)$.
Even if $\scA$ comes equipped with a duality, the duality on $\H\scA$ does not dependent on it.

Let $\scA = (\A,w,*,\can)$ be a dg-category with weak equivalences and duality.
The {\em forgetful dg form functor} $F:\scA \to \H\scA$ sends 
an object $A$ of $\scA$ to the object $(A,A^*)$, a map $f$ to $(f,f^*)$.
It is equipped with the duality compatibility morphism
$(1,\can): (A^*,A^{**}) \to (A^*,A)$.
If $\scA$ has direct sums, {\em e.g.}, when $\scA$ is pretriangulated, 
then we have the {\em hyperbolic dg form functor} 
$\H\scA \to \scA$ which sends an object $(A,B)$ to $A\oplus B^*$
and a map $(f,g)$ to the map $f\oplus g^*$.
It is equipped with the duality compatibility morphism
$\left(\begin{smallmatrix}0&1\\ \can &0\end{smallmatrix}\right): B\oplus A^* \to A^* \oplus B^{**}$.

We also have an exact dg form functor
\begin{equation}
\label{eqn:FunisH}
(F,\ffi): \Fun([1],\scA) \to \H\scA: (f:A_0\to A_1) \mapsto (A_0,A_1^*)
\end{equation}
with duality compatibility map $(1,\can):F(f^*)\to F(f)^*$.

\begin{lemma}
\label{lem:FunisH}
Let $\scA$ be a pretriangulated dg category with weak equivalences and duality.
Then for any integer $n\geq 1$, the functor (\ref{eqn:FunisH}) induces a homotopy equivalence of pointed topological spaces
$$|(w\RR_{\bullet}^{(n)}\Fun([1],\scA))_h| \stackrel{\sim}{\longrightarrow} |(w\RR_{\bullet}^{(n)}\H\scA)_h|.$$
\end{lemma}

\begin{proof}
Since the lemma holds for $K$-theory (same proof as below), we are reduced to showing that (\ref{eqn:FunisH}) induces a homotopy equivalence on Grothendieck-Witt spaces, in view of Proposition \ref{prop:AddtyFib1}.
It will be convenient to equip $Z^0\Fun([1],\scA)$ with the exact structure where a sequence of functors $[1] \to \scA$ is exact if it is exact in $Z^0\scA$ when evaluated at $0 \in [1]$ and $1 \in [1]$.
The identity functor on $Z^0\Fun([1],\scA)$ from the dg exact structure to this new exact structure induces a homotopy equivalence on Grothendieck-Witt spaces, by \cite[\S 5.1 Lemma 7]{myMV}.

The inverse of (\ref{eqn:FunisH}) on Grothendieck-Witt spaces is given by the form functor
\begin{equation}
\label{eqn:FunisHinverse}
(G,\psi): \H\scA \to \Fun([1],\scA): (A_0,A_1)\mapsto (0:A_0 \to A_1^*)
\end{equation}
with duality compatibility map $(\can_{A_1},1_{A_0^*}): G(A_1,A_0) \to G(A_0,A_1)^*$.
This is because $(1,\can): id_{\H\scA} \to (F,\ffi)\circ (G,\psi)$ is a natural weak equivalence of form functors and thus induces a homotopy of associated maps on Grothendieck-Witt spaces \cite[\S 2.8 Lemma 2]{myMV}.
Moreover, $(G,\psi)\circ (F,\ffi)$ is naturally weakly equivalent to the duality preserving functor $(f:A_0\to A_1) \mapsto (0:A_0\to A_1)$ which induces, up to homotopy, the same map on Grothendieck-Witt spaces as the identity functor, by Additivity \cite[\S3 Theorem 5]{myMV}, in view of the natural short exact sequence in $\Fun([1],\scA)$
$$(0 \to A_1) \to (f:A_0\to A_1) \to (A_0 \to 0).$$
\end{proof}

Consider the sequence of functors
\begin{equation}
\label{eqn:FunisK}
\renewcommand\arraystretch{1.5}
\begin{array}{ccccc}
(w\RR^{(n)}_{\bullet}\Fun([1],\scA))_h & \stackrel{(A,\ffi)\mapsto A}{\longrightarrow} & w\RR^{(n)}_{\bullet}\Fun([1],\scA)& &\\
& \longrightarrow & w\RR^{(n)}_{\bullet}\scA & 
\stackrel{A\mapsto A\iota}{\longrightarrow} & wS^{(n)}_{\bullet}\scA
\end{array}
\end{equation}
in which the non-labelled map is the functor $\Fun([1],\scA) \to \Fun([0],\scA) = \scA$ induced by the inclusion $[0] \to [1]: 0 \mapsto 0$.

\begin{corollary}
\label{cor:RFunIsSdot}
For every pretriangulated dg category with weak equivalences and duality $\scA$ and every integer $n\geq 1$, the composition of the maps in (\ref{eqn:FunisK}) induces a homotopy equivalence of pointed spaces
$$ |(w\RR^{(n)}_{\bullet}\Fun([1],\scA))_h| \stackrel{\sim}{\longrightarrow} 
|wS^{(n)}_{\bullet}\scA|.$$
\end{corollary}

\begin{proof}
The composition in (\ref{eqn:FunisK}) factors as
$$(w\RR^{(n)}_{\bullet}\Fun([1],\scA))_h
\stackrel{(F,\ffi)}{\longrightarrow}
(w\RR^{(n)}_{\bullet}\H\scA)_h = (\H w\RR^{(n)}_{\bullet}\scA)_h 
\longrightarrow w\RR^{(n)}_{\bullet}\scA \longrightarrow 
wS^{(n)}_{\bullet}\scA$$
in which the first map induces a homotopy equivalence by Lemma \ref{lem:FunisH}, the second map, by \cite[\S 2 Lemma 3]{myMV}, and the third map, by \cite[\S 2 Lemma 1]{myMV}.
\end{proof}

\section{Products and the Grothendieck-Witt spectrum}
\label{sec:GWspectrum}

From now on, we will work over a commutative base ring $\bk$ with 
$\frac{1}{2}\in \bk$. 
The purpose of this section is to construct the Grothendieck-Witt spectrum functor.
This is a symmetric monoidal functor
$$GW: \dgCatWD,{\otimes} \longrightarrow \Sp,\wedge$$
from the category of small (pointed) dg categories with weak equivalences and duality to the category of symmetric spectra of topological spaces.
For $\scA \in \dgCatWD$, the $n$-th space of the spectrum $GW(\scA)$ will be the pointed topological space
\begin{equation}
\label{eqn:GWAn}
GW(\scA)_n = \left| i\mapsto \left(w\RR_{i}^{(n)} \scA^{(n)}  \right)_h\right|.
\end{equation}
Here, $\scA^{(n)}$ denotes the dg category with weak equivalences and duality
$$\scA^{(n)} = \Z^n((\scC^{[1]})^{{\otimes} n}{\otimes}\scA),$$
where $\Z^n=\Z\times ...\times \Z$ is the n-fold cartesian product of the poset $\Z$ equipped with its usual ordering; see Section \ref{subsec:MoreOnExtCats}.
When $n\geq 1$, 
the category $\scA^{(n)}$ is equivalent to the pretriangulated hull of $\scA$ equipped with the $n$-th shifted duality; {\em cf.} Remark \ref{rmk:P0AP1AessentiallySurj} and Definitions \ref{dfn:pretr} and \ref{dfn:nshifteddgcat}.
When $n=0$, then it is simply $\scA$ itself.
The bonding maps of the spectrum will be defined in (\ref{eqn:bondingMaps}) below.
We will show in Theorem \ref{thm:GWAOmegaInfinity} and Proposition \ref{prop:GWspecIsGWspace} that the spectrum $GW(\scA)$ is a positive $\Omega$-spectrum with associated infinite loop space the Grothendieck-Witt space as defined in \cite{myMV} of the pretriangulated hull of $\scA$.
Its negative homotopy groups are  Balmer's Witt groups of the triangulated category of $\scA$, by Proposition \ref{prop:NegGWistriangularW} and Corollary \ref{cor:BalmerWnismyWn}.

\subsection{The Grothendieck-Witt functor as a symmetric sequence}
\label{subsec:GWsymmSeq}
Consider the following symmetric monoidal functors.
\vspace{1ex}

\begin{itemize}
\item[$(a)$]\hspace{7ex}
$\displaystyle{\dgCatWD \times \dgCatWD \to \dgCatWD: (\scB, \scA) \mapsto \scB \otimes \scA}$
\vspace{.7ex}

\noindent
with monoidal compatibility map the duality preserving exact dg functor
$$ (\scB_1 {\otimes} \scA_1) {\otimes} (\scB_2 {\otimes} \scA_2) \stackrel{id\otimes \tau \otimes id}{\longrightarrow}  (\scB_1 {\otimes} \scB_2) {\otimes} (\scA_1 {\otimes} \scA_2)$$ and 
unit map $\mu^{-1}:\bk \to \bk{\otimes} \bk$ where 
$\tau: \scA_1 {\otimes} \scB_2 \to \scB_2 {\otimes} \scA_1$ is the switch isomorphism and $\mu: \bk {\otimes} \bk \to \bk$ is the multiplication isomorphism $\bk \otimes \bk \to \bk$.
\vspace{1ex}

\item[$(b)$]
\hspace{16ex}$\poSetD_{\str} \times \dgCatD \to \dgCatD$
\vspace{.7ex}

\noindent
introduced in \ref{subsec:MoreOnExtCats} (\ref{eqn:poSetDdgCatDFun}).
\vspace{1ex}

\item[$(c)$]
\hspace{5ex} $(\CatD_{str})^{op} \times \dgCatWD \longrightarrow \dgCatWD: (D, \scA) \mapsto \Fun(D,\scA)$
\vspace{.7ex}

\noindent
with monoidal compatibility map the duality preserving exact dg functor
$$\Phi: \Fun(\D_1,\scA_1){\otimes} \Fun(\D_2,\scA_2) \to \Fun(\D_1\times \D_2,\scA_1{\otimes} \scA_2)$$
and unit map
$id: \bk \to \Fun(\pt,\bk)=\bk$
where $\Phi(A,B)_{D_1,D_2}=(A_{D_1},B_{D_2})$ and $\Phi(f\otimes g)_{D_1,D_2} = f_{D_1}\otimes g_{D_2}$.
\end{itemize}
Composition of symmetric monoidal functors thus yields the symmetric monoidal functor
\begin{equation}
\label{eqn:PrepForGWSymSeq}
\renewcommand\arraystretch{1.5}
\begin{array}{cl}
& (\CatD_{str})^{op} \times \poSetD_{\str} \times \dgCatWD \times \dgCatWD\\
\stackrel{1\times 1\times (a)}{\longrightarrow} & (\CatD_{str})^{op} \times \poSetD_{\str} \times \dgCatWD \\
\stackrel{1\times (b)}{\longrightarrow} & (\CatD_{str})^{op} \times  \dgCatWD\\
\stackrel{(c)}{\longrightarrow} & \dgCatWD.
\end{array}
\end{equation}

\begin{remark}
\label{rmk:symmonSymSeq}
Let $\U$, $\V$ and $\W$ be symmetric monoidal categories, and
assume that $\W$ has finite coproducts commuting with the monoidal tensor product.
The category $\W^{\Sigma}$ of symmetric sequences in $\W$ is endowed with a symmetric monoidal product; see Appendix \ref{subsec:SymmSeqProd}.
If $F:\U\times \V \to \W$ 
is a symmetric monoidal functor with monoidal compatibility map
$\Phi:F(U_1,V_1)\otimes F(U_2,V_2) \to F(U_1\otimes U_2,V_1\otimes V_2)$
and unit $\1 \to F(\1,\1)$, 
then any object $U$ of $\U$ defines a symmetric monoidal functor
$$F_U:\V \longrightarrow \W^{\Sigma}: V \mapsto \left\{n\mapsto F(U^{\otimes n},V)\right \}$$
where $\Sigma_n$ acts on $U^{\otimes n}$ by permuting the tensor factors.
The monoidal compatibility map is the map induced by the $\Sigma_n\times \Sigma_m$-equivariant map
$$\Phi:F(U^{\otimes n}, V_1)\otimes F(U^{\otimes m},V_2) \longrightarrow F(U^{\otimes n}\otimes U^{\otimes m}, V_1\otimes V_2),$$
and the unit is the map $\Phi: \1 \to F(\1,\1)=F(U^{\otimes 0},\1)$.
\end{remark}

We apply Remark \ref{rmk:symmonSymSeq} to the symmetric monoidal functor 
(\ref{eqn:PrepForGWSymSeq}) and the object 
$U = (\Ar(\underline{i}),  \Z, \scC^{[1]})$ of 
$(\CatD_{str})^{op} \times \poSetD_{\str} \times \dgCatWD$ and obtain the symmetric monoidal functor
$$\dgCatWD \to \dgCatWD^{\Sigma}: \scA \mapsto \left\{ n \mapsto \Fun\left(\Ar(\underline{i})^n,\Z^n((\scC^{[1]})^{{\otimes} n}\scA)\right) \right\}$$
varying simplicially with $i \in \Delta$.
If we restrict to the full dg subcategory $\RR^{(n)}_i$ we obtain the symmetric monoidal functor
$$\dgCatWD \to \dgCatWD^{\Sigma}: \scA \mapsto 
\left\{ n \mapsto \RR_i^{(n)}\Z^n((\scC^{[1]})^{{\otimes} n}\scA) \right\}$$
which varies simplicially with $i \in \Delta$.
Composing with the symmetric monoidal functor
$$\Delta^{op}\dgCatWD \to \Top_*: \scA_{\bullet} \mapsto |i\mapsto (w\scA_i)_h|,$$
we obtain a symmetric monoidal functor 
$$GW: \dgCatWD \to \Top_*^{\Sigma}: \scA \mapsto \left\{ n \mapsto  GW(\scA)_n\right\}$$
with $GW(\scA)_n$ as in (\ref{eqn:GWAn}).
Thus, we have defined $\Sigma_m\times \Sigma_n$-equivariant maps
$$\cup: GW(\scA)_m \wedge GW(\scB)_n \longrightarrow 
GW(\scA{\otimes} \scB)_{m+n}$$
functorial in $\scA,\scB \in \dgCatWD$ 
and a unit map $(S^0,\pt,\pt,\dots) \to GW(\bk)$ making the usual associativity and unit diagrams commute.

\subsection{The bonding maps in the Grothendieck-Witt spectrum}
Let $\scA$ be a dg category with weak equivalences and duality.
Write $\scA^{(0,n)}$ for the dg category with weak equivalences and duality
$$\scA^{(0,n)} = \Z^{1+n}(\scC^{[0]}{\otimes}(\scC^{[1]})^{{\otimes} n}{\otimes}\scA).$$
The inclusions $(\scC^{[1]})^{{\otimes}n} \subset \scC^{[0]}{\otimes}(\scC^{[1]})^{{\otimes}n}: E \mapsto \1\otimes E$ and $\Z^n \subset \Z^{1+n}: n\mapsto (0,n)$ define duality preserving inclusions
$\scA^{(n)} \to \scA^{(0,n)}$
which are equivalences of pretriangulated dg categories with weak equivalences and duality for $n\geq 1$, and for $n=0$ it is the inclusion of $\scA$ into its pretriangulated hull.

For $n\geq 0$, consider the composite functor
$$\Fun\left([1],\scA^{(0,n)}\right) \stackrel{\Cone}{\longrightarrow} \scC^{[1]}\scA^{(0,n)} \stackrel{\sim}{\longrightarrow} \scA^{(1+n)}$$
where the second functor is the equivalence
$$ \scC^{[1]}{\otimes}\Z^{1+n}(\scC^{[0]}{\otimes}\scB)
 \stackrel{\text{(\ref{eqn:MonoidalComMap})}}{\longrightarrow} 
\Z^{1+n}(\scC^{[1]}{\otimes}\scC^{[0]}{\otimes}\scB)
 \stackrel{\otimes}{\longrightarrow} 
\Z^{1+n}(\scC^{[1]}{\otimes}\scB)
$$
with $\scB = (\scC^{[1]})^{{\otimes} n}{\otimes}\scA$.
For $n\geq 0$, consider the commutative diagram
\begin{equation}
\label{eqn:GWnLoopGWn+1}
\xymatrix{
\left(w\RR^{(n)}_{\bullet}\scA^{(n)}\right)_h
\ar[rr] \ar[d] & & \pt \ar[d]\\
\left(w\RR^{(1+n)}_{\bullet}\scA^{(0,n)}\right)_h
\ar@{^(->}[r]^{\hspace{-3ex}I} \ar[d] &
\left(w\RR^{(1+n)}_{\bullet}\Fun([1],\scA^{(0,n)})\right)_h 
\ar[d]^{\Cone} \ar[r]^{\hspace{6ex}\sim}_{\hspace{8ex}\text{(\ref{eqn:FunisK})}} &
wS^{(1+n)}_{\bullet}\scA^{(0,n)}\\
\left(w\RR^{(1+n)}_{\bullet}(\scA^{(1+n)})^w\right)_h
\ar@{^(->}[r] &
\left(w\RR^{(1+n)}_{\bullet}\scA^{(1+n)}\right)_h.
&
}
\end{equation}

\begin{proposition}
\label{prop:GWtoGWtildeDiagram}
Let $\scA$ be a dg category with weak equivalences and duality such that $\frac{1}{2}\in \scA$.
Then, in diagram (\ref{eqn:GWnLoopGWn+1}), 
the lower square is homotopy cartesian
with contractible lower left corner for $n\geq 0$, and
the upper square 
is homotopy cartesian
for $n\geq 1$.
\end{proposition}

\begin{proof}
For $n\geq 1$, the upper square is homotopy cartesian, in view of Proposition \ref{prop:AddtyFib1} and the equivalence $\scA^{(n)} \to \scA^{(0,n)}$.
The lower square is homotopy cartesian, by Proposition \ref{prop:SpaceKaroubiProto1}.
\end{proof}

It follows that the sequence 
$$|(w\RR^{(n)}_{\bullet}\scA^{(n)})_h|, \hspace{4ex} n\in \N,$$
defines a positive $\Omega$-spectrum whenever $\frac{1}{2}\in \scA$.
To give a functorial and explicit definition of the bonding maps of that spectrum, we introduce the following notation which will only be used in this section.
For $n\geq 0$, let
$$
\renewcommand\arraystretch{2}
\begin{array}{lcl}
\overline{K}(\scA)_{1+n} & = & |wS_{\bullet}^{(1+n)}\scA^{(0,n)}|\\
\widetilde{K}(\scA)_{1+n} & = & |(w\RR_{\bullet}^{(1+n)}\Fun([1],\scA^{(0,n)}))_h|\\
GW(\scA)_{1+n}^w & = & |(w\RR_{\bullet}^{(1+n)}(\scA^{(1+n)})^w)_h|
\end{array}
$$

In this notation, diagram (\ref{eqn:GWnLoopGWn+1}) becomes the commutative diagram 
\begin{equation}
\label{eqn:DiagGWtoTildeGW}
\xymatrix{
&& GW(\scA)_n \ar[dll]  \ar[d]  \ar[drr] &&\\
GW(\scA)_{1+n}^w \ar[r] & GW(\scA)_{1+n} & \widetilde{K}(\scA)_{1+n} \ar[l] \ar[r]^{\sim} & \overline{K}(\scA)_{1+n} & \pt. \ar[l]}
\end{equation}
Let $\widetilde{GW}(\scA)_{n}$ be the homotopy limit of the lower row, that is,
$$\widetilde{GW}(\scA)_{n} = GW(\A)^w_{1+n}\underset{GW(\A)_{1+n}}{\stackrel{h}{\times}}\widetilde{K}(\A)_{1+n} \underset{\overline{K}(\A)_{1+n}}{\stackrel{h}{\times}}\pt.$$
For $n\geq 0$, diagram (\ref{eqn:DiagGWtoTildeGW}) defines a map
\begin{equation}
\label{eqn:GWtoTildeGW}
GW(\scA)_n \to \widetilde{GW}(\scA)_{n}.
\end{equation}
Furthermore, base-point inclusions (which are homotopy equivalences)
$$\pt \to GW(\scA)_{1+n}^w\hspace{2ex} \text{and}\hspace{2ex} \pt \to \widetilde{K}(\A)_{1+n} \underset{\overline{K}(\A)_{1+n}}{\stackrel{h}{\times}}\pt$$ 
define a map
\begin{equation}
\label{eqn:OmegaGWtoTildeGW}
\Omega GW(\A)_{1+n} = \pt \underset{GW(\A)_{1+n}}{\stackrel{h}{\times}}\pt \longrightarrow \widetilde{GW}(\scA)_{n}.
\end{equation}

\begin{proposition}
\label{prop:GWTildeGWOmega}
Let $\scA$ be a dg category with weak equivalences and duality such that
$\frac{1}{2}\in \scA$.
Then the map (\ref{eqn:GWtoTildeGW})
is an equivalence for all $n\geq 1$, and the map (\ref{eqn:OmegaGWtoTildeGW})
is an equivalence for all $n\geq 0$.
\end{proposition}

\begin{proof}
This is a restatement of Proposition \ref{prop:GWtoGWtildeDiagram}.
\end{proof}

Recall that $\bk$ denotes our base ring (assumed to satisfy $\frac{1}{2}\in \bk$) which the reader may take to be $\Z[1/2]$.
The inner product space $\bk\otimes \bk \to \bk:x\otimes y \to xy$ defines an object $\langle 1 \rangle $ of the hermitian category $(w\bk)_h$, and thus a pointed map $\langle 1 \rangle: S^0 \to |(w\bk)_h| = GW(\bk)_0$ sending the non-base point of $S^0$ to that object.
It follows from Proposition \ref{prop:GWTildeGWOmega} that
there is a pointed map 
$\epsilon:S^1 \to GW(\bk)_1$ whose adjoint
$\ad(\epsilon): S^0 \to \Omega GW(\bk)_1$
makes the following diagram commutative up to homotopy
\begin{equation}
\label{eqn:defineEpsilon}
\xymatrix{
S^0 \ar[d]_{\langle 1 \rangle} \ar[r]^{\ad(\epsilon)} & \Omega GW(\bk)_1 \ar[d]_{\simeq}^{\text{(\ref{eqn:OmegaGWtoTildeGW})}}\\
 GW(\bk)_0 \ar[r]_{\text{(\ref{eqn:GWtoTildeGW})}} & \widetilde{GW}(\bk)_0.
}
\end{equation}
We fix once and for all such a map $\epsilon$.

\begin{definition}
\label{dfn:GWspectrum}
For a dg category with weak equivalences and duality $\scA$, we define its {\em Grothendieck-Witt spectrum $GW(\scA)$} as the symmetric sequence 
$$GW(\scA) = \{ GW(\scA)_0, GW(\scA)_1, GW(\scA)_2, \dots \},$$
where the space  $GW(\scA)_n$ is the $\Sigma_n$ space (\ref{eqn:GWAn}).
The $\Sigma_n\times \Sigma_m$-equivariant bonding maps of the spectrum are the maps
\begin{equation}
\label{eqn:bondingMaps}
\epsilon_n: 
(S^1)^{\wedge n} \wedge GW(\scA)_m \stackrel{\epsilon^{\wedge n}\wedge 1}{\longrightarrow} (GW(\bk)_1)^{\wedge n}\wedge GW(\scA)_m \stackrel{\cup}{\longrightarrow} GW(\scA)_{n+m}.
\end{equation}
Cup product (Section \ref{subsec:GWsymmSeq}) makes the spectrum $GW(\bk)$ into a commutative symmetric ring spectrum, and $GW(\scA)$ into a module spectrum over $GW(\bk)$.
\end{definition}

\begin{theorem}
\label{thm:GWAOmegaInfinity}
Let $\scA$ be a dg category with weak equivalences and duality such that $\frac{1}{2}\in \scA$.
Then for all $m\geq 1$, $n\geq 0$, the adjoint of the map (\ref{eqn:bondingMaps})
is a homotopy equivalence of pointed topological spaces
$$\ad(\epsilon_n): GW(\scA)_m \stackrel{\sim}{\longrightarrow} \Omega^{n}GW(\scA)_{n+m}.$$
In other words, the spectrum $GW(\scA)$ is a positive $\Omega$-spectrum.
\end{theorem}

\begin{proof}
For $m\geq 1$, consider the diagram
$$\xymatrix{
S^0\wedge GW(\scA)_m \ar[r] \ar[d]^= \ar@/^2pc/[rr]^{\ad(\epsilon)\wedge 1}
 & \widetilde{GW}(\bk)_0 \wedge GW(\scA)_m \ar[d]^{\cup} & \Omega GW(\bk)_1 \wedge GW(\scA)_m \ar[l] \ar[d]^{\cup}\\
 GW(\scA)_m \ar[r]^{\sim} & \widetilde{GW}(\scA)_m  & \Omega GW(\scA)_{m+1} \ar[l]_{\sim}
}$$
in which the upper square (the one containing the bent arrow) is (\ref{eqn:defineEpsilon}) smashed with $GW(\scA)_n$ and thus commutes up to homotopy, the lower 
zigzag are the maps (\ref{eqn:GWtoTildeGW}) and (\ref{eqn:OmegaGWtoTildeGW}), and the vertical arrows are cup-product maps.
In particular, the two lower squares commute.
Since, by Proposition \ref{prop:GWTildeGWOmega}, the lower two maps labelled ``$\sim$'' are equivalences, the composition of upper horizontal (bent) arrow and right vertical arrow is also an equivalence.
But this composition is the map $\ad(\epsilon_1)$ in the proposition.
In particular, it is an equivalence.
Iterating, we obtain equivalences $\ad(\epsilon_n)$.
\end{proof}

\begin{proposition}
\label{prop:GWspecIsGWspace}
Let $\scA$ be a dg category with weak equivalences and duality such that $\frac{1}{2}\in \scA$.
Then the inclusion $\scA \to \scA^{\pretr}$ induces a stable equivalence
$$GW(\scA) \stackrel{\sim}{\longrightarrow}GW(\scA^{\pretr}),$$
and the infinite loop space $\Omega^{\infty}GW(\scA)$ of the spectrum $GW(\scA)$ is naturally equivalent to the Grothendieck-Witt space as defined in \cite{myMV} of the exact category with weak equivalences and duality $Z^0(\scA^{\pretr})$.
In particular, the map $(w\scA^{\pretr})_h = GW(\scA^{\pretr})_0 \to GW(\scA^{\pretr})$ induces an isomorphism
$$GW_0(\scA) \cong GW_0(\scA^{\pretr}) \stackrel{\cong}{\longrightarrow} \pi_0GW(\scA^{\pretr}) \cong \pi_0GW(\scA).$$
\end{proposition}

\begin{proof}
For all $n\geq 1$, the inclusion $\scA \to \scA^{\pretr}$ induces equivalences of dg categories with weak equivalences and duality $\scA^{(n)} \to (\scA^{\pretr})^{(n)}$. 
If follows that $GW(\scA)_n \to GW(\scA^{\pretr})_n$ is a homotopy equivalence for $n\geq 1$.
This implies that $GW(\scA) \to GW(\scA^{\pretr})$ is a stable equivalence.

By Theorem \ref{thm:GWAOmegaInfinity}, the spectrum $GW(\scA)$ is a positive $\Omega$-spectrum.
This implies that
$$\Omega^{\infty}GW(\scA) \simeq \Omega GW(\scA)_1$$
which is equivalent to $\widetilde{GW}(\scA)_0$, by Proposition 
\ref{prop:GWTildeGWOmega}.
The Grothendieck-Witt space of $\scA^{\pretr}$ was defined in \cite{myMV} as
$$(w\RR_{\bullet}\scA^{\pretr})_h \underset{K(\scA^{\pretr})_{1}}{\stackrel{h}{\times}}\pt,$$
that is, if we replace the upper left corner of diagram 
(\ref{eqn:GWnLoopGWn+1}) with this Grothendieck-Witt space when $n=0$, then 
the upper square of that diagram  becomes homotopy cartesian.
Since the lower square of that diagram is homotopy cartesian, by Proposition 
\ref{prop:GWtoGWtildeDiagram}, the natural map from the Grothendieck-Witt space to $\widetilde{GW}(\scA)_0$ is an equivalence.
The last statement follows from \cite[\S 3 Proposition 3]{myMV}.
\end{proof}

\red{
\begin{exampleNoNb}
Recall  from Example \ref{ex:ChbEDGwithWeaksAndDuals} the dg category with weak equivalences and duality $(\Ch^b\E,\quis,*,\eta)$ associated with the $k$-linear exact category with duality $(\E,*,\eta)$.
It follows from Proposition \ref{prop:GWspecIsGWspace} together with \cite[Proposition 6 and Remark 14]{myMV} that the infinite loop space $\Omega^{\infty}GW(\Ch^b\E,\quis,*,\eta)$ is naturally homotopy equivalent to the Grothendieck-Witt space of $(\E,*,\eta)$ as defined in \cite[Definition 4.4]{myEx} provided $\frac{1}{2}\in k$.
In particular, we have isomorphisms on homotopy groups for all $i\geq 0$
$$GW_i(\Ch^b\E,\quis,*,\eta) \cong GW_i(\E,*,\eta),$$%\hspace{3ex} i\geq 0,$$
where the right hand side denotes the higher Grothendieck-Witt groups of $\E$ as in \cite[Definition 4.12]{myEx}.
It follows from Proposition \ref{prop:NegGWistriangularW} below that for $i<0$ we have
$$GW_i(\Ch^b\E,\quis,*,\eta) \cong W^{-i}(\E,*,\eta),$$%\hspace{3ex} i< 0,$$
where the right hand side denotes Balmer's triangular Witt groups of $\E$ as defined in \cite{Balmer:TWGI}.
\end{exampleNoNb}
}

\subsection{Extended Functoriality}
As a symmetric monoidal functor between closed symmetric monoidal categories, the Grothendieck-Witt spectrum functor comes equipped with a natural map of spectra
$$GW\left(\dgFun(\scA,\scB)\right) \to \Sp\left(GW(\scA),GW(\scB)\right)$$
compatible with the internal composition on both sides.
The map above is the adjoint of the map
$$GW\left(\dgFun(\scA,\scB)\right)\wedge GW(\scA) \stackrel{\cup}{\longrightarrow} GW\left(\dgFun(\scA,\scB){\otimes} \scA\right) \stackrel{e}{\longrightarrow} GW(\scB).$$
It follows that two exact dg form functors $(F_i,\ffi_i):\scA \to \scB$, $i=1,2$, which give rise to the same element in $GW_0\dgFun(\scA,\scB)$ induce homotopic maps $GW(\scA) \to GW(\scB)$ of Grothendieck-Witt spectra and the same map $GW_i(\scA) \to GW_i(\scB)$ on higher Grothendieck Witt groups.
This happens, for instance, when there is a natural isomorphism of dg form functors $(F_1,\ffi_1) \cong (F_2,\ffi_2)$.

\begin{definition}
Let $\scA$ be a dg category with weak equivalences and duality such that $\frac{1}{2}\in \scA$.
Recall from Section \ref{subsec:DGshiftedDuals} the dg categories with weak equivalences and duality $\scA^{[n]}= \scC^{[n]}\scA$.
For $n\in \Z$, we define the {\em $n$-th shifted Grothendieck-Witt spectrum} of $\scA$ as
$$GW^{[n]}(\scA) = GW(\scA^{[n]}),$$
and we denote by 
$$GW_i^{[n]}(\scA) =\pi_iGW^{[n]}(\scA)$$
its $i$-th stable homotopy group.
From Proposition \ref{prop:GWspecIsGWspace} we infer that when $i=0$ and $n\in \Z$, this group is the Grothendieck-Witt group $GW_0^{[n]}(\scA)$ from Definition \ref{dfn:GW0A}.
\end{definition}

\subsection{Products}
\label{subsec:products}
We extend products to the shifted Grothendieck-Witt spectrum functors $GW^{[n]}$.
For $\scA, \scB\in \dgCatWD_{\bk}$ the following exact dg form functor
$$\xymatrix{
\scC^{[n]}\scA \otimes \scC^{[m]}\scB \ar[r]^{1\otimes \tau \otimes 1} & (\scC^{[n]}\otimes\scC^{[m]})\scA\scB \ar[r]^{\hspace{3ex}\mu_{n,m}\otimes 1} & \scC^{[n+m]}\scA\scB
}$$
induces upon application of the symmetric monoidal functor $GW$ an associative product map of spectra
$$\cup: GW^{[n]}(\scA) \wedge GW^{[m]}(\scB) \to GW^{n+m}(\scA\otimes\scB)$$
where $\tau:\scA \otimes \scC^{[m]} \to \scC^{[m]}\scA$ is the switch map 
and $\mu_{n,m}: \scC^{[n]} \otimes \scC^{[m]}\to \scC^{[n+m]}$ is the map
$\scC^{k[n]} \otimes \scC^{k[m]}\to \scC^{k[n]\otimes k[m]}$ from Section \ref{equn:CACBtoCAB}
with $A=k[n]$ and $B=k[m]$ followed by the map 
$\scC^{k[n]\otimes k[m]} \to \scC^{k[n+m]}$ induced by the multiplication map 
$k[n]\otimes k[m] \cong k[n+m]$.
For symmetric spectra
$X,Y$, there is a natural pairing map
$$\pi_iX\otimes \pi_jY \to \pi_{i+j}(X\wedge Y):a\otimes b \mapsto a\wedge b.$$
The cup-product of spectra above thus yields
the (associative) pairing
$$\cup: GW^{[n]}_i(\scA) \otimes GW^{[m]}_j(\scB) \to GW^{[n+m]}_{i+j}(\scA\otimes\scB).$$

\begin{proposition}
\label{prop:prod:commutativityCk}
Let $\scA\in \dgCatWD_{\bk}$.
Given elements $a\in GW^{[n]}_i(k)$ and $b\in GW^{[m]}_j(\scA)$, then in $GW^{[n+m]}_{i+j}(\scA)$ we have 
$$a\cup b = (-1)^{ij}\ \langle -1\rangle^{mn}\cup  b\cup a  $$
where $\langle -1\rangle \in GW_0(k)$ is the inner product space $x,y \mapsto -xy$ over $\bk$.
\end{proposition}

\begin{proof}
For any symmetric spectra $X$, $Y$,
the following diagram commutes up multiplication by $(-1)^{ij}$
$$\xymatrix{
\pi_iX \otimes \pi_jY \ar[r] \ar[d]^{\tau} & \pi_{i+j}(X\wedge Y) \ar[d]^{\tau}\\
\pi_jY \otimes \pi_iX \ar[r] & \pi_{i+j}(Y\wedge X).}$$
We apply this with $X=GW^{[n]}(k)$ and $Y=GW^{[m]}(\scA)$.
By virtue of the symmetric monoidal functor $GW$, the diagram of spectra
$$\xymatrix{
GW(\scC^{[n]})\wedge GW(\scC^{[m]}\scA) \ar[r]^{\cup} \ar[d]^{\tau} &
GW(\scC^{[n]}\otimes \scC^{[m]}\scA) \ar[d]_{\tau}\\
GW(\scC^{[m]}\scA) \wedge GW(\scC^{[n]}) \ar[r]^{\cup} &  
GW(\scC^{[m]}\scA\otimes \scC^{[n]})
}$$
commutes.
Finally, consider the following diagram in $\dgCatWD_{\bk}$
$$
\xymatrix{
\scC^{[n]}\otimes \scC^{[m]}\scA \ar[r]^{=} \ar[d]_{\tau} & (\scC^{[n]}\otimes \scC^{[m]})\scA \ar[r]^{\hspace{3ex}\mu_{n,m}\otimes 1} \ar[d]^{\tau\otimes 1} & \scC^{[n+m]}\scA \ar[d]^{\langle -1\rangle^{mn} \otimes 1}\\
\scC^{[m]}\scA \otimes \scC^{[n]} \ar[r]^{1\otimes \tau} &
(\scC^{[m]} \otimes \scC^{[n]})\scA \ar[r]^{\hspace{3ex}\mu_{m,n}\otimes 1} & \scC^{[n+m]}\scA
}$$
in which the left hand diagram commutes, by coherence in a symmetric monoidal category, and the right hand diagram commutes up to natural isomorphism of dg form functors, by Lemma \ref{lem:CiCjcommuting}.
On application of the functor $GW_i$, this diagram becomes commutative, and the right vertical map is cup-product with the central element $\langle -1 \rangle ^{mn} \in GW_0(k)$.
\end{proof}

\begin{remark}
\label{rmk:Cupmu}
Let $\mu \in GW^{[4]}_0(k)$ and $\mu' \in GW^{[-4]}_0(k)$
be the Grothendieck-Witt classes represented by the symmetric spaces $k[2]\otimes k[2] \to k[4]:x,y \mapsto xy$ and $k[-2]\otimes k[-2] \to k[-4]:x,y \mapsto xy$.
Then $\langle 1\rangle = \mu\cup\mu'=\mu'\cup\mu\in GW^{[0]}_0(k)$.
Therefore, cup-product with $\mu$ induces an isomorphism
$$GW^{[n]}_i(\scA) \cong GW^{[n+4]}_i(\scA):x\mapsto \mu\cup x = x\cup \mu$$
with inverse the cup-product with $\mu'$ for all $\scA\in \dgCatWD_{\bk}$.
\end{remark}

\begin{proposition}
\label{prop:GWKaroubiProto1}
Let $\scA$ be a dg category with weak equivalences and duality such that
$\frac{1}{2}\in \scA$.
Then the square $\square_{\scA}$ in Proposition \ref{prop:SpaceKaroubiProto1} induces a homotopy cartesian square of Grothendieck-Witt spectra with contractible lower left corner.
\end{proposition}

\begin{proof}
By Proposition \ref{prop:GWspecIsGWspace}, we can assume $\scA$ to be pretriangulated.
For such categories, the square $(\square_{\scA})^{(n)}$ is equivalent to the square
$\square_{\scA^{[n]}}$.
Therefore, by Proposition \ref{prop:SpaceKaroubiProto1}, the square $GW(\square_{\scA})_n$ is homotopy cartesian for all $n\geq 1$ with contractible lower left corner.
Since the square $GW(\square_{\scA})$ is a square of positive $\Omega$-spectra, this implies the claim.
\end{proof}

\section{Bott Sequence, Invariance and Localization}
\label{sec:PeriodInvLocn}

The homotopy cartesian square of Grothendieck-Witt spectra in Proposition \ref{prop:GWKaroubiProto1} induces 
a functorial exact triangle in the homotopy category of spectra
$$GW(\scA) \stackrel{I}{\longrightarrow} GW(\Fun([1],\scA)) \stackrel{\Cone}{\longrightarrow} GW^{[1]}(\scA) \stackrel{\delta}{\longrightarrow} S^1\wedge GW(\scA).$$
For $\scA=\scC^{[-1]}_{\bk}$, we obtain the boundary map
$\delta: GW^{[0]}(k) \to S^1\wedge GW^{[-1]}(k)$.
We let $\eta \in GW_{-1}^{[-1]}(k)$ be the element
\begin{equation}
\label{eqn:dfnEta}
\eta =-\delta\langle 1 \rangle \in GW_{-1}^{[-1]}(k)
\end{equation}
where $\langle 1 \rangle \in GW^{[0]}_0(k)$ is the Grothendieck-Witt class represented by the symmetric space $k\otimes k \to k:x\otimes y\mapsto xy$.
The element $\eta$ is represented by the map of spectra
$$S^0 \stackrel{\langle 1\rangle}{\longrightarrow} GW^{[0]}(k) 
\stackrel{-\delta}{\longrightarrow} S^1\wedge GW^{[-1]}(k).$$

From Lemma \ref{lem:FunisH} and Corollary \ref{cor:RFunIsSdot}, we have homotopy equivalences of pointed spaces
$(w\RR^{(n)}_{\bullet}\H\scA)_h\simeq wS^{(n)}_{\bullet}\scA$ 
for all pretriangulated dg categories with duality $\scA$ and $n\geq 1$.
It follows that $GW(\H\scA)$ is a connective delooping of the $K$-theory space of $\scA$.
Our model for the K-theory spectrum will therefore be
\begin{equation}
\label{eqn:Kmodel}
K(\scA) = GW(\H\scA).
\end{equation}

Recall from Section \ref{subsec:hyperforgetfuncs} \red{the} forgetful and hyperbolic dg form functors $F:\scA \to \H\scA$ and $H:\H\scA \to \scA^{\pretr}$.
They induce maps of spectra
$$F: GW^{[n]}(\scA) {\longrightarrow} K(\scA)\hspace{2ex}\text{and}\hspace{2ex}H: K(\scA) {\longrightarrow} GW^{[n+1]}(\A^{\pretr}) \stackrel{\sim}{\longleftarrow} GW^{[n+1]}(\A).$$

\begin{theorem}[Algebraic Bott Sequence]
\label{thm:PeriodicityExTriangle}
Let $\scA$ be a dg category with weak equivalences and duality for which $\frac{1}{2}\in \scA$.
Then the sequence of spectra
$$GW^{[n]}(\scA) \stackrel{F}{\longrightarrow} K(\scA) \stackrel{H}{\longrightarrow} GW^{[n+1]}(\scA) \stackrel{\eta\cup}{\longrightarrow}S^1\wedge GW^{[n]}(\scA)$$
is an exact triangle in the homotopy category of spectra.
\end{theorem}

\begin{proof}
In view of Proposition \ref{prop:GWspecIsGWspace}, we can assume $\scA$ pretriangulated.
The square in Proposition \ref{prop:SpaceKaroubiProto1} is multiplicative in the sense that tensor product induces a map of squares $\square_{\scB_0}\otimes \scB_1 \to \square_{\scB_0\otimes\scB_1}$ in $\dgCatWD_{\bk}$.
For $\scB_0=\scC^{[-1]}$ and $\scB_1=\scC^{[1]}\scA$, this shows commutativity of the right square in the following diagram
$$\xymatrix{
S^0\wedge GW^{[1]}(\scA) \ar[r]^{\hspace{-3ex}\langle 1\rangle \wedge 1} \ar[dr]_{=}&
GW^{[0]}(k)\wedge GW^{[1]}(\scA)  \ar[r]^{\hspace{-3ex}\delta\wedge 1} \ar[d]^{\cup} & 
S^1\wedge GW^{[-1]}(k)\wedge GW^{[1]}(\scA) \ar[d]^{\cup}\\
&
GW^{[1]}(\scA)  \ar[r]^{\delta}&
S^1\wedge GW^{[0]}(\scA).
}$$
Going first horizontally then vertically in the diagram yields
cup product from the left with $-\eta$.
Going diagonally down then horizontally in the diagram
yields the boundary map $\delta:  GW(\scA) \to S^1\wedge GW^{[-1]}(\scA)$.
Thus, we have a functorial exact triangle of spectra
$$GW(\scA) \stackrel{I}{\longrightarrow} GW(\Fun([1],\scA)) \stackrel{\Cone}{\longrightarrow} GW^{[1]}(\scA) \stackrel{-\eta\cup}{\longrightarrow} S^1\wedge GW(\scA).$$
By Lemma \ref{lem:FunisH}, the exact dg form functor (\ref{eqn:FunisH})
$$(F,\ffi): \Fun([1],\scA) \to \H\scA: (f:A_0\to A_1) \mapsto (A_0,A_1^*)$$
induces an equivalence of Grothendieck-Witt spectra with inverse 
the exact dg form functor (\ref{eqn:FunisHinverse})
$$(G,\psi): \H\scA \to \Fun([1],\scA): (A_0,A_1)\mapsto (0:A_0 \to A_1^*).$$
The composition $(F,\ffi) \circ I$ is the forgetful map 
$\scA \to \H\scA$.
The composition $\Cone\circ (G,\psi)$ is the dg form functor 
$H\circ T$ where $T:\H\scA \to \H\scA: (A,B)\mapsto (TA,TB)$ is the shift functor which induces multiplication by $-1$ on $K$-theory, and $H:\H\scA \to \scA^{[1]}$ is the hyperbolic functor.
With our model of the $K$-theory spectrum (\ref{eqn:Kmodel}), we therefore obtain an exact triangle in the homotopy category of spectra
$$GW(\scA) \stackrel{F}{\longrightarrow} K(\scA) \stackrel{-H}{\longrightarrow} GW^{[1]}(\scA) \stackrel{-\eta\cup}{\longrightarrow}S^1\wedge GW(\scA).$$
Replacing $\scA$ with $\scA^{[n]}$, we are done.
\end{proof}

Let $(\scA,w,*,\can)$ be a dg category with weak equivalences and duality such that $\frac{1}{2}\in \scA$.
For $\eps\in \{\pm 1\}$, write $_{\eps}GW^{[n]}(\scA)$ for $GW^{[n]}(\scA,w,*,\eps
\can)$, and $_{\eps}GW(\scA)$ for $_{\eps}GW^{[0]}(\scA)$.
Following Karoubi, one defines new theories $_{\eps}U(\scA)$ and $_{\eps}V(\scA)$
as the homotopy fibres of hyperbolic and forgetful functors
$$ K(\scA) \stackrel{H}{\to} {_{\eps}GW}(\scA)\phantom{123}{\rm and}\phantom{123} 
  _{\eps}GW(\scA) \stackrel{F}{\to} K(\scA).$$
The following theorem was proved by Karoubi in \cite{Karoubi:Annals} for rings with involution using different methods.
\red{Note, however,  that the negative homotopy groups in the theorem differ from Karoubi's negative homotopy groups, in general.}

\begin{theorem}[Karoubi's Fundamental Theorem]
\label{thm:KaroubiFundThm}
Let $(\scA,w,*,\can)$ be a dg category with weak equivalences and duality such that $\frac{1}{2}\in \scA$.
Then there are natural stable equivalences
$${_{-\eps}V}(\scA)\simeq \Omega\ {_{\eps}U}(\scA).$$
\end{theorem}

\begin{proof}
From the exact triangle in Theorem \ref{thm:PeriodicityExTriangle}, we obtain equivalences of spectra
$${_{\eps}U}(\scA)\simeq {_{\eps}GW}^{[-1]}(\scA)\phantom{123}{\rm and}\phantom{123} 
{_{\eps}V}(\scA)\simeq \Omega\ {_{\eps}GW}^{[1]}(\scA).
$$
Cup product with the symmetric space $k[1]\otimes k[1] \to k[2]:x\otimes y\mapsto xy$ in $(\scC^{[2]},-\can)$ induces an equivalence of dg categories with weak equivalences and duality $(\scA^{[-1]},\can) \simeq (\scA^{[1]},-\can)$, and therefore an equivalence of Grothendieck-Witt spectra ${_{\eps}GW}^{[-1]}(\scA)\simeq {_{-\eps}GW}^{[1]}(\scA)$.
Hence, we obtain a homotopy equivalence as claimed.
\end{proof}

Our next proposition shows that negative Grothendieck-Witt groups are Balmer's triangular Witt groups.

\begin{proposition}
\label{prop:NegGWistriangularW}
Let $\scA$ be a dg category with weak equivalences and duality such that $\frac{1}{2} \in \scA$.
Let $n\in \Z$ and $i<0$ be integers.
Then cup-product with $\eta$ induces an isomorphism
$$\eta\cup: GW^{[n]}_{i}(\scA) \stackrel{\cong}{\longrightarrow}
GW^{[n-1]}_{i-1}(\scA)$$
and the surjective map $\eta^{-i}\cup :GW^{n-i}_0(\scA) \to GW_i^{[n]}(\scA)$
induces an isomorphism
$$W^{n-i}(\scA) \cong GW_i^{[n]}(\scA).$$
\end{proposition}

\begin{proof}
From the exact triangle in Theorem \ref{thm:PeriodicityExTriangle} and the fact that connective $K$-theory $K(\scA)$ has trivial negative homotopy groups, we obtain an exact sequence
$$K_0(\scA) \stackrel{H}{\longrightarrow} GW^{[n]}_0(\scA) 
\stackrel{\eta\cup}{\longrightarrow} GW^{[n-1]}_{-1}(\scA) \longrightarrow 0$$
and isomorphisms
$$GW^{[n-1]}_{-1}(\scA)\underset{\cong}{\stackrel{\eta\cup}{\longrightarrow}}
GW^{[n-2]}_{-2}(\scA)\underset{\cong}{\stackrel{\eta\cup}{\longrightarrow}}
GW^{[n-3]}_{-3}(\scA)\underset{\cong}{\stackrel{\eta\cup}{\longrightarrow}} \cdots
$$
Replacing $n$ with $n-i$, the proposition follows.
\end{proof}

\begin{lemma}[Karoubi induction]
\label{lem:KaroubiInd}
Let $F:\scA \to \scB$ be an exact dg form functor between dg categories with weak equivalences and duality.
Assume that $\frac{1}{2}\in \scA,\scB$, 
that $F$ induces an equivalence in $K$-theory and isomorphisms on all shifted Witt groups $W^{n}(\scA) \cong W^{n}(\scB)$, $n \in \Z$.
Then $F$ induces an equivalence of Grothendieck-Witt spectra for all $n \in \Z$ 
$$GW^{[n]}(\scA) \stackrel{\simeq}{\longrightarrow} GW^{[n]}(\scB).$$
\end{lemma}

\begin{proof}
By Theorem \ref{thm:PeriodicityExTriangle}, the form functor $F$ induces a map of exact sequences for $n, i\in \Z$
$$
\xymatrix{
GW_{i+1}^{[n]}(\scA) \ar[r] \ar[d] & K_{i+1}(\scA) \ar[r] \ar[d]^{\cong} &
GW_{i+1}^{[n+1]}(\scA) \ar[r] \ar[d] & GW_{i}^{[n]}(\scA) \ar[r] \ar[d]
& K_{i}(\scA) \ar[d]^{\cong}\\
GW_{i+1}^{[n]}(\scB) \ar[r] & K_{i+1}(\scB) \ar[r]  &
GW_{i+1}^{[n+1]}(\scB) \ar[r]  & GW_{i}^{[n]}(\scB) \ar[r] 
& K_{i}(\scA).
}$$
By Proposition \ref{prop:NegGWistriangularW} and the hypothesis of the lemma, the functor $F$ induces isomorphisms $GW_i^{[n]}(F) = W^{n-i}(F)$ for all $n, i\in \Z$ with $i<0$.
By a version of the $5$-lemma, if $GW_i^{[n]}(F)$ and $K_*(F)$ are isomorphisms then
$GW_{i+1}^{[n+1]}(F)$ is a surjection.
If $GW_i^{[n]}(F)$  and $K_*(F)$ are isomorphisms and 
$GW_{i+1}^{[n]}(F)$ is a surjection then $GW_{i+1}^{[n+1]}(F)$ is an isomorphism.
By induction, we are done.
\end{proof}

\begin{theorem}[Invariance for $GW$]
\label{thm:Invariance}
Let $G:\scA \to \scB$ be an exact dg form functor between dg categories with weak equivalences and duality such that $\frac{1}{2}\in \scA,\scB$.
Suppose that the functor $G$ induces an equivalence $\T\scA \to \T\scB$ of associated triangulated categories, cf. Definition \ref{dfn:dgCatW}.
Then $G$ induces an equivalence of Grothendieck-Witt spectra for all $n \in
\Z$ 
$$GW^{[n]}(\scA) \stackrel{\simeq}{\longrightarrow} GW^{[n]}(\scB).$$
\end{theorem}

\begin{proof}
By a theorem of Thomason \cite[1.9.8]{TT}, the functor $G$ induces an equivalence of $K$-theory spectra.
By Proposition \ref{prop:NegGWistriangularW} and Corollary \ref{cor:BalmerWnismyWn}, the maps $GW_i^{[n]}(G)$ are isomorphisms for all $n, i \in \Z$ with $i<0$.
This finishes the proof in view of Lemma \ref{lem:KaroubiInd}.
\end{proof}

Call a sequence $\T_0 \to \T_1 \to \T_2$ of triangulated categories {\em exact} if the composition is trivial, the functor $\T_0 \to \T_1$ makes $\T_0$ into an epaisse subcategory of $\T_1$ (that is, a full subcategory which is closed under direct factors) and the induced functor from the Verdier quotient $\T_1/\T_0$ to $\T_2$ is an equivalence.
Call a sequence $\scA_0 \to \scA_1 \to \scA_2$ of dg categories with weak equivalences and duality {\em quasi-exact} if the associated sequence
$\T\scA_0 \to \T\scA_1 \to \T\scA_2$ of triangulated categories is exact.
An exact dg functor $(\A,w) \to (\B,v)$ of dg-categories with weak equivalences is called {\em quasi-equivalence} if the induced triangle functor $\T(\A,w) \to \T(\B,v)$ is an equivalence.

\begin{theorem}[Localization for $GW$]
\label{thm:LcnConn1}
Let $(\scA_0,w) \to (\scA_1,w) \to (\scA_2,w)$ be a quasi-exact sequence of dg categories with weak equivalences and duality.
Assume that $\frac{1}{2}\in \scA_0,\scA_1,\scA_2$. 
Then the commutative square of Grothendieck-Witt spectra
$$\xymatrix{
GW^{[n]}(\scA_0,w) \ar[r] \ar[d] & GW^{[n]}(\scA_1,w) \ar[d]\\
GW^{[n]}(\scA_2^w,w) \ar[r] & GW^{[n]}(\scA_2,w)
}$$
is homotopy cartesian, and the lower left corner is contractible.
In particular, for all $n\in \Z$ there is a homotopy fibration of Grothendieck-Witt spectra
$$GW^{[n]}(\scA_0) \to GW^{[n]}(\scA_1) \to GW^{[n]}(\scA_2).$$
\end{theorem}

\begin{proof}
Since the composition $\T\scA_0 \to \T\scA_1 \to \T\scA_2$ is trivial, the image of $\scA_0$ in $\scA_2$ is in the dg subcategory $\scA_2^w$ of $w$-acyclic objects.
In particular, the square of Grothendieck-Witt spectra in the theorem commutes.
Replacing $\scA_i$ with $\scA_i^{[n]}$ if necessary, we can assume $n=0$.
Furthermore, by Proposition \ref{prop:GWspecIsGWspace}, we can assume all dg categories to be pretriangulated.

Let $v$ be the set of morphisms in $\scA_1$ which are weak equivalences in $\scA_2$.
Then $(\scA_1,v)$ is a pretriangulated dg category with weak equivalences and duality.
By Proposition \ref{prop:pretrFibPrp}, the square of pretriangulated dg categories with weak equivalences and duality
$$\xymatrix{
(\scA_1^v,w)^{(n)} \ar[r] \ar[d] & (\scA_1,w)^{(n)} \ar[d]\\
(\scA_1^v,v)^{(n)} \ar[r] & (\scA_1,v)^{(n)}
}$$
induces a homotopy cartesian square of $|(w\RR^{(n)}_{\bullet})_h|$-spaces with contractible lower left corner.
By Proposition \ref{thm:GWAOmegaInfinity}, this square induces a homotopy cartesian square of Grothendieck-Witt spectra.
By assumption, the exact dg form functors
$(\scA_0,w) \to (\scA_1^v,w)$ and $(\scA_1,v) \to (\scA_2,w)$
induce equivalences on associated triangulated categories.
By Theorem \ref{thm:Invariance}, they induce equivalences of Grothendieck-Witt spectra.
\end{proof}

\begin{remark}
\label{rmk:FundThmFromLocnThm}
Let $\scA$ be a pretriangulated dg category with weak equivalences and duality.
Then the sequence
$$\scA \stackrel{I}{\longrightarrow} \Fun([1],\scA) \stackrel{\Cone}{\longrightarrow} \scA^{[1]}$$
induces an exact sequence of associated triangulated categories and hence a homotopy fibration of Grothendieck-Witt spectra, in view of Theorem \ref{thm:LcnConn1}.
By Additivity, the Grothendieck-Witt spectra of $\Fun([1],\scA)$ and $\H\scA$ are equivalent to the $K$-theory spectrum of $\scA$.
Therefore, the Localization Theorem \ref{thm:LcnConn1} implies the homotopy fibration in the Bott Sequence \ref{thm:PeriodicityExTriangle} and thus Karoubi's Fundamental Theorem as explained in the proof of Theorem \ref{thm:KaroubiFundThm}.
\end{remark}

\begin{proposition}[Additivity for $GW$]
\label{prop:AddGW}
Let $(\scU,{\vee})$ be a pretriangulated dg category with weak equivalences and duality such that $\frac{1}{2}\in \scU$.
Let $\scA \subset \scU$ be a full pretriangulated dg subcategory containing the 
$w$-acyclic objects $\scU^w$ of $\scU$.
Assume that $\T\scU(A,B^{\vee})=0$ for all $A,B\in \T\scA$, and that $\T\scU$ is generated as a triangulated category by $\T\scA$ and $\T\scA^{\vee}$.
Then the exact dg form functor 
$$\H\scA \to \scU: (A,B)\mapsto A\oplus B^{\vee}$$
induces a stable equivalence of Grothendieck-Witt spectra:
$$K(\scA) = GW(\H\scA) \stackrel{\sim}{\longrightarrow} GW(\scU).$$
\end{proposition}

\begin{proof}
The triangle functor $\T\scA^{\vee} \to \T\scU/\T\scA$ is fully faithful because we have
$\T\scU(A,B^{\vee})=0$ for all $A,B\in \T\scA$.
It is essentially surjective (hence an equivalence) because $\T\scU$ is generated by $\T\scA$ and $\T\scA^{\vee}$.
From the $K$-theory analog of the Localization Theorem \ref{thm:LcnConn1}, the sequence $(\scA,w) \to (\scU,w) \to (\scU,v)$ induces a homotopy fibration of $K$-theory spectra where $v$ is the set of morphisms which are isomorphisms in 
$\T\scU/\T\scA$.
In view of the Invariance Theorem for $K$-theory, the fibration splits since the composition $(\scA^{\vee},w) \to (\scU,w) \to (\scU,v)$ induces an equivalence of triangulated categories.
It follows that the functor $\H\scA \to \scU$ is a $K$-theory equivalence.

The Witt groups of $\H\scA$ are trivial as is the case for any hyperbolic category.
The Witt groups $W^*(\scU)$ of $\scU$ are also trivial.
It suffices to prove this for $W^0(\scU) = W^0(\T\scU)$, the arguments for $W^n$ being similar (or just apply the $W^0$ case to the $n$-th shifted triangulated category with duality $\T\scU^{[n]}$).
The hypothesis imply that for every $U \in \T\scU$, there is an exact triangle 
$A \to U \to B \to TA$ in $\T\scU$ with $A \in \T\scA$, $B\in \T\scA^{\vee}$ such that for any map $U \to U'$ and any exact triangle $A' \to U' \to B' \to TA'$  in $\T\scU$ with $A' \in \T\scA$, $B'\in \T\scA^{\vee}$ there is a unique map of triangles from the first to the second extending the given map $U \to U'$.
In particular, any symmetric isomorphism $\ffi:U \to U^{\vee}$ extends to a symmetric isomorphism of exact triangles showing that $[U,\ffi] = 0 \in W^0(\T\scU)$.

Since the functor $\H\scA \to \scU$ is a $K$-theory equivalence and an isomorphism on triangular Witt groups, the Karoubi-Induction Lemma \ref{lem:KaroubiInd} implies that it is also a $GW$-equivalence.
\end{proof}

\section{Relation with $L$-theory and Tate of $K$-theory}
\label{section:LandTate}

Recall from (\ref{eqn:dfnEta})
%Section \ref{sec:PeriodInvLocn}
 the element $\eta \in GW^{[-1]}_{-1}(k)$ corresponding to $\langle 1 \rangle \in W^0(k)$ under the isomorphism $GW^{[-1]}_{-1}(k) \cong W^0(k)$ of Proposition \ref{prop:NegGWistriangularW}.

\begin{definition}
\label{dfn:LthSp}
For $\scA\in \dgCatWD$, we define its $L$-theory spectrum $L(\scA)$ as
$$L(\scA) = \eta^{-1}GW(\scA).$$
This is the homotopy colimit spectrum of the sequence
$$GW(\scA) \stackrel{\eta\cup}{\longrightarrow} S^1\wedge GW^{[-1]}(\scA) 
\stackrel{S^1\wedge \eta^2}{\longrightarrow} S^2\wedge GW^{[-2]}(\scA) \longrightarrow \cdots$$
Since cup-product with the element $\mu\in GW^{[4]}_0(k)$ from Remark \ref{rmk:Cupmu} induces an equivalence $GW^{[n]}(\scA) \simeq GW^{[n+4]}(\scA)$, the spectrum
$\eta^{-1}GW(\scA)$ is also the localization of the $GW(k)$-module spectrum $GW(\scA)$ at the element $\eta^4\mu\in GW_{-4}(k)$ of the commutative ring spectrum
$GW(k)$.
In particular, $L(\scA) = \eta^{-1}GW(\scA)$ is a module spectrum over the commutative ring spectrum $L(k)=(\eta^4\mu)^{-1}GW(k)$ \cite{schwede:book}.
As usual, $L^{[n]}(\scA)$ denotes $L(\scA^{[n]})$,
and the homotopy groups of $L^{[n]}(\scA)$ are denoted by $L^{[n]}_i(\scA)$.

By definition, the ring spectrum $L(k)$ is $4$-periodic with periodicity isomorphism given by cup-product with $\eta^4\mu$.
As a module spectrum over $L(k)$, the spectrum $L(\scA)$ is also $4$-periodic. 
\end{definition}

\begin{proposition}
Let $\scA$ be a dg category with weak equivalences and duality such that $\frac{1}{2}\in \scA$.
Then there are natural isomorphisms for all $i,n\in \Z$
$$L^{[n]}_i(\scA) \cong W^{n-i}(\scA).$$
\end{proposition}

\begin{proof}
This is a consequence of Proposition \ref{prop:NegGWistriangularW}.
\end{proof}

\subsection{The $\Z/2$-action on $K$-theory}
\label{Z/2actionOnKth}
Call a pretriangulated dg category with weak equivalences and duality $\scA = (\scA,w,*,\can)$ 
{\em strict} if the duality is strict, that is, the natural transformation $\can: 1 \to **$ is the identity (in particular, $**=id$), and 
if it is equipped with a choice of direct sum and zero object $0$ satisfying $0^*=0$,  
$(X\oplus Y)^*=Y^*\oplus X^*$ and $0\oplus X = X =X \oplus 0$ for all objects $X, Y$ of $\scA$.
We denote by $\dgCatWD_{\str}$ the category of small strict pretriangulated dg categories with weak equivalences and duality.
The morphisms in $\dgCatWD_{\str}$ are the exact dg functors which commute with direct sum and duality, and which preserve the chosen zero object.

Let $\scA$ be a dg category with weak equivalences and duality.
By Lemma \ref{lem:Astr} below,
we can (functorially) associate with $\scA$ a strict pretriangulated dg category with weak equivalences and duality $\scA_{\str}$ that comes with an exact dg form functor $\scA \to \scA_{\str}$ which, by Theorem \ref{thm:Invariance}, induces an equivalence on $K$-theory and Grothendieck-Witt spectra.
So, in what follows, we may assume that $\scA$ is strict.

Let $\scA = (\scA,w,*)$ be a strict pretriangulated dg category with weak equivalences and duality.
The hyperbolic category $\H\scA= \scA\times \scA^{op}$ is equipped with the duality $(X,Y) \mapsto (Y,X)$ (ignoring the duality $*$ on $\scA$).
The functor 
$$\sigma: \H\scA \to \H\scA: (X,Y) \mapsto (Y^*,X^*)$$
is duality preserving (as the duality on $\scA$ is strict) and induces a $C_2 = \Z/2$ action on the dg category with weak equivalences and duality $\H\scA$. 
By functoriality, we obtain an induced $C_2$-action on the Grothendieck-Witt spectrum $GW(\H\scA)$, hence on the K-theory spectrum $K(\scA)=GW(\H\scA)$.
If $\scA$ is not strict, we will write $K(\scA)$ for the spectrum $K(\scA_{\str})$ with its $C_2$-action.
Moreover, we will write $K^{[n]}(\scA)$ for the $K$-theory spectrum $K(\scA^{[n]}) = GW(\H(\scA^{[n]})_{\str})$ of $\scA$ equipped with the $C_2$-action coming from the duality on $\A^{[n]}$ as described above.
Non-equivariantly, this spectrum is (equivalent to) the $K$-theory spectrum of $\scA$, but the $C_2$-action depends on $n$.

\begin{example}
We give an interpretation of the $C_2$-action on $K(R)$ for symmetric and alternating forms over a ring with involution $R$ in terms of an action on $BGl(R)^+$. 
For that consider first an additive category with strict duality $(\A,*)$.
If we equip the category $i\A$ with the $C_2$-action $A\mapsto A$, $f\mapsto (f^*)^{-1}$, then the equivalence of categories 
$i\A \to (i\H\A)_h: A\mapsto [(A,A),(id,id)] \hspace{1ex} f\mapsto (f,f^{-1})$ is $C_2$-equivariant.
This gives the $C_2$-action on the group completion of the classifying space of $i\A$, that is, on the usual $K$-theory space of $\A$.

For example, let $R$ be a ring with involution $a\mapsto \bar{a}$.
Let $\A=F(R)$ be the category of finitely generated free $R$-modules.
More precisely, objects of $F(R)$ are the free modules $R^n$, $n\in \N$.
Maps and composition of maps are given by matrices and their products.
The dual of $R^n$ is $R^n$ and the dual of a map $M \in M_n(R)$ is the matrix $M^* = {^t}{\bar{M}}$. 
In positive degrees, the hermitian $K$-theory of symmetric forms over $R$ is the hermitian $K$-theory of the additive category with strict duality $(F(R),*,id)$.
By the discussion above, the corresponding $C_2$-action on $BGl(R)^+$ is induced by  
$$BGl \to BGl: M \mapsto {^t}{\bar{M}}^{-1}$$
and the functoriality of the plus construction.

In the literature one frequently finds a differently looking $C_2$-action on $K$-theory; see for instance \cite{berrickKaroubi}. 
The point is that the hermitian $K$-theory of symmetric forms over $R$ is also the hermitian $K$-theory of the ring $M_2(R)$ with involution
$$\left(\begin{smallmatrix} a&b\\ c&d\end{smallmatrix}\right) \mapsto \left(\begin{smallmatrix} 0&1\\ 1&0\end{smallmatrix}\right) \left(\begin{smallmatrix} \bar{a}&\bar{c}\\ \bar{b}&\bar{d}\end{smallmatrix}\right)\left(\begin{smallmatrix} 0&1\\ 1&0\end{smallmatrix}\right) = \left(\begin{smallmatrix} \bar{d}&\bar{b}\\ \bar{c}&\bar{a}\end{smallmatrix}\right)$$
since both rings have equivalent associated hermitian categories of projective modules.
This yields the $C_2$-action on $BGl_{\infty}(R) = BGl_{2\infty}(R)$ given by
$$Gl_{2n}(R) \to Gl_{2n}(R): M \mapsto \left(\begin{smallmatrix} 0&1_n\\ 1_n&0\end{smallmatrix}\right) (^t\overline{M})^{-1} \left(\begin{smallmatrix} 0&1_n\\ 1_n&0\end{smallmatrix}\right).$$
To compare these two actions, consider the $C_2$-action on $BGl_{\infty}(R) = BGl_{3\infty}(R)$ given by
$$Gl_{3n}(R) \to Gl_{3n}(R): M \mapsto \left(\begin{smallmatrix} 0&1_n 0 \\ 1_n&0&0 \\ 0&0&1_n\end{smallmatrix}\right) (^t\overline{M})^{-1} \left(\begin{smallmatrix} 0&1_n 0 \\ 1_n&0&0 \\ 0&0&1_n\end{smallmatrix}\right).$$
The two maps
$$
\renewcommand\arraystretch{1.5}
\begin{array}{ccccc}
Gl_{2n}(R) & \longrightarrow & Gl_{3n}(R) & \longleftarrow & Gl_n(R) \\
M &\mapsto & \left(\begin{smallmatrix} M&0\\ 0&1_n\end{smallmatrix}\right), \hspace{3ex} \left(\begin{smallmatrix} 1_{2n}&0\\ 0&N\end{smallmatrix}\right) &
\mapsfrom & N
\end{array}
$$
are $C_2$-equivariant and yield, after passage to the colimit for $n\to \infty$  and taking (a functorial version of the) plus construction, $C_2$-equivariant maps 
$$K(R) \to K(R) \leftarrow K(R)$$
which are homotopy equivalences (forgetting the $C_2$-action).
In particular, the respective homotopy orbit (homotopy fixed point) spaces are all homotopy equivalent.

Finally, the hermitian $K$-theory of alternating forms over $R$ in positive degrees is the hermitian $K$-theory of the additive category with duality $(F(R),*,-id)$.
It is also the hermitian $K$-theory of symmetric forms over the ring  $M_2(R)$ with involution $\left(\begin{smallmatrix}a&b\\c&d\end{smallmatrix}\right) \mapsto \left(\begin{smallmatrix}\bar{d}&-\bar{b}\\-\bar{c}&\bar{a}\end{smallmatrix}\right)$ since both have equivalent associated hermitian categories.
Thus, the involution on $BGl(R)^+ = BGl(M_2(R))^+$ corresponding to alternating forms over $R$ is the involution  corresponding to symmetric forms over the ring with involution $M_2(R)$.
\end{example}

\subsection{The hypernorm in $K$-theory}
\label{KnormRevisited}
Let $G$ be a finite group.
For any spectrum $X$ with $G$-action its Tate spectrum $\hat{\bH}(G,X)$ is the homotopy cofibre of the hypernorm map $\tilde{N}: X_{hG} \to X^{hG}$ from the homotopy orbit spectrum $X_{hG}$ to the homotopy fix point spectrum $X^{hG}$; see Appendix \ref{Appx:TateSp} when $G=C_2$. 
In Lemma \ref{lem:KGWKisNorm} we give another description of this map for the $C_2$-spectrum $K(\scA)$ associated with a (strict pretriangulated) dg category with weak equivalences and duality $\scA$ (where the action was defined in Section \ref{Z/2actionOnKth}). 

Let $\scA$ be a strict pretriangulated dg category with weak equivalences and duality.
We consider $\scA$ equipped with the trivial $C_2$-action.
Recall from \S\ref{Z/2actionOnKth} that the hyperbolic category $\H\scA$ is equipped with a $C_2$-action (in the category of strict pretriangulated dg categories with weak equivalences and duality).
We will construct a sequence
\begin{equation}
\label{eqn:NormOnCatLevel}
\xymatrix{
\H\scA \ar[r]^{\hspace{1ex}H} & \bar{\scA} & \scA \ar[l]_{\hspace{1ex}\simeq} \ar[r]^{\hspace{-1ex}F} & \H\scA}
\end{equation}
of $C_2$ equivariant maps between strict pretriangulated dg categories with weak equivalences and duality in which the backward arrow is an equivalence of categories (forgetting the $C_2$ action).

Let $\bar{\scA}$ be the dg category whose objects are isomorphisms $a:A \stackrel{_{\cong}}{\to} B$ in $\scA$ and whose maps are commutative squares in $\scA$.
We define a strict duality on $\bar{\scA}$ on objects by $(a:A\to B)^*=((a^*)^{-1}:A^* \to B^*)$ and on morphisms by $(f,g)^*=(f^*,g^*)$.
The equivalence of categories $\scA \to \bar{\scA}: X \mapsto (1:X \to X), f\mapsto (f,f)$ is duality preserving.
%This allows us to equip $\bar{\scA}$ with the structure of an exact category with weak equivalences and strict duality in such a way that the functor $\scA \to \bar{\scA}$ becomes an equivalence of such categories.
The functor 
$$\bar{\scA} \to \bar{\scA}: (f:X \to Y) \mapsto (f^{-1}:Y \to X)$$
is duality preserving and induces a $C_2$-action on the strict exact category with weak equivalences and duality $\bar{\scA}$.
Note that the duality preserving equivalence $\scA \to \bar{\scA}$ is $C_2$-equivariant.
The functors $H$ and $F$ in diagram (\ref{eqn:NormOnCatLevel}) are defined as
$$
\renewcommand\arraystretch{1.5}
\begin{array}{cccc}
H:\H\scA \to \bar{\scA}:& (X,Y) & \mapsto &(X\oplus Y^*\stackrel{\tau}{\to} Y^*\oplus X) \hspace{4ex}\text{and}\\
F: \scA \to \H\scA: & X & \mapsto & (X,X^*)
\end{array}$$
where $\tau$ denotes the map switching the two factors.
Both functors are duality preserving, exact and $C_2$-equivariant.

Applying Grothendieck-Witt spectra to diagram (\ref{eqn:NormOnCatLevel})
yields a diagram 
\begin{equation}
\label{eqn:GeqKGWKseq}
\xymatrix{
K(\scA) \ar[r]^{H} & GW(\bar{\scA}) & GW(\scA) \ar[l]_{\hspace{1ex}\simeq} \ar[r]^{F} & K(\scA)}
\end{equation}
of ${C_2}$ -equivariant spectra where $C_2$ acts trivially on $GW(\scA)$.
Taking homotopy (co-) limits, we obtain a sequence of spectra
\begin{equation}
\label{eqn:KhGGWKhG} K_{h{C_2}} \stackrel{H}{\longrightarrow} \overline{GW}_{h{C_2}} \stackrel{\hspace{1ex}\small{\sim}}{\longleftarrow} GW_{h{C_2}} \longrightarrow GW \longrightarrow GW^{h{C_2}} \stackrel{F}{\longrightarrow} K^{h{C_2}}
\end{equation}
where we have suppressed the entries $\scA$ and where 
we wrote
$\overline{GW}$ for $GW(\bar{\scA})$.
 The non-labelled maps in the diagram are the natural maps $GW_{h{C_2}} \to GW_{C_2}=GW=GW^{C_2} \to GW^{h{C_2}}$.

\begin{lemma}
\label{lem:KGWKisNorm}
The map $K(\scA)_{h{C_2}} \to K(\scA)^{h{C_2}}$ in (\ref{eqn:KhGGWKhG})
is naturally equivalent to the hypernorm map \ref{subsec:hypernorm} (\ref{eqn:tildeN}) for the $C_2$-spectrum $K(\scA)$ defined in Section \ref{Z/2actionOnKth}.
\end{lemma}

\begin{proof}
Let $G_1=G_2=\Z/2=C_2$ and consider the sequence (\ref{eqn:GeqKGWKseq}) as a sequence of $G_1\times G_2$-spectra where $G_1\times G_2$ acts on the left two spectra via the projection $G_1\times G_2 \to G_1$ on the first factor.
The action on $GW(\scA)$ is trivial and the action on the last spectrum is via the projection $G_1\times G_2 \to G_2$ onto the second factor.
The map in (\ref{eqn:KhGGWKhG}) is the application of the functor 
$X \mapsto (X^{hG_2})_{hG_1}$ to the $G_1\times G_2$-equivariant string (\ref{eqn:GeqKGWKseq}) preceded by  
$(X^G_2)_{hG} \to (X^{hG_2})_{hG_1}$ and followed by $(X^{hG_2})_{hG_1} \to (X^{hG_2})_{G_1}$. 
We have to show that the sequence (\ref{eqn:GeqKGWKseq}) is $G_1\times G_2$-equivariantly weakly equivalent to the sequence 
\ref{subsec:hypernorm} (\ref{G1G2eqnorm}) where $X = GW(\H\scA)$.

By naturality of the $G_1$-equivariant map $\scA \to \bar{\scA}$ applied to the $G_2$-equivariant map $F:\scA \to \H\scA$, we have a $G_1\times G_2$-equivariant commutative diagram of strict pretriangulated dg categories with weak equivalences and duality
\begin{equation}
\label{eqn:PfOfNormMapEqual}
\xymatrix{\scA \ar[r]^{\sim} \ar[d]^F & \bar{\scA}  \ar[d]^{\bar{F}} \\
\H\scA \ar[r]^{\sim} & \overline{\H\scA}.
}
\end{equation}
Define a functor $J: \H\scA \times \H\scA \to \overline{\H\scA}$ by sending the object
$((X,Y),(U,V))$ to the object $(\tau,\tau): (X\oplus U, Y\oplus V) \longrightarrow (U\oplus X,V\oplus Y)$ and on maps in the obvious way.
The functor is exact, duality preserving and $G_1\times G_2$-equivariant 
where $G_1$ acts on $\H\scA \times \H\scA$ by interchanging the two factors $(x,y)\mapsto (y,x)$ and $G_2$ acts on $\H\scA \times \H\scA$ via $(x,y)\mapsto (\sigma y,\sigma x)$.
We have the following $G_1\times G_2$-equivariant commutative diagram of strict pretriangulated dg categories with weak equivalences and duality
$$
\xymatrix{
\H\scA \ar[r]^{\hspace{-4ex}(1,\sigma)} \ar[d]^H & 
\H\scA \times \H\scA \ar[d]^J & 
\H\scA \vee \H\scA \ar[d]^{1\vee 1} \ar[l]_{\hspace{2ex}\sim} \\
\bar{\scA} \ar[r]^{\bar{F}} & \overline{\H\scA} &
\H\scA \ar[l]_{\sim}
}$$
where, strictly speaking, the upper right corner only makes sense after application of the functor $GW$ using the convention $GW(\A\vee\B)=GW(\A)\vee GW(\B)$.
Going down then right is the $G_1\times G_2$-equivariant map (\ref{eqn:NormOnCatLevel}) in view of the commutative square (\ref{eqn:PfOfNormMapEqual}).
Going right then down is the map that defines the hypernorm map.
\end{proof}

Let $\scB$ be a pretriangulated dg category with weak equivalences that is equipped with a direct sum operation and a zero object $0$ such that $0\oplus B=B=B\oplus 0$ for all objects $B$ of $\scB$.
We make $\H\scB$ into a strict pretriangulated dg category with weak equivalences and duality by defining the direct sum in $\H\scB$ by $(A,B)\oplus (A',B')=(A\oplus A',B'\oplus B)$.
For $\scA=\H\scB$, the sequence (\ref{eqn:GeqKGWKseq}) yields the sequence of $C_2$-equivariant spectra
\begin{equation}
\label{eqn:GeqKGWKseqForH}
\xymatrix{
K(\H\scB) \ar[r]^{H} & GW(\overline{\H\scB}) & GW(\H\scB) \ar[l]_{\hspace{1ex}\simeq} \ar[r]^{F} & K(\H\scB)}
\end{equation}

\begin{lemma}
\label{lem:KhGGWKhGhypEq}
The sequence (\ref{eqn:GeqKGWKseqForH}) is $C_2$-equivariantly equivalent to the sequence 
$$K(\scB)\wedge (C_2)_+ \to K(\scB) \to \Sp^{C_2}((C_2)_+,K(\scB))$$
where $K(\scB)$ has trivial $C_2$-action, and both maps are induced by the $C_2$-equivariant map $C_2 \to \pt$.
\end{lemma}

\begin{proof}
We have a $C_2$-equivariant isomorphism of strict pretriangulated dg categories with weak equivalences and duality
$$
\renewcommand\arraystretch{1.5}
\begin{array}{cccc}
\Phi:& \H(\scB\times\scB) = (\scB\times\scB)\times (\scB\times\scB)^{op}  &  \stackrel{\cong}{\longrightarrow} &  \H(\H\scB) = \H\scB\times (\H\scB)^{op}  \\
& (A_0,A_1),(B_0,B_1)  & \mapsto & (A_0,B_1),(B_0,A_1)
\end{array}
$$
where $C_2$ acts on the left hand side $\H(\scB \times \scB)$ via the switch 
$\scB \times \scB \to \scB \times \scB: (A,B)\mapsto (B,A)$.
Then the following diagram of $C_2$-equivariant dg form functors commutes
$$
\xymatrix{\H(\scB\times 0) \vee \H(0\times \scB) \ar[d]^{\Phi} \ar[rr]^{\hspace{3ex}\nabla} &&
\H(\scB) \ar[d]^{=} \ar[r]^{\hspace{-3ex}\Delta} & \H(\scB\times\scB) \ar[d]^{\Phi}\\
\H(\H\scB) \ar[r]^H & \overline{\H\scB} & \H\scB \ar[l]_{\sim} \ar[r]^F & 
\H(\H\scB).}$$
It follows that the sequence
(\ref{eqn:GeqKGWKseqForH}) is $C_2$-equivariantly equivalent to the sequence
$$K(\scB)\vee K(\scB) \stackrel{\nabla}{\longrightarrow} K(\scB) \stackrel{\Delta}{\longrightarrow} K(\scB)\times K(\scB)$$
where $C_2$ acts on the left spectrum by switching the two summands, and on the right spectrum by switching the two factors.
\end{proof}

Diagram (\ref{eqn:KhGGWKhG}) is a diagram of module spectra over the commutative ring spectrum $GW(k)$.
This is because the diagram is obtained by applying the functor $GW$ to the $C_2$-equivariant diagram (\ref{eqn:NormOnCatLevel}) and taking various homotopy limits or colimits.
Thus, inverting $\eta \in GW^{[-1]}_{-1}(k)$ (or rather $\eta^4\mu \in GW_{-4}(k)$) in diagram (\ref{eqn:KhGGWKhG}) yields a commutative diagram of spectra
\begin{equation}
\label{eqnKhGWKhetalocn}
\xymatrix{
K(\scA)_{hC_2} \ar[r] \ar[d] & GW(\scA) \ar[r]\ar[d] & K(\scA)^{hC_2} \ar[d]\\
\eta^{-1}K(\scA)_{hC_2} \ar[r] & \eta^{-1}GW(\scA) \ar[r] & \eta^{-1}\left(K(\scA)^{hC_2}\right).
}
\end{equation}
For instance, $\eta^{-1}K(\scA)_{hC_2}$ is the homotopy colimit of the sequence
$$K(\scA)_{hC_2} \stackrel{\eta}{\to} (S^1\wedge K^{[-1]}(\scA))_{hC_2} \stackrel{\eta}{\to} (S^2\wedge K^{[-2]}(\scA))_{hC_2} \stackrel{\eta}{\to} \cdots$$
In particular, $\eta^{-1}K(\scA)_{hC_2} = (\eta^{-1}K(\scA))_{hC_2} = 0$ as $K_i(\scA)=0$ for $i<0$.
We have already seen that the spectrum $L(\scA)=\eta^{-1}GW(\scA)$ represents triangular Witt groups.
The following theorem identifies the lower right corner with the Tate spectrum $\hat{\bH}(C_2,K(\scA))$ of $K(\scA)$.

\begin{theorem}
\label{thm:williamsKobal}
Let $\scA$ be a dg category with weak equivalences and duality such that $\frac{1}{2}\in \scA$.
Then both squares of spectra in diagram (\ref{eqnKhGWKhetalocn}) are homotopy cartesian, the lower left corner is contractible, the lower middle term is $L(\scA)$ and the lower right term is the Tate spectrum of the $C_2$-spectrum $K(\scA)$.
In particular, there is a natural homotopy fibration 
$$K(\scA)_{hC_2} \to  GW(\scA) \to L(\scA)$$
and a homotopy cartesian square of spectra
$$
\xymatrix{
GW(\scA) \ar[r]\ar[d] & K(\scA)^{hC_2} \ar[d]\\
L(\scA) \ar[r] & \hat{\bH}(C_2,K(\scA)).
}$$
\end{theorem}

\begin{proof}
Consider the commutative diagram of spectra
$$
\xymatrix{
S^{n}\wedge K^{[-n]}(\scA)_{hC_2} \ar[r] \ar[d]^{\eta} & S^{n}\wedge GW^{[-n]}(\scA) \ar[r]\ar[d]^{\eta} & S^{n}\wedge K^{[-n]}(\scA)^{hC_2} \ar[d]^{\eta}\\
S^{n+1}\wedge K^{[-n-1]}(\scA)_{hC_2} \ar[r] & S^{n+1}\wedge GW^{[-n-1]}(\scA) \ar[r] & S^{n+1}\wedge K^{[-n-1]}(\scA)^{hC_2}.
}
$$
The maps on vertical homotopy fibres are equivalences by Theorem \ref{thm:PeriodicityExTriangle} and Lemma \ref{lem:KhGGWKhGhypEq}.
In other words both squares are homotopy cartesian.
Passing to the homotopy colimit for $n\to \infty$ shows that the squares in diagram (\ref{eqnKhGWKhetalocn}) are homotopy cartesian.
Since the lower left corner is contractible (see discussion preceding the theorem) and the composition of the upper two horizontal maps is the hypernorm map (Lemma \ref{lem:KGWKisNorm}), the lower right corner of diagram (\ref{eqnKhGWKhetalocn}) is the Tate spectrum $\hat{\bH}(C_2,K(\scA))$.
\end{proof}

\begin{remark}
\label{rem:LtoTatePeriodicity}
The map $L(\scA) \to \hat{\bH}(C_2,K(\scA))$ is a module map over the $4$-periodic ring spectrum $L(k)$.
In particular, both spectra are $4$-periodic and the map respects the periodicity maps given by the cup product with the element of $\eta^4\mu \in L_{-4}(k)$; compare \cite{someWeissWilliams}.
\end{remark}

\begin{remark}
The proof of Theorem \ref{thm:williamsKobal} is rather formal and uses only the natural exact triangle of Theorem \ref{thm:PeriodicityExTriangle}.
In particular, we could replace the functor $\scA \mapsto GW(\scA)$ with its localization $\scA \mapsto GW(\scA)[\frac{1}{2}]$ inverting $2$.
In this way, we obtain for any pretriangulated dg category with weak equivalences and duality such that $\frac{1}{2}\in \scA$ a homotopy cartesian square of spectra
$$
\xymatrix{
GW(\scA)[\frac{1}{2}] \ar[r]\ar[d] & \bH^{\bullet}\left(C_2,K(\scA)[\frac{1}{2}]\right) \ar[d]\\
L(\scA)[\frac{1}{2}] \ar[r] & \hat{\bH}\left(C_2,K(\scA)[\frac{1}{2}]\right).
}$$
The lower right corner in this diagram is contractible; see for instance \ref{lem:TateTwoinvertedX}.
Therefore, we have an equivalence of spectra
$$GW(\scA)[1/2] \stackrel{\sim}{\longrightarrow} L(\scA)[1/2] \oplus \bH^{\bullet}\left(C_2,K(\scA)[1/2]\right)$$
and an isomorphism of homotopy groups (Lemma \ref{lem:2divisiblehFixedPts})
$$GW_i(\scA)\otimes \Z[1/2] \cong W^{-i}(\scA)\otimes \Z[1/2] \oplus (K_i(\scA)\otimes \Z[1/2])^{C_2}.$$
\end{remark}

\subsection{Cofinality revisited}
\label{subsec:Kcofinal}
Let $\scA \to \scB$ be a map of pretriangulated dg categories with weak equivalences such that the induced map on associated triangulated categories $\T\scA \to \T\scB$ is cofinal, that is, it is fully faithful and
every object of $\T\scB$ is a direct factor of an object of $\T\scA$.
Let $K_0(\scB,\scA)$ be the monoid of isomorphism classes of objects in $\T\scB$ under direct
sum operation, modulo the submonoid of isomorphisms classes  of objects in
$\T\scA$. 
There is a canonical map $K_0(\scB) \to K_0(\scB,\scA): [B] \mapsto [B]$ which
induces an isomorphism between the cokernel of $K_0(\scA) \to K_0(\scB)$ and the group $K_0(\scB,\scA)$. 
By a theorem of Thomason \cite[1.10.1]{TT}, there is a homotopy fibration of spectra
\begin{equation}
\label{eqn:Kcofinal}
K(\scA) \to K(\scB) \to K_0(\scB,\scA)
\end{equation}
where $K_0(\scB,\scA)$ denotes the Eilenberg-MacLane spectrum whose only non-trivial homotopy group is in degree $0$ where it is $K_0(\scB,\scA)$.

Recall that for an Eilenberg-MacLane spectrum of a $C_2$-module $A$, we denote by $\bH^{\bullet}(C_2,A)$ the homotopy fixed point spectrum $A^{hC_2}$.
This spectrum represents cohomology of the group $C_2$ with coefficients in the $C_2$-module $A$.

\begin{theorem}[Cofinality]
\label{thm:cofinalGWL}
Let $\scA \to \scB$ be a map of dg categories with weak equivalences and duality such that $\frac{1}{2}\in \scA, \scB$.
If $\T\scA \to \T\scB$ is cofinal, then there are homotopy fibrations of spectra
$$
\renewcommand\arraystretch{1.5}
\begin{array}{lclcl}
GW(\scA) & \to & GW(\scB) & \to & \bH^{\bullet}(C_2,K_0)\\
L(\scA) & \to &L(\scB)& \to& \hat{\bH}(C_2,K_0)
\end{array}$$
where $K_0$ denotes the $C_2$-module $K_0(\scB,\scA)$ with $C_2$-action induced by the duality functor on $\scB$.
\end{theorem}

\begin{proof}
By Proposition \ref{prop:GWspecIsGWspace}, we can assume $\scA$ and $\scB$ pretriangulated.
Write ${GW}^?$ and ${L}^?$ for the homotopy cofibres of
$GW(\scA)  \to  GW(\scB)$ and $L(\scA)  \to L(\scB)$.
By cofinality in algebraic $K$-theory (\ref{eqn:Kcofinal}) and Theorem
\ref{thm:williamsKobal}, we have a homotopy cartesian square of spectra
$$\xymatrix{
{GW}^? \ar[r] \ar[d] & {\bH}^{\bullet}(C_2,K_0) \ar[d]\\
{L}^? \ar[r] & \hat{\bH}(C_2,K_0).}$$
By Cofinality for higher Grothendieck-Witt groups \cite[Theorem 7]{myMV} and the Invariance Theorem \ref{thm:Invariance}, the upper horizontal map is an isomorphism on homotopy groups in positive degrees since in this range both have zero homotopy groups.
It follows that the lower horizontal map is an isomorphism on homotopy groups
in degrees $\geq 2$.
Since the lower horizontal map is periodic as a module map over the periodic ring spectrum $L(k)$, it induces an isomorphism in all degrees.
It is therefore a stable equivalence.
It follows that the upper horizontal map is also a stable equivalence.
\end{proof}

Denote by $\dgCatWD_{\pretr} \subset \dgCatWD$ the full subcategory of pretriangulated dg categories with weak equivalences and duality.

\begin{lemma}
\label{lem:Astr}
There is a functor 
$$\dgCatWD_{\pretr} \to \dgCatWD_{\str}: \scA \mapsto \scA_{\str}$$
and a natural map $\scA \to \scA_{\str}$ in $\dgCatWD_{\pretr}$ which induces an equivalence $\T\scA \to \T\scA_{\str}$ on associated triangulated categories.
\end{lemma}

\begin{proof}
The construction in \cite[Lemma 4]{myMV} gives a functor that produces strict dualities without changing associated triangulated categories.
Thus, we can assume that $\scA$ has a strict duality and that functors commute with dualities, and we need to construct $\scA_{\str}$ satisfying 
$(X\oplus Y)^*=Y^*\oplus X^*$ and $0\oplus X = X =X \oplus 0$ for all objects $X, Y$ of $\scA_{\str}$.
The category $\scA_{\str}$ has objects the sequences 
$(A_0,...,A_n)$ of objects of $\scA$.
The mapping space from $(A_0,...,A_n)$ to $(B_0,...,B_m)$ is the mapping space from $A_0 \oplus \cdots \oplus A_n$ to $B_0 \oplus \cdots \oplus B_m$ in $\scA$ (which can be identified with the set of $m\times n$ matrices with $(i,j)$ entry the maps from $A_j$ to $B_i$, and thus is independent of the choice of a sum).
Composition is composition of maps in $\scA$ (or of matrices whichever is your point of view).
The direct sum of two sequences $(A_0,...,A_n)$ and $(B_0,...,B_m)$ is the sequence $(A_0,...,A_n,B_0,...,B_m)$, the empty sequence is the (base point) zero object, and the dual of a sequence is 
$(A_0,...,A_n)^* = (A_n^*,...,A_0^*)$.
The category $\scA_{\str}$ is a strict exact category with weak equivalences and duality and the functor 
$\scA \to \scA_{\str}: A \mapsto (A)$ 
is a duality preserving equivalence of categories and thus induces an equivalence on associated triangulated categories.
\end{proof}

\section{The Karoubi-Grothendieck-Witt spectrum}
\label{sec:KGWspectrum}

We saw in Theorem \ref{thm:LcnConn1} that the Grothendieck-Witt functor $GW$ sends exact sequences of triangulated categories to long exact sequences of higher Grothendieck-Witt groups. 
In applications, however, it is more useful to have a functor that sends sequences of triangulated categories that are exact only up to factors to long exact sequences of abelian groups.
In this section, we will construct such a functor $\GW$, the Karoubi-Grothendieck-Witt spectrum functor, together with a natural map of spectra $GW \to \GW$ which is an isomorphism in degrees $i\geq 1$ and a monomorphism in degree $i=0$.
In most cases of interest, this map is also an isomorphism in degree $i=0$.

The relation between the Grothendieck-Witt and Karoubi-Grothendieck-Witt functors $GW$ and $\GW$ is in some sense the same as the relation between connective and non-connective (or Bass) $K$-theory functors $K$ and $\bK$.
To construct non-connective $K$-theory from connective $K$-theory, 
there are essentially two ways.
\red{
In the geometric setting, one can use Bass' fundamental theorem as done by Thomason in \cite{TT}.
In the abstract setting, one can use algebraic suspension functors as done by Wagoner \cite{wagoner:deloop}, Gersten \cite{gersten:deloop} and Karoubi \cite{karoubi:deloop}, and first worked out in the generality applicable to dg categories with weak equivalences in \cite{mynegK}.
In this section, we follow the abstract approach.
In the following section, we will see in Theorem \ref{thm:Bassfund} and Remark \ref{rmk:ContractedFunctors} that the geometric opproach works as well.
}

\subsection{Cone and suspension of dg categories}
The {\em cone ring} is the ring $C$ of infinite matrices
$(a_{i,j})_{i,j\in\N}$ with coefficients $a_{i,j}$ in $\Z$ for which 
each row and each column has only finitely many non-zero entries.
Transposition of matrices ${^t(a_{i,j})} = (a_{j,i})$ makes $C$ into a ring
with involution.
As a $\Z$-module $C$ is torsion free, hence flat.

The {\em suspension ring} $S$ is the factor ring of $C$ by the two sided ideal
$M_{\infty} \subset C$ of those matrices which have only finitely many
non-zero entries. 
Transposition also makes $S$ into a ring with involution such that the
quotient map $C \twoheadrightarrow S$ is a map of rings with involution. 
For another description of the suspension ring $S$, consider the matrices $e_n
\in C$, $n\in \N$, with entries $(e_n)_{i,j}=1$ for $i=j\geq n$ 
and zero otherwise.
They are symmetric idempotents, \ie ${^te_n}=e_n=e_n^2$, and
they form a multiplicative subset of $C$ which satisfies the {\O}re condition,
that is, the multiplicative subset satisfies the axioms for a calculus of
fractions. 
One checks that the quotient map $C \twoheadrightarrow S$ identifies
the suspension ring $S$ with the 
localization of the cone ring $C$ with respect to the elements $e_n\in C$,
$n\in \N$. 
In particular, the suspension ring $S$ is also a flat $\Z$-module.

Let $\meps =1-e_1 \in C$ be the symmetric idempotent with entries $1$ at
$(0,0)$ and zero otherwise. 
The image $C\meps$ of the right multiplication map $\times \meps: C \to C$
is a finitely generated projective left $C$-module.
Denote by $\tilde{C}$ the full subcategory of left $C$-modules consisting of $C$ and $C\meps$.
Then $\tilde{C}(C,C)=C$, $\tilde{C}(C\meps,C) = \{ a\in C|\ a\meps = a\}$, 
$\tilde{C}(C,C\meps)=\{a\in C |\ \meps a = a\}$, and $\tilde{C}(C\meps,C\meps) = \{a\in C|\ a\meps = \meps a = a\} = \Z$.
The rule $\tilde{C}^{op}\to \tilde{C}: C \mapsto C,\ C\meps \mapsto C\meps,\ a\mapsto \red{^ta}$ makes $\tilde{C}$ into an additive category with strict duality.
Note that $\Z \to \tilde{C}:\Z \mapsto C\meps,\ n \mapsto n\meps$ defines a fully faithful duality preserving functor.
Since $\meps=0 \in S$, we have a duality preserving functor $\tilde{C} \to S:C \mapsto S,\ C\meps \mapsto 0$ (recall that $S$ is considered a pointed dg category) such that the composition
\begin{equation}\label{eqn:ZCSseq}\Z \to \tilde{C} \to S\end{equation}
is trivial.

For a dg category with weak equivalences and duality $\scA$, we denote by
$\tilde{C}\scA$ and $S\scA$ the dg categories with weak equivalences and duality $\tilde{C}\otimes_{\Z}\scA$ and $S\otimes_{\Z}\scA$. 
Tensoring the sequence (\ref{eqn:ZCSseq}) with $\scA\in \dgCatWD$, we obtain the the sequence
\begin{equation}\label{eqn:ACASAseq}\scA \to \tilde{C}\scA \to S\scA\end{equation}
\vspace{1ex}

Call a sequence $\T_0 \to \T_1 \to \T_2$ of triangulated categories {\em exact up to factors} if the composition is trivial, the functor $\T_0 \to \T_1$ is fully faithful and the induced functor from the Verdier quotient $\T_1/\T_0$ to $\T_2$ cofinal.

\begin{definition}
An exact dg functor $\scA \to \scB$ between dg categories with weak equivalences is called {\em Morita equivalence} if the associated triangle functor $\T\scA \to \T\scB$ is cofinal.
A sequence $\scA_0 \to \scA_1 \to \scA_2$ of dg categories with weak equivalences is called {\em Morita exact} if the
associated sequence of triangulated categories
$\T\scA_0 \to \T \scA_1 \to \T \scA_2$
is exact up to factors.
\end{definition}

\begin{lemma}
\label{lem:CSMoritaExactness}
The functors $C,S:\dgCatW \to \dgCatW$ preserve Morita equivalences and Morita exact sequences.
Moreover, for any dg category with weak equivalences $\scA$, the sequence
$$\scA \to \tilde{C}\scA \to S\scA$$
of dg categories with weak equivalences
is Morita exact.
\end{lemma}

\begin{proof}
Recall the following facts from the theory of dg categories \cite{Drinfeld:DGquotient}, especially \cite[Proposition 1.6.3]{Drinfeld:DGquotient}.

\begin{enumerate}
\item
Let $\B$ be a homotopically flat dg category \red{(that is, a dg category in which every mapping complex is homotopy equivalent to a complex of flat $k$-modules)}.
Then the functor 
$$\dgCat \to \dgCat:\A \mapsto \B\otimes\A$$
preserves Morita equivalences and Morita exact sequences.
\item
Let $\B_0 \to \B_1 \to \B_2$ be a Morita exact sequence of homotopically flat dg categories, then the sequence
$\B_0\otimes\A \to \B_1\otimes\A \to \B_2\otimes\A$
is Morita exact.
\item
\label{lem:item3:CSMoritaExactness}
For a full inclusion of dg categories $\A_0 \subset \A_1$ there is a dg category $\A_1\sslash\A_0$ together with a dg functor 
$\A_1 \to \A_1\sslash\A_0$ such that the sequence
$$\A_0 \to \A_1 \to \A_1\sslash\A_0$$
is Morita exact.
Moreover, this sequence is functorial in the pair $\A_0 \subset \A_1$.
\end{enumerate}
Since the dg categories $\tilde{C}$ and $S$ are flat dg categories, the lemma holds for dg categories.
Let $\B$ be a homotopically flat dg category, and $\scA = (\A,w)$ a dg category with weak equivalences.
By (\ref{lem:item3:CSMoritaExactness}) and the definition of $\T(\B\otimes \scA)$, the canonical  triangle functor 
$\T(\B\scA) \to \T(\B\A\sslash\B\A^w)$ is cofinal (in fact, this is an equivalence, but we won't need that).
Therefore, the lemma reduces to the dg category case.
\end{proof}

\begin{remark}
\label{rmk:myNegKsetup}
Lemma \ref{lem:CSMoritaExactness} shows that the functors $\tilde{C}$ and $S$ satisfy the requirements of the ``set-up'' of 
\cite[Section 2.2]{mynegK} on the category of dg categories with weak equivalences.
They can therefore be used to define the non-connective $K$-theory spectrum $\bK(\scA)$ of a dg category with weak equivalences $\scA$.
\end{remark}

\begin{lemma}
For any dg category with weak equivalences and duality $\scA$ such that $\frac{1}{2}\in \scA$, the Grothendieck-Witt spectrum $GW(\tilde{C}\scA)$ is contractible.
\end{lemma}

\begin{proof}
The proof is the same is the proof of \cite[Lemma 19]{myMV}, and we omit the details.
\end{proof}

In the following we will use terminology and results from Sections \ref{subsec:SpInMonCats} -- \ref{subsec:Bisp}.
Let $\tilde{S}^0 \to S^0$ be a cofibrant replacement of the monoidal unit $S^0\in \Sp$ in the positive projective stable model structure on the category $\Sp$ of spectra, and factor the composition $\tilde{S}^0 \to S^0 \to I$ into a cofibration (for the positive model structure) and a stable equivalence $\tilde{S}^{0} \rightarrowtail \tilde{I} \stackrel{\sim}{\to} I$ in $\Sp$.
Write $\tilde{S}^1$ for the quotient $\tilde{I}/\tilde{S}^0$.
The induced map on quotients $\tilde{S}^1 \to S^1$ is a stable equivalence, and $\tilde{S}^1$ is cofibrant in the positive stable model structure on $\Sp$.

The inner product space $\langle 1 \rangle$ of $\bk$ defines a map 
$\langle 1 \rangle: S^0 \to GW(\bk)$ of spectra.
Since in the positive stable model structure on $\Sp$, we have a cofibration $\tilde{S}^0 \to \tilde{I}$ and a trivial fibration $GW(\tilde{C}\bk) \to \pt$, the map 
$\tilde{S}^0 \to S^0 \to GW(\bk)$ extends to a map $\rho: \tilde{I} \to GW(\tilde{C}\bk)$ such that the diagram
$$\xymatrix{
\tilde{S}^0 \ar[r] \ar[d] & S^0 \ar[r]^{\hspace{-3ex}\langle 1\rangle} & GW(\bk) \ar[d] \\
\tilde{I} \ar[rr]^{\rho} & & GW(\tilde{C}\bk)}$$
commutes.
Moreover, the composition $GW(\bk) \to GW(\tilde{C}\bk) \to GW(S\bk)$ is trivial, 
and we obtain the
induced map of spectra $\sigma: \tilde{S}^1 \to GW(S\bk)$.

\begin{lemma}
\label{lem:idempCplteDeloop}
Let $\scA$ be a dg category with weak equivalences and duality such that $\frac{1}{2}\in \scA$.
Assume that the triangulated category $\T\scA$ of $\scA$ is idempotent complete.
Then the composition
$$\rho: \tilde{S}^1\wedge_SGW(\scA) \stackrel{\sigma \wedge 1}{\longrightarrow}
GW(S\bk)\wedge_SGW(\scA) \stackrel{\cup}{\longrightarrow} GW(S\scA)
$$
is a stable equivalence.
\end{lemma}

\begin{proof}
Since $\T\scA$ is idempotent complete, and $\tilde{C}\to S$ surjective on objects,  the Morita exact sequence
$$\T\scA \to \T\tilde{C}\scA \to \T S\scA$$
of Lemma \ref{lem:CSMoritaExactness} is in fact quasi-exact. 
Therefore, in the commutative diagram of spectra
$$\xymatrix{
\tilde{S}^0 \wedge_S GW(\scA) \ar[r] \ar[d]_{\langle 1\rangle \wedge 1} &
\tilde{I}\wedge_S GW(\scA) \ar[d]^{\rho\wedge 1} \ar[r] &
\tilde{S}^1 \wedge_S GW(\scA) \ar[d]^{\sigma\wedge 1}\\
GW(\bk)\wedge_S GW(\scA) \ar[r] \ar[d]_{\cup} &
GW(\bk)\wedge_S GW(\scA) \ar[d]^{\cup} \ar[r] &
GW(\bk) \wedge_S GW(\scA) \ar[d]^{\cup}\\
GW(\scA) \ar[r] & GW(\tilde{C}\scA) \ar[r] & GW(S\scA),
}$$
the last row is a homotopy fibration, by Theorem \ref{thm:LcnConn1}.
As a cofibre sequence, the first row is also a homotopy fibration.
Since the composition of the left two and middle two vertical maps are stable equivalences (the latter composition by contractibility of $\tilde{I}$ and $GW(\tilde{C}\scA)$), the map in the lemma is also a stable equivalence.
\end{proof}

\subsection{The Karoubi-Grothendieck-Witt spectrum}
\label{subsec:KGWspec}
In order to define a symmetric monoidal functor $\GW$ from dg categories to a category of spectra that sends Morita exact sequences to homotopy fibrations, 
we will work in the symmetric monoidal category of bispectra $\BiSp$.
The category of bispectra is yet another model for the stable homotopy category of spectra.
This category contains the category of spectra $\Sp$ as a symmetric monoidal subcategory.
The inclusion $\Sp \subset \BiSp$ preserves stable equivalences and induces an equivalence of associated homotopy categories; see Appendix \ref{subsec:Bisp}.
The Karoubi-Grothendieck-Witt spectrum functor most naturally has values in this category of bispectra.
If we are not interested in multiplicative properties of the Karoubi-Grothendieck-Witt spectrum, we can stay within the category $\Sp$ of spectra and equivalently work with a certain mapping telescope of spectra as in Remark \ref{rmk:BiSpVsSpForGW} below.

Recall that the category  of $\tilde{S}^1$-$S^1$-bispectra is the category of $\tilde{S}^1$-spectra in $\Sp$, that is, the category of left modules over the free commutative monoid $\tilde{S}$ generated by $(0,\tilde{S}^1,0,0,...)$ in $\Sp^{\Sigma}$.
Thus, to specify a symmetric monoidal functor
$\GW:\dgCatWD \to \BiSp$ into $\tilde{S}^1$-$S^1$-bispectra is the same as to specify a symmetric monoidal functor
$\GW:\dgCatWD \to \Sp^{\Sigma}$ into symmetric sequences of spectra together with a map of spectra $\sigma:\tilde{S}^1 \to \GW(\bk)_1$.

Application of Remark \ref{rmk:symmonSymSeq} with $F$ 
the symmetric monoidal functor 
$$\dgCatWD \times \dgCatWD \to \Sp: (\scB, \scA) \mapsto GW(\B\A)$$ 
and $U =S \in \dgCatWD$ yields the symmetric monoidal functor
$$\dgCatWD \longrightarrow \Sp^{\Sigma}: \scA \mapsto \left\{n \mapsto \GW(\scA)_n\right\}$$
with values in $\Sp^{\Sigma}$ 
where $\GW(\scA)_n$ is the spectrum $GW((S)^{\otimes {n}}\scA)$ with left $\Sigma_n$-action permuting the tensor factors $(S)^{\otimes {n}}$.
The map of spectra $\sigma:\tilde{S}^1 \to \GW(\bk)_1 = GW(S\bk)$ extends uniquely to a map of commutative monoids 
$$\tilde{S} \to \GW(\bk)$$
in $\Sp^{\Sigma}$ making every $\GW(\bk)$-module $\GW(\scA)$ into a module over $\tilde{S}$.
Thus, we have defined a symmetric monoidal functor
$$\GW:\dgCatWD \to \BiSp.$$

\begin{definition}
Let $\scA$ be a dg category with weak equivalences and duality such that $\frac{1}{2}\in \scA$.
The bispectrum $\GW(\scA)$ constructed in Section \ref{subsec:KGWspec} is called the {\em Karoubi-Grothendieck-Witt} spectrum of $\scA$.
It is equipped with a natural map of bispectra
\begin{equation}
\label{eqn:GWtoBBGW}
GW(\scA) = \GW(\scA)_0 \to \GW(\scA)
\end{equation}
Its homotopy groups are the {\em Karoubi-Grothendieck-Witt groups} of $\scA$:
$$\GW_i(\scA) = \pi_i\GW(\scA) = [\tilde{S}^n,\GW(\scA)]_{\BiSp}.$$
As usual, we write $\GW^{[n]}(\scA)$ and $\GW^{[n]}_i(\scA)$ for $\GW(\scA^{[n]})$ and $\GW_i(\scA^{[n]})$.
\end{definition}

Let $\T$ be a triangulated category.
Recall from \cite{BalmerMe} that its idempotent completion $\tilde{\T}$ is canonically a triangulated category.
If $\T$ is a triangulated category with duality, then the duality on $\T$ canonically extends to a duality on \red{$\tilde{\T}$} such that the inclusion $\T \subset \red{\tilde{\T}}$ is duality preserving.

\begin{proposition}
\label{prop:KGWbasicComps}
Let $\scA$ be a dg category with weak equivalences and duality such that $\frac{1}{2}\in \scA$.
Then the bispectrum $\GW(\scA)$ is semistable.
The Karoubi-Grothendieck-Witt groups are given by
$$\GW^{[n]}_i(\scA) = 
\left\{ 
\renewcommand\arraystretch{2}
\begin{array}{ll}
GW^{[n]}_i(\scA)& i>0\\
GW^{[n]}_0(\widetilde{\T\scA}) & i=0\\
GW^{[n]}_0(\widetilde{\T S^{-i}\scA}) & i<0.
\end{array}
\right.
$$
In particular, the map (\ref{eqn:GWtoBBGW})
$$GW^{[n]}_i(\scA) \to \GW^{[n]}_i(\scA)$$
is an isomorphism for $i\geq 1$ and a monomorphism for $i=0$.
If $\T\scA$ is idempotent complete, then this map is also an isomorphism for $i=0$.
\end{proposition}

\begin{proof}
The triangulated category $\T C\scA$ is idempotent complete.
This is because there is a triangle functor $F:\T C\scA \to \T C\scA$ and a natural isomorphism $F \oplus id \cong F$ (compare \cite[Lemma 19]{myMV}).
The triangle functor and the natural isomorphism extend to a triangle functor and a natural isomorphism on the idempotent completion of $\T C\scA$.
This forces the $K_0$ groups of $\T C\scA$ and its idempotent completion to be trivial.
In particular, the cofinal inclusion of $\T C\scA$ into its idempotent completion induces an isomorphism on $K_0$. 
By Thomason's classification theorem of dense subcategories \cite{Thomason:classification}, this inclusion is an equivalence, that is, $\T C\scA$ is already idempotent complete.

For a dg category with weak equivalences and duality $\scA$, let $\scA^c$ be the full dg subcategory of $(C\scA)^{\pretr}$ of those objects which are direct factors in $\T(C\scA)$ of objects of $\T(\scA)$.
A map in $\scA^c$ is a weak equivalence if it is an isomorphism in $\T(C\scA)$.
Then $\scA^c$ is pretriangulated, contains $\scA$, and has associated triangulated category $\T(\scA^c)$ the idempotent completion of $\T(\scA)$. 
Moreover, the Morita equivalence $S\scA \to S(\scA^c)$ is in fact a quasi-equivalence because both have trivial $K_0$ as a quotient of $K_0(C\scA) = K_0(C(\scA^c)) = 0$.
In the diagram
$$\xymatrix{ 
\tilde{S}^1\wedge_S GW(\scA) \ar[r] \ar[d]_{\sigma\cup} & \tilde{S}^1\wedge_S GW(\scA^c) \ar[d]_{\sigma\cup}\\
 GW(S\scA) \ar[r] & GW\left(S(\scA^c)\right), }$$
the right vertical map is a stable equivalence, by Lemma \ref{lem:idempCplteDeloop}.
By Cofinality (Theorem \ref{thm:cofinalGWL}) the top horizontal map is an isomorphism on $\pi_i$ for $i\geq 0$, and the bottom one is an isomorphism in all degrees, by the Invariance Theorem \ref{thm:Invariance}.
Therefore, the left vertical map is an isomorphism in degrees $i\geq 0$.
By Lemma \ref{lem:semistable}, the bispectrum $\GW$ is semi-stable, and 
$\GW_i(\scA)$ it the colimit of the sequence
$$GW_i(\scA) \stackrel{\sigma \cup }{\longrightarrow} GW_{1+ i}(S\scA) 
\stackrel{\sigma \cup }{\longrightarrow} \cdots 
GW_{k+i}(S^k\scA) \stackrel{\sigma \cup }{\longrightarrow} GW_{1+k+i}(S^{1+k}\scA) \cdots
$$
which consists of isomorphisms after $k+i \geq 1$.
This proves the claim in the proposition for $i>0$.
If $i\leq 0$, then 
$$\GW_i(\scA) = GW_1(S^{1-i}\scA) = GW_0((S^{-i}\scA)^c) = GW_0(\widetilde{\T S^{-i}\scA}).$$
\end{proof}

\begin{remark}
\label{rmk:BiSpVsSpForGW}
Recall that the inclusion $\Sp \subset \BiSp$ of spectra into bispectra preserves stable equivalences and induces an equivalence of homotopy categories.
By Remark \ref{rmk:OmegaSpforSemiStableBiSp} and Proposition \ref{prop:KGWbasicComps}, the spectrum corresponding to the bispectrum $\GW(\scA)$ under this equivalence is the mapping telescope of the sequence of spectra
$$GW(\scA) \stackrel{\sigma \cup}{\longrightarrow}  \Omega GW(S\scA) \stackrel{\sigma \cup}{\longrightarrow} \Omega^2 GW(S^2\scA) \stackrel{\sigma \cup}{\longrightarrow} \cdots $$
This telescope therefore has \red{the} correct homotopy type.
But it does not retain good point-set multiplicative properties.
In particular, it does not define a symmetric monoidal functor $\dgCatWD \to \Sp$.
This is the reason we consider $\GW$ as a symmetric monoidal functor into bispectra rather than as a functor with values in spectra defined by the mapping telescope above.
\end{remark}

\begin{theorem}[Invariance for $\GW$]
\label{thm:InvarianceKGW}
Let $\scA \to \scB$ be an exact dg form functor between dg categories with weak equivalences and duality such that $\frac{1}{2}\in \scA,\ \scB$.
If $\T\scA \to \T\scB$ is cofinal, then the map of Karoubi-Grothendieck-Witt spectra is a stable equivalence
$$\GW(\scA) \stackrel{\sim}{\longrightarrow} \GW(\scB).$$
\end{theorem}

\begin{proof}
The maps of abelian groups $\GW_i(\scA) \to \GW_i(\scB)$ are isomorphisms in view of Proposition \ref{prop:KGWbasicComps} and Lemma \ref{lem:CSMoritaExactness}.
\end{proof}

\begin{theorem}[Localization for $\GW$]
\label{thm:locnForKGW}
Let $\scA_0 \to \scA_1 \to \scA_2$ be a Morita exact sequence of dg categories with weak equivalences and duality.
Assume that $\frac{1}{2}\in \scA_0,\scA_1,\scA_2$.
Then the commutative square of Karoubi-Grothendieck-Witt spectra
$$\xymatrix{
\GW^{[n]}(\scA_0) \ar[r] \ar[d] & \GW^{[n]}(\scA_1) \ar[d]\\
\GW^{[n]}(\scA_2^w) \ar[r] & \GW^{[n]}(\scA_2)
}$$
is homotopy cartesian, and the lower left corner is contractible.
In particular, for all $n\in \Z$ there is a homotopy fibration of Karoubi-Grothendieck-Witt spectra
$$\GW^{[n]}(\scA_0) \to \GW^{[n]}(\scA_1) \to \GW^{[n]}(\scA_2).$$
\end{theorem}

\begin{proof}
To ease notation, we will prove the theorem for $n=0$.
The general case is then obtained from the case $n=0$ by replacing $\A_i$ with $\scA_i^{[n]}$.
By Remark \ref{rmk:BiSpVsSpForGW}, it suffices to check that 
the mapping telescope of the squares of spectra 
\begin{equation}
\label{equn:RnKGW1}
\xymatrix{
\Omega^nGW(S^n\scA_0) \ar[r] \ar[d] & \Omega^nGW(S^n\scA_1) \ar[d] \\
\Omega^nGW(S^n\scA^w_2) \ar[r]& \Omega^nGW(S^n\scA_2).
}
\end{equation}
is homotopy cartesian with contractible lower left corner.
Here, we have written $\Omega^n$ for $\Map_{\Sp}(\tilde{S}^n,\phantom{A})$.
One square maps into the next via the cup-product with $\sigma: \tilde{S}^1 \to GW(S\bk)$.

As in the proof of Proposition \ref{prop:KGWbasicComps}, let $\scA^c$ be the full dg subcategory of $(C\scA)^{\pretr}$ of those objects which are direct factors in $\T(C\scA)$ of objects of $\T(\scA)$.
So, $\T\scA \to \T\scA^c$ is an idempotent completion.
Let $(S^n\scA_2)^e \subset (S^n\scA_2)^c$ be the full dg subcategory of those objects which are in the essential image of the functor
$\T(S^n\scA_1)^c \to \T(S^n\scA_2)^c$.
The dg category $(S^n\scA_2)^e$ inherits the structure of a dg category with weak equivalences and duality from $(S^n\scA_2)^c$.
By Lemma \ref{lem:idempCplteDeloop} and the Cofinality theorem for $GW$ (or from the Invariance theorem \ref{thm:InvarianceKGW}), the telescope of the diagrams 
(\ref{equn:RnKGW1}) is stably equivalent to the mapping telescope over the diagrams of spectra
\begin{equation}
\label{equn:RnKGW2}
\xymatrix{
\Omega^nGW(S^{n}\scA_0)^c \ar[r] \ar[d] & \Omega^nGW(S^{n}\scA_1)^c \ar[d] \\
\Omega^nGW(S^{n}\scA^e_2)^w \ar[r]& \Omega^nGW(S^{n}\scA_2)^e.
}
\end{equation}
By Lemma \ref{lem:CSMoritaExactness}, the sequence
$$(S^n\scA_0)^c \to (S^n\scA_1)^c \to (S^n\scA_2)^e$$
is Morita exact.
In fact, this sequence is quasi-exact, since the first two dg categories with weak equivalences have idempotent complete associated triangulated categories, and the last functor is essentially surjective on associated triangulated categories.
By the Localization Theorem \ref{thm:LcnConn1} for $GW$, the squares \ref{equn:RnKGW2} are homotopy (co-) cartesian with contractible lower left corner.
Therefore, the same is true for the mapping telescope over the diagrams (\ref{equn:RnKGW2}) and (\ref{equn:RnKGW1}).
\end{proof}

As in Section \ref{sec:PeriodInvLocn}, set 
$$\bK(\scA) = \GW(\H\scA).$$
By Remark \ref{rmk:myNegKsetup}, this bispectrum represents non-connective algebraic $K$-theory of $\scA$ as defined in \cite{mynegK}.

\begin{theorem}[Algebraic Bott Sequence for $\GW$]
\label{thm:PeriodicityExTriangleForKGW}
Let $\scA$ be a dg category with weak equivalences and duality for which $\frac{1}{2}\in \scA$.
Then the sequence of Karoubi-Grothendieck-Witt spectra
$$\GW^{[n]}(\scA) \stackrel{F}{\longrightarrow} \bK(\scA) \stackrel{H}{\longrightarrow} \GW^{[n+1]}(\scA) \stackrel{\eta\cup}{\longrightarrow}S^1\wedge \GW^{[n]}(\scA)$$
is an exact triangle in the homotopy category of (bi-)spectra.
\end{theorem}

\begin{proof}
This follows from Theorem \ref{thm:PeriodicityExTriangle} in view of Remark \ref{rmk:BiSpVsSpForGW}.
Alternatively, it also follows from the Localization Theorem \ref{thm:locnForKGW} in view of Remark \ref{rmk:FundThmFromLocnThm}.
\end{proof}

As in Section \ref{Z/2actionOnKth}, we consider the bispectrum $\bK(\scA) = \GW(\H\scA)$ equipped with the $C_2$-action coming from the action on $\H\scA$ and the functoriality of $\GW$, and we obtain a commutative diagram of bispectra
\begin{equation}
\label{eqn:KobalWilliamsKGW}
\xymatrix{
\bK(\scA)_{hC_2} \ar[r] \ar[d] & \GW(\scA) \ar[r]\ar[d] & \bK(\scA)^{hC_2} \ar[d]\\
\eta^{-1}\bK(\scA)_{hC_2} \ar[r] & \eta^{-1}\GW(\scA) \ar[r] & \eta^{-1}\left(\bK(\scA)^{hC_2}\right).
}
\end{equation}
as in (\ref{eqnKhGWKhetalocn}).

\begin{definition}[Compare Definition \ref{dfn:LthSp}]
\label{dfn:stableWittThSp}
For $\scA\in \dgCatWD$, we define its stabilized $L$-theory spectrum $\bL(\scA)$ as
$$\bL(\scA) = \eta^{-1}\GW(\scA) = \eta^{-1}\sigma^{-1}GW(\scA).$$
This a module spectrum over the commutative ring spectrum $L(k)=(\eta^4\mu)^{-1}GW(k)$.
In particular, it is $4$-periodic with periodicity isomorphism induced by the cup-product with $\eta^4\mu$.
As usual, $\bL^{[n]}(\scA)$ denotes $\bL(\scA^{[n]})$,
and the homotopy groups of $\bL^{[n]}(\scA)$ are denoted by $\bL^{[n]}_i(\scA)$.
As in Remark \ref{rmk:BiSpVsSpForGW}, the spectrum $\bL(\scA)$ is the mapping telescope of the sequence of spectra
\begin{equation}
\label{eqn:LAsColim}
L(\scA) \stackrel{\sigma \cup}{\longrightarrow}  \Omega L(S\scA) \stackrel{\sigma \cup}{\longrightarrow} \Omega^2 L(S^2\scA) \stackrel{\sigma \cup}{\longrightarrow} \cdots 
\end{equation}
In particular, its homotopy groups $L_i(\scA)$ are given by the colimit
\begin{equation}
\label{eqn:LiAsColim}
L_i(\scA) \stackrel{\sigma \cup}{\longrightarrow}  L_{i+1}(S\scA) \stackrel{\sigma \cup}{\longrightarrow} L_{i+2}(S^2\scA) \stackrel{\sigma \cup}{\longrightarrow} \cdots 
\end{equation}
These groups are isomorphic to Karoubi's stabilized Witt groups \cite{Karoubi:stabilized}, but we won't need that fact here.
\end{definition}

With this definition we have the $\GW$-variant of Theorem \ref{thm:williamsKobal} which is proved in the same way as Theorem \ref{thm:williamsKobal}. 

\begin{theorem}
\label{thm:williamsKobalKGW}
Let $\scA$ be a dg category with weak equivalences and duality such that $\frac{1}{2}\in \scA$.
Then the squares of (bi-) spectra (\ref{eqn:KobalWilliamsKGW})
are homotopy cartesian, the lower left corner is contractible, and the lower right term is the Tate spectrum of the $C_2$-spectrum $\bK(\scA)$.
In particular, there is a natural homotopy fibration 
$$\bK(\scA)_{hC_2} \to  \GW(\scA) \to \bL(\scA)$$
and a homotopy cartesian square of spectra
$$
\xymatrix{
\GW(\scA) \ar[r]\ar[d] & \bH^{\bullet}(C_2,\bK(\scA)) \ar[d]\\
\bL(\scA) \ar[r] & \hat{\bH}(C_2,\bK(\scA)).
}$$
\Qed
\end{theorem}

Recall (\ref{eqn:GWtoBBGW}) that there is a natural map of (bi-) spectra
$GW(\scA) \to \GW(\scA)$
since the $0$-th spectrum of the bispectrum $\GW(\scA)$ is the spectrum $GW(\sc\A)$.

\begin{theorem}
\label{thm:KbKGWKGW}
Let $\scA$ be a dg category with weak equivalences and duality such that $\frac{1}{2}\in \scA$.
Then the following square of (bi-) spectra is homotopy cartesian
$$\xymatrix{
GW(\scA) \ar[r] \ar[d] & \GW(\scA) \ar[d]\\
K(\scA)^{hC_2} \ar[r] & \bK(\scA)^{hC_2}
}$$
\end{theorem}

\begin{proof}
By Theorems \ref{thm:williamsKobal} and \ref{thm:williamsKobalKGW} the map on vertical homotopy fibres is a map of $L(\bk)$-modules.
In particular, this map is $4$-periodic.
By Proposition \ref{prop:KGWbasicComps} and its $K$-theory analog, the maps $GW(\scA) \to \GW(\scA)$ and $K(\scA) \to \bK(\scA)$ are isomorphisms in positive degrees (and monomorphisms in degree $0$).
It follows that the horizontal homotopy fibres are trivial in high degrees.
In particular, the map between horizontal homotopy fibres is an isomorphism in high degrees.
Therefore, the map on vertical homotopy fibres is an isomorphism in sufficiently high degrees.
Since this map is periodic, it is an isomorphism in all degrees. 
\end{proof}

\begin{proposition}[Additivity for $\GW$]
\label{prop:AddKGW}
Let $(\scU,{\vee})$ be a pretriangulated dg category with weak equivalences and duality such that $\frac{1}{2}\in \scU$.
Let $\scA \subset \scU$ be a full pretriangulated dg subcategory containing the 
$w$-acyclic objects $\scU^w$ of $\scU$.
Assume that $\T\scU(A,B^{\vee})=0$ for all $A,B\in \T\scA$, and that $\T\scU$ is generated as a triangulated category by $\T\scA$ and $\T\scA^{\vee}$.
Then the exact dg form functor 
$$\H\scA \to \scU: (A,B)\mapsto A\oplus B^{\vee}$$
induces a stable equivalence of Karoubi-Grothendieck-Witt spectra:
$$\bK(\scA) = \GW(\H\scA) \stackrel{\sim}{\longrightarrow} \GW(\scU).$$
\end{proposition}

\begin{proof}
From the proof of Proposition \ref{prop:AddGW}, the functor 
$\H\scA \to \scU$ induces a $K$-theory equivalence.
The same argument applies to show that it induces an equivalence of non-connective $K$-theory spectra $\bK$.
Since it is a $GW$-theory equivalence, by Proposition \ref{prop:AddGW}, 
Theorem \ref{thm:KbKGWKGW} implies that it is also a an equivalence of Karoubi-Grothendieck-Witt spectra $\GW$.
\end{proof}

We end this section with a lemma which we need in the proof of Brumfiel's theorem in \cite{KSW}.

\begin{lemma}
Let $\scA$ be a dg category with weak equivalences and duality such that $\frac{1}{2}\in \scA$.
Then the natural map $L(\scA) \to \bL(\scA)$ induces a weak equivalence of spectra after inverting $2$:
$$L(\scA)[1/2] \stackrel{\simeq}{\longrightarrow} \bL(\scA)[1/2].$$
\end{lemma}

\begin{proof}
Recall from the proof of Proposition \ref{prop:KGWbasicComps} the dg categories with weak equivalences and duality $\scA^c$, $S\scA$ and $S(\scA^c)$.
The map $\scA \to \scA^c$ is a Morita equivalence, $\T(\scA^c)$ is idempotent complete, and $S(\scA) \to S(\scA^c)$ is a quasi-equivalence.
Lemma \ref{lem:idempCplteDeloop} holds with $L$ in place of $GW$.
In particular, in the commutative diagram
$$\xymatrix{
L(\scA) \ar[r]^{\sigma \cup} \ar[d] & \Omega L(S\scA) \ar[d]\\
L(\scA^c) \ar[r]^{\hspace{-3ex}\sigma \cup} & \Omega L((S\scA)^c)
}$$
the lower horizontal and the right vertical maps are equivalences.
The left vertical map is an equivalence after inverting $2$, by Cofinality (Theorem \ref{thm:cofinalGWL}), hence, so is the top horizontal map.
It follows that all maps in the sequence (\ref{eqn:LAsColim}) are equivalences after inverting $2$.
Since $\bL(\scA)[1/2]$ is the homotopy colimit of that sequence, we are done.
\end{proof}

\section{Higher Grothendieck-Witt groups of schemes}
\label{sec:GWschemes}

\subsection{Vector bundle Grothendieck-Witt groups}
Let $X$ be a scheme. 
Write $\Vect(X)$ for the category of finite rank locally free sheaves on $X$.
Denote by $\sPerf(X) = \Ch^b\Vect(X)$ the dg category of strictly perfect complexes on $X$, that is, the dg category of bounded  complexes in $\Vect(X)$; \red{see Example \ref{ex:ChbEduality}.}
\red{The underlying category $Z^0\sPerf(X)$ is a closed symmetric monoidal category with tensor product $\otimes_{O_X}$ and internal homomorphism complex $Hom^{\bullet}_{O_X}(E,F)$ given by formulas as in Subsection \ref{subsec:difk}.
The mapping complex of two objects in $\sPerf(X)$ is the complex of global sections of their internal homomorphism object.}
As in any closed symmetric monoidal category,
an object $A$ of $\sPerf(X)$ defines a duality 
$$\sharp_A:\sPerf(X)^{op} \to \sPerf(X): E \mapsto Hom^{\bullet}_{O_X}(E,A)$$ with double dual identification $\can^A:E \to E^{\sharp_A\sharp_A}$ given by the formula
$$\can^A_E(x)(f) = (-1)^{|x||f|}f(x).$$
If $A$ is an invertible strictly perfect complex on $X$, that is, a shift of a line bundle on $X$, then the double dual identification is a natural isomorphism.

Let $A$ be an invertible strictly perfect complex on $X$.
Together with the set $\quis$ of quasi-isomorphisms, we obtain a dg category with weak equivalences and duality
\begin{equation}
\label{eqn:PerfL}
\sPerf^{A}(X) = (\sPerf(X),\quis,\sharp_{A},\can^{A}).
\end{equation}
We say that $\frac{1}{2}\in X$ if $\frac{1}{2}\in \Gamma(X,O_X)$.

\begin{definition}
Let $X$ be a scheme with an ample family of line bundles \red{\cite[Definition 2.1]{TT}} such that $\frac{1}{2}\in X$, and
let $L$ a line bundle on $X$.
The {\em $n$-th shifted Grothendieck-Witt} and {\em Karoubi-Grothendieck-Witt spectra of $X$ with coefficients in $L$} are the $n$-th shifted Grothendieck-Witt and Karoubi-Grothendieck-Witt spectra of (\ref{eqn:PerfL}) for $A=L$, that is, 
$$
\renewcommand\arraystretch{2}
\begin{array}{lcl}
GW^{[n]}(X,L) & = & GW^{[n]}(\sPerf^L(X)),\ \text{and}\\
\GW^{[n]}(X,L) & = & \GW^{[n]}(\sPerf^L(X)).
\end{array}$$
Their homotopy groups are denoted by 
$$
\renewcommand\arraystretch{2}
\begin{array}{lcl}
GW^{[n]}_i(X,L) &=& \pi_iGW^{[n]}(X, L),\ \text{and}\\
\GW^{[n]}_i(X,L) &=& \pi_i\GW^{[n]}(X, L).
\end{array}$$
When $n=0$, or $L=O_X$, we may omit the decoration corresponding to it.
For instance, $GW(X)$ denotes $GW^{[0]}(X, O_X)$. 
\end{definition}

\red{
\begin{remarkNoNb}
In \cite{TT}, the $K$-theory of a quasi-compact and quasi-separated scheme $X$ is defined as the $K$-theory of the category of perfect complexes on X.
If $X$ has an ample family of line bundles, this is also the $K$-theory of the category of strictly perfect complexes.
If one wishes, one can develop the theory of higher Grothendieck-Witt groups for perfect complexes
in the generality of Thomason;  the technical foundations have been laid in the previous sections.
But at the moment I don't see the need to deal with the extra technicalities that come with this generality, and have decided to stick with schemes that have an ample family of line bundles. 
This is sufficient for most applications.
\end{remarkNoNb}
}

\begin{remark}
Let $X$ be a $k$-scheme with structure map $p:X \to \Spec(k)$.
The tensor product 
$$\scC^{[n]}_{\bk}\otimes_{\bk} \sPerf^L(X) \to \sPerf^{L[n]}(X): A \otimes_{\bk} E \mapsto p^*A\otimes_{O_X}E$$ with duality compatibility as in (\ref{eqn:dualityCompIsoForTensorProds})
is an equivalence of dg categories with weak equivalences and duality.
Therefore, the Grothendieck-Witt and Karoubi-Grothendieck-Witt spectra $GW^{[n]}(X, L)$ and $\GW^{[n]}(X, L)$ are really the Grothendieck-Witt and Karoubi-Grothendieck-Witt spectra of 
$\sPerf^{L[n]}(X)$ (without shift).
\end{remark}

\begin{proposition}
Let $X$ be a scheme with an ample family of line bundles such that $\frac{1}{2}\in X$.
Then the following hold.
\begin{enumerate}
\item
For all $n,i\in \Z$ with $i\geq 0$ and all line bundles $L$ on $X$, the map (\ref{eqn:GWtoBBGW}) is an isomorphism 
$$GW^{[n]}_i(X,L) \stackrel{\cong}{\longrightarrow} \GW^{[n]}_i(X,L).$$
\item
If $\bK_i(X) = 0$ for $i<0$ (e.g., $X$ is a regular noetherian separated scheme) then the map (\ref{eqn:GWtoBBGW}) is a stable equivalence for all line bundles $L$ on $X$ and all $n\in \Z$:
$$GW^{[n]}(X,L) \stackrel{\sim}{\longrightarrow} \GW^{[n]}(X,L).$$
\end{enumerate}
\end{proposition}

\begin{proof}
The first part follows from Proposition \ref{prop:KGWbasicComps} since $\T(\sPerf(X))=D^b\Vect(X)$ is idempotent complete; see for instance \cite{BalmerMe}.

For the second part, the map $K(X) \to \bK(X)$ is a stable equivalence, by assumption.
Therefore, the lower horizontal map in the square of Theorem \ref{thm:KbKGWKGW} is a stable equivalence.
Since this square is homotopy cartesian, the upper horizontal map of that square is also an equivalence.
\end{proof}

\subsection{Products}
Let $L_1$, $L_2$ be two line bundles on $X$ and denote by $L_1L_2$ their tensor product $L_1\otimes_{O_X}L_2$.
Tensor product of complexes defines an exact dg form functor
$$(\otimes_{O_X},\can): \sPerf^{L_1}(X) \otimes_{\bk} \sPerf^{L_2}(X) \longrightarrow \sPerf^{L_1L_2}(X)$$
with duality compatibility map $\can$ as defined in (\ref{eqn:dualityCompIsoForTensorProds}).
The tensor product is associative and unital up to natural isomorphism of dg form functors with unit the duality preserving functor $k \to \sPerf(X):k \mapsto O_X$.
Recall that the functors $GW:\dgCatWD \to \Sp$ and $\GW:\dgCatWD \to \BiSp$ are symmetric monoidal, and that the natural transformation $GW \to \GW$ commutes with the monoidal compatibility maps.
Therefore, the tensor product of strictly perfect complexes induces maps of (bi-) spectra
$$GW^{[n]}(X, L_1) \wedge_S
GW^{[m]}(X, L_2) \stackrel{\cup}{\longrightarrow}
GW^{[n+m]}(X, L_1L_2)
$$
$$\GW^{[n]}(X, L_1) \wedge_{\tilde{S}}
\GW^{[m]}(X, L_2) \stackrel{\cup}{\longrightarrow}
\GW^{[n+m]}(X, L_1L_2)
$$
which are associative and unital up to homotopy.
Taking homotopy groups, we obtain associative and unital cup products
$$GW^{[n]}_i(X, L_1) \otimes
GW^{[m]}_j(X, L_2) \stackrel{\cup}{\longrightarrow}
GW^{[n+m]}_{i+j}(X, L_1L_2)
$$
$$\GW^{[n]}_i(X, L_1) \otimes
\GW^{[m]}_j(X, L_2) \stackrel{\cup}{\longrightarrow}
\GW^{[n+m]}_{i+j}(X, L_1L_2)
$$
which are commutative in the sense that
$$\tau (a\cup b) = (-1)^{ij}\ \langle -1\rangle^{mn}\cup  b\cup a$$
for all $a\in GW^{[n]}_i(X, L_1)$ (or $a\in \GW^{[n]}_i(X, L_1)$), and 
$b \in GW^{[m]}_j(X, L_2)$ (or $b \in GW^{[m]}_j(X, L_2)$) where
$\tau:L_1 L_2 \cong L_2 L_1$ is the switch isomorphism and $\langle -1 \rangle = [O_X,-1] \in GW_0(X)$.
The proof is the same as the proof of Proposition \ref{prop:prod:commutativityCk}, and we omit the details.

\subsection{Functoriality}
\label{rmk:FunctorialityI}
Let $\Sch$ be a small category of schemes, e.g., schemes of finite type over some fixed base.
Let $\widetilde{\Sch}$
denote the category whose objects are the pairs $(X,L)$ where $X$ is a scheme in $\Sch$ and $L$ is a line bundle on $X$.
A morphism $(X,L) \to (Y,M)$ in this category is a pair $(f,\alpha)$ where $f:X \to Y$ is a map of schemes and $\alpha: f^*M \cong L$ is an isomorphism.
Composition of the two maps $(f,\alpha):(X,L) \to (Y,M)$ and $(g,\beta):(Y,M) \to (Z,N)$ is the pair 
$$(g,\beta)\circ (f,\alpha) = (g\circ f,\alpha \circ f^*(\beta)\circ \lambda_{f,g})$$
where $\lambda_{f,g}:(gf)^* \cong f^*g^*$ is the usual natural isomorphism of functors.
The higher Grothendieck-Witt groups $GW^{[n]}_i$ and $\GW^{[n]}_i$ are functorial for maps in this category.

More precisely, for a map $f:X \to Y$ of schemes, the functor 
$f^*:\sPerf(Y) \to \sPerf(X)$ is a symmetric monoidal functor between closed symmetric monoidal categories.
As such it is equipped with natural maps
$$\phi_f:f^*[E,F] \to [f^*E,f^*F]$$
such that, given $g:Y \to Z$ and $E,F\in \sPerf(Z)$, the following diagram commutes
$$\xymatrix{
f^*g^*[E,F] \ar[r]^{f^*\phi_g} & f^*[g^*E,g^*F] \ar[r]^{\phi_f} & [f^*g^*E,f^*g^*F] \ar[d]^{[\lambda_{f,g},1]}\\
(gf)^*[E,F] \ar[u]^{\lambda_{f,g}} \ar[r]_{\phi_{gf}} & [(gf)^*E,(gf)^*F] \ar[r]_{[1,\lambda_{f,g}]} & [(gf)^*E,f^*g^*F].}$$
Now, a map $(f,\alpha):(X,L) \to (Y,M)$ in $\widetilde{\Sch}$ defines an exact dg form functor
$$(f,\alpha)^*: \sPerf^M(Y) \to \sPerf^L(X): E \mapsto f^*E$$
with duality compatibility map
$f^*[E,M] \stackrel{\phi_f}{\longrightarrow}[f^*E,f^*M]
\stackrel{[1,\alpha]}{\longrightarrow}[f^*E,L].$
Given composable maps $(f,\alpha)$ and $(g,\beta)$ as above, then $\lambda_{f,g}$ defines a natural isomorphism of dg form functors
$$\lambda_{f,g}: (f,\alpha)^*(g,\beta)^* \cong ((g,\beta)\circ(f,\alpha))^*.$$
Therefore, the diagrams of Grothendieck-Witt and Karoubi-Grothendieck-Witt spectra
$$\xymatrix{GW^{[n]}(Z,N) \ar[r]^{(g,\beta)} \ar[dr]_{(g,\beta)\circ(f,\alpha)} &
GW^{[n]}(Y,M) \ar[d]^{(f,\alpha)} &
\GW^{[n]}(Z,N) \ar[r]^{(g,\beta)} \ar[dr]_{(g,\beta)\circ(f,\alpha)} &
\GW^{[n]}(Y,M) \ar[d]^{(f,\alpha)} \\
&GW^{[n]}(X,L) & & \GW^{[n]}(X,L)}$$
commute up to homotopy.
In particular, the higher Grothendieck-Witt groups $GW^{[n]}_i$ and $\GW^{[n]}_i$ are functors from $\widetilde{\Sch}$ to abelian groups.

Since tensor products of  complexes commute with $f^*$ up to natural isomorphism of dg functors, we have
$$f^* (a\cup b) = f^*(a)\cup f^*(b) \in GW^{[n+m]}_{i+j}(X)$$
for $a\in GW^{[n]}_i(Y)$ and $b\in GW^{[m]}_j(Y)$, and similarly for $\GW$ in place of $GW$.

\begin{remark}[Strict Functoriality]
Sometimes it is convenient to have functors 
$$GW^{[n]}: \widetilde{\Sch} \to \Sp\hspace{2ex}\text{and}\hspace{2ex}
\GW^{[n]}: \widetilde{\Sch} \to \BiSp$$
with values in (bi-) spectra before taking homotopy groups.
The trouble with the functoriality as noted in Remark \ref{rmk:FunctorialityI} is that the natural isomorphism of functors $\lambda_{f,g}:(gf)^*\cong f^*g^*$ is not the identity.
As explained in \cite[Appendix C.4]{FriedlanderSuslin}, if we replace the (somewhat unspecified) category of vector bundles on $X$ with the equivalent small category of big vector bundles on $X$, then we have equality $\lambda_{f,g} = id:(gf)^* =f^*g^*$.
In particular, the resulting assignment $X \mapsto \sPerf^L(X)$ is functorial for maps in $\widetilde{\Sch}$ and we obtain indeed a functor $GW^{[n]}: \widetilde{\Sch} \to \Sp$ and $\GW^{[n]}: \widetilde{\Sch} \to \BiSp$.
\end{remark}

Let $Z\subset X$ be a closed subset with open complement $X-Z$.
A complex $E \in \sPerf(X)$ has support in $Z$ if it is acyclic outside $Z$, that is if for all $x\in X-Z$ the complex $E_x \in \sPerf(O_{X,x})$ is an acyclic complex of $O_{X,x}$-modules.
We denote by $\sPerf_Z^L(X)$ the full dg subcategory with weak equivalences and duality of $\sPerf^L(X)$ of those complexes which have support in $Z$.
The {\em $n$-th shifted Karoubi-Grothendieck-Witt spectrum of $X$ with coefficients in $L$ and support in $Z$} is the Karoubi-Grothendieck Witt spectrum
of $\sPerf_Z^L(X)$,
denoted by $\GW^{[n]}(X\phantom{i}on\phantom{i}Z,\phantom{i}L)$.

The following theorem was proved in \cite{myMV} without the assumption $\frac{1}{2}\in X$.

\begin{theorem}[Localization for vector bundle Grothendieck-Witt groups]
\label{thm:schemeLocn}
Let $X$ be a scheme with an ample family of line-bundles and $\frac{1}{2}\in X$.
Let $Z \subset X$ be a closed subset with quasi-compact open complement $U=X-Z$.
Then for all line-bundles $L$ on $X$ and $n\in \Z$, the following is a homotopy fibration
$$
\GW^{[n]}(X\phantom{i}on\phantom{i}Z,\phantom{i}L) \longrightarrow
\GW^{[n]}(X,\phantom{i}L) \longrightarrow
\GW^{[n]}(U,\phantom{i}L).$$
\end{theorem}

\begin{proof}
By a result of Thomason \cite{TT}, the sequence of dg categories weak equivalences and duality
$$\sPerf^L_Z(X) \to \sPerf^L(X) \to \sPerf^L(U)$$
is Morita exact.
The result follows from the Localization Theorem \ref{thm:locnForKGW}.
\end{proof}

\begin{theorem}[Nisnevich Mayer-Vietoris]
\label{thm:NisnevichMV}
Consider a pull-back diagram of quasi-compact schemes 
$$\xymatrix{V \ar[r] \ar[d] & Y \ar[d]^p \\
U \ar[r]^j & X}$$
in which the map $p$ is \'etale and the map $j$ is an open immersion.
Assume that $X$ has an ample family of line bundles, and $\frac{1}{2}\in X$ (the same is then true for $Y$, $U$ and $V$).
If the map $p: (Y-V)_{red} \to (X-U)_{red}$ is an isomorphism, then 
for all line bundles $L$ on $X$ and all $n\in \Z$
the square of schemes induces a homotopy cartesian diagram of Karoubi-Grothendieck-Witt spectra
$$\xymatrix{\GW^{[n]}(X,L) \ar[r]\ar[d] & \GW^{[n]}(U,L)\ar[d]\\
\GW^{[n]}(Y,L)  \ar[r] & \GW^{[n]}(V,L).
}$$
\end{theorem}

\begin{proof}
Let $Z$ denote the common closed subset $(Y-V)_{red} \cong (X-U)_{red}$ of $X$ and $Y$.
By a result of Thomason \cite[Theorem 2.6.3]{TT}, the map $p$ induces a quasi-equivalence of dg categories
$$p^*:\sPerf_Z(X) \to \sPerf_Z(Y).$$
The result now follows from Theorems \ref{thm:InvarianceKGW} and \ref{thm:schemeLocn}.
\end{proof}

\begin{theorem}[Nisnevich descent]
\label{thm:NisnDescent}
Let $X$ be a noetherian scheme of finite Krull dimension with an ample family of line bundles and $\frac{1}{2}\in X$.
Then for all line bundles $L$ on $X$ and all $n\in \Z$, the map
on global sections of a globally fibrant replacement for the Nisnevich topology of $\GW^{[n]}(\phantom{X},L)$ on the small Nisnevich site on $X$ is a stable equivalence:
$$\GW^{[n]}(X,L) \stackrel{\sim}{\longrightarrow} \bH_{Nis}\left(X,\GW^{[n]}(\phantom{X},L)\right).$$
\end{theorem}

\begin{proof}
This follows from the Nisnevich Mayer-Vietoris Theorem \ref{thm:NisnevichMV} using for instance a spectrum version of \cite[Lemma 1.18]{MorelVoevodsky}.
Alternatively, the same arguments as in \cite[Theorem 10.8]{TT} using Theorem  \ref{thm:NisnevichMV} yield the result.
\end{proof}

\begin{theorem}[Homotopy Invariance]
Let $X$ be a noetherian regular separated scheme with $\frac{1}{2}\in X$.
Then for all line bundles $L$ on $X$ and all $n\in \Z$, the projection 
$p:X\times \bA^1 \to X$ induces a stable equivalence of Grothendieck-Witt spectra
$$GW^{[n]}(X,L) \stackrel{\sim}{\longrightarrow} GW^{[n]}(X\times \bA^1,p^*L).$$
\end{theorem}

\begin{proof}
The analogous statement for $K$-theory and triangular Witt groups hold, by \cite{quillen:higherI} and \cite[Theorem 3.4]{balmer:MV}.
The theorem follows from Theorem \ref{thm:williamsKobal} or from Lemma \ref{lem:KaroubiInd}.
\end{proof}

\begin{theorem}[Mayer-Vietoris for regular blow-ups]
\label{thm:blowup}
Consider a pull-back diagram of schemes 
$$\xymatrix{Y' \ar[r]^j \ar[d]_q & X' \ar[d]^p \\
Y \ar[r]^i & X}$$
in which the map $i$ is a closed immersion and the map $p$ is the blow-up of $X$ along $i$.
Assume that $X$  has an ample family of line bundles, and $\frac{1}{2}\in X$ (the same is then true for $Y$, $X'$ and $Y'$).
If the map  $i$ is a regular embedding of codimension $d$, then 
for all line bundles $L$ on $X$ and all $n\in \Z$
the square of schemes induces a homotopy cartesian diagram of Karoubi-Grothendieck-Witt spectra
$$\xymatrix{\GW^{[n]}(X,L) \ar[r]\ar[d] & \GW^{[n]}(Y,L)\ar[d]\\
\GW^{[n]}(X',L)  \ar[r] & \GW^{[n]}(Y',L).
}$$
\end{theorem}

\begin{proof}
The claim is local in $X$, by Theorem \ref{thm:NisnevichMV}.
Therefore, we may assume $X$ and $Y$ affine and $L=O_X$.
To ease notation, write $D_X$ for the triangulated category $\T\sPerf(X)$, and similarly for $Y$, $X'$, $Y'$.
So, $D_X$ and $D_Y$ are generated as idempotent complete triangulated categories by $O_X$ and $O_Y$, respectively.
Recall (e.g., from \cite[Lemma 1.4]{CHSW}) that the triangle functors
$Lp^*: D_X \to D_{X'}$,  $Lq^*:D_Y\to D_{Y'}$ and $Rj_*Lq^*: D_Y\to D_{X'}$ are fully faithful. 
Write $Lp^* D_X^{(k)}$ and $Rj_*Lq^*D_Y^{(k)}$ for the triangulated categories 
$O_{X'}(k)\otimes Lp^* D_X$ and $O_{X'}(k)\otimes  Rj_*Lq^*D_Y$.

For triangulated subcategories $\A_1,...,\A_n$ of an idempotent complete triangulated category $\A$, we denote by $\langle \A_1,...,\A_n\rangle$ the idempotent complete triangulated subcategory of $\A$ generated by
 $\A_1,...,\A_n$.
From the exact triangle in $D_{X'}$
$$O_{X'}(1) \to O_{X'} \to Rj_*O_{Y'} \to O_{X'}(1)[1],$$
we obtain the equality of triangulated subcategories of $D(X')$
\begin{equation}
\label{eqn:EqTrCats1}
\langle O_{X'}(1)\otimes Lp^*D_X,\ Rj_*Lq^*D_Y\rangle = \langle O_{X'}(1)\otimes Lp^*D_X,\ Lp^*D_X\rangle
\end{equation}
since the generators $O_{X'}(1)$ and $Rj_*O_{Y'}$ of the left hand side are in the right hand category, and the generators $O_{X'}(1)$ and $O_{X'}$ of the right hand side are in the left category.
If we denote by $\vee$ the duality functor $E \mapsto E^{\vee} = \Hom(E,O_{X'})$ in $D_{X'}$, then applying $\vee$ to the equality of triangulated subcategories yields the equality
\begin{equation}
\label{eqn:EqTrCats2}
\langle O_{X'}(-1)\otimes Lp^*D_X,\ (Rj_*Lq^*D_Y)^{\vee} \rangle = \langle O_{X'}(-1)\otimes Lp^*D_X,\ Lp^*D_X\rangle.
\end{equation}
For $c = \lfloor \frac{d-1}{2}\rfloor$ the largest integer $\leq \frac{d-1}{2}$, let
$\A$ be the triangulated subcategory 
$$\A = \left\langle Rj_*Lq^*D_Y^{(-1)},..., Rj_*Lq^*D_Y^{(-c)}\right\rangle$$
of $D_{X'}$.
From the equalities of triangulated categories above,
we obtain the equalities of triangulated categories
$$
\renewcommand\arraystretch{2}
\begin{array}{cl}
  &   \left\langle Lp^*D_X^{(c)},\ Rj_*Lq^*D_Y^{(c-1)},\ Rj_*Lq^*D_Y^{(c-2)},..., Rj_*Lq^*D_Y\right\rangle\\
 \stackrel{\text{(\ref{eqn:EqTrCats1})}}{=} &   \left\langle Lp^*D_X^{(c)},\ Lp^*D_X^{(c-1)},..., Lp^*D_X\right\rangle\\
 \stackrel{\text{(\ref{eqn:EqTrCats2})}}{=} & \left\langle \left( Rj_*Lq^*D_Y^{(-c)}\right)^{\vee},...,\left( Rj_*Lq^*D_Y^{(-1)}\right)^{\vee},\ Lp^*D_X\right\rangle\\
 = &   \left\langle \A^{\vee}, \ Lp^*D_X \right\rangle
\end{array}
$$
From \cite[Lemma 1.2 (1)]{CHSW}, this implies
$$
\renewcommand\arraystretch{2}
\begin{array}{rcl}
 D_{X'} &
  = & O_{X'}(c) \otimes D_{X'}\\
& = &  O_{X'}(c) \otimes \left\langle Lp^*D_X,\ Rj_*Lq^*D_Y^{(-1)},\ Rj_*Lq^*D_Y^{(-2)},..., Rj_*Lq^*D_Y^{(-d+1)}\right\rangle\\
& = &   \left\langle \A^{\vee}, \ Lp^*D_X,\ \A,\ \left[Rj_*Lq^*D_Y^{(-c-1)}\right] \right\rangle
\end{array}
$$
where the last term $[Rj_*Lq^*D_Y^{(-c-1)}]$ denotes $Rj_*Lq^*D_Y^{(-c-1)}$ if $d-1$ is odd and $0$ if $d-1$ is even.
As noted in the proof of \cite[Proposition 1.5]{CHSW}, every map from objects of
$$
\left\langle \A^{\vee}, \ Lp^*D_X \right\rangle = \left\langle Lp^*D_X^{(c)},\ Rj_*Lq^*D_Y^{(c-1)},..., Rj_*Lq^*D_Y\right\rangle$$
to objects of 
$$\left\langle Rj_*Lq^*D_Y^{(-1)},..., Rj_*Lq^*D_Y^{(c-d+1)}\right\rangle = \left\langle \A,\ \left[Rj_*Lq^*D_Y^{(-c-1)}\right] \right\rangle$$
is zero.

Similarly, let $\B$ be the triangulated subcategory 
$$\B = \left\langle Lq^*D_Y^{(-1)},..., Lq^*D_Y^{(-c)}\right\rangle$$
of $D_{Y'}$.
From \cite[Lemma 1.2 (2)]{CHSW}, we have
$$D_{Y'} = \left \langle \B^{\vee},\ Lq^*D_Y,\ \B,\ \left[Lq^*D_Y^{(-c-1)}\right]\right\rangle$$
and every map from objects of $\left\langle \B^{\vee}, \ Lq^*D_Y \right\rangle$ to objects of $\langle \B,\ [Lq^*D_Y^{(-c-1)}]\rangle$ are zero.
From \cite[Proposition 1.5]{CHSW} and the fact that $Lj^*$ commutes with dualities (up to natural isomorphism), we see that $Lj^*$ sends $\langle Lp^*D_X, \A\rangle$, $Lp^*D_X$ and $\langle \A^{\vee}, Lp^*D_X\rangle $ to $\langle Lq^*D_Y \B\rangle$, $Lq^*D_Y$ and $\langle \B^{\vee}, Lq^*D_Y\rangle$.
So, we have a map of filtrations of triangulated categories and dualities
\begin{equation}
\label{eqn:Filt1}
\xymatrix{
D_X \ar[d]_{Li^*} \ar[r]^{Lp^*}_{\sim} & Lp^*D_X \hspace{1ex} \xyhookrightarrow[r] \ar[d]_{Lj^*} & \langle \A^{\vee}, Lp^*D_X, \A \rangle \ar[d]_{Lj^*} \hspace{1ex} \xyhookrightarrow[r] & 
D_{X'} \ar[d]_{Lj^*} \\
D_Y  \ar[r]^{Lq^*}_{\sim} & Lq^*D_Y \hspace{1ex} \xyhookrightarrow[r] & \langle \B^{\vee}, Lq^*D_X, \B \rangle \hspace{1ex} \xyhookrightarrow[r] & D_{Y'}.
}
\end{equation}
To this diagram corresponds a diagram of dg categories with weak equivalences and duality
\begin{equation}
\label{eqn:Filt2}
\xymatrix{
\sPerf(X) \ar[d]_{i^*} \ar[r]^{p^*} & \sPerf^0(X') \hspace{1ex} \xyhookrightarrow[r] \ar[d]_{j^*} & \sPerf^1(X') \ar[d]_{j^*} \hspace{1ex} \xyhookrightarrow[r] & 
\sPerf(X') \ar[d]_{j^*} \\
\sPerf(Y)  \ar[r]^{q^*} & \sPerf^0(Y') \hspace{1ex} \xyhookrightarrow[r] & \sPerf^1(Y')  \hspace{1ex} \xyhookrightarrow[r] & 
\sPerf(Y')
 }
\end{equation}
where $\sPerf^i(X') \subset \sPerf(X')$ and $\sPerf^i(Y') \subset \sPerf(Y')$ are the full dg subcategories whose objects are in the triangulated subcategories  of $D_{X'}$ and $D_{Y'}$ as indicated in diagram (\ref{eqn:Filt1}), so that
the triangulated category diagram obtained from (\ref{eqn:Filt2}) is precisely the diagram (\ref{eqn:Filt1}).

I claim that all squares in (\ref{eqn:Filt2}) induce homotopy cartesian squares of $\GW$-spectra.
As a composition of homotopy cartesian squares, the outer diagram will then induce a homotopy cartesian of $\GW$-spectra as claimed in the theorem.
The left square of (\ref{eqn:Filt2}) induces a homotopy cartesian square of $\GW$-spectra, by the Invariance Theorem \ref{thm:InvarianceKGW} applied to the two horizontal arrows.

For the middle square of (\ref{eqn:Filt2}), the Localization Theorem \ref{thm:locnForKGW} identifies the horizontal homotopy cofibres as the $\GW$-spectra of the dg categories with weak equivalences and duality $(\sPerf^1(X'), v)$ and $(\sPerf^1(Y'), v)$
where the weak equivalences are the maps whose cones are in $Lp^*D_X$ and 
$Lq^*D_Y$.
Let $(\sPerf^1_0(X'),v) \subset (\sPerf^1(X'),v)$ and $(\sPerf^1_0(Y'),v)\subset (\sPerf^1_0(Y'),v)$ be the full dg subcategories with weak equivalences and duality corresponding to the triangulated categories
$\langle Lp^*D_X, \A \rangle  /Lp^*D_X$ and 
$\langle Lq^*D_Y, \B \rangle  /Lq^*D_Y$.
By the Additivity Theorem \ref{prop:AddKGW}, 
we have 
$$\GW(\sPerf^1(X'),v) \simeq \bK(\sPerf^1_0(X'),v), \hspace{2ex}\GW(\sPerf^1(Y'),v) \simeq \bK(\sPerf^1_0(Y'),v).$$
By \cite[Proposition 1.5]{CHSW}, the \red{following induced map is a stable equivalence}
$j^*: \bK(\sPerf^1_0(X'),v) \to \bK(\sPerf^1_0(Y'),v)$.
Therefore, the middle square in (\ref{eqn:Filt2}) induces a homotopy cartesian square of $\GW$-spectra.

The last square in (\ref{eqn:Filt2}) induces a homotopy cartesian square of $\GW$-spectra, by the Localization and Invariance Theorems \ref{thm:locnForKGW} and \ref{thm:InvarianceKGW}, since 
$Lj^*$ induces an equivalence after taking horizontal triangulated quotient categories, by \cite[Proposition 1.5]{CHSW}.
\end{proof}

\subsection{Projective line-bundle formula}
Write $\pp^1=\Proj(k[S,T])$ for the projective line $\pp^1_k$ over $k$ where $k$ is our base ring, e.g., $k=\Z[1/2]$.
Consider the category $\Vect(\pp^1)$ of vector bundles on $\pp^1$ equipped with the usual duality $\Hom(\phantom{A},O_{\pp^1})$.
Recall (Section \ref{subsec:CatsWDual}) that this makes the category $\Fun([1],\Vect(\pp^1))$ of morphisms in $\Vect(\pp^1)$ into a category with duality.
In that category, we have an object $S:O_{\pp^1}(-1) \to O_{\pp^1}$ (multiplication by $S \in k[S,T]$) equipped with the symmetric form $\tilde{\beta}$
$$\xymatrix{\ar@{}[d]_{\tilde{\beta}:} &
O_{\pp^1}(-1) \ar[r]^{S} \ar[d]_T & O_{\pp^1}\ar[d]^T\\
& O_{\pp^1} \ar[r]_S & O_{\pp^1}(1)}$$
(multiplication by $T$).
We consider $\tilde{\beta}$ as a symmetric form in the dg category $\Fun([1],\Ch^b\Vect(\pp^1))$ of morphisms of  complexes via the embedding $\Vect(\pp^1) \to \Ch^b\Vect(\pp^1)$ as complexes concentrated in degree $0$.
Its image under the cone dg form functor 
$$\Cone: \Fun([1],\Ch^b\Vect(\pp^1))^{[0]} \longrightarrow \left(\Ch^b\Vect(\pp^1)\right)^{[1]}$$
is denoted by $\beta$.
In view of the exact sequence
$$O_{\pp^1}(-1) \stackrel{\left(\begin{smallmatrix}S\\ T\end{smallmatrix}\right)}{\longrightarrow} O_{\pp^1}\oplus O_{\pp^1} \stackrel{(T,-S)}{\longrightarrow}  O_{\pp^1}(1),$$ the form on $\beta$ is a quasi-isomorphism of complexes of sheaves on $\pp^1$.
Therefore, $\beta$ is non-degenerate and defines an element
$$\beta = \Cone(\tilde{\beta})\in GW^{[1]}_0(\pp^1).$$

\begin{theorem}
\label{thm:projLineBld}
Let $X$ be a scheme with an ample family of line bundles and $\frac{1}{2}\in X$, and denote by $p:\pp^1_X \to X$ the structure map of the projective line over $X$.
Then for all line bundles $L$ on $X$ and all $n\in \Z$, the following are natural stable equivalences of (bi-) spectra
$$
\renewcommand\arraystretch{1.5}
\begin{array}{ccc}
GW^{[n]}(X,L) \oplus GW^{[n-1]}(X,L) & \stackrel{\sim}{\longrightarrow} & GW^{[n]}(\pp^1_X,p^*L)\\
\GW^{[n]}(X,L) \oplus \GW^{[n-1]}(X,L) & \stackrel{\sim}{\longrightarrow} & \GW^{[n]}(\pp^1_X,p^*L)\\
 (x,y) & \mapsto & p^*(x)+ \beta\cup p^*(y).
\end{array}$$
\end{theorem}

\begin{proof}
The triangle functor $p^*: \T\sPerf(X) \to \T\sPerf (\pp^1_X)$ is fully faithful, and $\beta$ induces an equivalence $\beta\otimes: \T\sPerf(X) \to \T\sPerf (\pp^1_X)/p^*\T\sPerf(X)$ of triangulated categories.
This is classical; see for instance \cite{Thomason:proj}, \cite[Theorem 3.5.1]{mySedano}. 
Denote by $w$ the set of morphisms in $\sPerf (\pp^1_X)$ which are isomorphisms in $\T\sPerf (\pp^1_X)/p^*\T\sPerf(X)$.
By Theorems \ref{thm:LcnConn1} and \ref{thm:locnForKGW}, the sequence
$$(\sPerf(X),\quis) \stackrel{p^*}{\longrightarrow} (\sPerf (\pp^1_X),\quis) \longrightarrow (\sPerf (\pp^1_X), w)$$
induces homotopy fibrations of $GW^{[n]}$ and $\GW^{[n]}$-spectra.
By Theorems \ref{thm:Invariance} and \ref{thm:InvarianceKGW}, these sequences split via the exact dg form functors
$$(\sPerf(X),\quis) \stackrel{\beta\otimes}{\longrightarrow} (\sPerf (\pp^1_X),\quis) \longrightarrow (\sPerf (\pp^1_X),w)$$ whose composition induces an equivalence of associated triangulated categories, and hence $GW$ and $\GW$-equivalences.
\end{proof}

\begin{remark}
\label{rmk:otherProjBdlFormulas}
Walter proved in \cite{Walter:proj} projective bundle formulas for the zero-th Grothendieck Witt groups $GW^{[n]}_0(\pp_X(\E))$ for arbitrary vector bundles $\E$ over $X$. 
His results immediately generalize to the higher Grothendieck-Witt groups
$GW^{[n]}_i(\pp_X(\E))$ and $\GW^{[n]}_i(\pp_X(\E))$.
\end{remark}

\begin{lemma}
\label{lem:BetaZero}
Let $k$ be a commutative ring with $\frac{1}{2}\in k$.
Then the element $\beta \in GW^{[1]}_0(\pp^1_k)$ is zero when restricted to the principal open subsets $D_+(S)=\Spec(k[T/S])$ and $D_+(T) = \Spec(k[S/T])$ of $\pp^1_k = \Proj(k[S,T])$.
\end{lemma}

\begin{proof}
It suffices to prove the claim for $k=\Z[1/2]$ as $\beta$ over $k$ is obtained from $\beta$ over $\Z[1/2]$ via base change along $\Z[1/2] \to k$.
In any case, it is clear that $\beta_{|D_+(S)} = 0$ as it is a form on a contractible complex.
For the other restriction, write $U=S/T$, so $D_+(T) = \Spec k[U]$.
Note that $\beta_{|D_+(T)}$ is supported on the closed subset $U=0$ of $\Spec k[U]$, that is, it is in the image of the left horizontal map
in the following diagram
$$\xymatrix{ 
GW^{[1]}_0(k[U]\phantom{i}on\phantom{i}U=0) \ar[r] & GW^{[1]}_0(k[U]) \ar[d]^{\cong}_{U=1} \ar[r] & GW^{[1]}_0(k[U,U^{-1}]) \ar[dl]^{U=1} \\
& GW^{[1]}_0(k). &
}$$
The top row of the diagram is exact, by the Localization Theorem \ref{thm:schemeLocn}, and the middle vertical map is an isomorphism, by Homotopy Invariance (recall that $k=\Z[1/2]$ is regular).
Therefore, the right horizontal map is injective, and the left horizontal map is the $0$ map.
In particular, $\beta_{|D_+(T)}=0$
\end{proof}

In the following theorem, denote by $X[T^+]= D_+(S)\cong \bA^1_X$ and $X[T^-] = D_+(T)\cong \bA^1_X$ the two principal open subschemes of the standard covering of $\pp^1_X= \red{X}\times_k \Proj(k[S,T])$ given by $S\neq  0$ and $T\neq 0$, and let $X[T^{\pm}]$ be their intersection.
For a line bundle $L$ on $X$ we also denote by $L$ the line bundle on 
$X[T^{+}]$, $X[T^{-}]$ and $X[T^{\pm}]$ induced from $L$ via the pull back along the projections $p: X[T^{+}],\ X[T^{-}],\ X[T^{\pm}] \to X$.

\begin{theorem}[Bass' Fundamental Theorem]
\label{thm:Bassfund}
Let $X$ be a scheme with an ample family of line bundles and $\frac{1}{2}\in X$.
Then for all line bundles $L$ on $X$ and all integers $n, i \in \Z$, there are natural exact sequences
$$
\renewcommand\arraystretch{1.5}
\begin{array}{lcccl}
0 \longrightarrow \GW^{[n]}_i(X,L) 
& \stackrel{(p^*,p^*)}{\longrightarrow}
& \GW^{[n]}_i(X[T^+],L) \oplus \GW^{[n]}_i(X[T^-],L) &\\
& \longrightarrow & 
\GW^{[n]}_i(X[T^{\pm}],L) &\\
 & \stackrel{\delta}{\longrightarrow} & \GW^{[n-1]}_{i-1}(X,L) & \longrightarrow 0.
\end{array}$$
\end{theorem}

\begin{proof}
To ease notation we will omit the line bundle in our notation.
From the Nisnevich Mayer-Vietoris Theorem \ref{thm:NisnevichMV} applied to the open covering $X[T^+] \cup X[T^-] = \pp^1_X$ and the Projective Line Bundle Theorem \ref{thm:projLineBld} we obtain a homotopy cartesian square
$$\xymatrix{
\GW^{[n]}(X) \oplus \GW^{[n-1]}(X) \ar[rr]^{\hspace{4ex}(p^*,\beta\cup p^*)} \ar[d]_{(p^*,\beta\cup p^*)} & &
\GW^{[n]}(X[T^+]) \ar[d]\\
\GW^{[n]}(X[T^-]) \ar[rr] && \GW^{[n]}(X[T^{\pm}]).
}$$
By Lemma \ref{lem:BetaZero}, the element $\beta \in GW^{[1]}_0(\pp_{\Z[1/2]})$ is zero in 
$GW^{[1]}_0({\Z[1/2][T^+]})$ and $GW^{[1]}_0({\Z[1/2][T^-]})$.
Therefore,
the two maps $\beta\cup p^*: \GW^{[n-1]}(X) \to \GW^{[n]}(X[T^+])$ and 
$\beta\cup p^*: \GW^{[n-1]}(X) \to \GW^{[n]}(X[T^-])$ are null-homotopic.
Since the two maps $p^*: \GW^{[n]}(X) \to \GW^{[n]}(X[T^+])$ and 
$p^*: \GW^{[n]}(X) \to \GW^{[n]}(X[T^-])$ are (split) injective, the long exact sequence of homotopy groups associated with the homotopy cartesian square decomposes into the exact sequences as in the theorem.
\end{proof}

\begin{remark}
\label{rmk:ContractedFunctors}
The exact sequences in Theorem \ref{thm:Bassfund} are split.
Splitting of the first non-trivial map in the sequence is clear, and the last non-trivial map is split via the cup-product with $[T] \in GW^{[1]}_1({\Z[1/2][T^{\pm}]})$.

As in Bass' formalism of contracted functors, 
it follows from Theorem \ref{thm:Bassfund} that one can define the functors $\GW_i^{[n]}$ on schemes for $i<0$ and all $n\in \Z$ inductively by starting with the functors $\GW_0^{[n]}$, $n\in \Z$.
\end{remark}

\subsection{Coherent Grothendieck-Witt groups}

There are many results about the triangular Witt groups of coherent sheaves, notably due to Stefan Gille. 
With our results in Sections \ref{sec:PeriodInvLocn} and \ref{section:LandTate}, they immediately generalize to higher Grothendieck-Witt groups.
In what follows, we will state and prove a few of them.

Let $X$ be a noetherian scheme, and denote by 
$\Qcoh^b(X)$ and $\Coh^b(X)$ the dg categories of bounded complexes of quasi-coherent and coherent $O_X$-modules.
These are closed symmetric monoidal categories under tensor product and internal homomorphism objects of quasi-coherent sheaves.
In particular, any object $A$ of $\Qcoh^b(X)$ defines a duality functor 
$$\sharp_A:\Qcoh^b(X)^{op} \to \Qcoh^b(X): E \mapsto Hom^{\bullet}_{O_X}(E,A)$$ with double dual identification $\can^A:E \to E^{\sharp_A\sharp_A}$ given by the formula
$$\can^A_E(x)(f) = (-1)^{|x||f|}f(x).$$
Denote by $\Qcoh^b_c(X)\subset \Qcoh^b(X)$ the full dg subcategory of those bounded complexes of quasi-coherent $O_X$-modules which have coherent cohomology.
Recall that a dualizing complex on $X$ is a bounded complex $I^{\bullet}$ of injective quasi-coherent $O_X$-modules such that for every $E \in \Qcoh^b_c(X)$, the double dual identification 
$\can^{I^{\bullet}}_E:E \to E^{\sharp_{I^{\bullet}}\sharp_{I^{\bullet}}}$
is a quasi-isomorphism.
By \cite[p. 258]{hartshorne:dual}, this only needs to be checked for $E=O_X$.

If $X$ is a noetherian scheme with a dualizing complex $I^{\bullet}$, we have a dg category with weak equivalences and duality
\begin{equation}
\label{eqn:dgCatCoh}
(\Qcoh^b_c(X),\quis,\sharp_{I^{\bullet}},\can^{I^{\bullet}}).
\end{equation}

\begin{definition}
\label{dfn:CohGW}
Let $X$ be a noetherian scheme over $\Z[1/2]$ with dualizing complex $I^{\bullet}$.
The $n$-th shifted {\em coherent Grothendieck-Witt spectrum} of $X$ with coefficients in $I^{\bullet}$ is the $n$-th shifted Grothendieck-Witt spectrum 
$$GW^{[n]}(X,I^{\bullet})$$
of  the dg category with weak equivalences and duality (\ref{eqn:dgCatCoh}).
The {\em $n$-th shifted coherent Grothendieck-Witt groups} of $X$ with coefficients in $I^{\bullet}$ are the homotopy groups of $GW^{[n]}(X,I^{\bullet})$:
$$GW^{[n]}_i(X,I^{\bullet}) = \pi_iGW^{[n]}(X,I^{\bullet}).$$
As usual, if $n=0$, we omit the label corresponding to $n$.

\end{definition}

\begin{remark}
Strictly speaking, the category $\Qcoh^b_c(X)$ is not small (not even essentially small), in general.
To obtain an honest spectrum, one would have to replace the category $\Qcoh^b_c(X)$ by a small full dg subcategory closed under the duality such that the inclusion into $\Qcoh^b_c(X)$ induces an equivalence of associated triangulated categories.
It is easy to see that one can always do that,
one only needs to bound the size of the quasi-coherent sheaves involved by a sufficiently large cardinal.
By Theorem \ref{thm:Invariance}, the choice of such a subcategory does not matter.
\end{remark}

\begin{remark}
As in Definition \ref{dfn:CohGW}, one can define a coherent Karoubi-Grothen-dieck-Witt spectrum
$\GW^{[n]}(X,I^{\bullet})$
as the $\GW^{[n]}$-spectrum associated with (\ref{eqn:dgCatCoh}).
Then the comparison map
$$GW^{[n]}(X,I^{\bullet}) \to \GW^{[n]}(X,I^{\bullet})$$
is a stable equivalence, by Theorem \ref{thm:KbKGWKGW}, because the negative $K$-groups of $\Qcoh^b_c(X)$ vanish.
This is because the inclusion $\Coh^b(X)\subset \Qcoh_c^b(X)$ induces an equivalence of associated triangulated categories, and $\Coh^b(X)$ has trivial negative $K$-groups, by \cite[Theorem 7]{mynegK}.
So, there is no need to develop a $\GW$-theory for coherent sheaves.
\end{remark}

Let $X$ be a regular noetherian separated scheme.
Then any coherent $O_X$-module has a finite injective resolution.
Any finite injective resolution $\rho: O_X \to I^{\bullet}$ of the structure sheaf $O_X$ defines a dualizing complex $I^{\bullet}$ on $X$.
Moreover, the inclusion 
\begin{equation}
\label{eqn:PerfIntoQcohforPoincare}
\left(\sPerf(X),\quis,\sharp_{O_X},\can^{O_X}\right) \longrightarrow \left(\Qcoh^b_c(X),\quis,\sharp_{I^{\bullet}},\can^{I^{\bullet}}\right)
\end{equation}
defines an exact dg form functor with duality compatibility map
$$Hom^{\bullet}_{O_X}(E,O_X) \to Hom^{\bullet}_{O_X}(E,I^{\bullet}):f \mapsto \rho\circ f.$$

\begin{theorem}[Poincar\'e duality]
Let $X$ be a regular noetherian separated scheme with $\frac{1}{2}\in X$, and let
$\rho: O_X \to I^{\bullet}$ be a finite injective resolution of the structure sheaf $O_X$.
Then the exact dg form functor 
(\ref{eqn:PerfIntoQcohforPoincare}) induces a stable equivalence of Grothendieck-Witt spectra
$$GW^{[n]}(X) \stackrel{\sim}{\longrightarrow} GW^{[n]}(X,I^{\bullet}).$$
\end{theorem}

\begin{proof}
Recall that for a regular noetherian separated scheme $X$, the inclusion 
$\sPerf(X) \subset \Coh^b(X)$ induces an equivalence of associated triangulated categories. This is classical. For a proof, see for instance \cite[Theorem 3.3.5]{mySedano}.
Since for any noetherian scheme, the inclusion
$\Coh^b(X) \subset \Qcoh^b_c(X)$ induces an equivalence of associated triangulated categories, the result follows from the Invariance Theorem \ref{thm:Invariance}.
\end{proof}

Let $X$ be a noetherian scheme with dualizing complex $I^{\bullet}$.
Let $i:Z\hookrightarrow X$ be a closed immersion with open complement $j:U=X-Z \hookrightarrow X$.
Then $j^*I^{\bullet}$ is a dualizing complex on $U$, and we obtain an exact dg form functor
\begin{equation}
\label{eqn:CohXtoU}
j^*: (\Qcoh^b_c(X),\quis,\sharp_{I^{\bullet}},\can) \longrightarrow (\Qcoh^b_c(Y),\quis,\sharp_{j^*I^{\bullet}},\can)
\end{equation}
with duality compatibility map the usual isomorphism
$$j^*Hom^{\bullet}_{O_X}(E,O_X) \stackrel{\cong}{\longrightarrow} Hom^{\bullet}_{O_U}(j^*E,j^*I^{\bullet}).$$
Similarly, the complex of quasi-coherent $O_Z$-modules
$$i^{\flat}I^{\bullet} = Hom^{\bullet}_{O_X}(i_*O_Z,I^{\bullet})$$
is a dualizing complex for $Z$ \cite[p. 260]{hartshorne:dual}, and we obtain an exact dg form functor
\begin{equation}
\label{eqn:CohZtoX}
i_*: (\Qcoh^b_c(Z),\quis,\sharp_{i^{\flat}I^{\bullet}},\can) \longrightarrow (\Qcoh^b_c(X),\quis,\sharp_{I^{\bullet}},\can)
\end{equation}
with duality compatibility map the isomorphism
$$i_*Hom^{\bullet}_{O_Z}(E,i^{\flat}I^{\bullet}) = 
Hom^{\bullet}_{O_X}(i_*E, Hom^{\bullet}_{O_X}(i_*O_Z,I^{\bullet}))
\stackrel{\cong}{\longrightarrow} 
Hom^{\bullet}_{O_X}(i_*E, I^{\bullet})
$$
induced by the map $O_X \to i_*O_Z$.

\begin{theorem}[Localization for coherent Grothendieck-Witt groups]
\label{thm:LocnCohGW}
Let $X$ be a noetherian scheme with dualizing complex $I^{\bullet}$, and
let $i:Z\hookrightarrow X$ be a closed immersion with open complement $j:U=X-Z \hookrightarrow X$.
If $\frac{1}{2}\in X$, then for all $n\in \Z$, the exact dg form functors 
(\ref{eqn:CohXtoU}) and (\ref{eqn:CohZtoX}) induce a homotopy fibration of coherent Grothendieck-Witt spectra
$$GW^{[n]}(Z,i^{\flat}I^{\bullet}) \stackrel{i_*}{\longrightarrow} GW^{[n]}(X,I^{\bullet}) \stackrel{j^*}{\longrightarrow} GW^{[n]}(U,j^*I^{\bullet}).$$
\end{theorem}

\begin{proof}
Denote by $\Qcoh^b_c(X\phantom{i}on\phantom{i}Z)$ the full dg subcategory of 
$\Qcoh^b_c(X)$ consisting of those complexes which are acyclic when restricted to $U$.
It inherits the structure of a dg category with weak equivalences and duality from (\ref{eqn:dgCatCoh}).
The sequence
$$(\Qcoh^b_c(X\phantom{i}on\phantom{i}Z), \quis) \longrightarrow (\Qcoh^b_c(X),\quis) \longrightarrow (\Qcoh^b_c(U),\quis)$$
is Morita exact. 
See for instance \cite[Theorem 3.3.2]{mySedano}.
By the Localization Theorem \ref{thm:LcnConn1}, the sequence induces a homotopy fibration of $GW^{[n]}$-spectra.

The exact dg form functor (\ref{eqn:CohZtoX}) factors through $\Qcoh^b_c(X\phantom{i}on\phantom{i}Z)$, and we obtain the exact dg form functor
$$i_*:(\Qcoh^b_c(Z),\quis,\sharp_{i^{\flat}I^{\bullet}},\can) \longrightarrow (\Qcoh^b_c(X\phantom{i}on\phantom{i}Z),\quis,\sharp_{I^{\bullet}},\can).
$$
This functor induces isomorphisms on $K$-groups and triangular Witt groups, by the results of \cite{quillen:higherI} and \cite[Theorem 3.2]{Gille:genlDevissage}.
By the Karoubi Induction Lemma \ref{lem:KaroubiInd} or by Theorem \ref{thm:williamsKobal}, it induces an equivalence of $GW^{[n]}$-spectra.
\end{proof}

\red{
We finish the section with an application to $L$-theory of schemes in characteristic zero.
In the following theorem we write $\bL(X)$ for the stabilized $L$-theory spectrum (Definition \ref{dfn:stableWittThSp}) of the dg category with weak equivalences and duality (\ref{eqn:PerfL})
where $A=O_X$.

\begin{theorem}[Cdh-descent for $L$-theory]
Let $k$ be a field of characteristic zero.
Then the stabilized $L$-theory functor $X \mapsto \bL(X)$ satisfies cdh-descent on the category 
of finite type separable $k$-schemes that have an ample family of line bundles.
\end{theorem}

\begin{proof}
We use the criterion given in \cite[Theorem 3.12]{CHSW}.
The functor $\bL$ satisfies Nisnevich descent and has the Mayer-Vietoris property for blow-ups along regularly embedded centers since it is a filtered homotopy colimit of functors 
$\GW^{[n]}$ which have these properties (Theorems \ref{thm:NisnevichMV} and \ref{thm:blowup}).
The functor $\bL$ satisfies excision and is invariant under infinitesimal extension, by \cite{Karoubi:stabilized}. 
\end{proof}
}

\section{Bott-periodicity}
\label{section:Bott}

Let $A$ be a topological ring (with involution) such that 
\begin{itemize}
\item[(*)]
the group of units $A^* \subset A$ is open in $A$ and the map $A^* \to A^*:a\mapsto a^{-1}$ is continuous.
\end{itemize}
For instance $A$ could be a Banach algebra with involution.
For a topological space $X$, the space of continuous maps $A(X)=\Map(X,A)$ in $\Top$ is a topological ring (with involution) with point-wise addition and multiplication (and involution).
If $X$ is compact and $A$ satisfies (*), then $A(X)$ trivially satisfies (*).
Moreover, if $A$ satisfies (*), then so does $M_n(A)$ for all $n\geq 1$ \cite[Corollary 1.2]{Swan:TopExs}.
Denote by $\Top_{\cpt}$ the category of compact topological spaces, and by $\Ab$ the category of abelian groups.

The following lemma is well-known (at least in the $K_0$-case).
We give a proof at the end of this section.

\begin{lemma}
\label{lem:topHtpyInv}
Let $A$ be a topological ring (with involution if appropriate) satisfying (*).
Then the following functors are homotopy invariant,
$$\Top_{\cpt} \to \Ab:X \mapsto K_0(A(X)), GW^{[n]}_0(A(X)),$$
that is, they send homotopic maps to the same map.
In particular, if $X$ is a contractible compact topological space, then the map of algebras $A \to A(X)$ induced by $X \to \pt$ yields isomorphisms
$$K_0(A) \cong K_0(A(X)),$$
$$GW^{[n]}_0(A) \cong GW^{[n]}_0(A(X)).$$
\end{lemma}

Let $A$ be a Banach algebra (with involution).
Let $\Delta_{top}^n$ be the standard topological $n$ simplex.
This is a contractible compact topological space.
Varying $n$ we have a cosimplicial space $n \mapsto \Delta_{top}^n$, hence a simplicial Banach algebra $A\Delta_{top}^*$.
We set $$K_{top}(A) = |K(A\Delta_{top}^*)|$$
$$GW^{[n]}_{top}(A) = |GW^{[n]}(A\Delta_{top}^*)|.$$
These are symmetric spectra, and module spectra over $GW(\rr)$.
Recall that for $\eps = \pm 1$, we may write $_{\eps}GW(A)$ instead of $GW^{\eps-1+4k}(A)$. 

For a spectrum $Z$, write $\Omega^{\infty}Z \in \Top$ for the infinite loop space associated with $Z$, that is, the zero space $\tilde{Z}_0$ of a functorial stable equivalence $Z \to \tilde{Z}$ into an $\Omega$-spectrum $\tilde{Z}$.
In particular, $\pi_iZ=\pi_i\Omega^{\infty}Z$ for $i\geq 0$.

\begin{proposition}
There are canonical homotopy fibrations of spaces
$$BGl^{top}A \to \Omega^{\infty}K_{top}(A) \to K_0(A)$$
$$B_{\eps}O_{\infty,\infty}^{top}(A) \to \Omega^{\infty}{_{\eps}GW}_{top}(A) \to {_{\eps}}GW_0(A),$$
where $G^{top}$ denotes the group $G$ with its usual Euclidean topology.
In particular, $\Omega^{\infty}K_{top}(A)$ and $\Omega^{\infty}{_{\eps}G}W_{top}(A)$ are the usual topological $K$-theory and $\eps$-hermitian $K$-theory spaces of $A$.
\end{proposition}

\begin{proof}
We only explain the hermitian $K$-theory case, the $K$-theory case being similar.
By Lemma \ref{lem:topHtpyInv} and \ref{cor:pi0RRn}, the simplicial groups
$$i \mapsto \pi_0GW^{[n]}(A\Delta^i)$$
are constant.
Therefore, the Bousfield-Friedlander Theorem \cite{BF:gamma} implies that the following canonical map is an equivalence
$$|\Omega^{\infty}GW^{[n]}(A\Delta_{top}^*)| \stackrel{\sim}{\longrightarrow}  \Omega^{\infty}|GW^{[n]}(A\Delta_{top}^*)| = \Omega^{\infty}GW^{[n]}_{top}(A).$$
By Corollary \ref{cor:+=Q}, Lemma \ref{lem:topHtpyInv} and the Bousfield-Friedlander Theorem \cite{BF:gamma}, the homotopy fibre of 
$\Omega^{\infty}{_{\eps}G}W_{top}(A) \to {_{\eps}}GW_0(A)$ is therefore $|B_{\eps}O_{\infty,\infty}(A\Delta^*)^+|$.
Consider the zigzag of maps
$$
\renewcommand\arraystretch{2}
\begin{array}{c}
 B_{\eps}O^{top}_{\infty,\infty}(A)  \longleftarrow \hspace{60ex}\phantom{1} \\
 B |\Sing_*{_{\eps}O}^{top}_{\infty,\infty}(A)| = |B\Sing_*{_{\eps}O}^{top}_{\infty,\infty}(A)| = |B_{\eps}O_{\infty,\infty}(A\Delta^*_{top})| \\
\phantom{1} \hspace{50ex}\longrightarrow   |B_{\eps}O_{\infty,\infty}(A\Delta_{top}^*)^+|.
\end{array}
$$
For a topological group $G$, the simplicial space $BG$ is good
in the sense of \cite{Segal:Gamma} if $G$ is well-pointed (at $1\in G$).
Thus, both maps in the diagram are maps of good simplicial spaces which are
degree-wise homology isomorphisms.
Hence, both maps are homology isomorphisms (after realization).
Since all spaces in the diagram are nilpotent spaces (in fact $H$-spaces), the homology isomorphisms are in fact weak equivalences.
\end{proof}

\begin{theorem}
\label{thm:BanachfundThmPeridicity}
Let $A$ be a Banach algebra with involution.
Then the sequence
$$GW_{top}^{[n]}(A) \stackrel{F}{\longrightarrow} K_{top}(A) \stackrel{H}{\longrightarrow} GW_{top}^{[n+1]}(A) \stackrel{\eta\cup}{\longrightarrow}S^1\wedge GW^{[n]}_{top}(A)$$
is an exact triangle in the homotopy category of spectra.
\end{theorem}

\begin{proof}
This follows from Theorem \ref{thm:PeriodicityExTriangle}.
\end{proof}

\begin{remark}
From our definitions, Lemma \ref{lem:topHtpyInv}, and Theorem \ref{thm:BanachfundThmPeridicity}, we have for $i<0$ the following:
$\pi_0GW^{[n]}_{top}(A)=GW_0^{[n]}(A)$, $\pi_i K_{top}(A) = 0$, and $\pi_iGW^{[n]}_{top}(A)=GW_i^{[n]}(A) = W^{n-i}(A)$. 
\end{remark}

\subsection{Classical Bott periodicity}
\label{subsec:ClassicalBottPeriod}
As in the proof of Theorem \ref{thm:KaroubiFundThm}, Theorem \ref{thm:BanachfundThmPeridicity} implies the topological version of Karoubi's fundamental theorem
\begin{equation}
\label{eqn:topKarPeriod}
_{-\eps}V_{top}(A) \simeq \Omega {_{\eps}U}_{top}(A)
\end{equation}
where $V_{top}$ and $U_{top}$ are the homotopy fibres of the topological forgetful and hyperbolic functors $F: \Omega^{\infty}{_{\eps}GW}_{top}(A) \to \Omega^{\infty}K_{top}(A)$ and $H: \Omega^{\infty}K_{top}(A) \to \Omega^{\infty}{_{\eps}GW}_{top}(A)$.

Consider $\rr$ and $\cc$ as Banach algebras with trivial involution, and consider the quaternions $\hh$ as Banach algebra with its usual involution sending $i,j,k$ to $-i,-j,-k$, respectively.
Then we obtain
$$
\renewcommand\arraystretch{1.5}
\begin{array}{ccc}
V(\rr) \sim \Z\times BO & V(\cc) \sim U/O & V(\hh) \sim \Z \times BSp \\
U(\rr) \sim O & U(\cc) \sim O/U & U(\hh) \sim Sp\\
{_-}V(\rr) \sim O/U & {_-}V(\cc) \sim U/Sp & {_-}V(\hh) \sim \Z \times Sp/U \\
{_-}U(\rr) \sim U/O & {_-}U(\cc) \sim Sp/U & {_-}U(\hh) \sim U/Sp
\end{array}$$
in view of the classical homotopy equivalences
$$
\renewcommand\arraystretch{1.5}
\begin{array}{ccc}
O_n \stackrel{\sim}{\hookrightarrow} Gl_n(\rr) & U_n \stackrel{\sim}{\hookrightarrow} Gl_n(\cc) & Sp_n \stackrel{\sim}{\hookrightarrow} Gl_n(\hh) \\
O_m\times O_n \stackrel{\sim}{\hookrightarrow} O_{m,n}(\rr) & O_{m+n} \stackrel{\sim}{\hookrightarrow} O_{m,n}(\cc) & Sp_m\times Sp_n \stackrel{\sim}{\hookrightarrow} O_{m,n}(\hh) \\
U_n \stackrel{\sim}{\hookrightarrow} Sp_n(\rr) & Sp_n \stackrel{\sim}{\hookrightarrow} Sp_n(\cc) & U_n \stackrel{\sim}{\hookrightarrow} Sp_n(\hh) 
\end{array}$$
induced by the polar decompositions of $Gl_n(\rr)$, $Gl_n(\cc)$ and $Gl_n(\hh)$.
Using the topological version of Karoubi's Fundamental Theorem (\ref{eqn:topKarPeriod}), we obtain the following table of homotopy equivalences
$$
\renewcommand\arraystretch{1.2}
\begin{array}{cc}
& O\sim \Omega(\Z\times BO) \\
{_-}V(\rr)\sim \Omega {_+}U(\rr) & O/U \sim \Omega O\\
{_-}V(\cc) \sim \Omega {_+}U(\cc) & U/Sp \sim \Omega O/U\\
{_+}V(\hh) \sim \Omega {_-}U(\hh) & \Z\times BSp \sim \Omega U/Sp\\
 & Sp \sim \Omega (\Z\times BSp)\\
{_-}V(\hh) \sim \Omega {_+}U(\hh) & Sp/U \sim \Omega Sp\\
{_+}V(\cc) \sim \Omega {_-}U(\cc) & U/O \sim \Omega Sp/U\\
{_+}V(\rr) \sim \Omega {_-}U(\rr) & \Z\times BO \sim \Omega U/O
\end{array}
$$
which yields the homotopy equivalence
$$\Z \times BO \sim \Omega^8(\Z\times BO).$$
Using the complex numbers $\cc$ with its usual involution $i \mapsto -i$, the topological version of Karoubi periodicity yields
$$\Z \times BU \sim \Omega^2(\Z \times BU).$$

\begin{proof}[Proof of Lemma \ref{lem:topHtpyInv}]
To simplify, write $GW$ generically for one of the functors $GW^{[n]}_0$ or $K_0$.
We need to show that the projection $X\times I \to X$ induces an isomorphism
$GW(A(X)) \to GW(A(X\times I))$.
Since $A(X\times I)=A(X)(I)$, and $A(X)$ satisfies (*) when ever $A$ does and $X$ is compact, it suffices to prove the case when $X$ is a point.
Any point $x:\pt \to I$ of $I$ induces a map $x^*:GW(A(I)) \to GW(A)$ such that the composition $GW(A) \to GW(A(I)) \to GW(A)$ is the identity.
Therefore, what we really have to show is the surjectivity of the map
$GW(A) \to GW(A(I))$.
For a ring $A$ (with involution), denote by $\P(A)$ the category of finitely generated projective $A$-modules, and denote by $\M(A)$ the category 
$$\M(A) = \left(w\Ch^b\P(A)\right)_h$$
of symmetric spaces in the dg category $\Ch^b\P(A)$ with weak equivalences the  homotopy equivalences \red{of complexes} and duality $\Hom(\phantom{X},A[n])$.
We will show that the functors $\P(A) \to \P(A(I))$ and $\M(A) \to \M(A(I))$
are essentially surjective on objects.
Clearly, this implies the surjectivity of the map $GW(A) \to GW(A(I))$
as both groups are generated by the isomorphism classes of objects in $\M(A)$ and $\M(A(I))$, respectively (or in $\P(A)$ and $\P(A(I))$ in the $K_0$-case).

By a version of the Serre-Swan theorem,
the category $\P(A(I))$ of finitely generated projective $A(I)$-modules is equivalent to the category $\Vect_A(I)$ of locally trivial $A$-module bundles with fibres in $\P(A)$.
This is similar to \cite{Swan:VectisProj} and follows for instance from \cite[\S 1]{Swan:TopExs}.
Therefore, the set of $A(I)$-module isomorphisms $F\to E \in \P(A(I))$ is the set of sections of the locally trivial bundle $Iso(\tilde{E},\tilde{F}) \to I$ of isomorphisms  between the associated $A$-module bundles $\tilde{E},\tilde{F} \in \Vect_A(I)$.
Any locally trivial map into a paracompact space is a fibration.
Thus, by contractibility of $I$, if $Iso(\tilde{E},\tilde{F}) \to I$ has a section over one point, it has a global section, that is, if $E$ and $F$ are isomorphic over one point in $I$, then they are isomorphic $A(I)$-modules.
In particular, any $E \in \P(A(I))$ is isomorphic to $E_0\otimes_A A(I)$, and the functor $\P(A) \to \P(A(I))$ is essentially surjective.

The same argument applies to show that $\M(A) \to \M(A(I))$ is essentially surjective using the equivalence
$$\M(A(I)) = \left(w\Ch^b\P(A(I))\right)_h \cong \left(w\Ch^b\Vect_A(I)\right)_h$$
and the fact that a bundle of complexes is the same as a complex of bundles.
Given two objects $E,F$ of $\M(A(I))$. These are bounded  complexes in $\P(A(I))$ equipped with symmetric homotopy equivalences.
Assume that the components $E_i$, $F_i$ of $E$, $F$ are zero outside $[a,b]$, for some $a<b\in \Z$.
Then the set of isomorphisms $E \to F$ in $\M(A(I))$ is the set of sections of a locally trivial continuous map $Iso(\tilde{E},\tilde{F}) \to I$ which is the subbundle of $\prod_{a \leq i\leq b}Iso(\tilde{E}_i,{F}_i)$ of those isomorphisms which commute with differentials and forms.
Again, $Iso(\tilde{E},\tilde{F}) \to I$ is a fibration with contractible target and thus has a global section whenever it has a section at a point. 
As above, this implies that the functor $\M(A) \to \M(A(I))$ is essentially surjective.
\end{proof}

\appendix

\section{Homology of the infinite orthogonal group}
\label{appdx:Homology}

\subsection{Group completions}
\label{subsec:GpCpl}
Let $X$ be a homotopy commutative $H$-space.
A group completion of $X$ is an $H$-space map $f: X \to Y$ into a homotopy
commutative $H$-spaces $Y$ such that $\pi_0Y$ is a group and such that the 
map 
$$(\pi_0X)^{-1}H_*(X,\Z) \to H_*(Y,\Z)$$ 
induced by $f$ is an isomorphism. 
If $(\scS,\oplus,0)$ is a symmetric monoidal category such that all
translations $\oplus B: \scS \to \scS: A \mapsto A \oplus B$ are faithful,
Quillen constructs a symmetric monoidal category $\scS^{-1}\scS$ and a
monoidal map $\scS \to \scS^{-1}\scS$ which induces a group completion of
associated classifying spaces \cite[Theorem,
p.\hspace{.2ex}221]{quillenGrayson}. 
If $(\scS,\oplus,0)$ is a symmetric strict monoidal category (up to
equivalence of categories, this can always be achieved \cite{may:permutative}), 
a group completion is also the map $B\scS \to \Omega \Bar B\scS$ where $\Bar
B\scS$ is the Bar construction (classifying space) of the topological monoid
$B\scS$ \cite{may:gpcpletions}.
The two group completions are equivalent via the zigzag
$\Omega \Bar B\scS \stackrel{\sim}{\longrightarrow} \Omega \Bar B\scS^{-1}\scS 
\stackrel{\sim}{\longleftarrow} B\scS^{-1}\scS$ of homotopy equivalences.

Let $(\A,*,\eta)$ be an additive category with duality.
Inclusion of degree zero simplices yields a map 
$(i\A)_h \to (i\RR_{\bullet}\A)_h$ such that its composition with
$(i\RR_{\bullet}\A)_h \to iS_{\bullet}\A$ as in (\ref{eqn:RtoSmap}) is the zero map.
Therefore, we obtain an induced map into the homotopy fibre of the last map
\begin{equation}
\label{eqn:Ahgpcompl}
(i\A)_h \longrightarrow |i\RR_{\bullet}\A|
\underset{|iS_{\bullet}\A|}{\stackrel{h}{\times}}\hspace{1ex}\pt  
\stackrel{\sim}{\longrightarrow}
\Omega^{\infty}GW(\A).
\end{equation}
where the second map is a weak equivalence, by \cite[\S 6 Proposition 6]{myMV} and Proposition \ref{prop:GWspecIsGWspace}.
In \cite{mygiffen}, we have shown that this map is a group completion if
$2$ is invertible in $\A$.
Our proof in \cite{mygiffen} uses Karoubi's Fundamental Theorem
\cite{Karoubi:Annals}.
The purpose of this section is to give a direct proof of that fact avoiding
the Fundamental Theorem, so that our proof of Theorem \ref{thm:KaroubiFundThm}
gives not only a generalization but also a new proof of Karoubi's theorem and its topological counter-part, classical Bott periodicity; see Section \ref{section:Bott}.

\begin{theorem}
\label{thm:hermGpCpl}
Let $(\A,*,\eta)$ be a split exact category with duality such that
$\frac{1}{2} \in \A$.
Then the natural map (\ref{eqn:Ahgpcompl}) is a group completion.
In particular, there is a natural homotopy equivalence
$$(i\A_h)^{-1}(i\A_h) \simeq \Omega^{\infty}GW(\A).$$
\end{theorem}

The proof of theorem \ref{thm:hermGpCpl} occupies most of this section.
Before going into the details we derive the statement that makes the link
between the homology of the infinite orthogonal group and higher
Grothendieck-Witt groups.
For that, let $R$ be a ring with involution $r\mapsto \bar{r}$ and $\eps = \pm
1$. 
We denote by 
$${_{\eps}\scP}(R) = (\scP(R),*,\eps \can).$$
the additive category with duality,
where $\scP(R)$ is the category of finitely generated projective right
$R$-modules, the dual $P^*$ of a right $R$-module $P$ is the left $R$-module
$P^*=\Hom_R(P,R)$ considered as a right module via the 
involution, and $\can_P:P \to P^{**}$ is the double dual identification with
$\can_P(x)(f)=\overline{f(x)}$. 
The associated category $(i{_{\eps}\scP}(R))_h$ is the usual category of
non-degenerated $\eps$-symmetric bilinear forms on finitely generated
projective $R$-modules with isometries as morphisms.
We denote by ${_{\eps}GW}(R)$, the Grothendieck-Witt spectrum
$${_{\eps}GW}(R) = GW(\scP(R),*,\eps \can).$$
By Theorem \ref{thm:hermGpCpl}, its $\Omega^{\infty}$-space is a group
completion of $(i{_{\eps}\scP}(R))_h$.
Let 
$${_{\eps}O}(R) = \colim_n Aut({_{\eps}H}(R^n))$$
denote the infinite $\eps$-orthogonal group of $R$, 
where $Aut({_{\eps}H}(P))$ denotes the group of isometries of the
$\eps$-hyperbolic space ${_{\eps}H}(P) = (P\oplus P^*,
\left(\begin{smallmatrix} 0& 1\\ \eps\can & 0 \end{smallmatrix} \right))$ and
the inclusion $Aut({_{\eps}H}(R^n)) \to Aut({_{\eps}H}(R^{n+1}))$ is induced
by the functor $\perp {_{\eps}H}(R)$.

\begin{corollary}
\label{cor:+=Q}
For any ring with involution $R$ such that $\frac{1}{2}\in R$, there is a
natural homotopy equivalence 
$$B{_{\eps}O}(R)^+ \stackrel{\sim}{\longrightarrow} \Omega^{\infty}_0\ {_{\eps}GW}(R)$$
where $\Omega^{\infty}_0\ {_{\eps}GW(R)}$ denotes the connected component of $0$ of the space
$\Omega^{\infty}\ {_{\eps}GW}(R)$ and $B{_{\eps}O}(R)^+$ is Quillen's plus construction applied
to $B{_{\eps}O}(R)$.
\end{corollary}

\begin{proof}
By Theorem \ref{thm:hermGpCpl}, the space $\Omega^{\infty}_0\ {_{\eps}GW}(R)$ is a group
completion of $(i{_{\eps}\scP}(R))_h$.
The proof now is the same as in \cite[Theorem,
p.\hspace{.2ex}224]{quillenGrayson}.
The main point is that if $\frac{1}{2}\in R$ then every non-degenerated
$\eps$-symmetric bilinear form is 
a direct factor of the hyperbolic form ${_{\eps}H}(R^n)$ for some $n$; see
also \cite[Proposition 3]{Weibel:Azumaya}.
\end{proof}

The proof of theorem \ref{thm:hermGpCpl} 
is based on the following proposition.
We introduce some notation.
For an exact category with duality $(\E,*,\eta)$, the category $E(\E)$ of
exact sequences in $\E$ is an exact category with duality where the dual 
$(E_{\bullet})^*$ of an exact sequence
$E_{\bullet}=(E_{-1} \rightarrowtail E_0 \twoheadrightarrow E_1)$ in $\E$ 
is the exact sequence $E_1^* \rightarrowtail E_0^*\twoheadrightarrow E_{-1}^*$ 
and the double dual identification $\eta_{E_{\bullet}}$ is
$(\eta_{E_{-1}},\eta_{E_0}, \eta_{E_1})$.

\begin{proposition}
\label{prop:S2=K}
Let $\A$ be a split exact category with duality such that $\frac{1}{2} \in
\A$.
Then the map $F: (iE\A)_h \to i\A: (E_{\bullet},\ffi) \mapsto E_{-1}$
of symmetric monoidal categories
induces a homotopy equivalence after group completion.
\end{proposition}

\begin{proof}
The functor $F$ has a section, the hyperbolic functor 
$$
\renewcommand\arraystretch{2}
\begin{array}{llcl}
\H:i\A \to (iE\A)_h:& 
A & \mapsto & \H(A) = 
\left(A \stackrel{\left(\begin{smallmatrix}1\\ 0
      \end{smallmatrix}\right)}{\rightarrowtail} 
A \oplus A^* \stackrel{\left(\begin{smallmatrix} 0 & 1
      \end{smallmatrix}\right)}{\twoheadrightarrow}
A,
 \left(\eta, \left(\begin{smallmatrix}0&1\\ \eta & 0
     \end{smallmatrix}\right),1 \right)
\right) \\
&g & \mapsto & \H(g) = \left( g , \left( \begin{smallmatrix} g&0 \\ 0 &
      (g^*)^{-1} \end{smallmatrix}\right), (g^*)^{-1} \right) 
\end{array}
$$
which is symmetric monoidal and thus induces an action of 
$i\A$ on $(iE\A)_h$ in the sense of \cite{quillenGrayson}.
In view of the Lemma \ref{lem:AppA:FA} (\ref{lem:AppA:FA:itemC}) below, the
hyperbolic functor induces on $\pi_0$ an isomorphism of abelian monoids
$$\pi_0(\H): S = \pi_0i\A \stackrel{\cong}{\longrightarrow} \pi_0 (iE\A)_h.$$
To abbreviate, we write $S$ for this monoid.
The actions of the symmetric monoidal category $i\A$ on $(iE\A)_h$ and on $i\A$
induce actions of the abelian monoid $S$ on the homology groups of $(iE\A)_h$
and of $i\A$ compatible with $F$.
We will show that the localized map
\begin{equation}
\label{eqn:gpcplhomolgyIso}
F: S^{-1}H_*((iE\A)_h,k) \to S^{-1}H_*(i\A,k)
\end{equation}
is an isomorphism when the coefficient group is $k=\ff_p, \Q$.
This implies that (\ref{eqn:gpcplhomolgyIso}) is an isomorphism for $k=\Z$,
that is,
the functor $F: (iE\A)_h \to i\A$ induces a homology (hence homotopy)
equivalence after group completion. 
\vspace{2ex}

In order to prove that (\ref{eqn:gpcplhomolgyIso}) is an isomorphism, 
we will analyze the spectral sequence associated with the functor $F: (iE\A)_h
\to i\A$.
We first  review some elementary properties of that spectral sequence 
in the more general context of an arbitrary functor $F:\B \to \C$ of small
categories. 
For that, we fix an abelian group $k$ (which we will suppress in most
formulas), and 
write $H_p(\C)=H_p(\C,k)$ for the homology of the category $\C$ with
coefficients in $k$.
More generally,
we write $H_p(\C,A)$ for the homology of $\C$ with
coefficients in a functor $A:\C \to \langle\text{ab\hspace{1ex}gps}\rangle$
from $\C$ to abelian groups.
This is the homology of the chain complex associated with the simplicial
abelian group
$$p \mapsto \bigoplus_{C_0 \to ... \to C_p}A(C_0)$$
of chains on $N\C$ with coefficients in $A$.

\begin{lemma}
\label{lem:spSeqFun}
\begin{enumerate}
\item
\label{lem:spSeqFun:item1}
Let $F:\B \to \C$ be a functor between small categories.
There is a strongly convergent first quadrant spectral sequence
$$E^2_{p,q}(F)=H_p(\C,H_q(F\downarrow \underline{\phantom{A}}\ )) \Rightarrow
H_{p+q}(\B)$$
with differentials $d^n_{p,q}:E^n_{p,q} \to E^n_{p-n,q+n-1}$ and edge morphism
$$H_p(\B) \twoheadrightarrow E^{\infty}_{p,0} \subset E^2_{p,0} =
H_p(\C,H_0(F\downarrow \underline{\phantom{A}}\ )) \to H_p(\C)$$
the map $H_p(F)$ induced by $F$ on homology.
Here, the last map is the change-of-coefficient map sending $(F\downarrow
\underline{\phantom{A}}\ )$ to the constant one-object-one-morphism
category.
\item
\label{lem:spSeqFun:item2}
If $F':\B' \to \C'$ is another functor, and $b:\B \to \B'$ and $c:\C \to \C'$
are functors such that $F'b=cF$, then there is a map of spectral sequences 
$E(b,c):E(F) \to E(F')$ compatible with edge-maps and which, on $E^2$-term and
abutment, is the usual maps on homology induced by $b$ and $c$.
\item
\label{lem:spSeqFun:item3}
Given two pairs of functors $b_i:\B \to \B'$, $c_i: \C \to \C'$ such that
$F'b_i=c_iF$, $i=0,1$.
If there are natural transformations $\beta:b_0 \Rightarrow b_1$ and $\gamma:
c_0 \Rightarrow c_1$ such that $F'\beta_B=\gamma_{FB}$ for all objects $B$ of 
$\B$, then the two induced maps of spectral sequences agree from the
$E^2$-page on:
$$E^n_{p,q}(b_0,c_0) = E^n_{p,q}(b_1,c_1): E^n_{p,q}(F) \to E^n_{p,q}(F'),\hspace{1ex} n\geq 2.$$
\end{enumerate}
\end{lemma}

\begin{proof}[Proof {\rm (Sketch)}]
The spectral sequence is obtained by filtering the bicomplex
$K_{p,q} = \bigoplus_{X_{p,q}}k$ associated with
the bisimplicial set $X_{\bullet, \bullet}$ whose $p,q$ simplices is the set
$$
  X_{p,q}= \{ B_0\to ... \to B_q,\ FB_q \to C_0 \to ... \to
  C_p\hspace{1ex} |\hspace{1ex}B_i\in \B, \hspace{1EX}C_i\in \C \}.
$$
The filtration is by subcomplexes 
$K_{\leq n,\bullet} \subset K_{\leq n+1,\bullet} \subset K$ 
with $(K_{\leq n,\bullet})_{p,q}=K_{p,q}$ if $p\leq n$ and $(K_{\leq
  n,\bullet})_{p,q}=0$ if $p > n$.
As in the proof of Quillen's theorem A \cite[\S 1]{quillen:higherI}, the
bisimplicial set $X_{\bullet,\bullet}$ is homotopy equivalent to the
bisimplicial set which is constant in the $p$-direction and is the nerve
$N_{\bullet}\B$ of $\B$ in the
$q$-direction.
In particular, the abutment of the spectral sequence computes the
homology of $\B$.
It is straight forward to identify the $E^2$-term as in
(\ref{lem:spSeqFun:item1}).
Functoriality as in (\ref{lem:spSeqFun:item2}) is clear since the pair $(b,c)$
induces a map of the corresponding bicomplexes.
The edge map for the spectral sequence of the identity functor $id_C$ of $\C$
is the identity map, so that the functoriality
in (\ref{lem:spSeqFun:item2}),
applied to the map $(F,id_{\C}): F \to id_{\C}$ of functors, yields the 
description of the edge map in (\ref{lem:spSeqFun:item1}).
For the somewhat extended functoriality in (\ref{lem:spSeqFun:item3}), 
one constructs an explicit homotopy between the two maps 
$$(b_i,c_i): \bigoplus_{C_0 \to ... \to C_p}H_j(F\downarrow C_0\ )
\longrightarrow 
\bigoplus_{C'_0 \to ... \to C'_p}H_j(F'\downarrow C'_0\ ),\hspace{2ex} i=0,1,
$$ 
of simplicial abelian groups which compute the $E^2$-terms of the spectral
sequences. 
Again this is straight forward, and we omit the details.
\end{proof}

We apply the lemma to the functor $F: (iE\A)_h \to i\A$ and obtain the spectral
sequence 
\begin{equation}
\label{eqn:spseqBeforeS}
E^2_{p,q} = H_p(i\A,H_q(F\downarrow \underline{\phantom{A}}\ )) \Rightarrow
H_{p+q}((iE\A)_h).
\end{equation}
Every object $A$ of $i\A$ induces functors $\oplus \H(A): (iE\A)_h \to
(iE\A)_h$ and $\oplus A: i\A \to i\A$ compatible with $F$.
By Lemma \ref{lem:spSeqFun} (\ref{lem:spSeqFun:item2}), these functors induce a
map of spectral sequences.
Every map $A \to A'$ in $i\A$ induces natural transformations 
$\oplus \H(A) \Rightarrow \oplus \H(A')$ and $\oplus A \Rightarrow \oplus A'$
compatible with $F$.
By Lemma \ref{lem:spSeqFun} (\ref{lem:spSeqFun:item3}), the objects $A$  and
$A'$ induce the same map of spectral sequences from $E^2$ on.
In other words, we obtain an induced action of the monoid $S$ on the spectral
sequence (\ref{eqn:spseqBeforeS}) from $E^2$ on.
Since localization with respect to a commutative monoid is exact, we obtain a strongly
convergent first quadrant spectral sequence
\begin{equation}
\label{eqn:spseqAfterS}
S^{-1}E^2_{p,q} = S^{-1}H_p(i\A,H_q(F\downarrow \underline{\phantom{A}}\ ))
\Rightarrow S^{-1}H_{p+q}((iE\A)_h).
\end{equation}

In order to analyze this spectral sequence, we will employ the symmetric
monoidal automorphism 
$$\Sigma: (iE\A)_h \to (iE\A)_h: (M,\ffi) \mapsto (M,\frac{1}{2}\ffi),
g\mapsto g.$$
Note that $F\circ \Sigma = F$, so that $\Sigma$ induces an action of the
spectral sequence (\ref{eqn:spseqBeforeS}).
For every object $A$ in $i\A$, there is a natural transformation of functors
$$\Sigma \circ (\oplus \H(A)) \Rightarrow (\oplus \H(A))\circ \Sigma$$
compatible with 
$F$. 
It is induced by the isometry $\Sigma \H(A) \cong \H(A)$ given by the
map 
$$\left(\begin{smallmatrix}1/2 & 0 \\ 0 & 1\end{smallmatrix}\right): A \oplus
A^* \to A \oplus A^*.$$
Therefore, the action of $S$ on the spectral sequence (\ref{eqn:spseqBeforeS})
commutes with the action of $\Sigma$ on it from the $E^2$-page on.
In particular, $\Sigma$ induces an action on the spectral sequence
(\ref{eqn:spseqAfterS}).

\begin{lemma}
\label{lem:AppA:FA}
\begin{enumerate}
\item
\label{lem:AppA:FA:itemA}
For every object $A$ of $i\A$, there is an equivalence of categories between 
the comma category $(F\downarrow A)$ and the one-object category whose
endomorphism set is the set 
$$\{a:A^* \to A|\ \eta\, a+ a^* = 0\} \subset Hom(A^*,A)$$ 
with addition as composition.
This is a uniquely $2$-divisible abelian group.
\item
\label{lem:AppA:FA:itemB}
Under the equivalence in (\ref{lem:AppA:FA:itemA}), the automorphism 
$\Sigma: (F\downarrow A) \to (F\downarrow A)$ corresponds to 
$a \mapsto 2a$.
\item
\label{lem:AppA:FA:itemC}
The functor $F$ induces an isomorphism of abelian monoids
$$\pi_0(iE\A)_h \stackrel{\cong}{\longrightarrow} \pi_0i\A.$$
\end{enumerate}
\end{lemma}

\begin{proof}
Choose for every inflation $M_{-1} \stackrel{i}{\rightarrowtail} M_0$ in $\A$
a retraction $r:M_0 \twoheadrightarrow M_{-1}$, so that we have $ri=1$.
For the inflation $\left(\begin{smallmatrix}1\\ 0
    \end{smallmatrix} \right): A 
\rightarrowtail A \oplus A^*$
we choose $r=\left(\begin{smallmatrix}1&0\end{smallmatrix}\right)$.
The choice of these retractions defines for every object
\begin{equation}
\label{App:eqnFAM}
M_{-1} \stackrel{i}{\rightarrowtail} M_0
\stackrel{p}{\twoheadrightarrow} M_1,\hspace{1ex} (\ffi_{-1},\ffi_0,\ffi_{1}), 
\hspace{2ex} M_{-1} \stackrel{g}{\to} A
\end{equation}
of the comma category $(F\downarrow A)$ an isomorphism to the object
\begin{equation}
\label{App:eqnFAA}
A \stackrel{\left(\begin{smallmatrix}1\\ 0
    \end{smallmatrix} \right)}{\rightarrowtail} A \oplus A^*
\stackrel{\left(\begin{smallmatrix} 0 & 1
    \end{smallmatrix} \right)}{\twoheadrightarrow} A^*, \hspace{1ex}
\left(\eta,
\left(\begin{smallmatrix}0 & 1\\ \eta & 0 \end{smallmatrix} \right)
,1 \right),\hspace{2ex} A \stackrel{1}{\to} A
\end{equation}
of that category given by the maps
\begin{equation}
\label{App:eqnFAiso}
g:M_{-1} \to A,\hspace{3ex}
\left(\begin{smallmatrix}gr + g\delta p\\ (g^*)^{-1}\ffi_1p
    \end{smallmatrix} \right): M_0 \to A\oplus A^*,
\hspace{3ex}
(g^*)^{-1}\ffi_1:M_1 \to A^*,
\end{equation}
where the map $r:M_0 \twoheadrightarrow M_{-1}$ is the chosen retraction of
$i:M_{-1} \rightarrowtail M_0$, 
the map $\delta: M_1 \to M_{-1}$ is defined as
$\delta =\frac{1}{2}(\ffi_1^*\eta)^{-1}\beta$,
and $\beta: M_1 \to M_1^*$ is the unique map satisfying
$p^*\beta p = \psi$ for the symmetric map $\psi = \ffi_0-p^*\ffi_1^*\eta r -
r^*\ffi_1 p: M_0 \to M_0^*$, the existence of $\beta$ being assured since
$\psi i =0$. 
Note that the map just defined is an isomorphism in $E\A$ since it is an
isomorphism on subobject and quotient object.
All these isomorphisms together define an equivalence of categories between
$(F\downarrow A)$ and the full subcategory of $(F\downarrow A)$
consisting of the single object (\ref{App:eqnFAA}).
The endomorphism set of (\ref{App:eqnFAA}) is the set consisting of the maps
\begin{equation}
\label{App:eqnFAendo}
1:A \to A, \hspace{3ex} 
\left(\begin{smallmatrix}1 & a\\ 0 & 1
    \end{smallmatrix} \right): A \oplus A^* \to A\oplus A^*,
\hspace{3ex}
1:A^* \to A^*,
\end{equation}
where $\eta\, a + a^* = 0$.
Composition corresponds to matrix multiplication, hence to addition of the
$a$'s.
By our assumption $\frac{1}{2}\in \A$, the group $Hom(A^*,A)$ is uniquely
$2$-divisible, so is the kernel of the map $a\mapsto \eta\, a+ a^*$.
This proves part (\ref{lem:AppA:FA:itemA}) of the lemma.

The existence of the isomorphism from (\ref{App:eqnFAM}) to (\ref{App:eqnFAA}) above (with $g=1$)
shows that every object of $(iE\A)_h$ is (up to isomorphism) in the image of the
hyperbolic functor $\H:i\A \to (iE\A)_h$.
Since $F\circ\H=id$, this proves (\ref{lem:AppA:FA:itemC}).

For part (\ref{lem:AppA:FA:itemB}), we note that $\Sigma$ sends the object
(\ref{App:eqnFAA}) to the object
\begin{equation}
\label{App:eqnFAA2}
A \stackrel{\left(\begin{smallmatrix}1\\ 0
    \end{smallmatrix} \right)}{\rightarrowtail} A \oplus A^*
\stackrel{\left(\begin{smallmatrix} 0 & 1
    \end{smallmatrix} \right)}{\twoheadrightarrow} A^*, \hspace{1ex}
\frac{1}{2} \left(\eta,
\left(\begin{smallmatrix}0 & 1\\ \eta & 0 \end{smallmatrix} \right)
,1 \right),\hspace{2ex} A \stackrel{1}{\to} A
\end{equation}
for which, according to our choice of
$r=\left(\begin{smallmatrix}1&0\end{smallmatrix}\right)$ (implying $\psi = 0$
and $\delta = 0$), the isomorphism (\ref{App:eqnFAiso}) is given by the maps
$$1:A \to A, \hspace{3ex} 
\left(\begin{smallmatrix}1 & 0\\ 0 & {1}/{2}
    \end{smallmatrix} \right): A \oplus A^* \to A\oplus A^*,
\hspace{3ex}
{1}/{2}:A^* \to A^*.$$
The functor $\Sigma$ sends the endomorphism (\ref{App:eqnFAendo})
of (\ref{App:eqnFAA}) to the endomorphism of (\ref{App:eqnFAA2}) given by the
same maps as in (\ref{App:eqnFAendo}).
Under the equivalence of categories, this corresponds to the endomorphism of 
(\ref{App:eqnFAA}) given by the maps
$$
1:A \to A, \hspace{3ex} 
\left(\begin{smallmatrix}1 & 2a\\ 0 & 1
    \end{smallmatrix} \right)=
\left(\begin{smallmatrix}1 & 0\\ 0 & {1}/{2}
    \end{smallmatrix} \right)
\left(\begin{smallmatrix}1 & a\\ 0 & 1
    \end{smallmatrix} \right)
\left(\begin{smallmatrix}1 & 0\\ 0 & {1}/{2}
    \end{smallmatrix} \right)^{-1}
: A \oplus A^* \to A\oplus A^*,
\hspace{3ex}
1:A^* \to A^*.
$$
This proves the lemma.
\end{proof}

Let $k=\Z/2$.
Then Lemma \ref{lem:AppA:FA} (\ref{lem:AppA:FA:itemA}) implies that the spectral
sequences (\ref{eqn:spseqBeforeS}) and (\ref{eqn:spseqAfterS}) degenerate at
the $E^2$-page to yield the isomorphism (\ref{eqn:gpcplhomolgyIso})
since for a uniquely $2$-divisible abelian group $G$, we have $H_q(G,\Z/2)=0$,
$q>0$. 

From now on we will assume $k=\ff_p$, $p$ prime $\neq 2$, or $k=\Q$.

\begin{lemma}
\label{lem:Eigenvalues}
Let $k$ be one of the fields $\F_p$, $p\neq 2$, or $\Q$.
Let $G$ be an abelian group and $t:G \to G: a \mapsto 2a$.
Then for all $q\geq 0$,
$$H_q(G,k)=\bigoplus_{q/2\leq i \leq q,\ i\in \N}V_i$$
where $H_q(t)$ acts on the $k$-vector space $V_i$ as multiplication by $2^i$.
\end{lemma}

\begin{proof}
This follows from the explicit description of 
$H_*(G)$ given, for instance, in \cite[Theorem 6.4 (ii), Theorem
6.6]{brown:cohgps}. 
Alternatively, it can be proved as follows.
The statement is clearly true for $q=0$ and $q=1$ (for the
latter use the natural isomorphism $G\cong \pi_iBG \cong H_1(BG,\Z)$), hence
it is true for $G=\Z$ (as $B\Z=S^1$).
Using the Serre spectral sequence associated with the fibration 
$B\Z \to B\Z \to B(\Z/l)$ we see that the statement is true for $G=\Z/l$.
Using the K\"unneth formula, if the statement is true for $G_1$ and $G_2$ then
it is true for $G_1\times G_2$.
Therefore, the lemma holds for finitely generated abelian
groups. 
Passing to limits, we conclude that it is true for all abelian groups.
\end{proof}

Lemma
\ref{lem:Eigenvalues} applies to the endomorphism $\Sigma$ of the 
$k$-module $H_q(F\downarrow A)$, in view of Lemma \ref{lem:AppA:FA}
(\ref{lem:AppA:FA:itemB}). 
We obtain a direct sum decomposition of the spectral sequences
(\ref{eqn:spseqBeforeS}) and (\ref{eqn:spseqAfterS}) into the eigen-space
spectral sequences of $\Sigma$. 
Note that the eigen-space spectral sequence of (\ref{eqn:spseqAfterS})
corresponding to the eigen-value $1$ has 
$E^2_{p,q}$-term $S^{-1}H_p(i\A,H_0(F\downarrow \underline{\phantom{A}}\ )) =
H_p(i\A,k)$ if $q=0$ and $E^2_{p,q}=0$ for $q\neq 0$.
In particular, it degenerates completely at the $E^2$-page.

The next lemma shows that $\Sigma$ acts as the identity on the abutment
of (\ref{eqn:spseqAfterS}), so that the eigen-space spectral sequence of
(\ref{eqn:spseqAfterS}) 
corresponding to the eigen-value $1$ has abutment the abutment of
(\ref{eqn:spseqAfterS}).
Degeneration of this spectral sequence at the $E^2$-page 
shows that its edge map (\ref{lem:spSeqFun} \ref{lem:spSeqFun:item1})
is an isomorphism, that is,
(\ref{eqn:gpcplhomolgyIso}) is also an isomorphism for $k=\ff_p$, $p$ prime
$\neq 2$ and $\Q$. 

\begin{lemma}
\label{lem:Abuttment}
Let $k$ be one of the fields $\ff_p$, $p$ prime $\neq 2$, or $\Q$.
\begin{enumerate}
\item
\label{lem:Abuttment:item1}
Let $(\C,\oplus,0)$ be a unital symmetric monoidal category with $\pi_0\C=0$.
Then the map $id\oplus id:\C \to \C:A \mapsto A \oplus A$ induces an
isomorphism on homology 
$$H_*(\C,k) \stackrel{\sim}{\longrightarrow} H_*(\C,k).$$
\item
\label{lem:Abuttment:item2}
The functor $\Sigma: (iE\A)_h \to (iE\A)_h$ induces the identity map on the
abutment $S^{-1}H_*(iE\A)_h$ of the spectral sequence (\ref{eqn:spseqAfterS}).
\end{enumerate}
\end{lemma}

\begin{proof}
For part (\ref{lem:Abuttment:item1}), consider $X=Y=B\C$ as pointed spaces
with base-point $0\in \C$, and write $f: X \to Y$ for
the map on classifying spaces induced by $id\oplus id:\C \to \C$.
Note that $f$ preserves base points and that $\pi_*(f)$ is multiplication by
$2$, so that 
$\pi_*(f) \otimes \Z[\frac{1}{2}]$ is an isomorphism.
Let $G$ be the cokernel of the map $\pi_1(f): \pi_1X \to \pi_1Y$ and write
$\tilde{Y} \to Y$ for the covering of $Y$ corresponding to the quotient map
$\pi_1Y \to G$, that is, $\tilde{Y}$ is the pullback of $EG$ along $Y \to
B\pi_1Y \to BG$.
By construction, the map $f:X \to Y$ factors through $\tilde{Y} \to Y$, and
the induced map $X \to \tilde{Y}$ is surjective on $\pi_1$.
Let $F$ be the homotopy fibre of $X \to \tilde{Y}$.
By construction, $F$ is a connected nilpotent space (in fact an infinite loop
space as all maps $X \to \tilde{Y} \to Y \to BG$ are infinite loop space
maps) with $\pi_*F \otimes \Z[\frac{1}{2}]= 0$.
It follows that $H_q(F,\Z[\frac{1}{2}])=0$, hence $H_q(F,k)=0$, $q>0$.
By the Serre spectral sequence, $X \to \tilde{Y}$ induces an isomorphism in
homology with coefficients in $k$.
Finally, the principal $G$-fibration $\tilde{Y} \to Y$ induces an isomorphism
in homology with coefficients in $k$, since every element in $G$ has order
invertible in $k$.

For part (\ref{lem:Abuttment:item2}), write $\scS = (iE\A)_h$.
The abutment of (\ref{eqn:spseqAfterS}) is the homology of the symmetric
monoidal category $\C = \scS^{-1}\scS$, in view of Lemma \ref{lem:AppA:FA}
(\ref{lem:AppA:FA:itemC}) and \cite{quillenGrayson}, see also
Section \ref{subsec:GpCpl}.
Note 
that $\pi_0\Sigma: \pi_0\scS \to \pi_0\scS$ is the identity because 
$F\circ \Sigma = F$ and $\pi_0F$ is an isomorphism, by Lemma \ref{lem:AppA:FA}
(\ref{lem:AppA:FA:itemC}).
In particular, $\Sigma$ induces the identity on $\pi_0\C$.
Let $\C_0 \subset \C$ be the connected component of $0$.
In view of the natural isomorphism $H_*(\C_0)\otimes_kk\pi_0\C \cong H_*\C$,
it suffices to show that the restriction $\Sigma_0$ of $\Sigma:\C \to \C$ to
$\C_0$ induces the identity on homology with coefficients in $k$.
To that end, we consider the functor $T=id\oplus id :\scS \to \scS$ and the
natural transformation $\Sigma T\cong T$ given by 
$$\frac{1}{2}\left(\begin{smallmatrix}1&1\\ -1&1\end{smallmatrix}\right):
(M,\ffi/2) \perp (M,\ffi/2) \stackrel{\cong}{\longrightarrow} (M,\ffi)\perp
(M,\ffi).$$
The functors $T$ and $\Sigma$ and the natural transformation induce functors 
$T$ and $\Sigma$ on $\C=\scS^{-1}\scS$ and a natural transformation $\Sigma
T\cong T$ of endofunctors of $\C$.
Restricted to the connected component $\C_0$ of $\C$, we obtain functors $T_0,
\scS_0:\C_0 \to \C_0$ and a natural transformation $\Sigma_0 T_0\cong T_0$.
By part (\ref{lem:Abuttment:item1}), $T_0$ induces an isomorphism on homology
with coefficients in $k$.
Since $H_*(\Sigma_0,k) \circ H_*(T_0,k)= H_*( T_0,k)$, by the existence of the
natural transformation $\Sigma_0 T_0\cong T_0$, we have $H_*(\Sigma_0,k)=id$.
\end{proof}

This finishes the proof of Proposition \ref{prop:S2=K}.
\end{proof}

For the statement of the next proposition, recall that for an exact category
with duality $(\A,*,\eta)$,  
the category $S_n\A$ is an exact category with duality, where
the dual $A^*$ of an object $A$ of $S_n\A$ is
$(A^*)_{i,j}=(A_{n-j,n-i})^*$ and the double dual identification $\eta_A$
satisfies $(\eta_A)_{i,j} = \eta_{A_{i,j}}$.

\begin{proposition}
\label{prop:AddAddty}
Let $\A$ be an additive category with duality.
Assume $\frac{1}{2} \in \A$.
Then the symmetric monoidal functors
$$
\renewcommand\arraystretch{1.5} 
\begin{array}{ll}
(iS_{2n}\A)_h \to (i\A)^n: &  (A_{\bullet},\ffi_{\bullet}) \mapsto
(A_{0,1},A_{1,2},...,A_{n-1,n})\\
(iS_{2n+1}\A)_h \to (i\A)^n\times (i\A)_h: &
(A_{\bullet},\ffi_{\bullet})\mapsto
(A_{0,1},A_{1,2},...,A_{n-1,n}),(A_{n,n+1},\ffi_{n,n+1}) 
\end{array}
$$
are homotopy equivalences after group completion. 
\end{proposition}

Before we give the proof, we note the following corollary of
Proposition \ref{prop:S2=K}.

\begin{corollary}
\label{cor:prop:S2=K}
Let $(A,*,\eta)$, $(\B,*,\eta)$ be additive categories with duality such that
$\frac{1}{2}\in \A, \B$.
\begin{enumerate}
\item
\label{cor:prop:S2=K:itemA}
Given a non-singular exact form functor $(F_{\bullet},\ffi_{\bullet}):
\A \to E\B$, that is,
a commutative diagram of additive functors $F_i: (\A,w) \to (\B,w)$, $i=-1,0,1$,
$$\xymatrix{
\hspace{2ex}F_{\bullet} \ar[d]^{\ffi_{\bullet}}_{\cong}:  & F_{1}
\xymono[r]^{d_1}\ar[d]^{\ffi_{1}}_{\cong} & F_0 \xyepi[r]^{d_0}\ar[d]^{\ffi_{0}}_{\cong}
&  F_{-1} \ar[d]^{\ffi_{-1}}_{\cong} \\
\hspace{2ex}F_{\bullet}^{\sharp}: & F_{-1}^{\sharp} \xymono[r]_{d_0^{\sharp}} & F_0^{\sharp}
\xyepi[r]_{d_1^{\sharp}} & F_{1}^{\sharp}}$$ 
with exact rows and  
$(\ffi_{1},\ffi_0,\ffi_{-1})=(\ffi_{-1}^{\sharp}\eta,\ffi_0^{\sharp}\eta,\ffi_{1}^{\sharp}\eta)$
an isomorphism.    
Then the two non-singular exact form functors 
$$(F_0,\ffi_0) \phantom{as}{\rm and}\phantom{wer} H\circ F_1$$
induce homotopic maps $(i\A)_h \to (i\B)_h$ after group completions.
\item
\label{cor:prop:S2=K:itemB}
Given a non-singular exact form functor $(F_{\bullet},\ffi_{\bullet}):
\A \to S_3\B$, 
then the two non-singular exact form functors 
$$(F_{03},\ffi_{03}) \phantom{as}{\rm and}\phantom{wer} (F_{12},\ffi_{12})
\perp H\circ F_{01}$$
induce homotopic maps $(i\A)_h \to (i\B)_h$ after group completions.
\end{enumerate}
\end{corollary}

\begin{proof}
Part (\ref{cor:prop:S2=K:itemA}) is a formal consequence of Proposition
\ref{prop:S2=K}. 
In a more general context, the implication is explained, for instance, in
\cite[3.10. Proof of theorem 3.3]{myMV}. 

For (\ref{cor:prop:S2=K:itemB}), write $\sim$ for ``homotopic after group completion''.
For any additive form
functor $(F,\ffi):\A \to \B$, we have a form functor $\A \to E\B$ given by
$$F \stackrel{\left(\begin{smallmatrix}1\\
      1\end{smallmatrix}\right)}{\rightarrowtail} F\oplus F
\stackrel{\left(\begin{smallmatrix}1 &
      -1\end{smallmatrix}\right)}{\twoheadrightarrow} F,
\hspace{3ex}(\ffi,\left(\begin{smallmatrix}\ffi & 0 \\
      0 & -\ffi \end{smallmatrix}\right), \ffi),
$$
so that $H(F)\sim (F,\ffi)\perp(F,-\ffi)$, by (\ref{cor:prop:S2=K:itemA}).

Now consider the given non-singular exact form functor 
$(F_{\bullet},\ffi_{\bullet}): \A \to S_3\B$.
It induces the non-singular exact form functor $\A \to E\B$ given by
$$\xymatrix{
F_{02} \hspace{1ex} \xymono[rr]^{\hspace{-3ex}\left(\begin{smallmatrix}F_{02\leq 12}\\
      F_{02\leq 03}\end{smallmatrix}\right)} &&
\hspace{2ex} F_{12}\oplus F_{03} \hspace{2ex}
\xyepi[rr]^{\hspace{3ex}\left(\begin{smallmatrix}-F_{12\leq 23} &
      F_{03\leq 23}\end{smallmatrix}\right)} && \hspace{3ex}F_{23}},
\hspace{3ex}(\ffi_{02},\left(\begin{smallmatrix}-\ffi_{12} & 0 \\
      0 & \ffi_{03} \end{smallmatrix}\right), \ffi_{23}),
$$
so that $H(F_{02})\sim (F_{12},-\ffi_{12})\perp(F_{03},\ffi_{03})$, by
(\ref{cor:prop:S2=K:itemA}). 
By Additivity in $K$-theory and the remark above, we have
$H(F_{02}) \sim H(F_{01}) \perp H(F_{12}) \sim H(F_{01}) \perp
(F_{12},\ffi_{12})\perp(F_{12},-\ffi_{12})$.
Cancelling $(F_{12},-\ffi_{12})$, the claim follows.
\end{proof}

\begin{proof}[Proof of Proposition \ref{prop:AddAddty}]
The cases $S_0$ and $S_1$ are trivial, and the case $S_2$ is proposition
\ref{prop:S2=K} since $E\A=S_2\A$.
Recall that the assignment $[n] \mapsto S_n\A$ is a simplicial category, where
a monotonic map $\vartheta: [n] \to [p]$ induces the functor
$\vartheta^!:S_{p}\A \to S_{n}\A: A \mapsto \vartheta^!A$ with
$(\vartheta^!A)_{i,j}=A_{\vartheta(i),\vartheta(j)}$.
If $\vartheta(n-i)=p-\vartheta(i)$, then $\vartheta^!$ preserves dualities.

We consider the following monotonic maps:
$$
\renewcommand\arraystretch{1.2} 
\begin{array}{ll}
\theta:[n-2] \to [n]:i \mapsto i+1, & \rho:[n] \to [n-2]: i \mapsto
\left\{\begin{array}{ll} 0 & i=0\\ i-1 & 1\leq i \leq n-1 \\ n-2 & i=
    n,\end{array} \right. \\
&\\
\iota: \hspace{1ex} [1] \to [n]: \hspace{1ex}i \mapsto i, & 
\pi: \hspace{2ex} [n] \to [1]: \hspace{2ex}
i \mapsto 
\left\{\begin{array}{ll} 0 & i=0\\ 1 & 1\leq i \leq n. \end{array} \right. 
\end{array}
$$
Note that $\theta^!$ and $\rho^!$ preserve dualities whereas, $\iota^!$ and
$\pi^!$  don't.
We will show that the functor
$$F = (\iota^!,\theta^!): (iS_{n}\A)_h \to i\A \times (iS_{n-2}\A)_h:(A,\ffi) \mapsto \iota^!A, \theta^!(A,\ffi)$$
is a homotopy equivalence after group completion with inverse the functor
$$G= \rho^!\perp H\pi^! : i\A \times (iS_{n-2}\A)_h \to (iS_{n}\A)_h:
A,(B,\ffi) \mapsto  \rho^!(B,\ffi) \perp H(\pi^!A),$$ 
where $H: iS_{n}\A \to (iS_{n}\A)_h$  denotes the hyperbolic functor 
$A \mapsto A \oplus A^*, \left(\begin{smallmatrix}0&1\\ \eta &
    0\end{smallmatrix}\right)$. 
One readily verifies $F\circ G = id$.
For the other composition, 
consider the following monotonic maps
$$
\renewcommand\arraystretch{1.2} 
\begin{array}{ll}
\sigma_{01}=\iota\pi:[n] \to [n], &
\sigma_{02}: \hspace{1ex} [n] \to [n]: \hspace{1ex} i \mapsto 
\left\{\begin{array}{ll} i & 0 \leq i\leq n-1 \\ n-1 & i=n, \end{array}
\right. \\ 
&\\
\sigma_{03}=id :[n] \to [n]:i\mapsto i, &
\sigma_{13}: \hspace{1ex} [n] \to [n]: \hspace{1ex} i \mapsto 
\left\{\begin{array}{ll} 1 & i=0 \\ i & 1 \leq i\leq n, \end{array}
\right. \\ 
&\\
\sigma_{12}=\theta\rho:[n] \to [n], &
\sigma_{23}: \hspace{1ex} [n] \to [n]: \hspace{1ex} i \mapsto 
\left\{\begin{array}{ll} n-1 & 0 \leq i\leq n-1 \\ n & i= n, \end{array}
\right. 
\end{array}
$$
and
note that $(i,j) \mapsto \sigma_{i,j}^!$ defines a
duality preserving functor $\sigma^!:S_n\A \to S_3S_n\A$.
By Corollary \ref{cor:prop:S2=K} \ref{cor:prop:S2=K:itemB},
the functors $G\circ F = \sigma_{12}^!\perp H\sigma_{01}^!$ and
$\sigma_{03}^!=id$ induce homotopic maps after
taking group completions.
\end{proof}

\begin{proof}[Proof of Theorem \ref{thm:hermGpCpl}]
This is the same as the proof of \cite[Theorem 4.2]{mygiffen}, or its version
\cite[Corollary 4.6]{mygiffen}.
The only place where Karoubi's Fundamental Theorem was used in \cite{mygiffen}
was in \cite[Lemma 4.4]{mygiffen} which asserts that a certain monoidal functor
$\beta_n:(i\A)^n\times (i\A)_h \to (iS_{2n+1}\A)_h$ 
is a homotopy equivalence after group completion. 
But this functor has a retraction which is a homotopy equivalence after
group completion, by Proposition \ref{prop:AddAddty}.
Thus $\beta_n$ is a homotopy equivalence after group completion.
\end{proof}

\section{Spectra, Tate spectra, and Bispectra}
\label{Appx:TateSp}

In this article we shall work with symmetric spectra of topological spaces 
where 
topological space means compactly generated topological space (that is, weak Hausdorff $k$-space) \cite{McCord}, \cite[Definition 2.4.21 and Proposition 2.4.22]{hovey:book}.
Symmetric spectra were first introduced in \cite{HSS:symSp} in order to construct a symmetric monoidal product on the level of spectra. 
Our main reference for symmetric spectra of topological spaces is \cite{MandellMayetc:diagrSp}, \cite{schwede:book}.
For the convenience of the reader and in order to fix notation we will review 
definitions and results needed in this article.
In Sections \ref{subsec:SpInMonCats} -- \ref{subsec:Bisp}, we consider bispectra used in the construction of the functor $\GW$.

\subsection{Spectra}
\label{Appx:Subsec:Spectra}
We start with fixing notation.
Let $I=[0,1] \subset \mathbb{R}$ be the unit interval, $S^0 = \{0,1\}\subset I$ the zero-sphere (both pointed at $1$), and let $S^1 = I/S^0$ be the unit circle.
Inclusion and quotient map define the pointed maps $i:S^0 \to I$ and $p:I \to S^1$.
Let $\Sigma_n$ be the symmetric group of bijections of the set $\{1,...,n\}$ of $n$-elements.
We consider the $n$-sphere 
$S^n= S^1\wedge ... \wedge S^1$ (smash product $n$-times)
as a pointed space with left $\Sigma_n$-action given by 
$\sigma(x_1\wedge \dots \wedge x_n) = x_{\sigma^{-1}(1)}\wedge ...\wedge x_{\sigma^{-1}(n)}$ for $\sigma \in \Sigma$, and we denote by
$$e_{n,m}:S^n\wedge S^m\to S^{n+m}: (x_1\wedge \dots x_n)\wedge (y_1\wedge \dots y_m) \mapsto x_1\wedge \dots x_n\wedge y_1\wedge \dots y_m$$
the canonical identification.
\vspace{2ex}

A {\em symmetric sequence} (in the category of pointed topological spaces) is
a sequence $X_n$, $n\in \N$, of pointed topological spaces with a continuous base-point preserving left action by the symmetric group $\Sigma_n$.
A map of symmetric sequences $X \to Y$ is a sequence of pointed $\Sigma_n$-equivariant maps $X_n \to Y_n$.

\begin{definition}
A {\em symmetric spectrum} is a symmetric sequence $X_n$, $n\in \N$, in the category of pointed topological spaces
together with pointed continuous $\Sigma_n\times \Sigma_m$-equivariant maps 
$$e_{n,m}: S^n\wedge X_m \to X_{n+m}$$
called {\em bonding maps} (or structure maps) where $\Sigma_n\times \Sigma_m$ acts on the space $X_{n+m}$ via the inclusion $\Sigma_n\times \Sigma_m \subset \Sigma_{n+m}: (\sigma,\tau)\mapsto \sigma\sqcup a\tau a^{-1}$ with $a \{1,...,m\} \to \{n+1,...,n+m\}: i\mapsto n+i$.
The bonding maps have to satisfy the following.
\begin{enumerate}
\item
The map $e_{0,n}:S^0\wedge X_n \to X_n$ is the usual identification $S^0\wedge X_n\cong X_n$.
\item
The following diagram commutes for all $k,n,m\in \N$
$$\xymatrix{
S^k\wedge S^n \wedge X_m \ar[r]^{1\wedge e_{n,m}} \ar[d]_{e_{k,n}\wedge 1} & S^k\wedge X_{n+m} \ar[d]^{e_{k,n+m}}\\
S^{k+n}\wedge X_m \ar[r]_{e_{k+n,m}} & X_{k+n+m}.}$$
\end{enumerate}
A map of symmetric spectra $X \to Y$ is a map of symmetric sequences
commuting with the bonding maps.
All spectra in this article will be symmetric.
For this reason, we will drop the adjective and call them simply spectra.
The category of spectra is denoted by $\Sp$.
\end{definition}

The prime example of a spectrum is the sphere spectrum 
$S=\{S^0,S^1,S^2,\dots\}$ with bonding maps $e_{n,m}$.
More generally, if $K$ is a pointed topological space, we denote by $\Sigma K$, or simply by $K$, the suspension spectrum of $K$ which has $n$-th space the pointed space $S^n\wedge K$ where $\Sigma_n$ acts on $S^n$ as above, and it acts trivially on $K$. The bonding maps are the maps $e_{n,m}\wedge 1: S^n\wedge S^m\wedge K \to S^{n+m}\wedge K$.
A spectrum $X$ is called {\em $\Omega$-spectrum} if  the adjoints $X_n \to \Omega X_{n+1}$ of the bonding maps $e_{1,n}$ are weak equivalences for all $n\geq 0$.
It is called {\em positive $\Omega$-spectrum} if  the adjoints $X_n \to \Omega X_{n+1}$ are weak equivalences for all $n\geq 1$.
The spectrum is called a {\em level CW-complex} if every space $X_n$ is a $CW$-complex.

The category of spectra is tensored and cotensored over the category of pointed spaces.
Let $K$ be a pointed space and $X$ a spectrum.
Their {\em smash-product} is the spectrum $X\wedge K$  with $n$-th space 
$X_n\wedge K$ and bonding maps $e_{n,m}\wedge 1: S^n\wedge X_m \wedge K\to X_{n+m}\wedge K$.
The {\em mapping (or power) spectrum} from $K$ to $X$ is the spectrum 
$\Sp(K,X)$
with $n$-th space the mapping space in the category of pointed topological spaces $\Map(K,X_n)$ and
bonding maps 
$$
\renewcommand\arraystretch{1.5} 
\begin{array}{ccccc}
\begin{array}{ccc}
S^n & \wedge & \Map(K,X_m)\\
t & \wedge & f
\end{array}
&
\hspace{-3ex}
\begin{array}{c}
\longrightarrow \\
\mapsto
\end{array}
&
\hspace{-3ex}
\begin{array}{c}
\Map(K,S^n\wedge X_{m})\\
(x\mapsto t\wedge f(x))
\end{array}
&
\hspace{-3ex}
\begin{array}{c}
\stackrel{e_{n,m}}{\longrightarrow}\\
\phantom{df}
\end{array}
&
\hspace{-3ex}
\begin{array}{c}
\Map(K,X_{n+m})\\
\phantom{df}
\end{array}
\end{array}
$$
Via the topological realization functor $K \mapsto |K|$ from simplicial sets to topological spaces, the category of spectra is also tensored and cotensored over the category of pointed simplicial sets.
We may write $X\wedge K$ and $\Sp(K,X)$ instead of $X\wedge |K|$ and $\Sp(|K|,X)$ when $K$ is a simplicial set and $X$ a spectrum.

\subsection{Model structures on the category of spectra}
Recall \cite[\S 2.4]{hovey:book} that the category of pointed topological spaces is a proper model category with weak equivalences the maps that induce isomorphisms on all homotopy groups (with respect to all choices of base points).
The fibrations are the Serre fibrations, and the cofibrations are the maps that have the left lifting property with respect to Serre fibrations which are weak equivalences.
Every topological space is fibrant, 
relative CW-complexes are cofibrations, and CW-complexes are cofibrant in this model structure.

The category of spectra supports various model structures \cite{MandellMayetc:diagrSp}, \cite{schwede:book}.
Call a map of spectra $X \to Y$ {\em level equivalence} ({\em level fibration}) 
if the maps $X_n \to Y_n$ are weak equivalences (fibrations) of topological spaces for all $n\geq 0$.
In the projective level model structure on $\Sp$, the weak equivalences and fibrations are the level equivalences and the level fibrations,  and the cofibrations are the maps that have the left lifting property with respect to fibrations which are level weak equivalences.
The suspension spectrum functor
$\Sigma:Top_* \to \Sp:K \mapsto \Sigma K$ preserves cofibrations.
In particular, the suspension spectrum
$\Sigma K$ of a pointed CW complex $K$ is cofibrant.
The level model structure is not particularly interesting in itself, but it is used to define the stable equivalences of spectra.

Denote by $[X,Y]'$ the set of maps from $X$ to $Y$ in the homotopy category of the projective level model structure on $\Sp$.
A map of spectra $X \to Y$ is a {\em stable equivalence} if for all $\Omega$-spectra $Z$, the map $[Y,Z]' \to [X,Z]'$ is a bijection.

In the projective stable model structure on the category of spectra,
the weak equivalences are the stable equivalences of spectra,
the cofibrations are the cofibrations of the projective level model structure, and 
the fibrations are the maps that have the right lifting property with respect to the cofibrations that are stable equivalences.
In this stable model structure, a map 
$X \to Y$ of spectra is a fibration if the maps $X_n \to Y_n$ are Serre fibrations for all $n\geq 0$, and the squares 
$$\xymatrix{
X_n \ar[r]^{\hspace{-5ex}e_{m,n}} \ar[d] & \Omega^{n+m}X_{n+m} \ar[d] \\
Y_n \ar[r]^{\hspace{-5ex}e_{m,n}} & \Omega^{n+m}Y_{n+m}}$$
are homotopy cartesian for all $n\geq 0$.
In particular, the fibrant objects in the projective model structure are precisely the $\Omega$-spectra, and 
suspension spectra of pointed CW complexes are cofibrant.

There is a useful variant of the projective model structures, namely, the positive projective model structures.
Call a map of spectra $X \to Y$ {\em positive level equivalence} ({\em positive level fibration}) 
if the maps $X_n \to Y_n$ are weak equivalences (fibrations) of topological spaces for all $n\geq 1$.
In the positive projective level model structure on $\Sp$, the weak equivalences and fibrations are the positive level equivalences and the positive level fibrations,  and the cofibrations are the maps that have the left lifting property with respect to fibrations which are positive level weak equivalences.
In the positive projective stable model structure on the category of spectra,
the weak equivalences are the stable equivalences of spectra,
the cofibrations are the cofibrations of the positive projective level model structure, and 
the fibrations are the maps that have the right lifting property with respect to the cofibrations that are stable equivalences.
The fibrant objects in the positive projective stable model structure are precisely the positive $\Omega$-spectra.
Note that the stable model structure and the positive stable model structure on $\Sp$ both have the same weak equivalences and therefore the same homotopy categories.

A map of (positive) $\Omega$ spectra $X \to Y$ is a stable equivalence if and only if for all $n\geq 0$ ($n\geq 1$), the maps $X_n \to Y_n$ are weak equivalences of topological spaces.
A square of (positive) $\Omega$-spectra 
$$\xymatrix{X \ar[r] \ar[d] & Y \ar[d]\\
Z \ar[r] & U}$$
is homotopy cartesian if and only if the square
$$\xymatrix{X_n \ar[r] \ar[d] & Y_n \ar[d]\\
Z_n \ar[r] & U_n}$$
is a homotopy cartesian square of spaces for all $n\geq 0$ ($n\geq 1$).
We denote by $[X,Y]$ the set of maps from $X$ to $Y$ in the stable homotopy category of spectra.
This set is canonically an abelian group.

\subsection{True and naive homotopy groups}
\label{subsec:trueVsNaiveHtpyGpsInSp}
The {\em true homotopy groups}, or simply {\em homotopy groups} $\pi_nX$, $n\in \Z$, of a spectrum $X$ are the groups
$$\pi_nX=[S^n,X].$$
The {\em naive homotopy groups} of a spectrum $X$ are the groups
$$\hat{\pi}_nX=\colim_k\pi_{n+k}(X_k)$$
where the colimit is taken over the system 
$$\pi_m(X_k) \to \pi_{m+1}(X_{k+1}):f \mapsto e_{1,k}\circ (S^1\wedge f).$$
If this colimit stabilizes, that is, if the maps $\pi_{n+k}(X_k) \to \pi_{n+k+1}(X_{k+1})$ are isomorphisms for large $k$ then the naive homotopy groups are canonically isomorphic to the homotopy groups. 
This happens, for instance, if the spectrum is a (positive) $\Omega$-spectrum, or if it is a suspension spectrum.

\subsection{Products}
The symmetric monoidal product $\wedge$ on pointed topological spaces
defines a symmetric monoidal product $\wedge$ on the 
category of symmetric sequences of pointed topological spaces; see \ref{subsec:SymmSeqProd}.
The maps $e_{m,n}:S^n\wedge S^m \to S^{n+m}$ define a map of symmetric sequences $e:S\wedge S \to S$.
Together with the map $u:\1 \to S$ of symmetric sequences given by $1:S^0 \to S^0$ and $\pt \to S^n$, $n\geq 1$, this 
makes the sphere spectrum $S$ into a commutative monoid.

Now, a symmetric spectrum is the same as a left module over the commutative monoid $S$ for the symmetric product $\wedge$ in the category of symmetric sequences.
It is a map of symmetric sequences $e:S\wedge X \to X$ such that the diagrams
$$\xymatrix{
S\wedge S \wedge X \ar[r]^{e\wedge 1} \ar[d]_{1\wedge e} & S\wedge X \ar[d]^e
& \1 \wedge X \ar[r]^{u\wedge 1} \ar[dr]_{\cong}& S\wedge X \ar[d]^e\\
S\wedge X \ar[r]^e & X & & X}
$$
commute.

The category of spectra $\Sp$ is symmetric monoidal under the smash product $X\wedge_SY$ of spectra defined as the coequalizer in the category of symmetric sequences
$$\xymatrix{
 \text{\raisebox{.7ex}{$X\wedge S\wedge Y$}} \hspace{1ex}
 \ar@<1.6ex>[r]^{\hspace{2ex}1\wedge e} \ar@<.2ex>[r]_{\hspace{2ex} e\circ \tau \wedge 1} 
& \hspace{1ex}  \text{\raisebox{.7ex}{$X\wedge Y$}}  \ar@<.9ex>[r]  
& \hspace{.5ex}\text{\raisebox{.7ex}{$X\wedge_S Y$}}
}
$$
where $\tau:X\wedge S \to S\wedge X$ is the switch isomorphism in the symmetric monoidal category of symmetric sequences.
If $X$ is a cofibrant spectrum for the projective stable model structure then the functors $Y \mapsto X\wedge_S Y$ and $Y \mapsto Y\wedge_S X$ preserve stable equivalences.
Thus, for any two spectra $X$ and $Y$, the derived smash-product $X\stackrel{L}{\wedge_S} Y$ can be computed as either $cX\wedge Y$, $cX\wedge_S cY$, or $X\wedge_S cY$ where $cX\to X$ and $cY \to Y$ are cofibrant approximations of $X$ and $Y$ in the projective stable model structure.

\subsection{The triangulated stable homotopy category}
The stable homotopy category $\SH$ is obtained from the category $\Sp$ of spectra by formally inverting the stable equivalences.
Recall that we denote by $[X,Y]$ the set of maps in $\SH$.
If $X$ is cofibrant, and $Y$ is fibrant then this is the set of homotopy classes of maps of spectra $X \to Y$, where two maps $f,g:X\to Y$ of spectra are homotopic if there is a map $h:X\wedge I_+ \to Y$ such that $f=h_0$ and $g=h_1$.
The stable homotopy category $\SH$ is a triangulated category with shift functor $$X \mapsto S^1\wedge_SX.$$
A triangle in $\SH$ is exact if it is isomorphic in $\SH$ to a standard exact triangle 
$$X \stackrel{f}{\longrightarrow} Y \stackrel{g}{\longrightarrow} C(f) \stackrel{h}{\longrightarrow} S^1\wedge X$$
associated with a map $f:X \to Y$ of spectra.
The maps in the standard exact triangle are
defined by the commutative diagram
$$\xymatrix{
X \ar[r]^{i\wedge 1} \ar[d]_f & I\wedge X \ar[r]^{p\wedge 1} \ar[d] & S^1\wedge X \ar[d]^{=}\\
Y \ar[r]^g & C(f) \ar[r]^h & S^1\wedge X}$$
in which the left square is cocartesian.

\subsection{Homotopy Orbit spectra}
Let $G$ be a group and let $EG$ be the usual contractible simplicial set on which $G$ acts freely on the left.
If $Ob[n]$ denotes the set of objects of the category $[n] = \{0<1<...<n\}$ then\begin{equation}
\label{eqn:EG}
E_nG=G^{n+1}=\Map_{Sets}(Ob[n],G)
\end{equation}
with simplicial structure induced by the cosimplicial set $n \mapsto Ob[n]$.
The left action is given by $gf(x)=g(f(x))$ for $g\in G$, $f:Ob[n]\to G$ and $x\in Ob[n]$.

Fix a functorial cofibrant approximation $cX \stackrel{\sim}{\longrightarrow} X$ on the category of spectra (in the projective stable model structure). 
If $X$ is a spectrum with right $G$-action then so is $cX$, and the stable equivalence $cX \to X$ is $G$-equivariant, by functoriality of the cofibrant approximation.
The {\em homotopy orbit spectrum} $X_{hG}$ of a spectrum $X$ with right $G$-action, also denoted $\bH_{\bullet}(G,X)$, is the spectrum 
$$\hspace{3ex}X_{hG}= \bH_{\bullet}(G,X)=  cX \wedge_GEG_+$$
where for a pointed topological space $K$ with left $G$-action the spectrum
$X\wedge_GK$  is defined by the coequalizer diagram
$$\xymatrix{
\text{\raisebox{.7ex}{${\displaystyle{\bigvee_{g\in G}}} X\wedge K$}} 
\hspace{1ex}
 \ar@<1.6ex>[r]^{\hspace{2ex}1\wedge g} \ar@<.2ex>[r]_{\hspace{2ex}g \wedge 1} 
& \hspace{1ex}  \text{\raisebox{.7ex}{$X\wedge K$}}  \ar@<.9ex>[r]  
& \hspace{.5ex}\text{\raisebox{.7ex}{$X\wedge_G K$}}
}
$$

The non-equivariant map $\pt \to G = E_0G \to EG: \pt \mapsto 1$ and 
the $G$-equivariant $EG \to \pt$ induce a string of spectra
$$cX = cX\wedge S^0 \stackrel{\sim}{\longrightarrow} cX \wedge EG_+ \longrightarrow  cX \wedge_G EG_+ \longrightarrow cX\wedge_G S^0 = cX/{G} \to X/G$$
such that the diagram
$$\xymatrix{cX \ar[r]^{\sim} \ar[d] & X \ar[d]\\
X_{hG} \ar[r] & X_G}$$
commutes.
By the coequalizer definition of homotopy orbits, the two compositions
$$\xymatrix{
\text{\raisebox{.7ex}{$ cX\wedge S^0 $}} \hspace{1ex}
 \ar@<1.6ex>[r]^{\hspace{-2ex}1\wedge g} \ar@<.2ex>[r]_{\hspace{-2ex}g \wedge 1} 
& \hspace{1ex}  \text{\raisebox{.7ex}{$cX\wedge EG_+$}}  \ar@<.9ex>[r]  
& \hspace{.5ex}\text{\raisebox{.7ex}{$cX\wedge_G EG_+$}}
}
$$
are equal for all $g\in G$.
The two maps $1\wedge g, 1 \wedge 1 :cX= cX \wedge S^0 \to cX\wedge EG_+$ are homotopic, by contractibility of $EG$.
It follows that the diagram
$$\xymatrix{
cX \ar[r] \ar[d]_g & X_{hG}\\
cX \ar[ru]&
}$$
is homotopy commutative for all $g\in G$.
In particular, for all $n\in \Z$
the map $\pi_n(X)=\pi_n(cX) \to \pi_n(X_{hG})$ factors through $\pi_n(X)/G \to \pi_n(X_{hG})$. 
In general, this map is not an isomorphism.

\begin{proposition}
\label{prop:htpyOrbitFacts}
\begin{enumerate}
\item
\label{prop:htpyOrbitFacts1}
If $X \to Y$ is a $G$-equivariant map of spectra which, forgetting the $G$-action, is a stable equivalence, then 
$X_{hG} \to Y_{hG}$ is also a stable equivalence.
\item
\label{prop:htpyOrbitFacts2}
If $X$ is a non-equivariant spectrum then $X\wedge G_+$ is a spectrum with right $G$-action and the natural map
$(X\wedge G_+)_{hG} \to (X\wedge G_+)/G=X$ is a stable equivalence.
\item
\label{prop:htpyOrbitFacts3}
If $X$ has a right $G$-action, then there is a homological spectral sequence  
$$E^2_{p,q}=H_p(G,\pi_qX) \Rightarrow \pi_{p+q}(X_{hG})$$
with differentials $d^r:E^r_{p,q} \to E^r_{p-r,q+r-1}$.
The spectral sequence converges strongly if $\pi_iX=0$ for $i<<0$.
\end{enumerate}
\end{proposition}

\begin{proof}
In any simplicial model category, the functor $?\wedge EG_+$ preserves weak equivalences between point-wise cofibrant objects \cite[Theorem 18.5.3.(1)]{Hirschhorn}.
This proves part (\ref{prop:htpyOrbitFacts1}).

For (\ref{prop:htpyOrbitFacts2}), the map
$$(X \wedge G_+) \wedge_G EG_+ = X \wedge EG_+ \to X \wedge S^0 = X$$
is a level equivalence since $|EG|_+ \to S^0$ is a homotopy equivalence.
It follows that this map is a stable equivalence.

For part (\ref{prop:htpyOrbitFacts3}),
the spectral sequence  is the spectral sequence of a filtered spectrum obtained by applying the functor $cX\wedge_G(\phantom{E})_+$ to the skeletal filtration of $EG$.
\end{proof}

\begin{remark}
If $X$ is a spectrum which is a level CW complex, then the following map is a stable equivalence 
$$X_{hG}=cX\wedge_GEG_+ \stackrel{\sim}{\longrightarrow} X\wedge_GEG_+.$$

This is because there is a functor $\Sing_*$ from symmetric spectra of topological spaces to symmetric spectra of simplicial sets which is right-adjoint to a functor $|\phantom{X}|$ from symmetric spectra of simplicial sets to symmetric spectra of topological spaces for which
the unit of adjunction $X \to |\Sing_*X|$ is a level equivalence.
Both functors $\Sing_*$ and $|\phantom{X}|$ preserve stable equivalences.
Therefore, in the commutative diagram
$$\xymatrix{
cX\wedge_G EG_+ \ar[r] \ar[d] & |\Sing_*cX|\wedge_G EG_+ = |\Sing_*cX\wedge_G EG_+| \ar[d] \\
X\wedge_G EG_+ \ar[r] & |\Sing_*X|\wedge_G EG_+ = |\Sing_*X\wedge_G EG_+| 
}$$
the horizontal maps are level weak equivalences because both sides are level cofibrant.
The right vertical map is a stable equivalence because in the injective stable model structure on the category of symmetric spectra of simplicial sets every spectrum is cofibrant.
Therefore, the left vertical map is a stable equivalence.
\end{remark}

\begin{example}
\label{ex:htpyOrbitMcLane}
If $X$ is a spectrum with $\pi_iX=0$ for $i<0$, then the spectral sequence \ref{prop:htpyOrbitFacts} (\ref{prop:htpyOrbitFacts3})
shows that the natural map $\pi_0(X)/G \to \pi_0(X_{hG})=H_0(G,\pi_0X)$ is an isomorphism.
If $X$ is an Eilenberg-MacLane spectrum (that is, $\pi_iX=0$ for $i\neq 0$),
then the spectral sequence collapses and yields an isomorphism for all 
$i \in \Z$
$$\pi_i(X_{hG})=H_i(G,\pi_0X).$$
\end{example}

\begin{example}
If $X$ is a spectrum with right $G$-action, then the canonical map 
$X \to X_{hG}$ is $G$-equivariant where $G$ acts trivially on $X_{hG}$.
More precisely, let $cX\wedge EG_+$ be the right $G$-spectrum where $g\in G$ acts as $g\wedge g^{-1}$ on $cX\wedge EG_+$.
Then we have a $G$-equivariant string of maps
$X \stackrel{\sim}{\leftarrow} cX \leftarrow cX\wedge EG_+ \to cX\wedge_GEG_+$
induced by the equivariant map $S^0 \leftarrow EG_+$ and the equivariant quotient map $X\wedge EG_+ \to X\wedge_GEG_+$.
The maps $X \stackrel{\sim}{\leftarrow} cX \leftarrow cX\wedge EG_+$ is a stable equivalence (forgetting the $G$-action).

We restrict now to the case $G=C_2=\Z/2$ the $2$-element group with 
$\sigma \in C_2$ the unique element of order $2$.
Let $X$ be a spectrum with $C_2=\Z/2$ action from the right.
We have a second $\Z/2$-action on $X$, denoted by $C_2^-$,
where the non-trivial element acts as $-\sigma$ on $X$.
From the discussion above, the map $X \to X_{hC_2^-}$ is $C_2^-$-equivariant where $X_{hC_2^-}$ carries the trivial action.
Therefore, the map $X \to X_{hC_2^-}$ is $C_2$-equivariant where $X$ carries the original $C_2$-action, and the non-trivial element $\sigma \in C_2$ acts on $X_{hC_2^-}$ as $-1$.
In summary, we have a $C_2$-equivariant map of spectra
\begin{equation}
\label{equn:XisX+X-}
X \to X^+\times X^- = X_{hC_2}\times X_{hC_2^-}
\end{equation}
where $X^+ = X_{hC_2}$ has the trivial action, and $\sigma$ acts on $X^-=X_{hC_2^-}$ by $-1$.
\end{example}

\begin{lemma}
\label{lem:2divisiblehorbit}
Let $X$ be a spectrum with $C_2$-action.
Assume that the homotopy groups $\pi_*X$ of $X$ are all uniquely $2$-divisible.
Then the $C_2$-equivariant map (\ref{equn:XisX+X-}) is a stable equivalence.
In particular, as an object of the triangulated stable homotopy category $SH$, the object $X^+$ is the image of the projector $(1+\sigma)/2$ of $X$, and
the homotopy groups of $X^+$ are given by the image of the map 
$(1+\sigma)/2: \pi_*X \to \pi_*X$, that is,
the following natural map is an isomorphism 
$$\pi_*(X)/C_2 \stackrel{\cong}{\longrightarrow} \pi_*(X_{hC_2}).$$
\end{lemma}

\begin{proof}
If $\pi_kX=0$ for $k<<0$, the claim follows from the spectral sequence of Proposition \ref{prop:htpyOrbitFacts} (\ref{prop:htpyOrbitFacts3}) which collapses at the $E^2$-page.
Since any spectrum can be written as a sequential (homotopy) colimit of spectra  $X$ for which $\pi_kX=0$ for $k<<0$, we are done.
\end{proof}

\subsection{Homotopy fixed point spectra}
Consider the simplicial set $EG$ from (\ref{eqn:EG}) equipped with the free right $G$-action defined by $fg(x)=f(x)g$ for $g\in G$, $f:Ob[n] \to G$ and $x\in Ob[n]$.
Fix a fibrant approximation $X \stackrel{\sim}{\to} X_f$ in the category of spectra (for the projective stable model structure).
If $X$ is a spectrum with right $G$-action then so is $X_f$, and the map $X \to X_f$ is $G$-equivariant.
The {\em homotopy fixed point spectrum} $X^{hG}$, also denoted $\bH^{\bullet}(G,X)$, of a spectrum $X$ with right $G$-action is the spectrum
$$X^{hG}=\bH^{\bullet}(G,X)=\Sp^G(EG_+,X_f)$$
where for a pointed space $K$ and a spectrum $Y$ both with right $G$-actions, the spectrum $\Sp^G(K,Y)$
is defined by the equalizer diagram
$$\xymatrix{
\text{\raisebox{.7ex}{$\Sp^G(K,Y)$}} \ar@<.9ex>[r] 
& \text{\raisebox{.7ex}{$\Sp(K,Y)$}}
 \ar@<1.6ex>[r]^{\hspace{-2ex}(g,1)}  \ar@<.2ex>[r]_{\hspace{-2ex}(1,g)} 
& \hspace{1ex}
\text{\raisebox{.7ex}{$\displaystyle{\prod_{g\in G}}{\Sp(K,Y).}$}}
}$$

The $G$-equivariant map $EG \to \pt$ and the non-equivariant map $\pt \to EG: \pt \mapsto 1$ induce a string of maps 
$$X^G \to (X_f)^G=\Sp^G(S^0,X_f) \to \Sp^G(EG_+,X_f)\to \Sp(EG_+,X_f) \to \Sp(S^0,X_f) = X_f$$
such that the diagram 
$$\xymatrix{X^G \ar[r] \ar[d] & X^{hG} \ar[d]\\
X \ar[r] & X_f
}$$
commutes. 
 
According to the equalizer definition of the homotopy fixed point spectrum, the two compositions
$$\xymatrix{
\text{\raisebox{.7ex}{$X^{hG}=\Sp^G(EG_+,X_f)$}} \ar@<.9ex>[r] 
& \text{\raisebox{.7ex}{$\Sp(EG_+,X_f)$}}
 \ar@<1.6ex>[r]^{\hspace{-2ex}(g,1)}  \ar@<.2ex>[r]_{\hspace{-2ex}(1,g)} 
& \hspace{1ex}\text{\raisebox{.7ex}{$\Sp(S^0,X_f)=X_f$}}
}$$
are equal. 
By contractibility of $EG$, the map $(g,1)$ in the diagram is homotopic to $(1,1)$.
Therefore, for all $g\in G$ the diagram
$$\xymatrix{X^{hG} \ar[r] \ar[dr] & X_f \ar[d]^g\\
& X_f}$$
homotopy commutes.
In particular, the map $\pi_n(X^{hG}) \to \pi_n(X_f) = \pi_n(X)$ factors through the fixed points $\pi_n(X)^G$ and we obtain a natural map 
$\pi_n(X^{hG}) \to  \pi_n(X)^G$ for all $n\in \Z$.
This map is not an isomorphism, in general.

\begin{proposition}
\label{prop:htpyFixPtFacts}

\begin{enumerate}
\item
\label{prop:htpyFixPtFacts1}
If $X \to Y$ is a $G$-equivariant map of spectra which, forgetting the $G$-action, is a stable equivalence, then 
$X^{hG} \to Y^{hG}$ is a stable equivalence.
\item
\label{prop:htpyFixPtFacts2}
If $X$ is a non-equivariant spectrum then the spectrum $\Sp(G_+,X)$ has a right $G$-action and the map 
$X=\Sp(G_+,X)^{G} \to \Sp(G,X)^{hG}$ is a stable equivalence.
\item
\label{prop:htpyFixPtFact3}
For a spectrum $X$ with right $G$-action there is a spectral sequence 
$$E_2^{p,q}=H^p(G,\pi_{-q}X) \Rightarrow \pi_{-p-q}(X^{hG})$$
which converges strongly if $\pi_iX=0$ for $i>>0$.
\end{enumerate}
\end{proposition}

\begin{proof}
The proof is dual to the proof of Proposition \ref{prop:htpyOrbitFacts}.
We omit the details.
\end{proof}

\begin{remark}
If $X$ is a positive $\Omega$-spectrum, then the map 
$$\Sp^G(EG_+,X) \to \Sp^G(EG_+,X_f) = X^{hG}$$
is a stable equivalence.
This is because the positive $\Omega$-spectra are the fibrant objects of the positive projective stable model structure on $\Sp$,
and in any simplicial model category, the functor $\Map^G(EG_+,?)$ preserves weak equivalences between point-wise fibrant objects \cite[Theorem 18.5.3.(2)]{Hirschhorn}.
\end{remark}

\begin{example}
\label{ex:htpyFixMcLane}
If $X$ is a spectrum with $\pi_iX=0$ for $i>0$ then the spectral sequence \ref{prop:htpyFixPtFacts} (\ref{prop:htpyFixPtFact3}) shows that $H^0(G,\pi_0X) = \pi_0(X^{hG})$, and the map $\pi_0(X^{hG}) \to (\pi_0X)^G $ is an isomorphism.

Let $X$ be an Eilenberg-MacLane spectrum ($\pi_iX=0$, $i\neq 0$)  with right $G$-action. 
Then the spectral sequence collapses and yields an isomorphism for all $i \in \Z$
$$\pi_{-i}(X^{hG})=H^i(G,\pi_0X).$$
\end{example}

\begin{lemma}
\label{lem:2divisiblehFixedPts}
Let $X$ be a spectrum with right $C_2$-action.
Assume that the homotopy groups $\pi_*X$ of $X$ are all uniquely $2$-divisible.
Then the $C_2$-equivariant stable equivalence (\ref{equn:XisX+X-}) induces 
stable equivalences
$$X^{hC_2} \stackrel{\sim}{\longrightarrow} (X^+)^{hC_2} \stackrel{\sim}{\longrightarrow} X^+$$
In particular, the following natural map is an isomorphism 
$$\pi_*(X^{hC_2}) \stackrel{\cong}{\longrightarrow} \pi_*(X)^{C_2}.$$
\end{lemma}

\begin{proof}
From the decomposition (\ref{equn:XisX+X-}) of Lemma \ref{lem:2divisiblehorbit}, it suffices to prove the claim in case $X$ has trivial action and in case the non-trivial element of $C_2$ acts as $-1$.
In the first case, the map $X^{hC_2} = \Sp(BG_+,X_f) \to X_f$ is a weak equivalence.
This follows from the spectral sequence \ref{prop:htpyFixPtFacts} (\ref{prop:htpyFixPtFact3}) provided $\pi_kX=0$ for $k>>0$.
Since every spectrum is a sequential homotopy limit of spectra $X$ with $\pi_kX=0$ for $k>>0$ (given by the Postnikov tower), the first case follows.
In the second case, the map $X^{hC_2} \to \pt$ is a weak equivalence.
Again, this follows from the spectral sequence \ref{prop:htpyFixPtFacts} (\ref{prop:htpyFixPtFact3}) provided $\pi_kX=0$ for $k>>0$, and in general by passing to homotopy limits.
\end{proof}

\subsection{Tate spectra}
\label{subsec:hypernorm}
Let $G$ be a finite group acting from the right on a spectrum $X$.
The {\em Tate spectrum} $\hat{\bH}(G,X)$ is the homotopy cofibre of the hyper norm map 
$\tilde{N}:X_{hG} \to X^{hG}$; see \cite{dwyerTate}, \cite{weissWilliamsAutomII}, \cite{Greenlees:AxTate}, \cite{jardine:etale}.
In case of a right $G$-module $A$, the Tate spectrum $\hat{\bH}(G,A)$ of the Eilenberg-MacLane spectrum associated with $A$ has $n$-th homotopy group naturally isomorphic to ordinary Tate cohomology $\hat{H}^{-n}(G,A)$ of $G$ with coefficients in $A$.
We review the relevant definitions and facts in case $G=\Z/2$.

Let $G=\Z/2 = \{1,\sigma\}$ where $\sigma \in G$ is the unique element of order $2$, and let $X$ be a spectrum with $G$-action.
We consider $X\times X$ as a spectrum with $G_1\times G_2$-action where $G_1=G_2=G$ and $(\sigma, 1) \in G_1\times G_2$ acts as the switch map $(x,y)\mapsto (y,x)$ and $(1,\sigma)$ acts as $(x,y)\mapsto (\sigma y, \sigma x)$.
This also induces a $G_1\times G_2$-action on $X\vee X \subset X\times X$. 
There is a string of $G_1\times G_2$-equivariant spectra
\begin{equation}
\label{G1G2eqnorm}
\xymatrix{X \ar[r]^{\hspace{-2ex}(1,\sigma)} & X\times X & X\vee X\ar[l]_{\hspace{2ex}\sim} \ar[r]^{\hspace{2ex}1\vee 1} & X}
\end{equation}
where $G_1\times G_2$ acts on the source $X$ via the projection $G_1\times G_2 \to G_1$ onto the first factor and $G_1\times G_2$ acts on the target $X$ of the map via the second projection $G_1\times G_2 \to G_2$.
The arrow in the wrong direction is a stable equivalence.
We write $N:X \to X$ for this string of maps and call it {\em norm map}.
The norm map is an honest map in the homotopy category of spectra and induces
$1+\sigma:\pi_n X \to \pi_nX$ on homotopy groups.

For a spectrum $Y$ with $G_1\times G_2$-action, we can consider $Y$ as a $G_i$-spectrum via the group homomorphisms $G_1\to G_1\times G_2:x \mapsto (x,1)$ and
$G_2\to G_1\times G_2:x \mapsto (1,x)$, and the two actions commute.
In particular, the homotopy fixed point spectrum $Y^{hG_2}$ is a $G_1$-spectrum and thus has a homotopy orbit spectrum $(Y^{hG_2})_{hG_1}$.
To abbreviate we write $FY=(Y^{hG_2})_{hG_1}$ for this spectrum.
Note that the functor $F$ preserves stable equivalences, and comes equipped with  natural maps
$(Y^{G_2})_{hG_1} \to FY \to (Y^{hG_2})_{G_1}$.
Since $G_2$ acts trivially on the source of (\ref{G1G2eqnorm}) and $G_1$ acts trivially on the target of (\ref{G1G2eqnorm}), we obtain the {\em hypernorm map} 
\begin{equation}
\label{eqn:tildeN}
\tilde{N}: X_{hG} = (X^{G_2})_{hG} \to FX \stackrel{FN}{\longrightarrow} FX \to (X^{hG})_{G_1} =  X^{hG}.
\end{equation}
More precisely, the hypernorm map is the string of maps $\tilde{N}$:
$$
X_{hG}=(X^{G_2})_{hG_1} \to FX \to F(X\times X) \stackrel{\sim}{\leftarrow} F(X\vee X) \to FX \to (X^{hG_2})_{G_1} = X^{hG}
$$
where the arrow in the wrong direction is a stable equivalence.

\begin{lemma}
\label{lem:NormHyperNorm}
The two maps $cX \to X_{hG} \stackrel{\tilde{N}}{\to} X^{hG} \to X_f$ and $cX \to X \stackrel{N}{\to} X \to X_f$ 
 are equal in the homotopy category of spectra.
\end{lemma}

\begin{proof}
For clarity of exposition, write $X$ and $Y$ for source and target of (\ref{G1G2eqnorm}) with their respective $G_1\times G_2$ actions.
So, the norm map is a $G_1\times G_2$-equivariant string of maps $X \to Y$ with the arrow pointing in the wrong direction a stable equivalence. 
By functoriality, we have a commutative diagram
$$\xymatrix{
X \ar[r] \ar[d]^N & X_f \ar[d]^N & X^{hG_2} \ar[l] \ar[r] \ar[d]^N & FX \ar[d]^N\\
Y \ar[r] & Y_f & Y^{hG_2} \ar[l] \ar[r] & FY}$$
where the vertical maps are actually strings of maps with the vertical arrows in the wrong direction being stable equivalences.
Every spectrum in the top row receives a natural map from $(cX)^{G_2}$ such that all triangles commute.
Similarly, every spectrum in the bottom row naturally maps to $(Y_f)_{G_1}$ such that all triangles commute.
The resulting composition $(cX)^{G_2} \to FX \to FY \to (Y_f)_{G_1}$ is the first map in the lemma, and the map $(cX)^{G_2} \to X \to Y \to (Y_f)_{G_1}$ is the second map in the lemma.
By the commutativity of the diagram, these two maps are equal.
\end{proof}

\begin{definition}
\label{dfn:TateSp}
The {\it Tate spectrum} 
$$\hat\bH(\Z/2,X)$$
of a spectrum $X$ with $G=\Z/2$ action is the homotopy cofibre of the hypernorm map $\tilde{N}: X_{hG} \to  X^{hG}$.
Since $\tilde{N}$ involves a stable weak equivalence in the wrong direction, we give the following more precise and functorial version on the level of spectra.
The Tate spectrum is defined to be the lower right corner in the diagram
$$
\xymatrix{
X_{hG} \ar@{}[dr]^{\hspace{-6ex} \text{\Large $\ulcorner$}} \ar[r] \xymono[d]& F(X\times X) \xymono[d] & F(X\vee X) \ar@{}[dr]^{\hspace{-6ex} \text{\Large $\ulcorner$}} \ar[l]_{\sim} \ar[dl]^f \xymono[d] \ar[r] & X^{hG} \xymono[d] \\
X_{hG}\wedge I  \ar[r] & P & Z(f) \ar[l]^{\sim} \ar[r] & \hat\bH(\Z/2,X)
}$$
where left and right squares are push-outs, all vertical arrows are cofibrations, the left vertical map is the inclusion $X \to  X\wedge I: x \mapsto x\wedge 0$ of $X$ into its cone ($1$ being the base-point of $I=[0,1]$), and where the diagram $X \to Z(f) \to Y$ is a factorial factorization of a map $f:X \to Y$ into a cofibration followed by a stable equivalence.
\end{definition}

The functor $\hat{\bH}(\Z/2,\phantom{X})$ preserves stable equivalences and it sends sequences of $\Z/2$-spectra which (forgetting the action) are homotopy 
fibre sequences to homotopy fibre sequences.

\begin{example}
If $X$ is an Eilenberg-MacLane spectrum  ($\pi_iX=0$ for $i\neq 0$)
equipped with a $\Z/2$-action, then the long exact sequence associated with the homotopy fibration 
$X_{hG} \to X^{hG} \to \hat{\bH}(G,X)$ 
together with the calculations of the homotopy groups of $X_{hG}$ and $X^{hG}$ in Examples \ref{ex:htpyOrbitMcLane} and \ref{ex:htpyFixMcLane} yield isomorphisms
$$
\pi_i\hat{\bH}(\Z/2,X)\cong\left\{
\renewcommand\arraystretch{1.5} 
\begin{array}{ll}
H_i(\Z/2,\pi_0X)&i\geq 2\\
H^{-i}(\Z/2,\pi_0X)&i\leq -1
\end{array}
\right.
$$
and an exact sequence
$$0 \to \pi_1\hat\bH(G,X) \to \pi_0(X_{hG}) \stackrel{\tilde{N}}{\to} \pi_0(X^{hG}) \to \pi_0\hat\bH(G,X) \to 0.$$
By Examples \ref{ex:htpyOrbitMcLane}, \ref{ex:htpyFixMcLane} and Lemma \ref{lem:NormHyperNorm} the middle map $\pi_0\tilde{N}$ is the usual norm map
$1+\sigma: (\pi_0X)/G \to (\pi_0X)^G$.
These properties characterize Tate cohomology of $G=\Z/2$ with coefficients in the $G$-module $\pi_0X$.
Therefore, we have natural isomorphisms for $i\in \Z$
$$\pi_i\hat\bH(\Z/2,X)\cong \hat{H}^{-i}(\Z/2,\pi_0X).$$
\end{example}

\begin{lemma}
\label{lem:TateTwoinvertedX}
If $X$ is a spectrum with $C_2$-action such that all its homotopy groups $\pi_*X$ are uniquely $2$-divisible, then its Tate spectrum is contractible
$$\hat{\bH}(C_2,X)\simeq \pt.$$
\end{lemma}

\begin{proof}
For the $C_2$-equivariant stable equivalence $X \simeq X^+\times X^-$ from Lemma \ref{lem:2divisiblehorbit}, we have $(X^-)_{hC_2}\simeq (X^-)^{hC_2}\simeq \pt$; see Lemmas \ref{lem:2divisiblehorbit} and \ref{lem:2divisiblehFixedPts}.
This reduces the proof of the lemma to the case when $X =X^+$ has trivial $C_2$-action.
In this case, the maps
$X \to X_{hC_2}$ and $X^{hC_2}\to X$ are weak equivalences; see Lemmas \ref{lem:2divisiblehorbit} and \ref{lem:2divisiblehFixedPts}.
The composition $cX \to X \stackrel{N}{\to} X \to X_f$ is multiplication by $2$ which is an equivalence since $2$ is invertible in the homotopy groups of $X$.
From Lemma \ref{lem:NormHyperNorm}, it follows that the 
hypernorm map $\tilde{N}:X_{hC_2} \to X^{hC_2}$ is a stable equivalence.
In particular, its homotopy cofibre, the Tate spectrum, is contractible. 
\end{proof}

\subsection{Spectra in monoidal model categories}
\label{subsec:SpInMonCats}
Let $(\C,\otimes,\1)$ be a cofibrantly generated closed symmetric monoidal model category for which the domains of the generating cofibrations are cofibrant, and let $K \in \C$ be a cofibrant object.
Recall that the category $\C^{\Sigma}$ of symmetric sequences in $\C$ is symmetric monoidal under the monoidal product of Section \ref{subsec:SymmSeqProd} which we will denote by $\wedge$.
We have evaluation functors $\Ev_n:\C^{\Sigma} \to \C: X \mapsto X_n$ and their left adjoints $G_n:\C \to \C^{\Sigma}:M\mapsto (0,...,0,(\Sigma_n)_+\otimes M,0,0,...)$.
The category $\Sp(\C,K)$ of $K$-spectra in $\C$ is the category of left modules over the
free commutative monoid $S:=\Sym(G_1K)= (\1,K,K^{\otimes 2},K^{\otimes 3},...)$ in $\C^{\Sigma}$ on the symmetric sequence $G_1K = (0,K,0,0,...)$.
The evaluation functors $\Ev_n:\Sp(\C,K) \to \C:X \mapsto X_n$ have left adjoints $F_n:\C \to \Sp(\C,K): M \mapsto \Sym(G_1K) \wedge G_nK$.
By abuse of notation, we will write $M$ for $F_0M$. 
As modules over the commutative monoid $S=\Sym(G_1K)$, the category $\Sp(\C,K)$ is symmetric monoidal with unit $S$ and monoidal product $X\wedge_{S} Y$ defined as the coequalizer of the two multiplication maps $X\wedge S \wedge Y \to X \wedge Y$ coming from $X\wedge S \cong S \wedge X \to X$ and $S \wedge Y\to Y$.

In \cite[Theorem 8.11]{hovey:SymSpGeneral}, Hovey constructs a (projective) stable symmetric monoidal model structure on $\Sp(\C,K)$ for which $K$ is cofibrant and the functor $\Sp(\C,K) \to \Sp(\C,K): X \mapsto K \wedge_{S} X$ is a left Quillen equivalence.
The weak equivalences for this model category are the stable equivalences which are the maps $X \to Y$ of $K$-spectra which induce bijections $[Y,E]' \to [X,E]'$ on morphism sets in the homotopy category of the level model structure on $\Sp(\C,K)$ for every $\Omega$-spectrum $E$.
Here, a $K$-spectrum $X$ is an $\Omega$-spectrum if $X_n$ is fibrant in $\C$ for all $n\in \N$ and the adjoint $X_n \to \Map_{\C}(K,X_{n+1})$ of the structure map $K\otimes X_n \to X_{n+1}$ is a weak equivalence in $\C$.
The functors $F_l:\C \to \Sp(\C,K)$ preserve weak equivalences and cofibrant objects.
For $M$ and $N \in \C$, there are natural isomorphisms $F_{m+n}(M\otimes N) \to F_{m}M\wedge_SF_nN$ (adjoint to the identity map $M\otimes N = \Ev_{m+n}(F_mM\wedge_SF_nN)$) which are associative and unital.
In particular, the functor $F_0:\C \to \Sp(\C,K)$ is strong symmetric monoidal.
If $\C \to \C: X \to K \otimes X$ is a left Quillen equivalence, then $F_0:\C \to \Sp(\C,K)$ is a left Quillen equivalence \cite[Theorem 9.1]{hovey:SymSpGeneral}.

\begin{example}
An example to keep in mind is $(\Top_*,\wedge,S^0)$ with the usual model structure and $K=S^1$.
In this case, the model structure in Section \ref{subsec:SpInMonCats} is the projective stable model structure on the category of spectra of topological spaces.
But the example we are really interested in is $(\Sp,\wedge_S,S^0)$ equipped with the positive projective stable model structure and $K=\tilde{S}^1$ a cofibrant replacement of $S^1\in \Sp$ in that model structure; see below.
\end{example}

\subsection{True and naive homotopy sets}
For two objects $X,Y$ of a model category $\M$, denote by $[X,Y]_{\M}$ the set of maps from $X$ to $Y$ in the homotopy category of $\M$.
Let $n\in Z$ be an integer.
For a $K$-spectrum $X$ define its {\em $n$-th naive $K$-homotopy set} as 
$$\hat{\pi}_n(X) = \colim_k [K^{\otimes k}K^{\otimes n},X_k]_{\C}$$
where the maps in the colimit are given by tensoring with $K$ and composing with the structure map $K\otimes X_k \to X_{1+k}$ of the $K$-spectrum $X$.
The {\em $n$-th true $K$-homotopy set} of a $K$-spectrum $X$ is the set 
$$\pi_n(X) = [K^{\otimes n},X]_{\Sp(\C,K)}.$$
Of course, in case $\C = \Top_*$ and $K=S^1$, these are precisely the definitions given in Section \ref{subsec:trueVsNaiveHtpyGpsInSp}.
There is a natural map
\begin{equation}
\label{eqn:mapPiHatToPi}
\hat{\pi}_n(X) \to \pi_n(X)\end{equation}
defined as follows:
$$
\renewcommand\arraystretch{1.8} 
\begin{array}{ccl}
[K^{\otimes k}K^{\otimes n},X_k]_{\C} 
&\stackrel{F_k}{\longrightarrow}& [F_k(K^{\otimes k}K^{\otimes n}),F_k(X_k)]_{\Sp(\C,K)}\\
& = & [(F_1K)^{\wedge k}K^{\otimes n},F_k\Ev_k(X)]_{\Sp(\C,K)}\\
&\stackrel{s_k}{\longrightarrow}& [(F_1K)^{\wedge k}K^{\otimes n},X)]_{\Sp(\C,K)}\\
&\underset{\cong}{\stackrel{\hspace{2ex}\lambda^{\wedge k}}{\longleftarrow}} & [K^{\otimes n},X)]_{\Sp(\C,K)}
\end{array}
$$
where $s_k:F_k\Ev_k \to 1$ is the counit of adjunction, and $\lambda: F_1K \to S$  is the adjoint of the identity map $K = \Ev_1(S)$.
Note that $\lambda$, and hence $\lambda^{\wedge n}$, is a stable equivalence \cite[Theorem 8.8]{hovey:SymSpGeneral}.

\subsection{Bispectra}
\label{subsec:Bisp}
Now, we specialize to the case $(\C,\otimes,\1)$ the category $(\Sp,\wedge_S,S)$ of spectra equipped with the {\em positive stable model structure} and $K=\tilde{S}^1$ a cofibrant replacement of $S^1\in \Sp$ in that model structure.
The resulting category $\Sp(\Sp,\tilde{S}^1)$ of symmetric spectra in $\Sp$ will be called the {\em category of $\tilde{S}^1$-$S^1$-bispectra}, or simply the {\em category of bispectra} 
denoted by $\BiSp$:
$$\BiSp = \Sp(\Sp,\tilde{S}^1).$$
It is equipped with the projective stable model structure of Hovey \cite[Theorem 8.11]{hovey:SymSpGeneral} induced by the {\em positive stable model structure on the category $\Sp$ of spectra}.
So, a bispectrum in our sense is a left module over the commutative monoid $\tilde{S} = (S^0,\tilde{S}^1, \tilde{S}^2, \tilde{S}^3,...) = \Sym(G_1\tilde{S})$ in the category of symmetric sequences $\Sp^{\Sigma}$ in $\Sp$
where 
$\tilde{S}^n$ denotes $\tilde{S}^1\wedge_S\tilde{S}^1\wedge_S ... \wedge_S\tilde{S}^1$  ($n$ factors).
The symmetric monoidal product in $\BiSp$ is denoted by $\wedge_{\tilde{S}}$.

Since in $\Sp$ smash product with $\tilde{S}^1$ is a Quillen equivalence, the symmetric monoidal inclusion
$F_0:\Sp \to \BiSp$ 
is a Quillen equivalence, too.
In other words the category of bispectra is yet another symmetric monoidal model for the stable homotopy category.

Recall the map $\lambda: F_1(\tilde{S}^1) \to \tilde{S}$  which is adjoint to the identity map $\tilde{S}^1 = \Ev_1(\tilde{S})$.
For a bispectrum $X$, write $\lambda^*:X \to RX = \Map_{\BiSp}(F_1(\tilde{S}^1),X)$ for the adjoint of the map $\lambda\wedge_{\tilde{S}} 1_X: F_1(\tilde{S}^1)\wedge_{\tilde{S}} X \to \tilde{S}\wedge_{\tilde{S}} X = X$.
A level fibrant bispectrum $X$ (\ie, a bispectrum $X=(X_0,X_1,...)$ such that $X_n$ is a positive $\Omega$-spectrum in $\Sp$) is called {\em semistable} if the map $\lambda^*:X \to RX$ is a $\hat{\pi}_*$-isomorphism.

\begin{lemma}
\label{lem:semistable}
Let $X$ be a level fibrant semistable $\tilde{S}^1$-$S^1$ bispectrum.
Then the natural map (\ref{eqn:mapPiHatToPi}) is an isomorphism for all $n\in \Z$:
$$\hat{\pi}_n(X) \stackrel{\cong}{\longrightarrow} \pi_n(X).$$
\end{lemma}

\begin{proof}
As in \cite{HSS:symSp}, denote by $R^n$ the $n$-fold iterate of the functor $R$, by $R^{\infty}X$ the mapping telescope of 
$$X \stackrel{\lambda^*}{\longrightarrow} RX
\stackrel{R\lambda^*}{\longrightarrow} R^2X
\stackrel{R^2\lambda^*}{\longrightarrow} R^3X 
{\longrightarrow} \cdots,$$
and by $i_X:X \to R^{\infty}X$ the map from the initial term into the mapping telescope.
For a level fibrant semistable $\tilde{S}^1$-$S^1$ bispectrum, the map 
$i_X:X \to R^{\infty}X$ is a $\hat{\pi}_*$-isomorphism, and $R^{\infty}X$ is an $\Omega$-spectrum.
For every $\Omega$-spectrum $X$, the map 
$\hat{\pi}_*(X) \to \pi_*(X)$ is an isomorphism. 
Moreover, for any $\hat{\pi}_*$-isomorphism is a stable equivalence.
These claims are proved precisely as in \cite{HSS:symSp} or \cite{schwede:book} with the exception of the last claim which we will prove below.
Therefore, in the commutative diagram
$$\xymatrix{\hat{\pi}_*(X) \ar[r] \ar[d] & \pi_*(X)\ar[d]\\
\hat{\pi}_*(R^{\infty}X) \ar[r] & \pi_*(R^{\infty}X)
}$$
the two vertical maps and the lower horizontal map are isomorphisms.
It follows that the top horizontal map is an isomorphism too.

The proofs that a $\hat{\pi}_*$-isomorphism is a stable equivalence in \cite[Theorem 3.1.11]{HSS:symSp} and \cite[Theorem 4.23]{schwede:book} 
both rely on the injective level model structure on the category of symmetric spectra of simplicial sets in which every object is cofibrant.
A corresponding model structure in the case of bispectra presumably exists but we don't really need it. 
A minor modification of their argument will do.
For an $\Omega$-spectrum $E$, the map $i_E:E \to R^{\infty}E$ is a level equivalence.
In particular, it is an isomorphism in the homotopy category of the level model structure on bispectra.
Denote by $[X,Y]'$ the set of maps from $X$ to $Y$ in that homotopy category.
For a map $f:X \to Y$ of bispectra and $E$ an $\Omega$-bispectrum, we have a commutative diagram
$$\xymatrix{
[Y,E]' \ar[r]^{\hspace{-3ex}R^{\infty}} \ar[d]^{[f,1]'} 
& [R^{\infty}Y,R^{\infty}E]' \ar[d]^{[R^{\infty}f,1]'} \ar[r]^{\hspace{2ex}[1,i_E^{-1}]'}
& [R^{\infty}Y,E]' \ar[d]^{[R^{\infty}f,1]'} \ar[r]^{\hspace{2ex}[i_Y,1]'} & [Y,E]' \ar[d]^{[f,1]'} \\
[X,E]' \ar[r]_{\hspace{-3ex}R^{\infty}} 
& [R^{\infty}X,R^{\infty}E]' \ar[r]_{\hspace{2ex}[1,i_E^{-1}]'}
& [R^{\infty}X,E]' \ar[r]_{\hspace{2ex}[i_X,1]'} & [Y,E]'
}$$
in which the two horizontal compositions are the identity maps.
This shows that the map of sets $[f,1]':[Y,E]' \to [X,E]'$ is a retract of 
the map of sets $[R^{\infty}f,1]':[R^{\infty}Y,E]' \to [R^{\infty}X,E]'$.
If $f$ is a $\hat{\pi}_*$-isomorphism of level fibrant bispectra then $R^{\infty}f$ is a level equivalence.
Hence, the map of sets $[R^{\infty}f,1]'$ is a bijection.
As a retract of a bijection, the map of sets $[f,1]'$ is also a bijection.
By definition, this means that $f$ is a stable equivalence.
\end{proof}

\begin{remark}
\label{rmk:OmegaSpforSemiStableBiSp}
Recall that the inclusion $\Sp \subset \BiSp$ of spectra into bispectra preserves stable equivalences and induces an equivalence of homotopy categories.
In fact, it is a left Quillen equivalence.
In particular, the quasi-inverse $\BiSp$ sends an $X\in \BiSp$ to the zero spectrum $Z_0$ of a fibrant replacement $Z$ of $X \in \BiSp$.
On the subcategory of level fibrant semistable bispectra, this quasi-inverse can be chosen to be $(R^{\infty}X)_0$.
That is, it is the mapping telescope of the diagram of spectra
$$X_0 \stackrel{\lambda^*}{\longrightarrow} \Omega X_1
\stackrel{R\lambda^*}{\longrightarrow} \Omega^2X_2
\stackrel{R^2\lambda^*}{\longrightarrow} \Omega^3X_3 
{\longrightarrow} \cdots$$
\end{remark}

\begin{example}
Let $X = (X_0,X_1,...)$ be a level fibrant bispectrum.
Recall that this means that each $X_n$ is a positive $\Omega$-spectrum in 
$\Sp$.
Assume that the map of spectra $X_n \to \Omega X_{1+n}$ adjoint to the structure maps induces an isomorphism $\pi_i(X_n) \to \pi_i(\Omega X_{1+n})$ whenever $i>0$.
Then $X$ is semistable.
This is because the map $X \to RX$ in degree $n$ is precisely the map 
$X_n \to \Omega X_{1+n}$, and 
$\hat{\pi}_n(X) \to \hat{\pi}_n(RX)$ is the map on colimits over $i$ of the horizontal sequences
$$\xymatrix{
\pi_{i+n}(X_i) \ar[r] \ar[d] & \pi_{1+i+n}(X_{1+i}) \ar[r] \ar[d] & \cdots\\
\pi_{i+n}(\Omega X_{1+i}) \ar[r]  & \pi_{1+i+n}(\Omega X_{2+i}) \ar[r] & \cdots
}$$
By assumption, all the maps in the diagram are isomorphisms when $i+n>0$.
Hence, $\hat{\pi}_n(X) \to \hat{\pi}_n(RX)$ is an isomorphism for all $n\in \Z$.
%Note, however, that the two maps 
%$\pi_{i+n}(X_i) \to \pi_{1+i+n}(X_{1+i})$ and 
%$\pi_{i+n}(X_i) \to \pi_{i+n}(\Omega X_{1+i})$ in the diagram are not the same but differ by an isomorphism.
%But this is irrelevant for the argument.
\end{example}

\subsection{Symmetric sequences \cite{HSS:symSp}}
\label{subsec:SymmSeqProd}
Recall \cite[2.1.1]{HSS:symSp} that a symmetric sequence in a category $\C$ is a functor $\Sigma \to \C$ from the category $\Sigma$ to $\C$ where $\Sigma$ has objects the finite sets $\overline{n}=\{1,2,...,n\}$, $\overline{0}=\emptyset$, for $n\in \N$, and the automorphisms of the sets $\overline{n}$ as its maps.
A morphism of symmetric sequences is a natural transformation of functors $\Sigma \to \C$.
This defines the category $\C^{\Sigma}$ of symmetric sequences in $\C$.

If the category $\C$ has finite coproducts $\bigsqcup$  then so does the category $\C^{\Sigma}$ of symmetric sequences in $\C$.
Coproducts in $\C^{\Sigma}$ are computed object-wise.
Let $(\C,\otimes,\1)$ be a symmetric monoidal category which has finite coproducts.
In particular, it has an initial object $\emptyset = \bigsqcup_{\emptyset}$.
Assume that the monoidal product commutes with finite coproducts, that is,
the natural maps $(A\otimes E) \sqcup (B\otimes E)  \to (A\sqcup B)\otimes E$ and $\emptyset \to \emptyset \otimes E$ are isomorphisms.
Then the category of symmetric sequences $\C^{\Sigma}$ is equipped with a symmetric monoidal product which is best described by replacing $\Sigma$ by a slightly larger but equivalent category \cite[Remark 2.1.5]{HSS:symSp}.

Let $\P$ be the category whose objects are the finite subsets of $\N$ and whose morphisms are the isomorphisms of sets.
The natural inclusion $\Sigma \subset \P$ is an equivalence of categories with inverse $\P \to \Sigma$ given by identifying a finite subset $P$ of $\N$ with the set $\overline{|P|} \in \Sigma$ via the unique order-preserving bijection $P \cong \overline{|P|}$ where $|P|$ denotes the cardinality of $P$.
The categories of functors $\Sigma \to\C$ and $\P \to \C$ are equivalent (via the functor $\Sigma \subset \P$ and its inverse $\P \to \Sigma$).
The tensor product of two functors $X,Y: \P \to \C$ is the functor $X\otimes Y$ which for a finite subset $P$ of $\N$ is
$$(X\otimes Y)(P) = \bigsqcup_{\underset{A\cap B=\emptyset}{A\cup B = P,}}X(A)\otimes Y(B).$$
An isomorphism $f:P\to Q$ of sets defines the map $(X\otimes Y)(f)$ which is the coproduct of the isomorphisms $X(A)\otimes Y(B) \to X(fA)\otimes Y(fB)$ induced by the isomorphisms $f:A \to f(A)$ and $f:B \to f(B)$.
The tensor product is equipped with maps
$$m_{A,B}:X(A)\otimes Y(B) \to (X\otimes Y)(A\sqcup B)$$
functorial in $A, B \in \P$. 
As in \cite[Lemma 2.1.6]{HSS:symSp}, tensor product of symmetric sequences makes the category $\C^{\Sigma}$ into a symmetric monoidal category with unit the symmetric sequence $(\1,\emptyset,\emptyset,\dots)$.

The tensor product of symmetric sequences has the following universal property.
Let $Z$ be a symmetric sequence in $\C$, and let
$$f_{A,B}:X(A)\otimes Y(B) \to Z(A\sqcup B)$$
be a family of morphisms in $\C$ functorial in $A,B\in \P$.
Then there is a unique morphism $f:X\otimes Y \to Z$ of symmetric sequences in $\C$ such that $f_{A,B}=f_{A\sqcup B}\circ m_{A,B}$ for all $A,B\in \P$.

If $\C$ and $\D$ are symmetric monoidal categories with finite coproducts that commute with the monoidal products, and if $F:\C \to \D$ is a symmetric monoidal functor commuting with finite coproducts then the induced functor
$F:\C^{\Sigma} \to \D^{\Sigma}:X \mapsto F\circ X$ 
on symmetric sequences has the same property that is, it is also symmetric monoidal and commutes with finite coproducts.

%\bibliography{gwdg6.bib}

%\end{document}

\end{document}